\documentclass [11pt,oneside,a4paper,mathscr]{amsart}

\usepackage{amscd,amsmath,amssymb,euscript}
\usepackage[frame,cmtip,curve,arrow,matrix,line,graph]{xy}
\usepackage{tikz-cd}
\usepackage{hyperref}
\usepackage{stmaryrd}
\usepackage{amssymb}
\usepackage{stackrel}

\title[Topological K-theory of quasi-BPS categories]{Topological K-theory of quasi-BPS categories of symmetric quivers with potential}
\author{Tudor P\u{a}durariu and Yukinobu Toda}
\newtheorem{thm}{Theorem}[section]
\newtheorem{cor}[thm]{Corollary}
\newtheorem{prop}[thm]{Proposition}

\newtheorem{lemma}[thm]{Lemma}

\theoremstyle{definition}
\newtheorem{defn}[thm]{Definition}

\newtheorem{thm*}[thm]{Theorem$^*$}
\newtheorem{remark}[thm]{Remark}

\newtheorem{example}[thm]{Example}

\newtheorem{assum}{Assumption}[section]
\newcommand{\comment}[1]{}

\renewcommand{\leq}{\leqslant}
\renewcommand{\geq}{\geqslant}

\newcommand{\X}{\mathscr{X}}

\newcommand{\id}{\operatorname{id}}

\newcommand{\Ind}{\operatorname{Ind}}
\newcommand{\Hom}{\operatorname{Hom}}

\newcommand{\Spec}{\operatorname{Spec}}

\newcommand{\ch}{\operatorname{ch}}

\newcommand{\C}{\mathbb{C}}

\newcommand{\dd}{\underline{d}}

\newcommand{\inclusion}{\ar@<-0.3ex>@{^{(}->}[r]}
\newcommand{\linclusion}{\ar@<-0.3ex>@{^{(}->}[rr]}

\newcommand{\Tr}{\mathop{\rm Tr}}

\newcommand{\ssslash}{/\!\!/}

\usetikzlibrary{positioning,fit,calc}
\tikzstyle{block}=[draw=black, width=1cm, minimum height=2cm, align=center] 
\tikzstyle{block2}=[draw=black, text width=2cm, minimum height=1cm, align=center] 
\tikzstyle{block3}=[draw=black, text width=2cm, minimum height=1cm, align=center] 

\makeatletter

\@addtoreset{equation}{section}	
\makeatother
\begin{document}

\begin{abstract}
In previous works, we introduced and studied certain categories called quasi-BPS categories associated to
symmetric quivers with potential, preprojective algebras, 
and local surfaces. 
They have properties reminiscent of BPS 
invariants/ cohomologies in enumerative geometry, for example they play
important roles in categorical wall-crossing 
formulas. 

In this paper, we make the connections between quasi-BPS categories and BPS cohomologies more precise
via the cycle map for topological K-theory. 
We show the existence of filtrations on 
topological K-theory of quasi-BPS categories whose 
associated graded are isomorphic to the 
monodromy invariant BPS cohomologies. 
Along the way, we also compute the topological K-theory of categories of matrix factorizations
in terms of the monodromy invariant vanishing cycles (a version of this comparison was already known by work of Blanc-Robalo-Toën-Vezzosi), prove a Grothendieck-Riemann-Roch theorem for matrix factorizations, and prove the compatibility between the Koszul equivalence in K-theory and dimensional reduction in cohomology.

In a separate paper, we use the results from this paper 
to show that the quasi-BPS categories of K3 surfaces recover the 
BPS invariants of the corresponding local surface, which are Euler characteristics of Hilbert schemes of
points on K3 surfaces.

\end{abstract}

\maketitle

 \setcounter{tocdepth}{2}
\tableofcontents

\section{Introduction}
\subsection{Background}

The BPS invariants are integer-valued invariants, which virtually 
count semistable (compactly supported) coherent 
sheaves on a smooth complex Calabi-Yau (CY) $3$-fold. 
They generalize Donaldson-Thomas (DT) invariants~\cite{Thom} which are only defined when there are no strictly semistable sheaves, 
and they were introduced by Joyce--Song~\cite{JS} and Kontsevich--Soibelman~\cite{K-S}
in the study of wall-crossing formulas for DT invariants. 
BPS invariants are fundamental enumerative invariants which determine many other enumerative invariants of interest for Calabi-Yau $3$-folds, such as Gromov-Witten, Donaldson-Thomas (DT), or Pandharipande-Thomas invariants, 
see the surveys~\cite[Section 2 and a half]{MR3221298} and~\cite[Section~6]{Tsurvey}. 

Let $X$ be a smooth Calabi-Yau $3$-fold, and consider $v\in H^{\ast}(X,\mathbb{Z})$
and a stability condition $\sigma$. 
We denote by $M^\sigma_X(v)$ the good moduli space of $\sigma$-semistable (compactly supported) coherent sheaves on $X$ of Chern character $v$. 
There is an associated BPS invariant, see~\cite[Equation (1.19)]{JS}
\begin{align*}\Omega^\sigma_X(v) \in \mathbb{Z}.
\end{align*}
It was originally defined as a rational number
in~\cite{JS, K-S}, and conjectured to be an integer. 
The above integrality conjecture is proved by Davison--Meinhardt~\cite{DM} for quivers with 
potential, 
by proving that the BPS invariant is the Euler characteristic of some graded vector space, called \textit{BPS cohomology}. 
%We note here that the BPS cohomology is the cohomology the \textit{BPS sheaf}, which is a perverse sheaf. 
In the case of a CY 3-fold, 
the integrality conjecture can be proved by 
combining the result of~\cite{DM} with 
the local description of the moduli stacks 
of semistable sheaves via quivers with potential~\cite{MR3811778}. 
However, natural 
BPS cohomology spaces for CY 3-folds are not yet available, 
and it is an important problem in enumerative algebraic geometry to define them. 
BPS cohomologies may be used to obtain 
a version of the wall-crossing formulas~\cite{JS, K-S} in 
cohomological DT theory~\cite{BDJS}, 
a Poincaré–Birkhoff–Witt type theorem for Hall algebras of CY 3-folds which is a global version of the Davison--Meinhardt theorem for Hall algebras of quivers with potential~\cite{DM}, 
and a construction of GV invariants via 
$sl_2$-actions~\cite{MT, TodGV}. 

We note here that BPS cohomology is expected to be defined as the cohomology of a \textit{BPS sheaf} on $M_X^{\sigma}(v)$, and that BPS cohomology is \textit{not} part of a general cohomology theory (for topological spaces or algebraic varieties). Indeed, it is expected to be defined only for a restricted class of spaces with singularities analogous to $M^\sigma_X(v)$ (for example, for good moduli spaces of $(-1)$-shifted symplectic stacks \cite{MR3090262}), which nonetheless provide a new tool to study moduli spaces of sheaves on a CY 3-fold. 
%Such a cohomology theory will be a new cohomology theory of $M^\sigma_X(v)$. 

Moreover, one could attempt to construct a natural dg-category 
\begin{equation}\label{BPS}
\mathscr{BPS}^\sigma_X(v)
\end{equation}
which recovers a $2$-periodic version of BPS cohomology (via periodic cyclic homology or topological K-theory \cite{Blanc}), and thus also the BPS invariant $\Omega^\sigma_X(v)$. 
The BPS cohomology, the BPS category \eqref{BPS} and the K-theory of \eqref{BPS} (which we call \textit{BPS K-theory}) are alternatives to their classical counterparts for $M^\sigma_X(v)$.
%Such a category will then be a replacement of the derived category of coherent sheaves on $M^\sigma_X(v)$, and its K-theory a replacement of the usual K-theory of $M^\sigma_X(v)$. 
One may hope that the constructed BPS spaces are more tractable than their classical counterparts, and that they will have applications in noncommutative algebraic geometry, 
such as in the construction of non-commutative crepant resolutions~\cite{PTquiver, PTK3}, and in geometric representation theory, for example in the study of quantum groups and their representations~\cite{MR3618057, VarVas, RSYZ2, BuRap}.

Locally, the spaces $M^\sigma_X(v)$ can be described as good moduli spaces of representations of certain Jacobi algebras constructed from a symmetric quiver with potential~\cite{JS, MR3811778}. 
One could then attempt to construct the BPS category \eqref{BPS} in two steps: first, construct a BPS category for symmetric quivers with potential; second, glue these categories and obtain \eqref{BPS}.
We refer to~\cite{halpK32, T, P2} for related arguments 
on the second step, i.e. gluing categories locally 
defined in terms of quivers with potential. 

\subsection{The main result}
The purpose of this paper is to complete the first step
mentioned in the previous subsection 
for symmetric quivers $Q=(I,E)$ with a potential $W$. 
For a dimension vector 
$d=(d^i)_{i\in I}\in\mathbb{N}^I$,
we denote by \[\X(d)=R(d)/G(d)\] the stack of 
representations of $Q$ of dimension $d$, 
and consider its good moduli space
\begin{equation}\label{map:gms}
\X(d) \to X(d).
\end{equation}
The BPS cohomology of $(Q,W)$ is the cohomology of a perverse sheaf (called \textit{BPS sheaf}~\cite{DM}) 
\[\mathcal{BPS}_d\in \mathrm{Perv}(X(d)).\]

In~\cite{P}, 
the first author 
introduced 
natural candidates for the category~\eqref{BPS}
called \textit{quasi-BPS categories} in~\cite{PTzero}:
\begin{equation}\label{quasiBPS}
\mathbb{S}(d; \delta)\subset \mathrm{MF}(\X(d), \mathrm{Tr}\,W).
\end{equation}
Here, 
$\delta$ is a (real) character of $G(d)$. The 
potential $W$ induces a regular function 
$\mathrm{Tr}\,W\colon \X(d)\to\mathbb{C}$
and the category $\mathrm{MF}(\X(d), \mathrm{Tr}\,W)$ is the 
dg-category of matrix factorizations of the function 
$\mathrm{Tr}\,W$, whose objects 
consist of tuples
\begin{align}\label{intro:tuple}
    (\alpha \colon E \rightleftarrows F \colon \beta), \ 
    \alpha \circ \beta=\mathrm{Tr}\,W \cdot \id_F, \ 
    \beta \circ \alpha=\mathrm{Tr}\,W \cdot \id_E,
\end{align}
where $E, F$ are coherent sheaves on $\X(d)$. 
%The categories (\ref{quasiBPS})
%share numerous properties analogous to BPS 
%invariants/cohomologies. 
The construction of the subcategory (\ref{quasiBPS}) 
is based on the 
work by \v{S}penko--Van den Bergh~\cite{SVdB}
on non-commutative 
resolutions of GIT quotients, and contains matrix factorizations (\ref{intro:tuple}) 
such that the weights of $E, F$
are contained in a certain polytope of the weight lattice. 

For any dg-category, Blanc~\cite{Blanc} defined the topological K-theory spectrum. 
In this paper, we study the topological K-theory of the quasi-BPS category 
\begin{align*}
    K^{\rm{top}}(\mathbb{S}(d; \delta)) \in \mathrm{Sp}. 
\end{align*}
 The following is a corollary of the main theorem of this paper, which shows that the topological 
K-theory~\cite{Blanc} 
of the quasi-BPS category (\ref{quasiBPS})
recovers the dimensions of monodromy 
invariant BPS cohomologies:
%(note that the maps we use to state results in the introduction are not canonical, but we reference the precise statements in the body of the paper).

\begin{thm}\emph{(Corollary~\ref{corollarytheorem61})}
\label{t}
Let 
$(Q, W)$ be a symmetric quiver with potential. 
Then, for $i\in \mathbb{Z}$ and each dimension vector $d$, there exists a Weyl-invariant real weight $\delta$ 
such that 
\begin{equation}\label{ineqintro}
\dim_\mathbb{Q} K^{\mathrm{top}}_i(\mathbb{S}(d; \delta))_{\mathbb{Q}}= \sum_{j\in\mathbb{Z}} \dim_\mathbb{Q} H^j(X(d), \mathcal{BPS}_d)^{\mathrm{inv}}.
\end{equation}
Here, the right hand side is the monodromy invariant part of the BPS cohomology. 
% For any $\ell\in \mathbb{Z}$, there is an injective map from the $\ell$th (rational) topological K-theory space (defined by Blanc \cite{Blanc}) to the monodromy invariants BPS cohomology:
% \begin{equation}\label{m}
%     K^{\mathrm{top}}_\ell(\mathbb{S}(d)_v)\hookrightarrow \bigoplus_{j\in\mathbb{Z}}H^j(X(d), \mathcal{BPS}_d)^{\mathrm{inv}}.
% \end{equation}
% Let $Q^\circ$ be a quiver whose number of unoriented edges between any two vertices is even.
%     Assume that $w$ is coprime to $\sum_{i\in I} d^i$. Then the topological K-theory of the category \eqref{quasiBPSpreproj} is isomorphic to the preprojective BPS cohomology.

%     For a K3 surface $S$, assume that $w$ is coprime to $d$. Then the zeroth topological K-theory of the category  \eqref{quasiBPSK3} is isomorphic to the even BPS cohomology of $S$. %(alternatively, of $X=\mathrm{Tot}_SK_S$)  \cite{DHSM}.
\end{thm}

\subsection{Cohomological DT theory}

 We briefly review cohomological DT theory, which 
 was first introduced by Kontsevich--Soibelman~\cite{MR2851153},
 and later developed by Joyce et al.~\cite{BDJS}.
 Under the assumption 
 that 
 $M_X^{\sigma}(v)$ consists of only stable 
 sheaves, the invariant 
 $\Omega^\sigma_X(v)$ equals 
 the DT invariant 
 $\mathrm{DT}^\sigma_X(v)$
 counting stable sheaves~\cite{Thom}. 
Joyce et al.~\cite{BDJS} introduced a perverse sheaf $\varphi_{\mathrm{JS}}$ on $M^\sigma_{X}(v)$ whose Euler characteristic recovers the DT invariant:
\[\sum_{j\in\mathbb{Z}}(-1)^j \dim H^j(M^\sigma_{X}(v), \varphi_{\mathrm{JS}})=\mathrm{DT}^\sigma_X(v).\]
The cohomology $H^{\ast}(M^\sigma_{X}(v), \varphi_{\mathrm{JS}})$ has some advantages over the approach discussed in~\cite[Section 7.3]{Sz}, and is more computable \cite[Section 6.2]{Sz} than the singular cohomology of $M^\sigma_{X}(v)$. 

If there are strictly $\sigma$-semistable sheaves, then $\mathrm{DT}^\sigma_X(v)$ may not be an integer
and is related to the BPS invariant $\Omega_X^{\sigma}(v)$ 
via the multiple cover formula~\cite[(1.19)]{JS}. 
It is thus more natural to search for categorifications of the (integer-valued) BPS invariant $\Omega^\sigma_X(v)$, namely we look for a BPS sheaf on $M_X^{\sigma}(v)$ 
whose hypercohomology (which is the BPS cohomology) recovers 
the BPS invariant as its Euler characteristic. 
%the (derived) stacks $\mathfrak{M}^\sigma_X(v)$ do not have finite dimensional cohomology. Instead, 
%One attempts to define a Hall algebra on the direct sum (after $v$) of the critical cohomologies of $\mathfrak{M}^\sigma_X(v)$, and extract (as the space of generators) a finite dimensional cohomology theory for $M^\sigma_X(v)$ whose Euler characteristic equals $\Omega^\sigma(v)$. 
%Currently there is no definition of BPS cohomology for an arbitrary stability condition on a smooth Calabi-Yau $3$-fold, but 
Davison--Meinhardt \cite{DM} defined the BPS cohomology for all symmetric quivers with potentials, thus for all local models of moduli spaces of sheaves on CY 3-folds~\cite{MR3811778}. 
Davison--Hennecart--Schlegel Mejia \cite{DHSM} defined it for $X=\mathrm{Tot}_SK_S=S\times \mathbb{A}^1_\mathbb{C}$, where $S$ is a Calabi-Yau surface. For a general CY 3-fold, up to the existence of certain orientation data, the BPS sheaf for moduli spaces of 
one-dimensional sheaves is defined in~\cite[Definition~2.11]{TodGV}. 

In the case of 
$X=\mathrm{Tot}_SK_S$
for a CY surface $S$, 
we can also regard, via dimensional reduction, the BPS cohomology in~\cite{DHSM}
as cohomology spaces for good moduli spaces of 
semistable sheaves on $S$. 
Local models of these moduli spaces are given by moduli of representations of preprojective algebras~\cite{Sacca0, DavPurity}. 
For a quiver $Q^{\circ}$,
the moduli stack $\mathscr{P}(d)$ of the representations of its preprojective algebra is a quasi-smooth derived stack 
with a good moduli space \[\mathscr{P}(d)^{\rm{cl}} \to P(d).\] 
The BPS sheaf on its tripled quiver with potential 
gives rise to the perverse sheaf (called \textit{preprojective BPS sheaf})
\begin{align*}
    \mathcal{BPS}^p_d\in \mathrm{Perv}(P(d))
\end{align*}
 whose cohomology is the BPS cohomology of the preprojective algebra of $Q^\circ$.

% BPS cohomology has some remarkable (conjectured) properties, for example the $\chi$-independence phenomenon is expected to hold for the BPS cohomology of $M^\sigma_S(v)$, but not for its singular or intersection cohomology, see Subsection \ref{chiindep} and \cite[Section 0.4]{MaulikShen}. 

\subsection{Categorical DT theory}

We are interested in constructing a category \eqref{BPS} which recovers (and has analogous properties to) the BPS invariants/cohomologies. If there are no strictly $\sigma$-semistable sheaves of Chern character $v$, such a category will recover the DT invariants by taking the Euler characteristic of its periodic cyclic homology, see \cite{T} for a definition for local surfaces and \cite{RHH} for a recent work addressing the general case.

% We are interested in the problem of categorifying BPS (and thus DT) cohomology, that is, finding natural dg-categories 
% \begin{equation}\label{BPScategories}
% \mathscr{BPS}^\sigma_X(v)
% \end{equation}
% which recover the (mod $2$-periodic version of) BPS cohomology via an additive invariant. In this way, we obtain replacements of the difficult to study category $D^b\mathrm{Coh}(M^\sigma_X(v))$ and of its K-theory, or, by dimensional reduction for $X=\mathrm{Tot}_SK_S$, of $D^b\mathrm{Coh}(M^\sigma_S(v))$ and of its K-theory. 
% We expect that such replacements can be used to construct representations of (doubles of) Hall algebras \cite{MR3618057}, \cite{VarVas}.
% Such categorifications also provide examples of noncommutative spaces with rich geometries, see \cite{PTquiver, PTK3}.

% In this paper, we consider (rational) topological K-theory as defined by Blanc \cite{Blanc}, and thus search for BPS categories \eqref{BPScategories} such that there are natural isomorphisms for $i\in\{0,1\}$:
% \begin{equation}\label{categorify}
%     K^{\mathrm{top}}_i(\mathscr{BPS}^\sigma_X(v))\cong \bigoplus_{j\in\mathbb{Z}}H^{i+2j}(M^\sigma_X(v), \mathcal{BPS}).
% \end{equation}
% If there are no strictly $\sigma$-semistable sheaves of support $v$, such a category will recover the DT invariants by taking the Euler characteristic of its topological K-theory, see \cite{T} for a definition for local surfaces and \cite{RHH} for work in progress addressing the general case.

In previous works, we introduced and studied quasi-BPS categories in the following cases:
\begin{itemize}
    \item for symmetric quivers with potential \eqref{quasiBPS} in~\cite{P, PTquiver},
    \item for preprojective algebras~\cite{PTzero, PTquiver}, which we denote by
\begin{equation}\label{quasiBPSpreproj}
    \mathbb{T}(d; \delta)\subset D^b(\mathscr{P}(d)),
\end{equation}
\item for zero-dimensional sheaves on smooth surfaces \cite{P0, PT2}, and
\item for semistable sheaves on K3 surfaces \cite{PTK3}.
\end{itemize}
These categories have analogous properties to those of BPS cohomologies. Indeed, in the above cases, 
there are semiorthogonal decompositions of the categorical Hall algebras~\cite[Theorem 1.1]{P}, \cite{P2, PTquiver, PTK3}, or of Donaldson-Thomas categories~\cite{PTzero, PTquiver} into products of quasi-BPS categories. These semiorthogonal decompositions are analogous to the PBW theorem for cohomological Hall algebras~\cite{DM, DHSM}, or of the DT/BPS wall-crossing of
Meinhardt--Reineke for framed quivers \cite{MeRe}. 
We also proved categorical versions of the Davison support lemma~\cite[Lemma~4.1]{Dav}
about the support of BPS sheaves on tripled quivers~\cite{PT1, PTquiver, PTK3}. However, we observed in \cite{PTzero} that quasi-BPS categories do not categorify BPS cohomology for an arbitrary $\delta$. 
%Quasi-BPS categories are thus natural candidates for being the categories \eqref{BPS}. We observed in \cite{PTzero} that quasi-BPS categories do not categorify BPS cohomology for every $v\in\mathbb{Z}$. In this paper, we compute the topological K-theory of the quasi-BPS categories \eqref{quasiBPS} and \eqref{quasiBPSpreproj} in terms of BPS cohomology, and in particular we find quasi-BPS categories which categorify (the monodromy invariant) BPS cohomology. In \cite{PTK3}, we prove the same results for quasi-BPS categories of K3 surfaces. 

\subsection{Matrix factorizations and vanishing cycles}

When it is defined, the BPS sheaf
is, locally on $M_X^{\sigma}(v)$, isomorphic to the vanishing cycle sheaf 
of the IC sheaf of some good moduli space 
of a smooth quotient stack, see~\eqref{BPSequ} in the case of quivers with potentials.
We thus first study vanishing cycle sheaves
for a regular function 
\[f\colon \X\to \mathbb{C},\]
where $\X=X/G$ is a smooth quotient stack with 
$G$ a reductive group and $X$ a smooth affine variety
with a $G$-action.  

It is well-known that the category of matrix factorizations $\mathrm{MF}(\X, f)$ is a categorification of vanishing cohomology $H^{\ast}(\X, \varphi_f\mathbb{Q}_\X)$, see \cite{Eff, BRTV}. 
Let $\mathrm{T}$ be the monodromy operator on $\varphi_f\mathbb{Q}_{\X}$. 
The \textit{monodromy invariant vanishing cycle sheaf}
$\varphi^{\mathrm{inv}}_f\mathbb{Q}_{\X}$ is defined to 
be the cone of the endomorphism $1-\mathrm{T}$ on $\varphi_f\mathbb{Q}_\X$.
Inspired by \cite{Eff, BRTV, BD}, 
we construct a Chern character map
\begin{equation}\label{chintro}
    \mathrm{ch}\colon K^{\mathrm{top}}_i(\mathrm{MF}(\X,f))\to \prod_{j\in\mathbb{Z}}H^{i+2j}(\X, \varphi^{\mathrm{inv}}_f\mathbb{Q}_\X),
\end{equation}
which is an isomorphism (after tensoring with $\mathbb{Q}$) if $\X$ is a variety, see Subsection~\ref{subsec:ChernMF}. 
%We note that the construction of \eqref{chintro} is fairly elementary: by the Koszul equivalence and dimensional reduction, both sides are isomorphic to relative theories; under this identification, \eqref{chintro} is the Chern character from relative topological K-theory~\cite[Chapter~10]{Huse} to relative singular cohomology.
The Chern character map \eqref{chintro} induces a cycle map on an associated graded of topological K-theory:
\begin{equation}\label{cintro}
    \mathrm{c}\colon \mathrm{gr}_\ell K^{\mathrm{top}}_i(\mathrm{MF}(\X, f))\to H^{2\dim\X-i-2\ell}(\X, \varphi_{f}^{\mathrm{inv}}\mathbb{Q}_{\X}).
\end{equation}
In Section \ref{s4}, we discuss functoriality of \eqref{chintro} and \eqref{cintro}, in particular we prove a Grothendieck-Riemann-Roch theorem, see Theorem \ref{GRRMFtop}.

\subsection{Quasi-BPS categories for symmetric quivers with potential}\label{sub14}

We briefly explain the construction of the quasi-BPS categories \eqref{quasiBPS}, and then give a
more precise statement of the main result. 
Let $(Q, W)$ be a symmetric quiver with potential. 
For a dimension vector $d$
and a real character $\delta$ of $G(d)$, \v{S}penko--Van den Bergh \cite{SVdB} constructed 
a subcategory 
\begin{equation}\label{magic}
\mathbb{M}(d; \delta)\subset D^b(\X(d))
\end{equation}
such that the endomorphism algebra of its generator 
gives a twisted non-commutative crepant resolution of the 
affine variety $X(d)$. 
The generator of $\mathbb{M}(d; \delta)$ is a direct sum of vector bundles 
of highest weight contained in 
some polytope inside the weight lattice of the maximal torus $T(d)$ of $G(d)$, see Subsection~\ref{subsec27} 
for details. 
The quasi-BPS category 
\[\mathbb{S}(d; \delta):=\mathrm{MF}(\mathbb{M}(d; \delta), \mathrm{Tr}\,W)\subset \mathrm{MF}(\X(d), \mathrm{Tr}\,W)\] is defined 
to be the category consisting 
of matrix factorizations $(\alpha \colon E\rightleftarrows F \colon \beta)$
as in (\ref{intro:tuple}), where $E, F$ are direct sums of generating vector bundles of $\mathbb{M}(d; \delta)$.

We define a set $S^d_{\delta}$ of partitions of $d$ from the combinatorics of the polytope used to define \eqref{magic}, see Subsection~\ref{subsec612}.
%The category \eqref{magic} is generated by certain lattice points inside a closed polytope, andthe set $S^d_v$ is determined by the lattice points lying on the boundary of the polytope. 
For each partition $A\in S^d_{\delta}$, there is a corresponding perverse sheaf $\mathcal{BPS}_A$
on $X(d)$, see Subsection~\ref{sub613}. We define 
\begin{equation}\label{def:BPSsheafdw}
\mathcal{BPS}_{d,\delta}:=\bigoplus_{A\in S^d_\delta}\mathcal{BPS}_A[-\ell(A)],
\end{equation}
where $\ell(A)$ is the length of the partition $A$. 
Its monodromy invariant version $\mathcal{BPS}_{d, \delta}^{\rm{inv}}$ is 
also defined in Subsection~\ref{s6}. 
If $A$ is 
the length one partition $\{d\}$, then 
\begin{equation}\label{BPSequ}
\mathcal{BPS}_{A}=\mathcal{BPS}_d:=\begin{cases}
\varphi_{\mathrm{Tr}\,W}\mathrm{IC}_{X(d)}[-1],\text{ if }X(d)^{\circ}\neq \emptyset,\\0, \text{ otherwise}.
\end{cases}
\end{equation}
Here $X(d)^{\circ} \subset X(d)$ is the open subset 
parameterizing simple representations. 
There is thus a monodromy action on the cohomology of $\mathcal{BPS}_{d, \delta}$  induced from the monodromy of vanishing cycles.
The following is the main result in this paper:

\begin{thm}\emph{(Theorem~\ref{thm1})}\label{thmintro1}
Let $(Q, W)$ be a symmetric quiver with potential
and let $d\in \mathbb{N}^I$. 
For any $i, \ell\in\mathbb{Z}$, 
the cycle map (\ref{cintro}) for $(\X, f)=(\X(d), \Tr W)$ induces an isomorphism 
\begin{align*}\label{m2}
   \mathrm{c}\colon \mathrm{gr}_\ell K_i^{\mathrm{top}}\left(\mathbb{S}(d; \delta)\right)_{\mathbb{Q}}\stackrel{\cong}{\to} H^{\dim\X(d)-2\ell-i+1}(X(d), \mathcal{BPS}_{d,\delta}^{\mathrm{inv}}).
\end{align*}
\end{thm}
We note the following corollary of the above theorem. 
\begin{cor}\emph{(Corollary~\ref{corollarytheorem61})}\label{cor:dimequal}
There is an equality 
\begin{align*}
    \dim_{\mathbb{Q}} K_i^{\rm{top}}(\mathbb{S}(d; \delta))_{\mathbb{Q}}=
    \sum_{j\in \mathbb{Z}}\dim_{\mathbb{Q}}H^j(X(d), \mathcal{BPS}_{d, \delta})^{\rm{inv}}. 
    \end{align*}
    \end{cor}
For any $d$, there exists $\delta$ such that 
$S_{\delta}^d=\{d\}$, see Proposition~\ref{prop:delta}. 
Therefore, the above corollary implies Theorem~\ref{t}. 
In the case that $S_{\delta}^d$
contains non-trivial partitions, the topological K-theory 
of quasi-BPS categories may have bigger dimensions 
than monodromy invariant BPS cohomologies, and we need 
to take into account the direct sum of BPS sheaves as in (\ref{def:BPSsheafdw}). 
%\begin{align*}
%\mathrm{gr}_\ell K^{\mathrm{top}}_i(\mathbb{S}(d; \delta))_{\mathbb{Q}}\subset \mathrm{gr}_\ell K^{\mathrm{top}}_i(\mathrm{MF}(\X(d), \mathrm{Tr}\,W))_{\mathbb{Q}}
%\stackrel{c}{\to} H^{2\dim \X(d)-i-2l}(\X, %\varphi_f^{\rm{inv}}\mathbb{Q}_{\X}).
%\end{align*}
%Finally, to show that \eqref{m2} is an isomorphism under the hypothesis above, we use
%cohomological DT/BPS wall-crossing 
%by Meinhardt--Reineke~\cite{MeRe}, and compare its %categorical 
%version 
%proved in~\cite{PTquiver}. 
%A first ingredient in the proof of Theorem \ref{thmintro1} is the explicit computation of the pushforward $\pi_*\mathrm{IC}_{\X(d)}$ as a sum of shifted perverse sheaves, where $\pi\colon \X(d)\to X(d)$ is the good moduli space map, due to Meinhardt--Reineke \cite{MeRe} and Davison--Meinhardt \cite{DM}.

The potential zero case of Theorem \ref{thmintro1} is related to the categorification of intersection cohomology for good moduli spaces of smooth symmetric stacks pursued in \cite{P3}.

\begin{thm}\emph{(Theorem~\ref{thmWzero})}
\label{thm13zeropote}
Let $Q=(I, E)$ be a symmetric quiver
and let $d \in \mathbb{N}^I$.  
If $S_{\delta}^d=\{d\}$, 
  there is an isomorphism of $\mathbb{Q}$-vector spaces for all $\ell\in\mathbb{Z}$:
    \[\mathrm{c}\colon \mathrm{gr}_\ell K^{\mathrm{top}}_0(\mathbb{M}(d; \delta))_{\mathbb{Q}}\xrightarrow{\sim} I\!H^{\dim \X(d)-2\ell-1}(X(d)).\]
    Moreover, we have $K_1^{\rm{top}}(\mathbb{M}(d; \delta))_{\mathbb{Q}}=0$. 
\end{thm}

%In the case that $S_{\delta}^d$ contains non-trivial partitions, the topological K-theory of quasi-BPS categories may have bigger dimensions than monodromy invariant BPS cohomologies, and we need 
%to take account of the direct sum of BPS sheaves as in (\ref{def:BPSsheafdw}). In this case, we prove that the cycle map is an isomorphism under some assumption of the quiver and $\delta=v\tau_d$ for $v \in \mathbb{Z}$, see Subsection~\ref{section:prelim} for the notation. 
%\begin{thm}\emph{(Theorem~\ref{thm2})}\label{thm2.5}
%Suppose that $Q$ has an even number of edges between 
%any two different vertices and an odd number of loops 
%at every vertex. Then for $\delta=v\tau_d$
%with $v \in \mathbb{Z}$, the map (\ref{m2}) is an 
%isomorphism. 
%\end{thm}
%In the above case, we have $S_{\delta}^d=\{d\}$ if 
%$(v, \dd)$ is coprime, where $\dd$ is the 
%total dimension of $d$. However otherwise 
%$S_{\delta}^d$ may contain non-trivial partitions. 
%The above result is proved by combining the injectivity 
%of the cycle map together with the semiorthogonal 
%decompositions into quasi-BPS categories proved in~\cite{PTquiver}. 

\subsection{Main ingredients used in the proof}
There are two main ingredients of the proof of Theorem~\ref{thmintro1}. 
The first one is 
the construction of 
a cycle map from the topological K-theory of the quasi-BPS category to monodromy invariant 
BPS cohomology, see Theorem \ref{thm2}:
 \begin{align}\label{m2.5}
   \mathrm{c}\colon \mathrm{gr}_\ell K_i^{\mathrm{top}}\left(\mathbb{S}(d; \delta)\right) \to H^{\dim\X(d)-2\ell-i+1}(X(d), \mathcal{BPS}_{d,\delta}^{\mathrm{inv}}).
\end{align}   
The idea is to construct coproduct-like maps associated with 
each partition of $d$ on both the topological K-theory of the category of matrix 
factorizations 
and on the cohomology of monodromy invariant vanishing cycles. 
These coproduct-like maps are used to restrict the image of the cycle map of the topological K-theory of quasi-BPS 
categories. Indeed,
one can show that the above image lies in the summand 
of the monodromy invariant vanishing cycles 
corresponding to $S_{\delta}^d$ under the BBDG decomposition theorem for the pushforward of the vanishing cycle sheaf along the map \eqref{map:gms}. 

The map (\ref{m2.5}) is injective by 
construction. Therefore, to show Theorem \ref{thmintro1}, it is enough to argue that 
the dimension of the left hand side in (\ref{m2.5}) is bigger than or equal to the 
dimension of the right hand side. In order to prove the above inequality of dimensions, we use the second crucial ingredient, which is the theory of relative topological K-theory developed 
by Moulinos~\cite{Moulinos}. 
For a complex variety $M$ and $\mathrm{Perf}(M)$-linear dg-category 
$\mathscr{D}$, it associates a sheaf of spectra
$\mathcal{K}_M^{\rm{top}}(\mathscr{D})$ on $M$. Its rationalization 
$\mathcal{K}_M^{\rm{top}}(\mathscr{D})_{\mathbb{Q}}$
is a complex of $\mathbb{Q}$-vector spaces on $M$. 
In Subsection~\ref{subsec:reltop},
we study $\mathcal{K}_{X(d)}^{\rm{top}}(\mathbb{S}(d; \delta))_{\mathbb{Q}}$ and 
show that it contains 
$\mathcal{BPS}^{\rm{inv}}_{d, \delta}[\beta^{\pm 1}][-\dim \X(d)]$ 
as a direct summand. Here, $\beta$ is a formal variable of degree two. 
It implies the desired inequality of the dimensions
mentioned above. 

The combination of these two ingredients also implies 
the following sheaf-theoretic version of Theorem~\ref{thmintro1}, which 
is stronger than Corollary~\ref{cor:dimequal}:
\begin{thm}\emph{(Theorem~\ref{thm:Ksheaf2})}
\label{thm:Ksheaf}
There is an isomorphism of (rational) constructible complexes on $X(d)$:
\begin{align*}
    \mathcal{K}_{X(d)}^{\rm{top}}(\mathbb{S}(d; \delta))_{\mathbb{Q}}
    \cong \mathcal{BPS}^{\rm{inv}}_{d, \delta}[\beta^{\pm 1}][-\dim \X(d)]. 
\end{align*}
    \end{thm}

\subsection{Quasi-BPS categories for preprojective algebras}\label{sub15}
The result of 
Theorem \ref{thmintro1} can also be used, in conjunction with dimensional reduction, to compute the topological K-theory of quasi-BPS categories for preprojective algebras.
For a quiver $Q^\circ$, 
consider its tripled quiver with potential $(Q,W)$. The subcategory 
\eqref{quasiBPSpreproj} is 
Koszul equivalent \cite{I} to the subcategory of graded matrix factorizations with summands in $\mathbb{M}(d; \delta)$:
\[\mathbb{S}^{\mathrm{gr}}(d; \delta):=\mathrm{MF}^{\mathrm{gr}}\left(\mathbb{M}(d; \delta), \mathrm{Tr}\,W\right).\]
We define perverse sheaves $\mathcal{BPS}^p_{d,\delta}$ on $P(d)$ 
as a direct sum of preprojective BPS sheaves 
as in \eqref{def:BPSsheafdw}. There is a cycle map 
\begin{equation}\label{cintro2}
\mathrm{c}\colon \mathrm{gr}_\ell G^{\mathrm{top}}_0(\mathscr{P}(d))\to H^{\mathrm{BM}}_{2\ell}(\mathscr{P}(d))\end{equation} induced from the Chern character map of $\mathscr{P}(d)$.

\begin{thm}\emph{(Corollary~\ref{cor615} and Theorem~\ref{thm1plus})}\label{thmintro2}
Let $Q^\circ$ be a quiver, $d\in \mathbb{N}^I$ and $\ell \in \mathbb{Z}$. Then $K^{\mathrm{top}}_1(\mathbb{T}(d; \delta))_{\mathbb{Q}}=0$ and the cycle map \eqref{cintro2} induces an isomorphism:
\begin{equation}\label{m3}
   \mathrm{c}\colon \mathrm{gr}_\ell K^{\mathrm{top}}_0(\mathbb{T}(d; \delta))_{\mathbb{Q}}\stackrel{\cong}{\to} H^{-2\ell}(P(d), \mathcal{BPS}^p_{d, \delta}).
\end{equation}
\end{thm}
As a corollary, we have the following. 
\begin{cor}\emph{(Corollary~\ref{corollarythm1plus})}
\label{corintro2}
There is an equality 
\begin{align*}
    \dim_{\mathbb{Q}} K_0^{\rm{top}}(\mathbb{T}(d; \delta))_{\mathbb{Q}}
    =\dim_{\mathbb{Q}} H^{\rm{even}}(P(d), \mathcal{BPS}_{d, \delta}^p). 
\end{align*}
\end{cor}

The moduli stacks of semistable sheaves
on Calabi-Yau surfaces are locally described by
those of representations of preprojective algebras of quivers~\cite{Sacca0, DavPurity}.
In a separate paper~\cite{PTK3}, we use the above
local results 
to construct natural categorifications \eqref{BPS} of the BPS invariants for non-compact Calabi-Yau $3$-folds $X=S\times \mathbb{C}$, where $S$ is a K3 surface.

%Similarly to Theorem~\ref{thm2.5}, we also have the following: 
%\begin{thm}\label{thmiintro2.5}
%Let $Q^{\circ}$ be a quiver such that 
%there is an even number of unoriented edges 
%between them. Then for $\delta=v\tau_d$ 
%with $v \in \mathbb{Z}$, the map (\ref{m3})
%is an isomorphism.     
%\end{thm}

%The preprojective algebras which locally model the moduli of semistable sheaves on a K3 surface are of quivers $Q^\circ$ with the property that, for any two different vertices, there is an even number of unoriented edges between them. 
%In Section \ref{subsec67}, we prove a version of Theorem \ref{thmintro2} for \'etale covers of stacks of representations of preprojective algebra of such quivers, which suffices to compute the topological K-theory of quasi-BPS categories of K3 surfaces \cite{PTK3}.

\subsection{Weight-independence}

We revisit the discussion from Subsection \ref{sub14}, but the same observations apply in the setting of Subsection \ref{sub15}. Let $Q=(I,E)$ be a quiver with an even number of edges between any two different vertices and an odd number of loops at every vertex, and let $W$ be a 
potential of $Q$. Note that the quivers which locally model moduli of semistable sheaves on a K3 surface are of this form. 
For $v \in \mathbb{Z}$, we write $\mathbb{S}(d)_v :=\mathbb{S}(d; v\tau_d)$, 
see Subsection~\ref{subsec:character} for the notation. 
Note that there are equivalences, where $k\in\mathbb{Z}$:
\[\mathbb{S}(d)_v\simeq \mathbb{S}(d)_{v+k\dd},\,\, \mathbb{S}(d)_v\simeq \mathbb{S}(d)^{\text{op}}_{-v}\]
given by tensoring with the $k$th power of the determinant line bundle and by taking the derived dual, respectively. There are no other obvious relations between $\mathbb{S}(d)_v$ and $\mathbb{S}(d)_{v'}$ for $v,v'\in\mathbb{Z}$. However, by Corollary~\ref{cor:dimequal}, we obtain:
\begin{cor}\label{corintro}
    Let $v,v'\in\mathbb{Z}$ be such that $\gcd(v,\dd)=\gcd(v',\dd)$. Let $i\in\mathbb{Z}$. Then there is an equality of dimensions:
    \[\dim_\mathbb{Q} K^{\mathrm{top}}_i(\mathbb{S}(d)_v)_{\mathbb{Q}}=\dim_\mathbb{Q} K^{\mathrm{top}}_i(\mathbb{S}(d)_{v'})_{\mathbb{Q}}.\]
\end{cor}

Note that the statement is reminiscent of the $\chi$-independence phenomenon \cite{MaulikShen, KinjoKoseki}, see especially \cite[Corollary 1.5]{KinjoKoseki}. We observed an analogous statement for quasi-BPS categories of K3 surfaces in \cite{PTK3}. We do not know whether a stronger categorical statement, or at the level of algebraic K-theory, should hold for quivers with potential, see \cite[Conjecture 1.4]{PTK3} for a conjecture in the case of K3 surfaces.

It is natural to ask whether one can use a coproduct to define a primitive part \begin{align*}\mathrm{P}K^{\mathrm{top}}_i(\mathbb{S}(d)_v)_{\mathbb{Q}}\subset K^{\mathrm{top}}_i(\mathbb{S}(d)_v)_{\mathbb{Q}}
\end{align*}
of dimension equal to the dimension of the (total) monodromy invariant BPS cohomology, and thus independent of $v\in\mathbb{Z}$. We defined such spaces in the localized equivariant algebraic K-theory for the tripled quiver with potential in \cite{PT1}. We do not pursue this idea further in this paper.

\subsection{Plan of the paper}

We briefly review the content of each section. In Section \ref{sec:replim}, we list the main notation used in the paper, and we mention some constructions and results about categories of constructible sheaves, and about categories of coherent sheaves and matrix factorizations. 

In Section \ref{sec:quivers}, we recall the construction of quasi-BPS categories for symmetric quivers with potential and for preprojective algebras.

In Section \ref{sec:topK:dg}, we discuss some results about the topological K-theory of dg-categories (due to Blanc) and of its relative version (due to Moulinos).

In Section \ref{s3}, we review the Chern character, the cycle map, and the (topological) Grothendieck-Riemann-Roch theorem for quotient stacks. The main tool we use is the approximation of quotient stacks by varieties. 

In Section \ref{s4}, we revisit the known relation between matrix factorizations and vanishing cohomology, and compute the (relative) topological K-theory of matrix factorizations in terms of (monodromy invariant) vanishing cohomology.

In Section \ref{s5}, we compare the Koszul equivalence for dg-categories (and its induced isomorphism in K-theory) with the dimensional reduction theorem in cohomology. In particular, we construct a Chern character map from the topological K-theory of a class of graded matrix factorizations to vanishing cohomology.

Sections \ref{s6} and \ref{sec:proof} are the main parts of the paper, where we prove the main theorems mentioned in the introduction.

In Section \ref{subsection:preproj}, 
we discuss the computations of topological K-theory of quasi-BPS categories of preprojective algebras in terms of BPS cohomology. In Section \ref{subsec67}, we discuss an extension of these results to étale covers of moduli of representations of preprojective algebras. These two sections are the main inputs in computing topological K-theory of quasi-BPS categories of local K3 surfaces.  

In Section \ref{s8}, we discuss some explicit computations of the topological K-theory of quasi-BPS categories. We mention two examples. First, let $g\geq 0$. The coarse space of representations of the $2g+1$ loop quiver is the variety of matrix invariants $X(d)=\mathfrak{gl}(d)^{2g+1}\ssslash G(d)$.
By Theorem \ref{thm13zeropote}, we obtain a combinatorial formula for the dimensions of the intersection cohomology $I\!H^\bullet(X(d))$, which recovers a formula of Reineke \cite{MR2889742}. Second, we compute the topological K-theory for quasi-BPS categories of points in $\mathbb{C}^3$, see Proposition \ref{prop89}. 

In Section \ref{sec:aux}, we prove some technical results used in the previous subsections.

\subsection{Acknowledgements}
We thank Andrei Okounkov, Alex Perry, and \v{S}pela \v{S}penko for useful discussions. 
T.~P.~is grateful to Columbia University in New York and to Max Planck Institute for Mathematics in Bonn for their hospitality and financial support during the writing of this paper.
Y.~T.~is supported by World Premier International Research Center
	Initiative (WPI initiative), MEXT, Japan, and JSPS KAKENHI Grant Numbers JP19H01779, JP24H00180.

\section{Preliminaries}\label{sec:replim}
\subsection{Some notations and conventions}
For a $\mathbb{Z}$-graded vector space $V=\bigoplus_{j\in\mathbb{Z}}V^j$, we set $\widetilde{V}^i:=\prod_{j\in\mathbb{Z}}V^{i+2j}$. 

For a set $S$, let $\# S$ be the cardinal of $S$.

For an abelian group $V$, we 
write $V_{\mathbb{Q}}:=V \otimes_{\mathbb{Z}} \mathbb{Q}$.

A map of abelian groups $V_1 \to V_2$ is
called \textit{an isomorphism (resp. injective) over $\mathbb{Q}$}
if it is an isomorphism (resp. injective) after tensoring with $\otimes_{\mathbb{Z}}\mathbb{Q}$. 

We list the main notation used in the paper in 
Figure~\ref{table:notation}.

\begin{figure}
	\centering
\scalebox{0.7}{
	\begin{tabular}{|l|l|l|}
		\hline
	Notation & Description & Location defined \\\hline
 $\omega_\X$ & dualizing (constructible) sheaf & Subsection \ref{subsectionconshvs}
\\ \hline
$\phi_f, \psi_f$ & vanishing and nearby fibers & Equation \eqref{psiphi}
\\ \hline
$G^{\mathrm{top}}, K^{\mathrm{top}}$ & K-homology, topological K-theory (of a stack) &
Equation \eqref{HLPSiso}
\\ \hline
$l\colon\X^{\mathrm{cl}}\to\X$  & inclusion of the classical stack & Subsection~\ref{subsec:stacks}
\\ \hline
$V_n, U_n, S_n$  & representations (and open and closed subsets) of a group & Subsection \eqref{subsub25}
\\ \hline
$\kappa$  & Koszul equivalence & Equation \eqref{Koszul}
\\ \hline
$j$  & inclusion of a Koszul stack in a smooth stack & Diagram \eqref{diagmapskoszul}
\\ \hline
$\eta$  & projection from the total space of a vector bundle & Diagram \eqref{diagmapskoszul}
\\ \hline
$\X(d)=R(d)/G(d)$ & stack of representations of a quiver $Q$ &
Subsection \ref{section:prelim}
\\ \hline
$\pi_{X,d}\colon \X(d)\to X(d)$ & good moduli space map for $\X(d)$ &
Subsection \ref{section:prelim}
\\ \hline
$T(d)\subset G(d)$ & maximal torus &
Subsection \ref{section:prelim}
\\ \hline
$M(d)$, $M(d)_\mathbb{R}$ & weight spaces & Subsection \ref{section:prelim}
\\ \hline
$\chi$ & (dominant, integral) weight &
Subsection \ref{section:prelim}
\\ \hline
$1_d$ & identity cocharacter &
Subsection \ref{section:prelim}
\\ \hline
$\underline{d}$ & sum of components of a dimension vector of a quiver &
Subsection \ref{section:prelim}
\\ \hline
$\rho$ & half the sum of positive roots &
Subsection \ref{section:prelim}
\\ \hline
$\tau_d$ & Weyl-invariant weight with sum of coefficients $1$ & Subsection \ref{section:prelim}
\\ \hline
$\X^f(d)$, $\X^{\alpha f}(d)$ & stack of representations of the framed quiver $Q^f$, $Q^{\alpha f}$ &
Subsection \ref{subsec:framedquiver}
\\ \hline
$\mathscr{Y}(d)$ & stack of representations of the doubled quiver $Q^{\circ,d}$ &
Subsection \ref{subsec26}
\\ \hline
$\mathscr{P}(d)$ & stack of representations of the preprojective algebra of $Q^{\circ}$ &
Subsection \ref{subsec26}
\\ \hline
$\pi_{P,d}\colon\mathscr{P}(d)^{\mathrm{cl}}\to P(d)$ & good moduli space map of $\mathscr{P}(d)^{\mathrm{cl}}$ &
Subsection \ref{subsec26}
\\ \hline
$(Q,W)$ & tripled quiver with potential & 
Subsection \ref{subsec29}
\\ \hline
$\delta$ & Weyl-invariant real weight & 
Subsection \ref{subsec27}
\\ \hline
$\lambda$ & cocharacter (usually antidominant) associated to a partition & 
Subsection \ref{subsec27}
\\ \hline
$p_\lambda, q_\lambda$ & maps used to define the Hall product & 
Subsection \ref{subsec27}
\\ \hline
$\textbf{W}(d)$ & polytope used to define quasi-BPS categories
& Equation \eqref{defpolytope} 
\\ \hline
$\mathbb{M}(d; \delta), \mathbb{M}(d)_v$  & magic categories for a quiver with potential
& Equation \eqref{defmdw} 
\\ \hline
 $\mathbb{S}(d; \delta), \mathbb{S}(d)_v$ &  quasi-BPS categories for a quiver with potential
& Equation \eqref{def:quasiBPS} 
\\ \hline
$\mathbb{T}(d; \delta), \mathbb{T}(d)_v$ & preprojective quasi-BPS categories 
& Equation \eqref{kappamagic} 
\\ \hline
$\X(d)^{\mathrm{red}}, \mathbb{M}(d;\delta)^{\mathrm{red}}$ etc. & the reduced stack, its magic category etc. &
Subsection \ref{subsec27}
\\ \hline
$l'\colon\mathscr{P}(d)^{\mathrm{red}}\to\mathscr{P}(d)$ & inclusion of the reduced stack &
Subsection \ref{subsec27}
\\ \hline
$K^{\mathrm{top}}$ & topological K-theory (of a dg-category) &
Subsection \ref{subsec:topK:review}
\\ \hline
$\mathcal{K}_M^{\mathrm{top}}$ & relative topological K-theory &
Subsection \ref{subsec:reltopKth}
\\ \hline
$E_\ell$, $\mathrm{gr}_\ell$ & filtration and the associated graded & Equations \ref{def:filtration}, \eqref{cherngraded}
\\ \hline
 $\mathrm{c}$ &  the cycle map &  Equation \eqref{cherngraded}
\\ \hline
$\iota\colon\X_0\to\X$ & inclusion of the derived zero fiber of a regular function & Section \ref{s4}
\\ \hline
$\mathrm{T}$ & monodromy of vanishing cycles &
Subsection \ref{subsection:vanishing}
\\ \hline
$\varphi_f^{\mathrm{inv}}$ & cone of $1-\mathrm{T}$ on vanishing cycles &
Equations \eqref{defC}, \eqref{thirdcolumn}
\\ \hline
$\Theta$ & forget-the-grading map &
Equation \eqref{forgetthepotential}
\\ \hline
$\varepsilon_{\lambda, \delta}$  & integer used to measure magic categories &
Equation \eqref{varepsilondeltalambda}
\\ \hline
$S^d_\delta, S^d_v$ & sets of partitions &
Subsection \ref{subsec612}
\\ \hline
$\mathbf{d}$ & partition of $d$ &
Subsection \ref{subsec612}
\\ \hline
$\mathcal{BPS}_d, \mathcal{BPS}_A, \mathcal{BPS}_{d,\delta}$ & BPS sheaves for a quiver with potential &
Subsection \ref{sub613}
\\ \hline
$\pi_{\alpha f,d}$ & proper map from the variety of stable framed representations
& Subsection \ref{subsection:decompositiontheorem}
\\ \hline
$\mathrm{P}_A, \mathrm{Q}_A$ & constructible complexes & Equation \eqref{defqa}
\\ \hline
$Q^\sharp$ & auxiliary quiver & Subsection \ref{subsec:reduce}
\\ \hline
$a_\lambda$ & addition map at the level of stacks &
Subsection \ref{subsec64}
\\ \hline
$\X(d)'^\lambda$ & auxiliary stack used to define a coproduct-like map &
Subsection \ref{subsec64}
\\ \hline
$c_{\lambda,\delta}$ & width of a magic category &
Equation \eqref{def:width}
\\ \hline
$\Delta_\lambda$ & coproduct-like maps & Equation \eqref{defdeltamap}
\\ \hline
$\mathcal{BPS}^p_d$ & preprojective BPS sheaves &
Subsection \ref{subsec71}
%\\ \hline
%$L$, $\mathscr{L}$ & \'etale covers of $P(d)$, $\mathscr{P}(d)$ &
%Subsection \ref{subsec91}
%\\ \hline
%$F$, $\mathscr{F}$ & \'etale covers of $X(d)$, $\mathscr{X}(d)$ &
%Subsection \ref{subsec91}
%\\ \hline
%$\mathcal{BPS}^L$, $\mathcal{BPS}^L_{d,v}$ & BPS sheaves for $L$ &
%Equation \eqref{def:BPSLsheaf}
%\\ \hline
%$\mathcal{BPS}^F$, $\mathcal{BPS}^F_{d,v}$ & BPS sheaves for $F$ &
%Subsection \ref{subsec672}
\\ \hline
	\end{tabular}
}
	\vspace{.5cm}
	 \caption{Notation introduced in the paper}
	 \label{table:notation}
\end{figure}

\subsection{Stacks and semiorthogonal decompositions}\label{subsec:stacks}
The spaces $\X=X/G$ considered are quasi-smooth (derived) quotient stacks over $\mathbb{C}$, where $G$ is a reductive group acting on $X$. 
Here a morphism of derived stacks $\X \to \mathscr{Y}$ is 
called \textit{quasi-smooth} if its relative cotangent complex 
is perfect of amplitude $[-1, 1]$, and $\X$ is quasi-smooth 
if $\X \to \Spec \mathbb{C}$ is quasi-smooth. 
%All the groups considered $G$ are reductive groups obtained as the complexification of a compact Lie group $M$.
The classical truncation of $\X$ is denoted by 
$\X^{\rm{cl}}=X^{\rm{cl}}/G$, 
and $l \colon \X^{\rm{cl}} \hookrightarrow \X$
is the natural closed immersion. 
We assume that $X^{\rm{cl}}$ is quasi-projective. 
We denote by $\mathbb{L}_\X$ the cotangent complex of 
$\X$. 
When $X$ is affine, we denote by $X\ssslash G$ the quotient derived scheme with dg-ring of regular functions $\mathcal{O}_X^G$. 

We use the terminology of \textit{good moduli spaces} of Artin stacks, see \cite[Example 8.3]{MR3237451}.
In the case of the quotient stack $\X=X/G$, its good moduli space 
is $X\ssslash G$. 

We denote by $D^b(\X)$ the dg-category of bounded derived category of coherent 
sheaves on $\X$, which is a pre-triangulated dg-category. 
In particular, its homotopy category is a triangulated category. 
We will consider semiorthogonal decompositions \begin{equation}\label{SOD}
    D^b(\X)=\langle \mathcal{C}_i \mid i\in I\rangle,
\end{equation}
where $I$ is a totally ordered set. 

Consider a morphism $\pi\colon \X\to S$. We say the semiorthogonal decomposition \eqref{SOD} is \textit{$S$-linear} or 
\textit{$\mathrm{Perf}(S)$-linear} if $\mathcal{C}_i\otimes \pi^*\mathrm{Perf}(S)\subset \mathcal{C}_i$ for all $i\in I$. 

A morphism of classical stack $f\colon \X \to \mathscr{Y}$ is called 
\textit{smoothable lci} if it admits a factorization 
$\X \stackrel{i}{\hookrightarrow} A \stackrel{p}{\to} \mathscr{Y}$
such that $A$ is smooth, $i$ is a regular closed immersion and $p$ is smooth.

% In this paper, we will only consider regular functions
% \[f\colon \X\to\mathbb{C}\] such that $\text{crit}(f^{-1}(0))\subset \X_0:=f^{-1}(0)$.

\subsection{Constructible sheaves}\label{subsectionconshvs}
For $\X=X/G$ a quotient stack, denote by
$D(\mathrm{Sh}_{\mathbb{Q}}(\X))$ the derived 
category of complexes of $\mathbb{Q}$-constructible sheaves on the classical stack (for the analytic topology) $\X^{\rm{cl}}$ and by $D^b(\mathrm{Sh}_{\mathbb{Q}}(\X))$ the subcategory of bounded complexes, see \cite{MR2312554}. 
For $F\in D(\mathrm{Sh}_{\mathbb{Q}}(\X))$, we use the notation $H^\bullet(\X, F)$ for individual cohomology spaces (that is, $\bullet$ is an arbitrary integer) and $H^{\ast}(\X, F)$ for the total cohomology $H^{\ast}(\X, F):=\bigoplus_{i\in\mathbb{Z}} H^i(\X, F)$. 
%Assume $\X$ has all the irreducible components of the same dimension $d$.
Let $\mathbb{D}$ denote the Verdier duality functor on $D^b(\mathrm{Sh}_{\mathbb{Q}}(\X))$, and set $\omega_\X:=\mathbb{D}\mathbb{Q}_{\X^{\rm{cl}}}$. If $\X$ is a smooth stack, equidimensional of dimension $d$, then $\omega_\X=\mathbb{Q}_\X[2d]$.

For $\X=X/G$, we denote by 
\begin{align*}
 H^i(\X):=H^i(\X, \mathbb{Q})=H^i_G(X, \mathbb{Q}), \ 
 H^{\mathrm{BM}}_i(\X)=H^{\mathrm{BM}}_i(\X, \mathbb{Q})=H^{\mathrm{BM}}_{i, G}(X, \mathbb{Q})
 \end{align*}
 the singular cohomology of $\X$ and 
 the Borel-Moore homology of $\X$, 
 respectively (with rational coefficients). 
For $i, j \in \mathbb{Z}$, there
is an intersection product \begin{align}\label{intersect}H^i(\X)\otimes H^{\mathrm{BM}}_j(\X)\to H^{\mathrm{BM}}_{j-i}(\X). 
\end{align}

We denote by $\mathrm{Perv}(\X)\subset D^b(\mathrm{Sh}_{\mathbb{Q}}(\X))$ the abelian category of perverse sheaves on $\X^{\rm{cl}}$, see \cite{MR2480756}. 
The truncation functors with respect to the perverse 
t-structure is denoted by 
\begin{align*}
{}^p\tau^{\leq \bullet}\colon D^b(\mathrm{Sh}_{\mathbb{Q}}(\X))\to D^b(\mathrm{Sh}_{\mathbb{Q}}(\X)), \
{}^p\tau^{\geq \bullet}\colon D^b(\mathrm{Sh}_{\mathbb{Q}}(\X))\to D^b(\mathrm{Sh}_{\mathbb{Q}}(\X)),
\end{align*}
and the cohomology 
functors with respect to the perverse t-structure is denoted by 
\[
{}^p\mathcal{H}^i=({}^p\tau^{\leq i})
({}^p\tau^{\geq i})
\colon D^b(\mathrm{Sh}_{\mathbb{Q}}(\X))\to \mathrm{Perv}(\X) \subset D^b(\mathrm{Sh}_{\mathbb{Q}}(\X)).\] 
For $F\in D^b(\mathrm{Sh}_{\mathbb{Q}}(\X))$, its total cohomology and perverse cohomology are denoted by: 
\[
\mathcal{H}^{\ast}(F):=\bigoplus_{i\in \mathbb{Z}}\mathcal{H}^i(F)[-i], \ 
{}^p\mathcal{H}^{\ast}(F):=\bigoplus_{i\in\mathbb{Z}}{}^p\mathcal{H}^i(F)[-i].\]
We say $F\in D^b(\mathrm{Sh}_{\mathbb{Q}}(\X))$ is \textit{a shifted perverse sheaf in degree $\ell$} if $F[\ell]\in \mathrm{Perv}(\X)$ and \textit{a shifted perverse sheaf} if there exists $\ell\in\mathbb{Z}$ such that $F[\ell]\in \mathrm{Perv}(\X)$. 
% We denote by $D^+_{\mathrm{con}}(\X)$ the category of locally bounded below complexes of constructible sheaves on $\X$: if $\X$ is connected, then $D^+_{\mathrm{con}}(\X)$ is the limit of the diagram of categories $D_n:=D^b_{\mathrm{con}}(\X)$ for all $n\in\mathbb{N}$ and for the functors ${}^p\tau^{\leq n'}\colon D^n\to D^{n'}$; for general $\X$, we have $D^+_{\mathrm{con}}(\X)=\prod_{\X'\in \pi_0(\X)} D^+_{\mathrm{con}}(\X')$.

\subsection{Nearby and vanishing cycles}\label{subsec:nearby}
For $\X$ a smooth quotient stack, let 
\begin{equation}\label{f}
f\colon \X\to\mathbb{C}
\end{equation}
be a regular function. In this paper, we consider regular functions \eqref{f} such that $0 \in \mathbb{C}$ is the only critical value, equivalently that $\mathrm{Crit}(f)\subset \X_0:=f^{-1}(0)$ set 
theoretically.
We consider the vanishing and nearby cycle functors:
\begin{equation}\label{psiphi}
\varphi_f(-), \psi_f(-)\colon D(\mathrm{Sh}_{\mathbb{Q}}(\X))\to D(\mathrm{Sh}_{\mathbb{Q}}(\X_0)).
\end{equation}
We refer to~\cite[Subsection 2.2]{MR3667216}, ~\cite[Proposition 2.13]{DM}
for details on vanishing and nearby cycles
for quotient stacks. 
By abuse of notation, their push-forward along the closed 
immersion $\iota\colon \X_0 \hookrightarrow\X$ 
are also denoted by $\varphi_f(-)$, $\psi_f(-)$. 

From the definition of nearby and vanishing cycle functors, 
there is an exact triangle:
\[\iota_*\iota^*(-)\to \psi_f(-)\to \varphi_f(-)\to \iota_*\iota^*(-)[1].\] 
The functors \eqref{psiphi} restrict to functors \[\varphi_f[-1], \psi_f[-1]\colon \mathrm{Perv}(\X)\to \mathrm{Perv}(\X_0).\] Further, $\varphi_f[-1]$ and $\psi_f[-1]$ commute with 
the dualizing functor $\mathbb{D}$.
They are also equipped with monodromy operators 
\begin{align*}
    \mathrm{T} \colon \psi_f(-) \to \psi_f(-), \ 
    \mathrm{T} \colon \varphi_f(-) \to \varphi_f(-). 
\end{align*}
We will abuse notation and let 
\begin{align*}\varphi_f:=\varphi_f\mathbb{Q}_\X, \ 
\psi_f:=\psi_f\mathbb{Q}_\X, \ \varphi_f\mathrm{IC}:=\varphi_f\mathrm{IC}_\X, \ \psi_f\mathrm{IC}:=\psi_f\mathrm{IC}_\X.
\end{align*}
We may drop $f$ from the notation if there is no danger of confusion.

\subsection{Matrix factorizations}\label{subsection:matrixfactorizations}
 We refer to \cite[Subsection 2.6]{PTzero} for definitions and references related to categories of matrix factorizations. 
 
Let $\X=X/G$ be a quotient stack with $X$ affine smooth and $G$ reductive. For a regular function $f\colon\X\to\mathbb{C}$, we denote the corresponding dg-category of matrix factorizations by 
\[\mathrm{MF}(\X, f).\] 
Its objects are tuples $\left(\alpha\colon E\rightleftarrows F\colon\beta\right)$ such that $E, F\in \mathrm{Coh}(\X)$ and
\begin{align*}
\alpha \circ \beta=f \cdot \id_F, \ \beta \circ \alpha=f \cdot \id_E.
\end{align*}
If $\mathbb{M}\subset D^b(\X)$ is a pre-triangulated dg-subcategory, let 
\[\mathrm{MF}(\mathbb{M}, f)\subset \mathrm{MF}(\X, f)\]
be the subcategory consisting of totalizations of tuples $(E, F, \alpha, \beta)$ with $E, F \in \mathbb{M}$. The category $\mathrm{MF}(\mathbb{M}, f)$ has a description in terms of categories of singularities \cite[Subsection 2.6]{PTzero}. In this paper, we consider categories $\mathbb{M}$ generated by a collection $\mathscr{C}$ of vector bundles, then 
$\mathrm{MF}(\mathbb{M}, f)$ is the category of matrix factorizations with summands $E, F$ which are direct sums of vector bundles in $\mathscr{C}$, see \cite[Lemma 2.3]{PTzero}.

Suppose that there exists an extra action of $\mathbb{C}^*$ on $X$ which commutes with the action of $G$ on $X$.
Assume that $f$ is of weight one with respect to 
the above $\mathbb{C}^{\ast}$-action. 
We also consider the dg-category of graded matrix factorizations $\mathrm{MF}^{\mathrm{gr}}(\mathscr{X}, f)$. 
The objects of $\mathrm{MF}^{\mathrm{gr}}(\mathscr{X}, f)$ can be described as tuples 
\begin{align}\label{tuplet:graded}
(E, F, \alpha \colon 
E\to F(1), \beta \colon F\to E),
\end{align}
where $E$ and $F$ are $G\times\mathbb{C}^{\ast}$-equivariant coherent sheaves
on $X$, $(1)$ is the twist by the character 
$G \times \mathbb{C}^{\ast} \to \mathbb{C}^{\ast}$, 
and $\alpha$ and $\beta$ are $\mathbb{C}^{\ast}$-equivariant morphisms
such that $\alpha\circ\beta$ and $\beta\circ\alpha$ are multiplication by $f$.
% Assume 
% there is an action of $\mathbb{C}^*$ on $\X$ for which $f$ is of weight $d$, that is, if $z\cdot x$ denotes the action of $z\in \mathbb{C}^*$ on $x\in X(\mathbb{C})$, we have that $f(z\cdot x)=z^df(x)$.
% In this case, consider the category of graded matrix factorizations
% \[\mathrm{MF}^{\mathrm{gr}}(\X, f).\]
% Its objects are tuples $(A, B, \alpha, \beta)$ such that $A, B\in \mathrm{Coh}_{\mathbb{C}^*}(\X)$, $\alpha\colon A\to B$, $\beta\colon B\to A$ are homogeneous morphisms of degrees $d$ and $0$, respectively, and $\alpha\circ\beta$ and $\beta\circ\alpha$ are multiplication by $f$.
For a pre-triangulated dg-subcategory $\mathbb{M}\subset D^b_{\mathbb{C}^*}(\X)$, we define 
\[\mathrm{MF}^{\mathrm{gr}}(\mathbb{M}, f)\subset \mathrm{MF}^{\mathrm{gr}}(\X, f)\] 
the subcategory of totalizations of tuples $(E, F, \alpha, \beta)$ with $E, F \in \mathbb{M}$. 
Alternatively, if $\mathbb{M}$ is generated by a collection of $\mathbb{C}^*$-equivariant vector bundles $\mathscr{C}$, then 
$\mathrm{MF}^{\mathrm{gr}}(\mathbb{M}, f)$ is the category of matrix factorizations with summands $E, F$ which are direct sums of vector bundles in $\mathscr{C}$.

In this paper, we will consider either ungraded or 
graded categories of matrix factorizations. 
When considering the tensor product $\mathcal{D}_1 \otimes \mathcal{D}_2$ of two categories of matrix factorizations, 
we define it in the context of the Thom-Sebastiani theorem:
the tensor product is over $\mathbb{C}(\!(\beta)\!)$ for $\beta$ of homological degree $-2$ in the ungraded case, see \cite[Theorem 4.1.3]{Preygel}, and is over $\mathbb{C}$ in the graded case, see \cite[Corollary 5.18]{MR3270588}. 

\subsection{The categories of singularities}
Let $\mathscr{X}$ be a smooth quotient stack of dimension $d$, and $f\colon \mathscr{X}\to\mathbb{C}$ a regular function with $0 \in \mathbb{C}$ the only singular value. Let $\iota\colon \X_0\hookrightarrow \X$ be the (derived) zero locus of $f$. 
%and by $H^{\cdot}(\mathscr{X}_0, \varphi_f)$ the cohomology of vanishing cycles $\varphi_f=\varphi_f\mathbb{Q}_\X$, also referred to as vanishing or critical cohomology. 
%The setting above suffices to study the Hall algebras of tripled quivers with potential by \cite{DavPurity}.
% Note that both $\X$ and $\X_0$ contract onto $\X^{\mathbb{C}^*}$, and thus also \begin{equation}\label{zerocohx0}
%     H^{\mathrm{odd}}(\X_0)=0.
% \end{equation}
The \textit{category of singularities} is defined by 
\begin{equation}\label{sg}
D_{\mathrm{sg}}(\mathscr{X}_0):=D^b(\mathscr{X}_0)/\text{Perf}(\mathscr{X}_0). 
\end{equation}
It is related to the category of matrix factorizations 
as follows: 
\begin{thm}\emph{(\cite{o2, EfPo})}\label{thm:orlov}
There is an equivalence
\begin{equation}\label{dsgmf}
\mathrm{MF}(\X, f)\stackrel{\sim}{\to} D_{\mathrm{sg}}(\mathscr{X}_0).
\end{equation}
\end{thm}

\subsection{The Koszul equivalence}\label{subsectionKoszul}
In this subsection, we recall the Koszul equivalences for quasi-smooth 
quotient stacks. 
Let $\mathscr{X}$ be a smooth quotient stack
and let $\eta\colon \mathscr{E}\to \mathscr{X}$ be a rank $r$ vector bundle with a section $s\in \Gamma(\X, \mathscr{E})$. 
Let  
\begin{equation}\label{def:K}
j\colon\mathscr{P}:=s^{-1}(0)\hookrightarrow \X
\end{equation}
be the derived zero locus of $s$, and define 
the regular function 
\begin{equation}\label{deffunctionf}
f\colon \mathscr{E}^{\vee}\to \mathbb{C}
\end{equation}
to be $f(x, v_x)=\langle s(x), v_x\rangle$ for $x\in \X$ and $v_x\in \mathscr{E}^\vee|_x$.
Let $\mathscr{E}_0^{\vee}$ be the derived zero locus of $f$. Consider 
the following diagram 
\begin{equation}\label{diagmapskoszul}
    \begin{tikzcd}
        \mathscr{E}^{\vee}|_{\mathscr{P}} \arrow[d, "\eta'"] \arrow[rr, bend left, hook, "j'"] \arrow[r, hook, "\tau"]  & \mathscr{E}_0^{\vee} \arrow[r, hook, "\iota"] & \mathscr{E}^{\vee} 
   \arrow[r, "f"] \arrow[d, "\eta"] & \mathbb{C} \\
    \mathscr{P} \arrow[rr, hook, "j"] & & \X.
    \end{tikzcd}
\end{equation}

% \begin{align*}
%     \xymatrix{
%     \mathscr{E}^{\vee}|_{\mathscr{P}} \ar[d]_-{p'} \inclusion^-{\tau} & \mathscr{E}_0^{\vee} \inclusion^-{\iota} & \mathscr{E}^{\vee} 
%     \ar[r]^-{f} \ar[d]_-{p} & \mathbb{C} \\
%     \mathscr{P} \linclusion^-{i} & & \X
%     }
% \end{align*}
% and let $i':=\iota\circ\tau.$
Let $\mathbb{C}^*$ act with weight one on the fibers of $\mathscr{E}^{\vee}$ and consider the corresponding graded category of matrix factorizations $\mathrm{MF}^{\mathrm{gr}}(\mathscr{E}^{\vee}, f)$.
The Koszul equivalence is stated 
as follows: 
\begin{thm}\emph{(\cite{I, Hirano, T})}\label{thm:Koszul}
There is an equivalence 
\begin{align}\label{Koszul}
\kappa \colon D^b(\mathscr{P})\xrightarrow{\sim} \mathrm{MF}^{\mathrm{gr}}(\mathscr{E}^{\vee}, f).
\end{align}
\end{thm}
The functor $\kappa$ 
is given by a tensor product with 
the Koszul factorization, see~\cite[Section~2.3.2]{T}. 
We also have the following explicit description on complexes from the classical stack: 
%We consider the maps:\[\mathscr{P}\xleftarrow{p}\mathscr{E}^\vee|_{\mathscr{P}}\xrightarrow{i}\mathscr{X}.\]
%Assume $\mathscr{E}$ has rank $r$. 
%When \eqref{def:K} is classical, the equivalence \eqref{Koszul} is given by 
\begin{prop}\label{prop511}
    The composition 
    \[D^b(\mathscr{P}^{\mathrm{cl}})\stackrel{l_{\ast}}{\to} D^b(\mathscr{P})\xrightarrow{\kappa} \mathrm{MF}^{\mathrm{gr}}(\mathscr{E}^{\vee}, f)\] is isomorphic to the functor $j'_*\eta'^*l_*$. 
    \end{prop}
\begin{proof}
In the case that $s$ is a regular section, 
the closed immersion $l \colon \mathscr{P}^{\rm{cl}} \hookrightarrow \mathscr{P}$ is an equivalence 
and $\kappa=j_{\ast}'\eta^{'\ast}$ 
coincides with the equivalence constructed in~\cite{Hirano}, 
see~\cite[Remark~2.3.5]{T}. 
Therefore the proposition is obvious. 
In general, the proposition follows from 
the explicit formula 
for $\kappa$ via the Koszul factorization
in \cite[Section 2.3.2]{T}.
\end{proof}

\section{Review of quasi-BPS categories for symmetric quivers}\label{sec:quivers}
In this section, we review the definition and properties of 
quasi-BPS categories for symmetric quivers. 
The references for this section are~\cite{P0, P, PTzero, PTquiver}. 

\subsection{Quivers}\label{section:prelim}
Let $Q=(I,E)$ be a quiver, where $I$ is the set of 
vertices and $E$ is the set of edges. 
For a dimension vector $d=(d^i)_{i\in I}\in\mathbb{N}^I$, 
we denote by  
\[\X(d)=R(d)/G(d)\]
the stack of representations of $Q$ of dimension $d$, alternatively the stack of representations of dimension $d$ of the path algebra $\mathbb{C}[Q]$.
Here $R(d)$, $G(d)$ are given by 
\begin{align*}
    R(d)=\bigoplus_{(i\to j) \in E}\Hom(V^i, V^j), \ 
    G(d)=\prod_{i\in I}GL(V^i), 
\end{align*}
where $V^i$ is a $\mathbb{C}$-vector space of dimension $d^i$. 
Consider the good moduli space
\[\pi_d \colon \X(d)\to X(d):=R(d)\ssslash G(d).\]
The quotient variety $X(d)$ parametrizes semisimple 
$Q$-representations. We denote by $X(d)^{\circ} \subset X(d)$
the open subset corresponding to simple $Q$-representations. 

A \textit{potential} of $Q$ is an element 
$W \in \mathbb{C}[Q]/[\mathbb{C}[Q], \mathbb{C}[Q]]$. 
It is represented by a linear combination of 
cycles in $Q$. 
A pair $(Q, W)$ is called a \textit{quiver with potential}. 
A potential $W$ determines a regular function 
\begin{align}\label{map:trW}
    \Tr W \colon \X(d) \to \mathbb{C}
\end{align}
given by the trace of the compositions 
of maps along with cycles in $W$.

\subsection{Character and cocharacter lattices}\label{subsec:character}
Let $Q$ be a quiver and $d$ a dimension vector. 
We denote by $T(d)$ a maximal torus of $G(d)$, by $M(d)$ the weight lattice of $T(d)$, 
% and by $M(d)_0\subset M(d)$ the hyperplane of weights with sum of coefficients equal to zero. 
and by $\mathfrak{g}(d)$ the Lie algebra of $G(d)$. 
Let $M(d)_{\mathbb{R}}=M(d)\otimes_{\mathbb{Z}}\mathbb{R}$. Note that we have 
\begin{align*}
M(d)_{\mathbb{R}}=\bigoplus_{i\in I} \bigoplus_{1\leq a\leq d^i}
\mathbb{R} \beta_a^i
\end{align*}
where $\beta_a^i$ for $1\leq a \leq d^i$ 
are the weights of the standard representation of 
$T(d^i)$.
We denote by $\mathfrak{S}_a$ the permutation group on $a$ letters (where $a\in \mathbb{N}$), 
and by $W_d:=\times_{i\in I}\mathfrak{S}_{d^i}$ the Weyl 
group of $G(d)$.

% and $M(d)_{0,\mathbb{R}}=M(d)_0\otimes_{\mathbb{Z}}\mathbb{R}$. 
We denote by $M(d)_\mathbb{R}^+\subset M(d)_\mathbb{R}$
the dominant chamber consisting of weights 
\begin{align*}\chi=\sum_{i\in I}\sum_{a=1}^{d^i}c^i_a\beta^i_a, \ c^i_a\geq c^i_b \mbox{ for all }i\in I, d^i\geq a\geq b\geq 1.
\end{align*}
For $\chi \in M(d)^+:=M(d) \cap M(d)_{\mathbb{R}}^+$, we denote by $\Gamma_{G(d)}(\chi)$ the irreducible 
representation of $G(d)$ with highest weight $\chi$. 
We denote by $\rho$ half the sum of positive roots of $\mathfrak{g}(d)$, 
\begin{align*}
    \rho=\frac{1}{2}\sum_{i\in I} \sum_{a>b} (\beta_a^i-\beta_b^i). 
\end{align*}
We denote by $1_d$ the diagonal cocharacter of $T(d)$, which 
acts on $\beta_a^i$ by weight one. 
% For a $T(d)$-representation $V$, by abuse of notation 
% we denote by $V \in M(d)$ 
% the sum of weights of $V$. 
For $d=(d^i)_{i\in I}$,
denote by $\underline{d}=\sum_{i\in I}d^i$. We define the following weights
\begin{align*}
    \sigma_{d}:=\sum_{i\in I, 1\leq a\leq d^i}\beta^i_a\in M(d), \ 
    \tau_d:=\frac{\sigma_d}{\dd}\in M(d)_\mathbb{R}.
\end{align*}

Let $N(d)$ be the cocharacter lattice of $T(d)$. 
We have the perfect pairing 
\begin{align*}
    \langle -, - \rangle \colon 
    N(d) \times M(d) \to \mathbb{Z}, \ 
    N(d)_{\mathbb{R}} \times M(d)_{\mathbb{R}} \to \mathbb{R}. 
\end{align*}
If $\lambda$ is a cocharacter of $T(d)$ and $V$ is a $T(d)$-representation, we may abuse notation and write
$\langle \lambda, V\rangle=\langle \lambda, \det(V)\rangle$
to simplify the notation.
A cocharacter $\lambda=(\lambda_a^i) \in N(d)$
for $i \in I$ and $1\leq a\leq d^i$ is \textit{antidominant} 
if $\lambda_a^i \geq \lambda_b^i$ for $1\leq b\leq a \leq d^i$. 
An antidominant cocharacter of $T(d)$ uniquely 
determines a decomposition $d=d_1+\cdots+d_k$ 
in $\mathbb{N}^I$, 
such that $\lambda$ acts on $\beta_a^i$ with weight 
$\lambda_j$ for \begin{align*}\sum_{l<j} d_l^i <a \leq \sum_{l\leq j} d_l^i.\end{align*}
The partition $(d_i)_{i=1}^k$ is called \textit{associated 
partition} of $\lambda$.

 For a cocharacter $\lambda$ of $T(d)$ and a 
 $T(d)$-representation $V$, we set
 $V^{\lambda \geq 0} \subset V$ to be 
 spanned by $T(d)$-weights $\beta$
 of $V$ such that $\langle \lambda, \beta \rangle \geq 0$. 
 The subspace $V^{\lambda} \subset V$ is defined 
 to be the $\lambda$-fixed subspace.

\subsection{Framed quivers}\label{subsec:framedquiver}
For a quiver $Q=(I,E)$
and $\alpha \in \mathbb{N}$, 
its $\alpha$-\textit{framed quiver} 
$Q^{\alpha f}=(I^f, E^{\alpha f})$ is defined as follows~\cite{EnRei}. 
The vertex set is $I^f=\{\infty\} \sqcup I$
and the edge set is 
\begin{align*}E^{\alpha f}=E\sqcup \{e_{i, j}\mid i\in I, 1\leq j \leq \alpha\},
\end{align*}
where $e_{i, j}$ is an edge from $\infty$ to $i\in I$. 
For $d=(d^i)_{i\in I}\in\mathbb{N}^I$, 
we set \[V^{\alpha}(d)=\bigoplus_{i\in I}(V^i)^{\oplus \alpha},\] where $\dim V^i=d^i$, 
and define 
 \begin{align}\label{Rf:framed}R^{\alpha f}(d):=R(d)\oplus V^{\alpha}(d).
 \end{align}
 Note that $R^{\alpha f}(d)$ is 
 the affine space of representations of $Q^{\alpha f}$ of dimension 
 vector $(1, d)$. 
 The moduli stack of $\alpha$-framed representations of $Q$ is given by 
\[\X^{\alpha f}(d):=R^{\alpha f}(d)/G(d).\]
The character $\sigma_{d}$ defines 
the GIT stability on $Q^{\alpha f}$, 
whose (semi)stable locus 
consists of 
$Q^{\alpha f}$-representations with 
no subrepresentations of dimension $(1,d')$ for $d' \neq d$,  
see~\cite{EnRei}, \cite[Lemma~5.1.9]{T}. 
By taking the GIT quotient, 
we obtain the smooth quasi-projective variety of 
stable framed representations 
\[\X^{\alpha f}(d)^{\text{ss}}:=R^{\alpha f}(d)^{\text{ss}}/G(d).\]
% Let $W$ be a potential of $Q$.
% We refer to the categories 
% \[D^b\left(\X^f(d)^{\text{ss}}\right)\text{ and } \mathrm{MF}\left(\X^f(d)^{\text{ss}}, \mathrm{Tr}\,W\right)\] as the \textit{Donaldson-Thomas (DT) categories} of $Q$ and $(Q,W)$, respectively.

\subsection{The categorical Hall product}
For a quiver $Q=(I, E)$
and a dimension vector $d$, 
let $\lambda$ be an antidominant cocharacter of $T(d)$. 
We have the attracting and fixed stacks given by 
\begin{align*}
    \X(d)^{\lambda \geq 0}=R(d)^{\lambda \geq 0}/G(d)^{\lambda \geq 0},  \ 
    \X(d)^{\lambda}=R(d)^{\lambda}/G(d)^{\lambda}
    =\times_{i=1}^k \X(d_i). 
\end{align*}
Here $(d_i)_{i=1}^k$ is the 
partition associated with $\lambda$. 
We have the natural maps 
\begin{align}\label{diagram:natural}
\xymatrix{
\X(d)^{\lambda \geq 0} \ar[r]^-{q_{\lambda}} \ar[d]_-{p_{\lambda}} & \X(d)^{\lambda} 
\ar[ld]^-{a_{\lambda}} \ar[r]^-{\pi_{\lambda}} & X(d)^{\lambda} \ar[d]_-{i_{\lambda}} \\
\X(d) \ar[rr]^-{\pi_d} & & X(d). 
}
\end{align}
Here $\pi_d$, $\pi_{\lambda}$ are good moduli space 
morphisms, 
and 
$p_{\lambda}$, $a_{\lambda}$, $q_{\lambda}$ are
induced by the inclusions, projection, 
respectively
\begin{align*}
    R(d)^{\lambda \geq 0} \hookrightarrow R(d), \ 
    R(d)^{\lambda} \hookrightarrow R(d), \ 
    R(d)^{\lambda \geq 0} \twoheadrightarrow R(d)^{\lambda}. 
\end{align*}
In this paper, we often use the notation 
in the diagram (\ref{diagram:natural}). 

The categorical Hall product of $Q$
is defined to be the functor \cite{P0}: 
\begin{align}\label{def:Hallprod}
m_{\lambda}:= p_{\lambda*}q_\lambda^{\ast} \colon 
\boxtimes_{a=1}^k D^b(\mathscr{X}(d_a)) \to D^b(\mathscr{X}(d)). 
\end{align}
The above functor is a categorical version 
of the cohomological Hall product~\cite{MR2851153}, 
which is induced by 
the following sheaf-theoretic one, see~\cite{DM}:
\begin{align}\label{coha}
  m_{\lambda}^{\rm{co}} := p_{\lambda*}q_\lambda^{\ast} \colon 
  i_{\lambda \ast}\pi_{\lambda\ast}\mathrm{IC}_{\X(d)^{\lambda}} \to 
  \pi_{d\ast}\mathrm{IC}_{\X(d)}. 
\end{align}

Let $W$ be a potential of $Q$, and $\Tr W$
the function on $\X(d)$ as in (\ref{map:trW}). 
The categorical Hall product of $(Q, W)$ is defined to 
be the functor 
\begin{align}\label{def:Hallprod2}
    p_{\lambda*}q_\lambda^{\ast} \colon \boxtimes_{a=1}^k \mathrm{MF}(\mathscr{X}(d_a), \Tr W) \to \mathrm{MF}(\mathscr{X}(d), \Tr W). 
\end{align}
If there is a grading, there is also
a categorical Hall product for graded matrix 
factorizations. 
%Later we also use the following 
%notation on the set of weights
%\begin{align}\label{def:alambda}
%    \mathscr{A}_\lambda&:=\{\beta\in \mathscr{A} \mid \langle \lambda, \beta\rangle>0\},\\
%   \notag \mathfrak{g}_\lambda&:=\{\beta^i_a-\beta^i_b \mid 
% i\in I, 1\leq a,b\leq d^i, \langle \lambda, \beta^i_a-\beta^i_b\rangle>0\}.
%\end{align}
% We abuse notation and write
% \begin{align*}
%    R(d)^{\lambda>0}=\sum_{\beta\in\mathscr{A}_\lambda}\beta,\
%     \mathfrak{g}(d)^{\lambda>0}=\sum_{\beta\in\mathfrak{g}_\lambda}\beta,
% \end{align*} and use analogous notations for other determinants of various representations. 
%Define the weights $v_a\in\frac{1}{2}\mathbb{Z}$ by the formula \begin{equation}\label{defv} \sum_{a=1}^k v_a\tau_{d_a}=\sum_{a=1}^k w_a\tau_{d_a}-\frac{1}{2}\left(\sum_{\beta\in A_\lambda}\beta-\sum_{\beta\in\mathfrak{g}_\lambda}\beta\right). \end{equation}
%Explicitly, the weights $v_a$ are given by \[v_a=w_a-\frac{1}{2}\left(A\cdot\dd_a\left(-\sum_{b>a}\dd_b+\sum_{b<a}\dd_b\right)-\sum_{i\in I}d^i_a\left(-\sum_{b>a}d^i_b+\sum_{b<a}d^i_b\right)\right)\] for $1\leq a\leq k$.

\subsection{Doubled quivers and preprojective algebras}\label{subsec26}

% In this section, we recall the construction of double and tripled quivers with potential.

Let $Q^\circ=(I, E^\circ)$ be a quiver. For an edge $e$ of $Q^{\circ}$, denote by $\overline{e}$ the edge with opposite orientation.
Let $E^{\circ, d}$ be the multiset $E^{\circ, d}=\{e, \overline{e} \mid e\in E^{\circ}\}$. 
The \textit{doubled quiver}
is defined by \[Q^{\circ, d}=(I, E^{\circ, d}).\] 
Let $\mathscr{I}\subset \mathbb{C}[Q^{\circ, d}]$ be the two-sided ideal generated by $\sum_{e\in E^\circ}[e, \overline{e}]$.
The preprojective algebra of $Q^\circ$ is $\Pi_{Q^\circ}:=\mathbb{C}[Q^{\circ, d}]/\mathscr{I}$.

For $d\in \mathbb{N}^I$, recall the stack of representations of dimension $d$ of $Q^{\circ}$:
\[\X^\circ(d)=R^\circ(d)/G(d).\]
The stack of representations of $Q^{\circ, d}$ is
given by 
\[\mathscr{Y}(d):=(R^\circ(d)\oplus R^\circ(d)^\vee)/G(d).\]
The stack of representations of the preprojective algebra $\Pi_{Q^{\circ}}$ is given by 
\[\mathscr{P}(d):=T^*\left(\X^\circ(d)\right)=\mu^{-1}(0)/G(d),\]
where \[\mu\colon T^*R^\circ(d)=R^\circ(d)\oplus R^\circ(d)^\vee\to \mathfrak{g}(d)^\vee\cong \mathfrak{g}(d)\] is the moment map 
defined by the relation 
$\sum_{e \in E^{\circ}}[e, \overline{e}]$,
and $\mu^{-1}(0)$ is the derived zero locus of $\mu$. 
% For an edge $e$ of $Q$, denote by $\overline{e}$ the edge with opposite orientation.
% Define the multiset $E^{\circ, d}=\{e, \overline{e}|\,e\in E\}$.
% The affine space $T^*R^\circ(d)$ can be also described as the space of representations of \textit{the double quiver} $Q^{\circ, d}=(I, E^{\circ, d})$. 

The image of $\mu$ lies in the traceless Lie subalgebra $\mathfrak{g}(d)_0\subset \mathfrak{g}(d)$, and thus induces a map $\mu_0\colon T^*R^\circ(d)\to \mathfrak{g}(d)_0$.  
The reduced version of the moduli stack 
is defined to be  
\[\mathscr{P}(d)^{\mathrm{red}}:=\mu_0^{-1}(0)/G(d).\]
Note that $\mathscr{P}(d)^{\mathrm{red, cl}}=\mathscr{P}(d)^{\rm{cl}}$. 
We have the good moduli space map:
\[\pi_{P,d}\colon \mathscr{P}(d)^{\mathrm{cl}}\to P(d):=\mu^{-1}(0)^{\rm{cl}}\ssslash G(d).\] 
% For any Zariski open $U\hookrightarrow P(d)$, there exists $\mathscr{U}\hookrightarrow\mathscr{P}(d)$ such that the following square are cartesian:
% \begin{equation}\label{diagdefU}
%     \begin{tikzcd}
%         \mathscr{U}\arrow[r, hook]& \mathscr{P}(d)\\
%         \pi_{P,d}^{-1}(U)\arrow[u]\arrow[r, hook]\arrow[d]\arrow[dr, phantom, "\square"]\arrow[ur, phantom, "\square"]& \mathscr{P}(d)^{\mathrm{cl}}\arrow[u]\arrow[d, "\pi_{P,d}"]\\
%         U\arrow[r, hook]& P(d).
%     \end{tikzcd}
% \end{equation}
% %We abuse notation and also write $\mathscr{U}=\pi_{P,d}^{-1}(U)\subset \mathscr{P}(d)$.

\subsection{Tripled quivers with potential}\label{subsec29}

Let $Q^\circ=(I, E^\circ)$ be a quiver and let $Q^{\circ, d}=(I, E^{\circ, d})$ be
its doubled quiver. 
\textit{The tripled quiver with potential} \[(Q,W)\] is defined as follows. 
The quiver $Q=(I, E)$ has set of edges $E=E^{\circ, d}\sqcup\{\omega_i \mid i\in I\}$, where $\omega_i$ is a loop at the vertex $i \in I$. The potential $W$ is 
\[W:=\left(\sum_{i\in I}\omega_i\right)
\left(\sum_{e\in E^\circ}[e, \overline{e}]\right)\in \mathbb{C}[Q]/[\mathbb{C}[Q], \mathbb{C}[Q]].\]
We say $(Q,W)$ is a tripled quiver with potential if it is obtained as above for some quiver $Q^\circ$.

For a tripled quiver $Q$ as above, the moduli stack 
of its representations is given by 
\begin{align*}
    R(d)=T^{\ast}R^{\circ}(d)\oplus \mathfrak{g}(d)=
    R^{\circ}(d) \oplus R^{\circ}(d)^{\vee} \oplus \mathfrak{g}(d). 
\end{align*}
We have the moduli stack of $Q$-representations 
together with the regular function determined by $W$
\[\Tr W \colon 
\X(d)=R(d)/G(d) \to \mathbb{C}.\]
Let $\mathbb{C}^{\ast}$ acts on the $\mathfrak{g}(d)$-factor on $R(d)$ 
with weight one.  
The Koszul equivalence \eqref{Koszul} implies an equivalence \begin{equation}\label{Koszulequivalence}
\kappa\colon D^b\left(\mathscr{P}(d)\right)\stackrel{\sim}{\to} \mathrm{MF}^{\mathrm{gr}}\left(\X(d), \mathrm{Tr}\,W\right).
\end{equation}

\subsection{Quasi-BPS categories}\label{subsec27}
Here we recall the definition of quasi-BPS categories for 
symmetric quivers with potential. 
Here a quiver $Q=(I, E)$ is called \textit{symmetric}
if for any $i, j \in I$ the number of edges from $i$ to $j$ is the same 
as that from $j$ to $i$. 
%We say $Q$ is a \textit{very symmetric quiver} if there exists a natural number $A$ such that, for any vertices $i,j\in I$, the number of edges from $i$ to $j$ is $A$. 
Let $(Q, W)$ be a symmetric quiver with potential $W$. 
For a dimension vector 
$d=(d^i)_{i\in I}\in\mathbb{N}^I$, 
let $\mathscr{A}$ be the following multiset of $T(d)$-weights on $R(d)$:
\[\mathscr{A}:=\{\beta^i_a-\beta^j_b \mid i,j\in I, 
    (i \to j) \in E, 1\leq a\leq d^i, 1\leq b\leq d^j\}.\] 
The above multiset defines the following polytope in $M(d)_{\mathbb{R}}$:
\begin{equation}\label{defpolytope}
\mathbf{W}(d):=\frac{1}{2}\mathrm{sum}_{\beta\in \mathscr{A}}[0, \beta]\subset M(d)_{\mathbb{R}},
\end{equation}
where the sums above are Minkowski sums in the space of weights $M(d)_{\mathbb{R}}$.

For a weight $\delta\in M(d)_{\mathbb{R}}^{W_d}$,
%\mu\in \mathbb{R}$, 
the subcategory \begin{equation}\label{defmdw}
    \mathbb{M}(d; \delta)\subset D^b(\X(d))
\end{equation}
is defined to be the full subcategory of $D^b(\X(d))$ generated by vector bundles $\mathcal{O}_{\X(d)}\otimes\Gamma_{G(d)}(\chi)$, where $\chi \in M(d)^+$ is
such that
\[\chi+\rho-\delta\in \mathbf{W}(d).\] 
    
    Note that any complex $F\in D^b(B\mathbb{C}^*)$ splits as a direct sum \[F=\bigoplus_{w\in\mathbb{Z}}F_w\] such that $\mathbb{C}^*$ acts with weight $w$ on $F_w$. We say \textit{$w\in \mathbb{Z}$ is a weight of $F$} if $F_w\neq 0$.
For a cocharacter $\lambda$ of $T(d)$, we 
define 
   \begin{align}\label{nlambdadef}
n_{\lambda} :=\langle \lambda, \mathbb{L}_{\X(d)}|_{0}^{\lambda>0} \rangle 
=\langle \lambda, (R(d)^{\vee})^{\lambda>0}-
(\mathfrak{g}(d)^{\vee})^{\lambda>0}\rangle. 
\end{align}
 The category \eqref{defmdw} has the following alternative description. 
\begin{lemma}\emph{(\cite[Corollary 3.11]{PTquiver})}\label{lemma:alt}
    The category $\mathbb{M}(d; \delta)$ is the subcategory of $D^b(\X(d))$ of objects $F\in D^b(\X(d))$ such that, for any $\nu \colon B\mathbb{C}^*\to \X(d)$
    corresponding to a point $x \in R(d)$ and a cocharacter 
    $\lambda$ of $T(d)$ which fixes $x$, the weights of $\nu^*F$ are contained in the set 
    \begin{align*}\left[-\frac{1}{2}n_{\lambda}+\langle \lambda, \delta\rangle, \frac{1}{2}n_{\lambda}+\langle \lambda, \delta\rangle\right].
    \end{align*}  
    % is generated by vector bundles $\mathcal{O}_{\X(d)} \otimes \Gamma$
    % for a $G(d)$-representation $\Gamma$
    % such that any $T(d)$-weight $\chi$ of $\Gamma$
    % satisfies 
    % \begin{align*}
    % -\frac{1}{2}n_{\lambda} \leq 
    %     \langle \lambda, \chi-\delta\rangle \leq \frac{1}{2}n_{\lambda},
    % \end{align*}
    % for any cocharacter $\lambda$.
\end{lemma}

The quasi-BPS category for quiver
with potential is defined as follows: 
\begin{defn}
    Let $(Q, W)$ be a symmetric quiver with potential. 
The \textit{quasi-BPS category} for $(Q, W)$ is defined by 
    \begin{align}\label{def:quasiBPS}
\mathbb{S}(d; \delta):=\mathrm{MF}\left(\mathbb{M}(d; \delta), \mathrm{Tr}\,W\right) \subset 
\mathrm{MF}(\X(d), \Tr W). 
\end{align}
\end{defn}
\begin{remark}
If there is a grading, the graded quasi-BPS 
category 
\begin{align*}
    \mathbb{S}^{\rm{gr}}(d; \delta) \subset 
    \mathrm{MF}^{\rm{gr}}(\X(d), \Tr W)
\end{align*}
is defined in a similar way, see Subsection~\ref{subsection:matrixfactorizations} for details about categories of graded matrix factorizations. 
\end{remark}
If $\delta=v\tau_d$, we use the notations:
\[\mathbb{M}(d)_v:=\mathbb{M}(d; v\tau_d), \ 
\mathbb{S}(d)_v:=\mathbb{S}(d; v\tau_d).\]

\subsection{Quasi-BPS categories for preprojective algebras}\label{subsec:qbps:pre}
Let $Q^{\circ}$ be a quiver and $Q^{\circ, d}$ its double as in 
Subsection~\ref{subsec26}. 
There is a subcategory, 
called \textit{preprojective quasi-BPS category}:
\begin{align*}\mathbb{T}(d; \delta)\subset D^b(\mathscr{P}(d)),
\end{align*}
see also \cite[Definition 2.14]{PTquiver} for its 
definition. 
We only use the property that, under the Koszul equivalence \eqref{Koszulequivalence}, 
we have an equivalence :
\begin{equation}\label{kappamagic}
\kappa\colon \mathbb{T}(d; \delta)\xrightarrow{\sim} \mathbb{S}^{\mathrm{gr}}(d; \delta). 
\end{equation}

Let 
$\X(d)^{\mathrm{red}}$ be defined by 
\begin{align*}
\X(d)^{\mathrm{red}}:=(T^* R^\circ(d)\oplus \mathfrak{g}(d)_0)/G(d).
\end{align*}
By the Koszul equivalence in Theorem~\ref{thm:Koszul}, we have 
an equivalence 
\[\kappa'\colon D^b(\mathscr{P}(d)^{\mathrm{red}})\xrightarrow{\sim} \mathrm{MF}^{\mathrm{gr}}\big(\X(d)^{\mathrm{red}}, \mathrm{Tr}\,W\big).\]
Similarly to above, 
there is a subcategory 
called \textit{reduced quasi-BPS category}
\begin{align*}
 \mathbb{T}(d; \delta)^{\mathrm{red}}\subset D^b(\mathscr{P}(d)^{\mathrm{red}})   
\end{align*}
such that, under the Koszul equivalence $\kappa'$, we have an equivalence 
\[\kappa'\colon \mathbb{T}(d; \delta)^{\mathrm{red}}\xrightarrow{\sim} \mathbb{S}^{\mathrm{gr}}(d; \delta)^{\mathrm{red}}.\]
Here the right hand side is defined similarly to (\ref{def:quasiBPS}). 

We next discuss the compatibility between reduced and non-reduced quasi-BPS categories. 
We have 
the decomposition 
$\mathfrak{g}(d)= \mathfrak{g}(d)_0\oplus \mathbb{C}$
of $G(d)$-representation, and the projection onto the first factor induces a map $t\colon \X(d)\to \X(d)^{\mathrm{red}}$. Note that $\mathrm{Tr}\,W \circ t=\mathrm{Tr}\,W$. Let $l'\colon \mathscr{P}(d)^{\mathrm{red}}\hookrightarrow \mathscr{P}(d)$
be the natural closed immersion. The next proposition follows from \cite[Lemma 2.4.4]{T}:

\begin{prop}\label{redunred}
    The following diagram commutes:
    \begin{equation*}
        \begin{tikzcd}
            D^b(\mathscr{P}(d)^{\mathrm{red}})\arrow[d, "\sim" {anchor=south, rotate=90}, "\kappa'"]\arrow[r, "l'_*"]& D^b(\mathscr{P}(d))\arrow[d, "\sim" {anchor=south, rotate=90}, "\kappa"]\\
            \mathrm{MF}^{\mathrm{gr}}(\X(d)^{\mathrm{red}}, \mathrm{Tr}\,W)\arrow[r, "t^*"]& \mathrm{MF}^{\mathrm{gr}}(\X(d), \mathrm{Tr}\,W).
        \end{tikzcd}
    \end{equation*}
    It induces a commutative diagram:
    \begin{equation*}
        \begin{tikzcd}
            \mathbb{T}(d; \delta)^{\mathrm{red}}\arrow[d, "\sim" {anchor=south, rotate=90}, "\kappa'"]\arrow[r, "l'_*"]& \mathbb{T}(d; \delta)\arrow[d, "\sim" {anchor=south, rotate=90}, "\kappa"]\\
            \mathbb{S}^{\mathrm{gr}}(d; \delta)^{\mathrm{red}}\arrow[r, "t^*"]& \mathbb{S}^{\mathrm{gr}}(d; \delta).
        \end{tikzcd}
    \end{equation*}
\end{prop}

\subsection{Semiorthogonal decompositions}\label{subsec:sod}

We recall (a simplified version of) the semiorthogonal decompositions 
of $D^b(\X(d))$, see~\cite[Theorem 1.1]{P}. 
Note that the diagonal cocharacter 
$1_d \colon \mathbb{C}^{\ast} \to G(d)$ acts on $R(d)$ trivially. 
Therefore, there is an orthogonal decomposition 
\begin{align*}
    D^b(\X(d))=\bigoplus_{v\in \mathbb{Z}}D^b(\X(d))_v,
\end{align*}
where $D^b(\X(d))_v$ is the subcategory of 
weight $v$ with respect to the diagonal cocharacter. 

\begin{thm}\label{theorem266}\emph{(\cite{P})}
Let $Q=(I, E)$ be a symmetric quiver, 
and take a dimension vector $d\in \mathbb{N}^I$. 
Let $\delta\in M(d)^{W_d}_\mathbb{R}$ 
such that $\langle 1_d, \delta\rangle=v \in \mathbb{Z}$.
The category $\mathbb{M}(d; \delta)$ is right admissible in $D^b(\X(d))_v$,
so there exist $X(d)$-linear semiorthogonal decompositions:
\begin{equation}\label{SODtheorem266}
D^b(\X(d))_v=\langle \mathbb{B}(d;\delta), \mathbb{M}(d; \delta) \rangle. 
\end{equation}

\end{thm}

\begin{remark}\label{rmk:version}
    Applying a version of Theorem~\ref{theorem266}
    for the categories of matrix factorizations associated with 
    tripled quivers with potential and 
    using the Koszul equivalence, the category $\mathbb{T}(d; \delta)$ is right admissible in $D^b(\mathscr{P}(d))_v$, so there is a semiorthogonal decomposition:
\begin{equation}\label{SODtheorem266Pd}
D^b(\mathscr{P}(d))_v=\langle \mathbb{A}(d; \delta), \mathbb{T}(d; \delta)\rangle.
\end{equation}
Similarly the reduced quasi-BPS category $\mathbb{T}(d; \delta)^{\mathrm{red}}$ is
right admissible in $D^b(\mathscr{P}(d)^{\mathrm{red}})_v$. 
\end{remark}

The following framed version is proved in~\cite{PTquiver} 
under some assumption on $Q$ and $\delta$, but with explicit description 
of the semiorthogonal summands. In this paper, we only use the right admissibility, 
which is deduced from Theorem~\ref{theorem266}.
The proof follows by arguments which already appeared in several 
references~\cite{PTzero, T, KoTo, PTquiver}, but
the statement is not mentioned so we will include the proof in 
Subsection~\ref{subsec:sod2}. 
\begin{prop}\label{prop226}
    In the same setting of Theorem~\ref{theorem266}, 
    there is a $X(d)$-linear semiorthogonal decomposition 
    \begin{align*}
        D^b(\X^{\alpha f}(d)^{\rm{ss}})=\langle \mathbb{B}(d; \delta)', \mathbb{M}(d; \delta) \rangle. 
    \end{align*}
\end{prop}

% \begin{thm}\label{theorem26}
%     Let $Q^\circ$ be a quiver satisfying Assumption~\ref{assum1}. There is a semiorthogonal decomposition:
%     \[D^b(\mathscr{P}(d))=\Big\langle \bigotimes_{i=1}^k \mathbb{T}\left(d_i; \theta_i+v_i\tau_{d_i}\right)\Big\rangle,\] where the right hand side is an is Theorem \ref{sodfullstackB}, except that the weights $\theta_i\in M(d_i)$ are defined by
%     \[\sum_{i=1}^k \theta_i=-\frac{1}{2}R(d)^{\lambda>0}+\frac{3}{2}\mathfrak{g}(d)^{\lambda>0}.\]
% \end{thm}

\section{Topological K-theory of dg-categories}\label{sec:topK:dg}
In this section, we review the topological K-theory 
of dg-categories due to Blanc~\cite{Blanc}
and its relative version due to Moulinos~\cite{Moulinos}. 
We also prove some lemmas which are not available 
in the references. 
\subsection{Topological K-theory}\label{subsec:topK:review}
We use the notion of \textit{spectrum}, which is an object 
representing a generalized cohomology theory and is a central object of study in stable homotopy theory. 
Basic references are~\cite{Adams2, May}. 
A spectrum consists of a sequence of pointed spaces $E=\{E_n\}_{n\in \mathbb{N}}$ 
together with maps $\Sigma E_n \to E_{n+1}$ where $\Sigma$ is the suspension 
functor. 
There is a notion of homotopy groups of a spectrum 
denoted by $\pi_{\bullet}(E)$.
The $\infty$-category of spectra is 
constructed in~\cite[Section~1.4]{LHAG} and 
denoted by $\mathrm{Sp}$.
A map of spectra 
$E \to F$ 
in $\mathrm{Sp}$
is called an \textit{equivalence} if it induces
isomorphisms on homotopy groups. Below, for an $\infty$-category $\mathcal{C}$, we 
denote by $\mathrm{Ho}(\mathcal{C})$ its homotopy category, e.g. $\mathrm{Ho}(\mathrm{Sp})$ is the 
homotopy category of spectra.

For a topological space $X$, we denote by $K^{\rm{top}}_{i}(X)$ its $i$-th
topological K-group.
There is a Chern character map 
\begin{align}\label{ch:Ktheory}
\ch \colon 
    K^{\rm{top}}_i(X) \to H^{i+2\ast}(X, \mathbb{Q})
\end{align}
which is an isomorphism over $\mathbb{Q}$ 
if $X$ is a finite CW complex
by the Atiyah-Hirzebruch theorem. 
The K-groups $K^{\rm{top}}_{\bullet}(X)$ 
are obtained as homotopy groups of the spectrum $KU=\{BU\times \mathbb{Z}, U, \ldots\}$, 
called \textit{the topological K-theory spectrum}. 
For $KU_X:=[\Sigma^{\infty}X, KU]$, we have that (see~\cite[Sections~22-24]{May}):
\begin{align*}
	K_{\bullet}^{\rm{top}}(X)=\pi_{\bullet}KU_X. 
\end{align*}
Here $[-, KU]$ is the mapping spectrum with target $KU$
and $\Sigma^{\infty}(-)$ is the infinite suspension functor. 
We call $K^{\rm{top}}(X)$ the \textit{topological K-theory spectrum of }$X$. 

For a dg-category $\mathscr{D}$,
Blanc \cite{Blanc}
defined the
\textit{topological K-theory spectrum
	of $\mathscr{D}$}, denoted by
\begin{align}\label{def:topK:dg}K^{\mathrm{top}}(\mathscr{D}) \in \mathrm{Sp}.
\end{align}
For $\mathcal{D}=\mathrm{Perf}(X)$, where $X$ is a separated $\mathbb{C}$-scheme of finite type, 
there is an isomorphism 
\[K^{\rm{top}}(\mathcal{D}) \simeq KU_X.\]
For $i\in \mathbb{Z}$, we consider the $i$-th rational homotopy groups of~\eqref{def:topK:dg}, which is a $\mathbb{Q}$-vector space:
\begin{align*} 
K^{\mathrm{top}}_{i}(\mathscr{D}):=\pi_i(K^{\mathrm{top}}(\mathscr{D})),\,\,K^{\mathrm{top}}_i(\mathscr{D})_{\mathbb{Q}}:=K^{\mathrm{top}}_i(\mathscr{D})\otimes_\mathbb{Z}\mathbb{Q}.
\end{align*}
There is an isomorphism $K^{\mathrm{top}}_i(\mathscr{D})\cong K^{\mathrm{top}}_{i+2}(\mathscr{D})$ for every $i\in \mathbb{Z}$ given by multiplication with a Bott element, see \cite[Definition 1.6]{Blanc}.
The topological K-theory spectrum sends exact triangles of dg-categories to exact triangles of spectra \cite[Theorem 1.1(c)]{Blanc}.
We denote the $\mathbb{Z}/2$-graded total rational 
topological K-theory of $\mathscr{D}$ by
\begin{align*}K^{\mathrm{top}}_{\ast}(\mathscr{D})_{\mathbb{Q}}:=K^{\mathrm{top}}_0(\mathscr{D})_{\mathbb{Q}}\oplus K^{\mathrm{top}}_1(\mathscr{D})_{\mathbb{Q}}.
\end{align*}
Given a filtration (indexed by integers) on $K^{\mathrm{top}}_i(\mathscr{D})$ for some $i\in \mathbb{Z}$, we consider the associated graded pieces $\mathrm{gr}_\bullet K^{\mathrm{top}}_i(\mathscr{D})$ for $\bullet\in\mathbb{Z}$ and we let \[\mathrm{gr}_{\ast}K_i^{\mathrm{top}}(\mathscr{D}):=\bigoplus_{j\in \mathbb{Z}}\mathrm{gr}_j K_i^{\mathrm{top}}(\mathscr{D}).\]

\subsection{Relative topological K-theory}\label{subsec:reltopKth}
Here we review the relative version of topological
K-theory developed by Moulinos in~\cite{Moulinos}. 
Below we use the notation in~\cite[Section~2]{GS}
for sheaves of spectra. 

For a complex variety $M$ and an $\infty$-category with arbitrary 
small limits, we denote by 
$\mathrm{Sh}_{\mathcal{C}}(M^{\rm{an}})$ the $\infty$-category of
$\mathcal{C}$-valued sheaves on $M^{\rm{an}}$
as in~\cite[Section~1.3.1]{LSAG}, 
where $M^{\rm{an}}$ is the underlying complex analytic 
space. Note that the $\infty$-category $\mathrm{Sh}_{\mathcal{C}}(M^{\rm{an}})$ satisfies the hyperdescent as 
$M^{\rm{an}}$ is finite dimensional, see~\cite[Cor.~7.2.1.12]{Ltopos}, \cite[Prop.~A.20]{GS}.
There is a rationalization functor, see~\cite[Definition~2.6]{GS}:
\begin{align*}
	\mathrm{Rat} \colon \mathrm{Sh}_{\mathrm{Sp}}(M^{\rm{an}}) \to 
	\mathrm{Sh}_{D(\mathbb{Q})}(M^{\rm{an}})=D(\mathrm{Sh}_{\mathbb{Q}}(M^{\rm{an}}))
\end{align*}
given by $\mathcal{F} \mapsto \mathcal{F} \wedge H\mathbb{Q}$. 
Here, $H\mathbb{Q}$ is the Eilenberg-Maclane spectrum
of $\mathbb{Q}$ and 
the right hand side is the derived category of sheaves of $\mathbb{Q}$-vector 
spaces on $M^{\rm{an}}$.
We will write $\mathcal{F}_{\mathbb{Q}} :=\mathrm{Rat}(F)$. 
It satisfies that 
$\pi_{\bullet}(\mathcal{F}_{\mathbb{Q}})=\pi_{\bullet}(\mathcal{F})\otimes \mathbb{Q}$. 

By~\cite[Lemma~2.7]{GS}, 
the sheaf of topological K-theory spectra $\underline{KU}_M$ on $M^{\rm{an}}$ satisfies
\begin{align}\label{isom:KU}
	(\underline{KU}_{M})_{\mathbb{Q}} \cong \mathbb{Q}_M[\beta^{\pm 1}]
	=\bigoplus_{n\in \mathbb{Z}} \mathbb{Q}_M[2n],
\end{align}
where $\beta$ is of degree $2$. 
The above isomorphism follows from the degeneration of 
Atiyah--Hirzebruch spectral sequence. 
The natural morphism 
\begin{align*}
    \ch \colon \underline{KU}_M \to (\underline{KU}_M)_{\mathbb{Q}} \cong \mathbb{Q}_M[\beta^{\pm 1}]
\end{align*}
induces the Chern character map (\ref{ch:Ktheory}) by taking global sections. 

Let 
$\mathrm{Cat}^{\rm{perf}}(M)$ be
the $\infty$-category of $\mathrm{Perf}(M)$-linear stable 
$\infty$-categories. 
By~\cite{Moulinos}, there is a functor 
\begin{align}\label{def:reltopK:dg}
	\mathcal{K}_{M}^{\rm{top}} \colon 
	\mathrm{Cat}^{\rm{perf}}(M) \to \mathrm{Sh}_{\mathrm{Sp}}(M^{\rm{an}})  
\end{align}
such that, for a Brauer class $\alpha$ on $M$, 
we have 
\begin{align}\label{isom:ktopalpha}
	\mathcal{K}_M^{\rm{top}}(\mathrm{Perf}(M, \alpha))=\underline{KU}^{\hat{\alpha}}_M. 
\end{align}
In the above, 
$\mathrm{Perf}(M, \alpha)$ is the dg-category of $\alpha$-twisted 
perfect complexes, 
$\hat{\alpha} \in H^3(M, \mathbb{Z})$ is the class associated with $\alpha$
via the natural map $\mathrm{Br}(M) \to H^3(M, \mathbb{Z})$, 
and the right hand side is the sheaf of $\hat{\alpha}$-twisted topological K-theory spectrum.

For $M=\Spec \mathbb{C}$, 
\eqref{def:reltopK:dg} 
agrees with (\ref{def:topK:dg}), that is
$\mathcal{K}_{\mathrm{Spec}\,\mathbb{C}}^{\rm{top}}(\mathscr{D})=K^{\rm{top}}(\mathscr{D})$. 

\subsection{Topological G-theory}\label{subsec:topG}
Let $M$ be a quasi-projective scheme. 
There is an embedding $\mathrm{Perf}(M) \subset D^b(M)$ which 
is not an equivalence unless $M$ is smooth. 
A version of topological K-theory for $D^b(M)$ is 
called \textit{topological K-homology}. It is 
introduced by Thomason~\cite{ThomasonAS}, 
and we denote it by $G^{\rm{top}}_{\bullet}(M)$. 
There is a Chern character map 
\begin{align}\label{ch:Gtheory}
    \ch \colon G_i^{\rm{top}}(M) \to H^{\mathrm{BM}}_{i+2\ast}(M, \mathbb{Q})
\end{align}
which is an isomorphism over $\mathbb{Q}$, 
see \cite[Section 3.6 and Section 6]{MR2762553} or \cite[Section 11]{MR0679698}. 

There is also a spectrum 
version $KU_{M, c}^{\vee}$, 
called \textit{locally compact supported K-homology}, see~\cite[Lemma~2.6]{HLP}, 
satisfying 
\begin{align*}
G_{\bullet}^{\rm{top}}(M)=
    \pi_{\bullet}(KU_{M, c}^{\vee}).
\end{align*}
The spectrum $KU_{M, c}^{\vee}$ is 
characterized by the property that $KU_{M, c}^{\vee}=KU_M$ if $M$ is smooth, 
and for a closed immersion $M\hookrightarrow A$ for a smooth $A$ 
with complement $U=A\setminus M$,
there is a fiber sequence of spectra, see~\cite[Lemma~2.6]{HLP}
\begin{align}\label{fib:Kdual}
    KU_{M, c}^{\vee} \to KU_{A, c}^{\vee} \to KU_{U, c}^{\vee}.
\end{align}
In~\cite[Theorem~2.10]{HLP}, it is proved 
that there is an equivalence 
\begin{align}\label{equiv:csupp}
	K^{\rm{top}}(D^b(M)) \simeq KU_{M, c}^{\vee}. 
\end{align}
We denote by $\underline{KU}_{M, c}^{\vee}$ the presheaf of locally 
compact supported K-homology.  
\begin{lemma}\label{lem:sheaf}
The presheaf $\underline{KU}_{M, c}^{\vee}$ is a sheaf. 
\end{lemma}
\begin{proof}
We take a closed immersion $i\colon M\hookrightarrow A$ for a smooth $A$
with complement $j\colon U:=A\setminus M \hookrightarrow A$.
The presheaf $\underline{KU}_{A, c}^{\vee}$ is a sheaf since it equals
$\underline{KU}_A$ and the latter is a sheaf. The same applies to $\underline{KU}_{U, c}^{\vee}$. 
Let $\widetilde{\underline{KU}}_{M, c}^{\vee}$ be the fiber of the map 
\begin{align*}
    \underline{KU}_{A, c}^{\vee} \to j_{*}\underline{KU}_{U, c}^{\vee}
\end{align*}
in $\mathrm{Sh}_{\mathrm{Sp}}(A^{\rm{an}})$. Since the global section functor 
$\Gamma \colon \mathrm{Sh}_{\mathrm{Sp}}(M^{\rm{an}}) \to \mathrm{Sp}$ preserves finite limits, 
we have $\Gamma(\underline{KU}_{M, c}^{\vee})\simeq KU_{M, c}^{\vee}$, hence 
$\underline{KU}_{M, c}^{\vee}$ is a sheaf. 
\end{proof}

If $M$ is smooth, then 
we have $\underline{KU}_{M, c}^{\vee} \cong \underline{KU}_M$. There is a natural map
\begin{align*}
    \ch \colon \underline{KU}_{M, c}^{\vee} \to 
    (\underline{KU}_{M, c}^{\vee})_{\mathbb{Q}} \cong 
    \omega_M[\beta^{\pm 1}]
\end{align*}
which induces the Chern character map (\ref{ch:Gtheory})
by taking the global section, 
where recall that $\omega_M=\mathbb{D} \mathbb{Q}_M$ is the dualizing 
complex. 
The following is the sheaf version of the equivalence (\ref{equiv:csupp}):
\begin{lemma}\label{lem:KBM}
	For a quasi-projective scheme $M$, 
	there is a natural isomorphism in $\mathrm{Ho}(\mathrm{Sh}_{\mathrm{Sp}}(M^{\rm{an}}))$:
	\begin{align}\label{equiv:singular}
		\mathcal{K}_M^{\rm{top}}(D^b(M)) \stackrel{\cong}{\to} \underline{KU}_{M, c}^{\vee}. 
	\end{align}
	In particular, there is an isomorphism 
			$\mathcal{K}_M^{\rm{top}}(D^b(M))_{\mathbb{Q}} \cong \omega_M[\beta^{\pm 1}]$. 
		\end{lemma}
The proof of the above lemma is not
available in references. 
We prove it after discussing some 
functoriality properties of relative 
topological K-theories in the next subsection. 

\subsection{Push-forward of relative topological K-theories}
We will use the following property of 
topological K-theory under proper push-forward. 

\begin{thm}\emph{(\cite[Proposition~7.8]{Moulinos}, \cite[Theorem~2.12]{GS})}\label{thm:phiproper}
	Let $\phi \colon M \to M'$ be a proper morphism.
	Then, for $\mathscr{D} \in \mathrm{Cat}^{\rm{perf}}(M)$, we have 
	\begin{align*}
		\mathcal{K}_{M'}^{\rm{top}}(\phi_{\ast}\mathscr{D}) \cong \phi_{\ast}\mathcal{K}_M^{\rm{top}}(\mathscr{D}), 
	\end{align*}
	where $\phi_{\ast}\mathscr{D}$ is the category $\mathscr{D}$
	with $\mathrm{Perf}(M')$-linear structure induced by 
	$\phi^{\ast} \colon \mathrm{Perf}(M') \to \mathrm{Perf}(M)$. 
\end{thm}
Below we often write $\mathcal{K}_{M'}^{\rm{top}}(\phi_{\ast}\mathscr{D})$ as 
$\mathcal{K}_{M'}^{\rm{top}}(\mathscr{D})$ when $\phi$ is clear from the context. 
An isomorphism as in Theorem~\ref{thm:phiproper} is not known for 
non-proper morphisms in general. 
The following is a version of Theorem~\ref{thm:phiproper}
for open immersions. 
As the proof uses motivic stable homotopy theory as in~\cite[Proposition~7.7]{Moulinos} and is not self-contained, 
we postpone the proof to Subsection~\ref{subsec:motivic}.

\begin{lemma}\label{lem:openimm}
	Let $j \colon U \hookrightarrow M$ be an open immersion. 
	Then there is a natural isomorphism in $\mathrm{Ho}(\mathrm{Sh}_{\rm{Sp}}(M^{\rm{an}}))$ 
	\begin{align}\label{equiv:open}
		\mathcal{K}_M^{\rm{top}}(j_{\ast}\mathrm{Perf}(U)) \stackrel{\sim}{\to} 
		j_{\ast}\mathcal{K}_U^{\rm{top}}(\mathrm{Perf}(U)). 
	\end{align}
\end{lemma}

We give a proof of Lemma~\ref{lem:KBM} using the 
above results. 
\begin{proof}[Proof of Lemma~\ref{lem:KBM}]
	Let $i \colon M \hookrightarrow A$ be a closed immersion 
	for a smooth $A$
	with complement $j \colon U \hookrightarrow A$. 
	We have an exact sequence of dg-categories
	\begin{align*}
		D_M^b(A) \to D^b(A) \stackrel{j^{\ast}}{\to} D^b(U). 
	\end{align*}
 Here $D_M^b(A) \subset D^b(A)$ is the subcategory 
 of objects supported on $M$. 
	Then there is a fiber sequence
	\begin{align*}
		\mathcal{K}_A^{\rm{top}}(D^b(M)) \to \mathcal{K}_A^{\rm{top}}(D^b(A)) \to \mathcal{K}_A^{\rm{top}}(D^b(U)). 
	\end{align*}
	The above fiber sequence exists since $\mathcal{K}_A^{\rm{top}}(-)$ satisfies 
	devissage, which follows since algebraic K-theory satisfies devissage, see~\cite[Example~2.3]{HLP}. 
	By Theorem~\ref{thm:phiproper} and Lemma~\ref{lem:openimm}, 
	the above sequence is 
	\begin{align*}
		i_{\ast}\mathcal{K}_M^{\rm{top}}(D^b(M)) \to \mathcal{K}_A^{\rm{top}}(D^b(A)) \to j_{\ast}\mathcal{K}_{U}^{\rm{top}}(D^b(U)). 
	\end{align*}
	
	On the other hand, there is also a fiber sequence,
	see the proof of Lemma~\ref{lem:sheaf}
	\begin{align}\label{triangle:A}
		i_{\ast}\underline{KU}_{M, c}^{\vee} \to \underline{KU}_{A, c}^{\vee} \to 
		\underline{KU}_{U, c}^{\vee}.
	\end{align}
	We have the diagram (cf.~\cite[Equation (2)]{HLP}):
	\begin{align}\label{dia:Ktop:KU}
		\xymatrix{
			i_{\ast}\mathcal{K}_M^{\rm{top}}(D^b(M)) \ar[r] \ar@{.>}[d]_-{\rho_{M\subset A}} & \mathcal{K}_A^{\rm{top}}(D^b(A)) 
			\ar[r] \ar[d]_-{\simeq} & j_{\ast}K_U^{\rm{top}}(D^b(U)) \ar[d]_-{\simeq} \\
			i_{\ast}\underline{KU}_{M, c}^{\vee} \ar[r] & \underline{KU}_{A, c}^{\vee} \ar[r] & 
			j_{\ast}\underline{KU}_{U, c}^{\vee}    
		}
	\end{align}
    which is commutative in $\mathrm{Ho}(\mathrm{Sh}_{\mathrm{Sp}}(M^{\rm{an}}))$. 
	The middle vertical arrow is an equivalence 
	since $A$ is smooth, so $D^b(A)=\mathrm{Perf}(A)$ and
	$\underline{KU}_{A} \cong \underline{KU}_{A, c}^{\vee}$. Similarly, the right vertical arrow is also 
 an equivalence. 
	Therefore, we obtain an equivalence $\rho_{M\subset A}$ for the left vertical 
    arrow. 
    
	We need to show that $\rho_{M\subset A}$ is independent of a choice of $M \hookrightarrow A$ 
in $\mathrm{Ho}(\mathrm{Sh}_{\mathrm{Sp}}(M^{\rm{an}}))$, which is given in Lemma~\ref{lem:rho} below, and whose proof is postponed to Subsection~\ref{subsec:lem:rho}. 
\end{proof}

We have the following lemma. Since its proof is 
the same~\cite[Lemma~2.9]{HLP}, a proof will be given in Subsection~\ref{subsec:lem:rho}. 
\begin{lemma}\label{lem:rho}
An equivalence $\rho_{M\subset A}$ is independent of a choice of $M \hookrightarrow A$ 
in $\mathrm{Ho}(\mathrm{Sh}_{\mathrm{Sp}}(M^{\rm{an}}))$. 
\end{lemma}
 
\subsection{Relative topological K-theories of semiorthogonal summands}
We also need a version of Theorem for non-proper $\phi$.  We prove it for semiorthogonal summands of $\mathrm{Perf}(M)$, $D^b(M)$, or 
of categories of matrix factorization. 
We first discuss a preliminary lemma: 

\begin{lemma}\label{lem:gsection}
	Let $B$ be a quasi-projective scheme with a closed immersion 
    $i\colon B\hookrightarrow A$ for a smooth quasi-projective scheme $A$.
	Then, for $\phi \colon B \to \Spec \mathbb{C}$
	and for $\mathscr{D} \in \mathrm{Cat}^{\rm{perf}}(B)$, there is a morphism
	\begin{align}\label{nat:eta}
		\eta_{B\subset A} \colon K^{\rm{top}}(\mathscr{D}) \to \phi_{\ast}\mathcal{K}_{B}^{\rm{top}}(\mathscr{D}). 
	\end{align}
    Moreover, it is a natural transformation: for any $B$-linear functor $F\colon \mathscr{D}_1 \to \mathscr{D}_2$, there is a commutative diagram in $\mathrm{Ho}(\mathrm{Sp})$
\begin{align}\label{nat:trans}
    \xymatrix{
    K^{\rm{top}}(\mathscr{D}_1) \ar[r]^-{F_{*}} \ar[d]_-{\eta_{B\subset A}}
    &K^{\rm{top}}(\mathscr{D}_2) \ar[d]_-{\eta_{B\subset A}} \\
    \phi_{*}\mathcal{K}_{B}^{\mathrm{top}}(\mathscr{D}_1) \ar[r]^-{F_{*}} &
     \phi_{*}\mathcal{K}_{B}^{\mathrm{top}}(\mathscr{D}_2). 
    }
\end{align}
\end{lemma}
\begin{proof}
Let $\phi' \colon A \to \Spec \mathbb{C}$. 
Since $A$ is smooth, there is a natural transformation
\begin{align}\label{nat:trans2}
    K^{\rm{top}}(\mathscr{D}) \to \phi'_{*}\mathcal{K}_A^{\rm{top}}(i_{*}\mathcal{D})
\end{align}
which is constructed in~\cite[Subsection~7.4]{Moulinos}, 
	also see~\cite[Proposition~2.13]{GS}. Then using 
	$i_{\ast}\mathcal{K}_B^{\rm{top}}(\mathscr{D})=\mathcal{K}_A^{\rm{top}}(i_{\ast}\mathscr{D})$ by Theorem~\ref{thm:phiproper},
	a desired morphism (\ref{nat:eta}) is obtained. The commutative diagram (\ref{nat:trans}) is obvious since (\ref{nat:trans2}) is a natural transformation.
\end{proof}

\begin{prop}\label{prop:gsection}
Let $h \colon M \to B$ be a proper morphism and consider the constant map $\phi\colon B\to \mathrm{Spec}\,\mathbb{C}$. 
	Let $\mathscr{D}$ be either $\mathrm{Perf}(M)$, $D^b(M)$ 
	or $\mathrm{MF}(M, f)$ for a non-zero function $f$ on $M$, 
	where we assume that $M$ is smooth in the last case. 
	Let $\mathscr{D}=\langle \mathcal{C}_1, \mathcal{C}_2 \rangle$ be a 
	$B$-linear semiorthogonal decomposition. 
	Then the maps 
	\begin{align*}
		\eta_i \colon K^{\rm{top}}(\mathcal{C}_i) \to \phi_{\ast}\mathcal{K}_B^{\rm{top}}(\mathcal{C}_i)
	\end{align*}
	in Lemma~\ref{lem:gsection} are isomorphisms in $\mathrm{Ho}(\mathrm{Sp})$, and independent 
    of a choice of a closed immersion $B\hookrightarrow A$. 
\end{prop}
\begin{proof}
	By Lemma~\ref{lem:gsection}, by taking a closed immersion $B\hookrightarrow A$ into smooth $A$, 
    there is a map 
	\begin{align}\label{nat:etaK}
		K^{\rm{top}}(\mathscr{D}) \to 
		\phi_{\ast}\mathcal{K}^{\rm{top}}_B(h_{\ast}\mathscr{D}). 
	\end{align}
	We first show that the above map is an equivalence and independent of a choice of 
    $B\hookrightarrow A$
    as morphisms in $\mathrm{Ho}(\mathrm{Sp})$.
	Suppose that $\mathscr{D}=\mathrm{Perf}(M)$. 
	Then, by Theorem~\ref{thm:phiproper} and 
	(\ref{isom:ktopalpha}), the map (\ref{nat:etaK}) is identified 
	with the natural map 
	\begin{align*}
		KU_M \to (\phi \circ h)_{\ast}\underline{KU}_M
	\end{align*}
	which is an equivalence as $\underline{KU}_M$ is a sheaf, and 
	independent of
    $B\hookrightarrow A$ in $\mathrm{Ho}(\mathrm{Sp})$.
    
	Suppose that $\mathscr{D}=D^b(M)$. 
	By (\ref{equiv:csupp}) and Lemma~\ref{lem:KBM} the map (\ref{nat:etaK}) is identified with 
	\begin{align}\label{Khomology}
		KU_{M, c}^{\vee} \to (\phi \circ h)_{\ast}\underline{KU}_{M, c}^{\vee}. 
	\end{align}
	Since $\underline{KU}_{M, c}^{\vee}$ is a sheaf by Lemma~\ref{lem:sheaf}, 
	the map (\ref{Khomology}) is an equivalence and independent of  
    $B\hookrightarrow A$ in $\mathrm{Ho}(\mathrm{Sp})$.
	
	Suppose that $\mathscr{D}=\mathrm{MF}(M, f)$. 
	We set $M_0=f^{-1}(0)$. 
	The category $\mathrm{MF}(M, f)$ is equivalent to 
	the singularity category of $M_0$, see
 Theorem~\ref{thm:orlov}. 
	Therefore, there is an exact sequence of dg-categories 
	\begin{align*}
		\mathrm{Perf}(M_0) \to D^b(M_0) \to \mathscr{D}. 
	\end{align*}
	There is a commutative diagram 
	\begin{align*}
		\xymatrix{
			K^{\rm{top}}(\mathrm{Perf}(M_0)) \ar[r]\ar[d] & K^{\rm{top}}(D^b(M_0)) \ar[r] \ar[d] & 
			K^{\rm{top}}(\mathscr{D}) \ar[d] \\
			\phi_{\ast}\mathcal{K}_B^{\rm{top}}(\mathrm{Perf}(M_0)) \ar[r] & \phi_{\ast}\mathcal{K}_B^{\rm{top}}(D^b(M_0))
			\ar[r] & \phi_{\ast}\mathcal{K}_B^{\rm{top}}(\mathscr{D}), 
		}
	\end{align*}
	where each horizontal sequence is a fiber sequence.
	By arguments as above, the left and middle vertical maps 
	are equivalences and independent of
    $B\hookrightarrow A$ in $\mathrm{Ho}(\mathrm{Sp})$. Therefore, the right vertical map is also an equivalence. 
	
	By Lemma~\ref{lem:sod:Ktheory} and the diagram (\ref{nat:trans}), 
 the map (\ref{nat:etaK}) is identified with 
	\begin{align*}
		\eta_1 \oplus \eta_2 
		\colon K^{\rm{top}}(\mathcal{C}_1)\oplus
		K^{\rm{top}}(\mathcal{C}_2)
		\to \phi_{\ast}\mathcal{K}_B^{\rm{top}}(\mathcal{C}_1) \oplus
		\phi_{\ast}\mathcal{K}_B^{\rm{top}}(\mathcal{C}_2)
	\end{align*}
    in $\mathrm{Ho}(\mathrm{Sp})$.
	Therefore both of $\eta_1$, $\eta_2$ are equivalences. 
\end{proof}

\section{The Chern character maps of quotient stacks}\label{s3}

In this section, we recall the definition of the Chern character maps for quotient stacks. The main tool we use for their construction is the approximation of cohomology theories of quotient stacks by varieties \cite{Totaro, EdidinGraham}. We also construct a Chern character map for quasi-smooth quotient stacks and discuss a Grothendieck-Riemann-Roch theorem for quasi-smooth morphisms. The results are most probably well-known to the experts, but we do not know a reference for them. 

In Appendix~A, we will derive the versions of topological GRR theorem from 
references in motivic stable homotopy theory. In this section, we assume known versions of the topological GRR theorem for maps of varieties, and we 
extend the topological GRR theorem to quotient stacks via approximations by varieties.

\subsection{The Chern character map}\label{subsec:chernmap}

Let $\X=X/G$ be a quotient stack such that $G$ is reductive and $X^{\mathrm{cl}}$ is quasi-projective. Let $M\subset G$ be a compact Lie group such that $G$ is the complexification of $M$.
We denote by \begin{align*}
 K^{\mathrm{top}}(\X):=K^{\mathrm{top}}_{M}(X), \ 
 G^{\rm{top}}(\X):=G^{\rm{top}}_{M}(\X)
 \end{align*}
 the $M$-equivariant topological K-theory of $X$ defined by Atiyah-Segal \cite{Segal}, 
 the $M$-equivariant topological K-homology of $X$ defined by Thomason \cite{ThomasonAS} (which is also referred to as the dual of compactly supported equivariant topological K-theory in the literature), respectively.  
We refer to \cite{BF02392091} for a brief review of properties of topological K-theory and K-homology of varieties and of the Grothendieck-Riemann-Roch theorems for varieties. 
For references on topological K-homology, see \cite{MR0679698} for the non-equivariant case and \cite[Subsection 2.1.2]{HLP} for the equivariant case.
By \cite[Theorem C and the remark following it]{HLP}, we have equivalences:
\begin{align}\label{HLPSiso}
    K^{\mathrm{top}}(\mathrm{Perf}(\X))\simeq K^{\mathrm{top}}(\X),\
   K^{\mathrm{top}}(D^b(\X))\simeq G^{\mathrm{top}}(\X).
\end{align}
Note that $K^{\mathrm{top}}(\X)=G^{\mathrm{top}}(\X)$ if $\X$ is smooth.
% We denote by $H^{\ast}(\X):=H^{\ast}(\X, \mathbb{Q})$ the $G$-equivariant singular cohomology of $X$ and by 
% $H_{\ast}^{\mathrm{BM}}(\X):=
% H_{\ast}^{\mathrm{BM}}(\X, \mathbb{Q})$ the $G$-equivariant Borel-Moore hoomology of $X$. 

For a quotient stack $\X$, we have Chern character maps for $i\in\mathbb{Z}$:
\begin{align*}
    \mathrm{ch}\colon K^{\mathrm{top}}_i(\X)\to \prod_{j\in\mathbb{Z}}H^{i+2j}(\X),\
    \mathrm{ch}\colon G^{\mathrm{top}}_i(\X)\to \prod_{j\in\mathbb{Z}}H_{i+2j}^{\mathrm{BM}}(\X).
\end{align*}
The construction of the above Chern character map will be recalled in 
Subsection~\ref{section31}. 
As we mentioned in the previous section, 
the above Chern character maps are 
isomorphisms over $\mathbb{Q}$ 
when $\X$ is a complex variety, 
or more generally for a scheme of finite 
type over $\mathbb{C}$. 
%The first one is the usual Atiyah-Hirzebruch theorem. The second one follows as the dual of the analogous isomorphism for compactly supported topological K-theory, see \cite[Section 3.6 and Section 6]{MR2762553} or \cite[Section 11]{MR0679698}. 
However, for a quotient stack $\X$, the 
above Chern character maps are not necessarily isomorphisms. 
In Section \ref{section31}, we review the approximation of the Chern character of quotient stacks by Chern characters for varieties following Edidin--Graham \cite{EdidinGraham}.

\begin{remark}
The above Chern character maps can be also obtained from the Chern character 
maps
\[\ch \colon K^{\mathrm{top}}_i(-)\to \mathrm{HP}_i(-)\] from topological K-theory to periodic cyclic homology applied to the dg-categories $\mathrm{Perf}(\X)$ and $D^b(\X)$, respectively, see \cite[Section 4.4]{Blanc}.
\end{remark}

\subsection{Approximation of stacks by varieties}\label{subsub25}

In the study of cohomology theories for quotient stacks, it is useful to 
approximate quotient stacks by varieties. We use the method of Totaro \cite{Totaro}, Edidin--Graham \cite{EdidinGraham}. 
We exemplify the method for Borel-Moore homology and singular cohomology, but it can be applied in many other situations such as equivariant Chow groups, see loc. cit., approximation of the algebraic or topological Chern character, see loc. cit. and Subsection \ref{section31}, or vanishing cohomology, see \cite[Subsection 2.2]{DM}.

Let $\X=X/G$ be a quotient stack with $G$ an algebraic group and $X$ quasi-projective scheme with a $G$-linearized
action. 
%Let $V$ be a representation of $G$ and let $S\subset X\times V$ be the closed set of points with non-trivial stabilizer. 
We choose $G$-representations
$\{V_n\}_{n\in \mathbb{N}}$ with surjections 
$V_n\twoheadrightarrow V_{n-1}$ such that
the closed subset $S_n \subset V_n$ 
of points with non-trivial stabilizers 
has codimension at least $n$. Further we may choose $V_n$ such that, for $U_n:=V_n\setminus S_n$, the quotient $U_n/G$ is a scheme~\cite[Lemma 9]{EdidinGraham2}. 
Then the quotient $(X\times U_n)/G$ is also a scheme because $X$ is quasi-projective \cite[Proposition 23]{EdidinGraham2}.
For $\ell$ fixed and for $n$ large enough, there are isomorphisms of Borel-Moore homologies 
induced by pullbacks
\begin{align}\label{isom:BM:pullback}
H^{\mathrm{BM}}_\ell(\mathscr{X})\xrightarrow{\cong}
 H^{\mathrm{BM}}_{\ell+2\dim V_n}((X\times V_n)/G)\xrightarrow{\cong}  H^{\mathrm{BM}}_{\ell+2\dim V_n}\left((X\times U_n)/G\right).
 \end{align}
 We also have isomorphisms of cohomologies 
 induced by pull-backs
 \begin{align}\label{isom:coh:pullback}
H^{\ell}(\mathscr{X})\xrightarrow{\cong} H^\ell((X\times V_n)/G)\xrightarrow{\cong} H^\ell\left((X\times U_n)/G\right).
\end{align}

\subsection{The Chern character map for a classical quotient stack}\label{section31}

Let $\X=X/G$ be a quotient stack, 
 where $G$ is a connected reductive group and $X$ is a classical quasi-projective scheme with an action of $G$. Let $M$ be a compact Lie group such that $G$ is the complexification of $M$.
% Recall that $G^{\mathrm{top}}_i(\mathscr{X})$ the (rational) topological K-theory of $D^b(\mathscr{X})$. 
Let $EM$ be a contractible CW complex with a free action of $M$. For $i\in \mathbb{Z}$, 
the Chern character map of the CW complex $EM\times_M X$
defines the Chern character map for $\X$:
\begin{align}\label{chernstack}
\mathrm{ch}\colon G^{\mathrm{top}}_i(\mathscr{X})&=G^{\mathrm{top}}_i(EM\times_M X) \\
\notag&\to \widetilde{H}^{\mathrm{BM}}_i(EM\times_M X)=\widetilde{H}^{\mathrm{BM}}_i(\mathscr{X}):=\prod_{j\in\mathbb{Z}}H^{\mathrm{BM}}_{i+2j}(\mathscr{X}).
\end{align}

\begin{remark}
The above Chern character map is, in general, neither injective nor surjective. It becomes an isomorphism when we take the completion of the 
the K-group with respect to an augmentation ideal by theorems of Atiyah--Segal \cite{AtiyahSegal}, and Edidin--Graham \cite{EdidinGraham} in the algebraic case. 
\end{remark}
% Further, \eqref{chernstack} is not injective if $G$ is allowed to be an arbitrary group, for example for $\X=BG$ for $G$ finite and non-trivial. We explain that \eqref{chernstack} is injective if $G$ is a connected reductive group. We will deduce the statement by convenient approximations of the stack $\X$ by varieties (and thus finite CW complexes for which the Chern character is an isomorphism), following Edidin-Graham \cite{EdidinGraham}.

The Chern character map for a stack can be approximated by the Chern character map for varieties as follows \cite{EdidinGraham}, see also Subsection \ref{subsub25}.
For $V$ a representation of $G$, denote by $S\subset V$ the closed set of points with non-trivial stabilizers. Let $U:=V\setminus S$. We may choose $V$ such that $U/G$ and $(X\times U)/G$ are schemes.
Then the following diagram commutes, where the vertical maps are pullback maps and the bottom map is an isomorphism
over $\mathbb{Q}$ by the Atiyah-Hirzebruch theorem:
\begin{equation*}
    \begin{tikzcd}
        G^{\mathrm{top}}_i(\mathscr{X})\arrow[r, "\mathrm{ch}"]\arrow[d, "\sim" {anchor=south, rotate=90}]\arrow[dd,bend right=75, "\mathrm{res}_V"']& \widetilde{H}^{\mathrm{BM}}_i(\mathscr{X})\arrow[d, "\sim" {anchor=south, rotate=90}]\arrow[dd,bend left=75, "\mathrm{res}_V"]\\
        G^{\mathrm{top}}_i((X\times V)/G)\arrow[r, "\mathrm{ch}"]\arrow[d, "\mathrm{res}"]& \widetilde{H}^{\mathrm{BM}}_i((X\times V)/G)\arrow[d, "\mathrm{res}"]\\
        G^{\mathrm{top}}_i\left((X\times U)/G\right)\arrow[r, "\mathrm{ch}_V"]& \widetilde{H}^{\mathrm{BM}}_i\left((X\times U)/G\right).
    \end{tikzcd}
\end{equation*}
Let $\{V_n\}_{n\in \mathbb{N}}$ 
be $G$-representations 
with surjections $V_n\twoheadrightarrow V_{n-1}$, 
and take closed subsets $S_n \subset V_n$ as in Subsection~\ref{subsub25}.
For $\ell$ fixed and for $n$ large enough, recall that 
we have isomorphisms (\ref{isom:BM:pullback}) induced by pullbacks. 
Then $\mathrm{ch}(y)$ 
for $y \in G_i^{\rm{top}}(\X)$
equals the limit 
\begin{align}\label{equal:limit}
    \ch(y)=\lim_{n\to \infty} \ch_{V_n}(\mathrm{res}_{V_n}(y)).
\end{align}
\begin{remark}
    In the algebraic case, Edidin--Graham show in \cite[Proposition 3.1]{EdidinGraham} that the limit of $\mathrm{ch}_{V_n}(\mathrm{res}_{V_n}(y))$ is well-defined and use it to define the Chern character. 
\end{remark}
In a similar way, there is also a Chern character map for $i\in\mathbb{Z}$:
\begin{equation}\label{chernstack2}
\mathrm{ch}\colon K^{\mathrm{top}}_i(\mathscr{X})\to \widetilde{H}^{i}(\mathscr{X}):=\prod_{j\in\mathbb{Z}}H^{i+2j}(\mathscr{X}).
\end{equation} 
As above, the Chern character map \eqref{chernstack2} can be approximated by Chern character maps of varieties.
The Chern character maps \eqref{chernstack} and \eqref{chernstack2} are compatible 
by the following commutative diagram. 
\begin{equation}\label{chernKG}
    \begin{tikzcd}
        K^{\mathrm{top}}_i(\X)\arrow[d, "\mathrm{ch}"']\arrow[r, "\varepsilon'"]& G^{\mathrm{top}}_i(\X)\arrow[d, "\mathrm{ch}"]\\
        \widetilde{H}^i(\X)\arrow[r, "\varepsilon"]& \widetilde{H}^{\mathrm{BM}}_{i}(\X).
    \end{tikzcd}
\end{equation}
Here $\varepsilon$ and $\varepsilon'$ are the maps induced by intersecting with the fundamental class, see \cite[Section 5 and Property 2 from Section 4.1]{BF02392091}, or the diagram of spectra (\ref{com:fiber}). 

%where the maps $\varepsilon$ are the maps induced by intersecting with the fundamental class, see \cite[Section 5 and Property 2 from Section 4.1]{BF02392091}.

\subsection{The cycle maps for quotient stacks}

%Let $\omega$ be the dualizing complex on $\mathscr{X}$. Then $H^{\mathrm{BM}}_{\ast}(\X)=H^{\ast}(\X, \omega)$.
Let $\X$ be a classical quotient stack as in the previous subsection.
For $i\in\mathbb{Z}$, we consider the increasing filtration 
\begin{equation}\label{def:filtration}
E_{ \ell}G^{\mathrm{top}}_i(\mathscr{X}):=\mathrm{ch}^{-1}\left(H^{\mathrm{BM}}_{\leq i+2\ell}(\mathscr{X})\right)\subset G^{\mathrm{top}}_i(\mathscr{X}).
\end{equation}
Note that $E_{\ell}G^{\mathrm{top}}_i(\mathscr{X})=G^{\mathrm{top}}_i(\mathscr{X})$ for $\ell$ large enough and that $E_{\ell}G^{\mathrm{top}}_i(\mathscr{X})$ contains the kernel of the Chern character for all $\ell$. 
The associated graded of the above filtration is denoted by 
\[\mathrm{gr}_\ell G^{\mathrm{top}}_i(\mathscr{X}):=E_{\ell}G^{\mathrm{top}}_i(\mathscr{X})/E_{\ell-1}G^{\mathrm{top}}_i(\mathscr{X}).\]
The Chern character induces a map, which we call \textit{the cycle map}:
\begin{equation}\label{cherngraded}
\mathrm{c}\colon \mathrm{gr}_\ell G^{\mathrm{top}}_i(\mathscr{X})\to H^{\mathrm{BM}}_{i+2\ell}(\mathscr{X}).
\end{equation}
Note that the cycle map is injective by construction.
\begin{example}\label{exam:BGL}
Let $\X=BGL(d)$. Then the Chern character map is 
\begin{align*}
    \ch \colon G_0^{\rm{top}}(\X)=\mathbb{Z}[q_1^{\pm 1}, \ldots, q_d^{\pm 1}]^{\mathfrak{S}_d}
    \to \widetilde{H}_0^{\rm{BM}}(\X)=\mathbb{Q}[[h_1, \ldots, h_d]]^{\mathfrak{S}_d}
\end{align*}
given by $q_i \mapsto e^{h_i}$. 
Let $I \subset \mathbb{Z}[q_1^{\pm 1}, \ldots, q_d^{\pm 1}]$ be the ideal 
generated by $q_i-1$ for $1\leq i\leq d$. 
The induced filtration on 
$G_0^{\rm{top}}(\X)$ is 
\begin{align*}
    \cdots \subset I^2 \cap G_0^{\rm{top}}(\X) \subset 
    I \cap G_0^{\rm{top}}(\X) \subset G_0^{\rm{top}}(\X). 
\end{align*}
\end{example}

For $\X$ a quotient stack as above, recall that there is an intersection product 
(\ref{intersect}). We note the following immediate statement.

\begin{prop}\label{propchprime}
Let $\alpha\in \widetilde{H}^0(\X)$ 
be of the form $\alpha=1+\alpha'$ for $\alpha'\in \prod_{i\geq 1}H^i(\X)$. 
The twisted Chern character map 
$\mathrm{ch}'(-):=\mathrm{ch}(-)\cdot\alpha$
induces a map
\begin{align*}
    c' \colon \mathrm{gr}_{\ell}G_i^{\rm{top}}(\X) \to H_{i+2\ell}^{\rm{BM}}(\X)
\end{align*}
which equals the cycle map \eqref{cherngraded}.
\end{prop}

\subsection{The Grothendieck-Riemann-Roch theorem}

A morphism $f:X\to Y$ of schemes over $\mathbb{C}$ is called a smoothable lci morphism if it factors $X\xrightarrow{j} Z\xrightarrow{p} Y$, where $j$ is a regular closed immersion and $p$ is a smooth morphism.
We state the (topological) Grothendieck-Riemann-Roch (GRR) theorem for smoothable lci morphisms of (classical and possibly singular) stacks (for varieties, we refer to Appendix~\ref{subsec:append}), which extends the works of Baum--Fulton--MacPherson \cite{MR0412190, BF02392091} to higher topological $G$-theory.

\begin{thm}\label{GRRS}
Let $\X$ and $\mathscr{Y}$ be classical quotient stacks and let $f\colon \X\to\mathscr{Y}$ be a representable morphism. 

(i) The following diagrams commute, where $f$ is assumed to be smoothable lci in the first diagram:
\begin{equation}\label{diagpullb}
    \begin{tikzcd}
        G^{\mathrm{top}}_{\bullet}(\mathscr{Y})\arrow[d, "\mathrm{ch}_\mathscr{Y}"']\arrow[r, "f^*"]& G^{\mathrm{top}}_{\bullet}(\mathscr{X})\arrow[d, "\mathrm{ch}_\mathscr{X}"]\\
        \widetilde{H}^{\mathrm{BM}}_{\bullet}(\mathscr{Y})\arrow[r, "f^*"]& \widetilde{H}^{\mathrm{BM}}_{\bullet}(\mathscr{X}).
    \end{tikzcd}
    \text{ and }
    \begin{tikzcd}
        K^{\mathrm{top}}_{\bullet}(\mathscr{Y})\arrow[d, "\mathrm{ch}_\mathscr{Y}"']\arrow[r, "f^*"]& K^{\mathrm{top}}_{\bullet}(\mathscr{X})\arrow[d, "\mathrm{ch}_\mathscr{X}"]\\
        \widetilde{H}^{\bullet}(\mathscr{Y})\arrow[r, "f^*"]& \widetilde{H}^{\bullet}(\mathscr{X}).
    \end{tikzcd}
\end{equation}

(ii) Assume that $f$ is proper and smoothable lci. 
Let $\mathbb{T}_f$ be the tangent complex of $f$ and consider its Todd class $\mathrm{Td}(\mathbb{T}_f)\in \widetilde{H}^0(\X)$. 
By setting 
\[f'_*(-):=f_*(\mathrm{Td}(\mathbb{T}_f)\cdot(-)),\] 
the following diagrams commute, where $\X$ and $\mathscr{Y}$ are assumed to be smoothable lci in the first diagram:
\begin{equation}\label{diagpushf}
    \begin{tikzcd}
        G^{\mathrm{top}}_{\bullet}(\mathscr{X})\arrow[d, "\mathrm{ch}_\mathscr{X}"']\arrow[r, "f_*"]& G^{\mathrm{top}}_{\bullet}(\mathscr{Y})\arrow[d, "\mathrm{ch}_\mathscr{Y}"]\\
        \widetilde{H}^{\mathrm{BM}}_{\bullet}(\mathscr{X})\arrow[r, "f'_*"]& \widetilde{H}^{\mathrm{BM}}_{\bullet}(\mathscr{Y}).
    \end{tikzcd} \text{ and }
    \begin{tikzcd}
        K^{\mathrm{top}}_{\bullet}(\mathscr{X})\arrow[d, "\mathrm{ch}_\mathscr{X}"']\arrow[r, "f_*"]& K^{\mathrm{top}}_{\bullet}(\mathscr{Y})\arrow[d, "\mathrm{ch}_\mathscr{Y}"]\\
        \widetilde{H}^{\bullet}(\mathscr{X})\arrow[r, "f'_*"]& \widetilde{H}^{\bullet}(\mathscr{Y}).
    \end{tikzcd}
\end{equation}
\end{thm}

\begin{proof}
When $\X$ and $\mathscr{Y}$ are schemes, both of (i) and (ii) follow from the corresponding diagrams for spectra (\ref{GRR:top}). 
For quotient stacks, these results follow by the approximations
by varieties. 
We can write $\X=X/G$ and $\mathscr{Y}=Y/G$ for varieties $X$ and $Y$ as $f$ is representable. 
We choose
$G$-representations $V_n$ with open subsets $U_n \subset V_n$ as in
Subsection~\ref{section31}. Set
\[
\X_n:=(X\times U_n)/G, \qquad \mathscr{Y}_n:=(Y\times U_n)/G.
\]
For $n \gg 0$, the morphism $f\colon X \to Y$ induces a morphism
\[
f_n \colon \X_n \to \mathscr{Y}_n
\]
of schemes of the same type as $f$. By the scheme case proved above, the corresponding Grothendieck--Riemann--Roch
diagram for $f_n$ is commutative.

Moreover, the transition morphisms in the approximation system are induced by pull-backs
along vector bundle maps, hence the identifications
\[
H^\ast(\mathscr{Y})\cong H^\ast(\mathscr{Y}_n), \quad H^\ast(\X)\cong H^\ast(\X_n),
\]
\[
H^{\mathrm{BM}}_\ast(\mathscr{Y})\cong H^{\mathrm{BM}}_{\ast+2\dim V_n}(\mathscr{Y}_n), \quad
H^{\mathrm{BM}}_\ast(\X)\cong H^{\mathrm{BM}}_{\ast+2\dim V_n}(\X_n)
\]
for $n \gg 0$ are compatible with pull-backs, proper push-forwards, and products by
Todd classes. By the construction of the Chern character for quotient stacks in Subsection~\ref{section31}
as the stable limit of the Chern characters for the schemes $X_n$ and $Y_n$, the commutative
diagrams for $f_n$ pass to the limit and yield the desired commutative diagrams for $f$.

\end{proof}

\subsection{The Chern character map for quasi-smooth stacks}\label{subsec331}

Let $\X=X/G$ be a quotient stack with $X$ a quasi-smooth derived 
scheme. 
We define 
\begin{align*}
K^{\mathrm{top}}(\X):=K^{\mathrm{top}}(\mathrm{Perf}(\X)), \ 
    G^{\mathrm{top}}(\X):=K^{\mathrm{top}}(D^b(\X)). 
\end{align*}
Note that the closed immersion $\iota \colon \X^{\mathrm{cl}} \hookrightarrow \X$ induces an equivalence 
\begin{align*}
\iota_{*} \colon G^{\mathrm{top}}(\X^{\mathrm{cl}}) 
\stackrel{\sim}{\to} G^{\mathrm{top}}(\X).
\end{align*}
We define the Chern character from the classical truncation $\X^{\mathrm{cl}}$;
\begin{align}\label{chquasismooth}
&\ch \colon G_{\bullet}^{\mathrm{top}}(\X) \simeq G_{\bullet}^{\mathrm{top}}(\X^{\mathrm{cl}})\to \widetilde{H}_{\bullet}^{\mathrm{BM}}(\X):=\widetilde{H}_{\bullet}^{\mathrm{BM}}(\X^{\mathrm{cl}}), \\
&\notag \ch \colon K_{\bullet}^{\mathrm{top}}(\X) \stackrel{\iota^*}{\to} K_{\bullet}^{\mathrm{top}}(\X^{\mathrm{cl}})\to \widetilde{H}^{\bullet}(\X):=\widetilde{H}^{\bullet}(\X^{\mathrm{cl}}).
\end{align}
We have the following: 

\begin{prop}\label{propo36}
Let $\X$ and $\mathscr{Y}$ be quasi-smooth stacks and let $f\colon \X\to\mathscr{Y}$ be a representable morphism. 

(i) Suppose that $f$ is quasi-smooth and proper. Let $\mathbb{T}_f$ be the tangent complex of $f$ and consider the Todd class $\mathrm{Td}(\mathbb{T}_f)\in \widetilde{H}^0(\X)$. By setting \[f'_*(-):=f_*(\mathrm{Td}(\mathbb{T}_f)\cdot(-)),\] 
the following diagram commutes:
\begin{equation}\label{GRR}
    \begin{tikzcd}
        G^{\mathrm{top}}_\bullet(\X)\arrow[d, "\mathrm{ch}_\X"']\arrow[r, "f_*"]& G^{\mathrm{top}}_\bullet(\mathscr{Y})\arrow[d, "\mathrm{ch}_{\mathscr{Y}}"]\\
        \widetilde{H}^{\mathrm{BM}}_\bullet(\X)\arrow[r, "f'_*"]& \widetilde{H}^{\mathrm{BM}}_\bullet(\mathscr{Y}).
    \end{tikzcd}
\end{equation}

(ii) Suppose that $f$ is smooth. Then the following diagram commutes:
\begin{equation}\label{diagpull}
    \begin{tikzcd}
        G^{\mathrm{top}}_\bullet(\mathscr{Y})\arrow[d, "\mathrm{ch}_\mathscr{Y}"']\arrow[r, "f^*"]& G^{\mathrm{top}}_\bullet(\mathscr{X})\arrow[d, "\mathrm{ch}_{\mathscr{X}}"]\\
        \widetilde{H}^{\mathrm{BM}}_\bullet(\mathscr{Y})\arrow[r, "f^*"]& \widetilde{H}^{\mathrm{BM}}_\bullet(\mathscr{X}).
    \end{tikzcd}
\end{equation}
\end{prop}
\begin{proof}
    When $\X$ and $\mathscr{Y}$ are quasi-smooth derived schemes, both of (i) and (ii) follow from the corresponding diagrams for spectra (\ref{GRR2}). 
    For quotient stacks, these follow by the approximations by varieties
    of the classical truncations of $\X$ and $\mathscr{Y}$
    as in the proof of Theorem~\ref{GRRS}. 
\end{proof}

Similarly to (\ref{def:filtration}), 
the Chern character map (\ref{chquasismooth})
defines a filtration 
\begin{align*}
    E_{\ell}G_i^{\rm{top}}(\X) \subset G_i^{\rm{top}}(\X)
\end{align*}
and the cycle map 
from the associated graded
\begin{align}\label{cherngraded3}
    c \colon \mathrm{gr}_{\ell}G_i^{\rm{top}}(\X) \to H_{i+2\ell}^{\rm{BM}}(\X). 
\end{align}
%The above cycle map is also an isomorphism over 
%$\mathbb{Q}$ 
%by a relative version of Proposition \ref{prop:cherniso} and Proposition \ref{propchprime}. 

%We have that $\mathrm{Td}(T_f)=1+x\in \widetilde{H}^0(\X)$ for $x\in \prod_{i\geq 2}H^i(\X)$. 
%Define a filtration on $G^{\mathrm{top}}_\bullet(\X)$, the associated graded, and a cycle map as in \eqref{def:filtration2} and \eqref{cherngraded2}.
We record the following corollary of the diagrams \eqref{GRR} and \eqref{diagpull}, see also Proposition \ref{propchprime}.

\begin{cor}\label{functoriality:Ktop2}
Let $f\colon \X\to \mathscr{Y}$ be a quasi-smooth morphism of quasi-smooth quotient stacks of relative dimension $d$. Let $i, \ell\in\mathbb{Z}$.
If $f$ is smooth, then it induces a pullback map:
\begin{align*}
    f^*\colon \mathrm{gr}_\ell G^{\mathrm{top}}_i(\mathscr{Y})\to \mathrm{gr}_{\ell+d} G^{\mathrm{top}}_i(\X).
\end{align*}
If $f$ is proper, then it induces a pushforward map:
\begin{align*}
    f_*\colon \mathrm{gr}_\ell G^{\mathrm{top}}_i(\X)\to \mathrm{gr}_\ell G^{\mathrm{top}}_i(\mathscr{Y}).
\end{align*}
\end{cor}

We also note the following lemma to be used later: 
\begin{lemma}\label{prop4zero}
Let $\X$ be a smooth quotient stack and set
\[
\X[\epsilon]:=\X\times \Spec \C[\epsilon],
\qquad \deg \epsilon=-1.
\]
Then the natural map
\[
\varepsilon' \colon
K_i^{\mathrm{top}}(\X[\epsilon])
\to
G_i^{\mathrm{top}}(\X[\epsilon])
\]
induced by the inclusion
\[
j\colon \mathrm{Perf}(\X[\epsilon])\hookrightarrow D^b(\X[\epsilon])
\]
is zero.
\end{lemma}

\begin{proof}
Let
$
i\colon \X\hookrightarrow \X[\epsilon]$
be the closed immersion defined by $\epsilon=0$.
Since $\C[\epsilon]=\C\oplus \C[-1]$ is a square-zero extension, for any
$P\in \mathrm{Perf}(\X[\epsilon])$, there is a functorial cofiber sequence in
$D^b(\X[\epsilon])$
\[
i_*i^*P[1]\to P\to i_*i^*P.
\]
Hence, in the stable $\infty$-category
\[
\mathrm{Fun}^{\mathrm{ex}}(\mathrm{Perf}(\X[\epsilon]),D^b(\X[\epsilon])),
\]
there is a cofiber sequence of exact functors
\[
i_*i^*[1]\to j\to i_*i^*.
\]
Applying the additive invariant $K^{\mathrm{top}}$ and using additivity for
pointwise cofiber sequences of exact functors, we obtain
\[
[K^{\mathrm{top}}(j)]
=
[K^{\mathrm{top}}(i_*i^*[1])]
+
[K^{\mathrm{top}}(i_*i^*)]
\]
in
\[
\pi_0\mathrm{Map}(
K^{\mathrm{top}}(\mathrm{Perf}(\X[\epsilon])),
K^{\mathrm{top}}(D^b(\X[\epsilon]))
).
\]
Since the shift functor $[1]$ acts by $-1$ on any additive invariant, we have
\[
[K^{\mathrm{top}}(i_*i^*[1])]
=
-[K^{\mathrm{top}}(i_*i^*)].
\]
Hence
$
[K^{\mathrm{top}}(j)]=0$,
so $K^{\mathrm{top}}(j)$ is nullhomotopic.
\end{proof}

\section{Topological K-theory of categories of matrix factorizations}\label{s4}

In this section, we compute the topological K-theory of categories of matrix factorizations in terms of the monodromy invariant vanishing cycle cohomologies. The results are inspired by the work of Preygel \cite[Section 6]{Preygelthesis}, Efimov \cite{Eff}, Blanc--Robalo--Toën--Vesozzi \cite{BRTV}, and Brown--Dyckerhoff \cite{BD}.
\smallskip

\subsection{Monodromy invariant vanishing cycles}\label{subsection:vanishing}
In this subsection, we discuss the monodromy invariant vanishing cycles. 
Let $\X$ be a quotient stack 
with a regular function $f \colon \X \to \mathbb{C}$. 
We denote by 
$\iota \colon \X_0 \hookrightarrow \X$
the (derived) zero locus of $f$. 
Recall the vanishing cycle functor $\varphi_f$ and 
the monodromy operator $\mathrm{T}$, see Subsection~\ref{subsec:nearby}. 
\begin{defn}\label{def:moninv}
For $F \in D(\mathrm{Sh}_{\mathbb{Q}}(\X))$, 
we define its
\textit{monodromy invariant vanishing cycle}
$\varphi_f^{\rm{inv}}(F) \in D(\mathrm{Sh}_{\mathbb{Q}}(\X_0))$ to be the cone of $1-\mathrm{T}$:
\begin{align}\label{thirdcolumn}
    \varphi_f(F)\xrightarrow{1-\mathrm{T}}\varphi_f(F)\to \varphi_f^{\mathrm{inv}}(F)\to \varphi_f(F)[1].
\end{align}
\end{defn}
The push-forward
of $\varphi_f^{\rm{inv}}(F)$ 
via $\iota \colon \X_0 \hookrightarrow \X$
is also denoted by $\varphi_f^{\mathrm{inv}}(F)$. 
For $F=\mathbb{Q}_{\X}$, we write 
$\varphi_f^{\rm{inv}}:=\varphi_f^{\rm{inv}}(\mathbb{Q}_{\X})$
for simplicity. 
It is supported on $\mathrm{Crit}(f) \subset \X$ 
if $\X$ is smooth. 

\begin{defn}
We define $H^\bullet(\X, \varphi_f(F))^{\mathrm{inv}}$
and $H^\bullet(\X, \varphi_f(F))_{\mathrm{inv}}$ to be 
\begin{align*}
    H^\bullet(\X, \varphi_f(F))^{\mathrm{inv}}&:=\text{ker}(1-\mathrm{T})\subset H^\bullet(\X, \varphi_f(F)),\\
    H^\bullet(\X, \varphi_f(F))_{\mathrm{inv}}&:=H^\bullet(\X, \varphi_f(F))/\text{image}(1-\mathrm{T}).
\end{align*}
\end{defn}
\begin{remark}
Directly from the above definition, there are 
  short exact sequences:
\begin{equation}\label{SESC}
0\to H^i(\X, \varphi_f(F))_{\mathrm{inv}}\to H^i(\X, \varphi_f^{\mathrm{inv}}(F))\to  H^{i+1}(\X, \varphi_f(F))^{\mathrm{inv}}\to 0.
\end{equation}
Moreover, there is a (non-canonical) isomorphism 
\begin{align}\label{minv:isom}
    H^i(\X, \varphi_f(F))_{\rm{inv}} \cong H^i(\X, \varphi_f(F))^{\rm{inv}}.
\end{align}
\end{remark}

Below we assume that $\X$ is smooth, and without loss of generality we assume that $\X$ is connected. If $f$ is non-zero, 
$\X_0$ is a classical lci stack. In this case, 
let $\alpha$ be the map \[\alpha\colon \mathbb{Q}_{\X_0}=\iota^*\mathbb{Q}_{\X}\to \iota^!\mathbb{Q}_{\X}[2]=\omega_{\X_0}[-2\dim \X_0]\] given by capping with the fundamental class of the lci stack $\X_0$. 
%If $f$ is not the zero map, then this is the usual construction. 
If $f$ is the zero map, 
then $\X_0\cong \X\times \Spec \mathbb{C}[\epsilon]$ 
for $\epsilon$ in $\deg \epsilon=-1$. In this case, 
the map $\alpha$ is defined to be the zero map. 
Or equivalently, it equals to the capping with the virtual fundamental 
class of $\X_0$, which is zero. 

\begin{remark}\label{rmk:zero}
In this paper, we often treat the case of $f=0$ separately. This is because
that, for a connected quotient stack $\X$, the stack $\X_0$ is a classical lci stack if and only if $f\neq 0$. Otherwise, it is a quasi-smooth stack, 
and some references do not cover such a case. Nevertheless, it is 
a very simple quasi-smooth stack, namely $\X_0=\X\times \Spec \mathbb{C}[\epsilon]$, and the required properties will be easily checked separately. 
\end{remark}

\begin{remark}\label{rmk:comzero}
By Lemma~\ref{prop4zero} together with the above definition of $\alpha$, the diagram (\ref{chernKG}) also commutes for $\X[\epsilon]$, indeed both horizontal arrows are zero in this case.
\end{remark}

\begin{lemma}\label{lem:dist:minv}
There is a distinguished triangle 
\begin{equation}\label{defC}
\iota_*\mathbb{Q}_{\X_0}\xrightarrow{\alpha}\iota_*\iota^!\mathbb{Q}_{\X}[2]\to \varphi_f^{\mathrm{inv}}\to \iota_*\mathbb{Q}_{\X_0}[1].
\end{equation}
\end{lemma}
\begin{proof}
We recall two distinguished triangles relating the vanishing and cycle functors applied to the constant sheaf. 
A reference is \cite[Chapter 3]{Massey}, especially \cite[pages 24-28]{Massey}. The results in loc. cit. are stated for varieties, but they also hold for quotient stacks using the argument in \cite[Subsection 2.2]{DM}.
By the definition of vanishing cycle sheaves, 
there is an exact triangle in $D^b(\mathrm{Sh}_{\mathbb{Q}}(\X))$: 
\begin{equation}\label{defcan}
\iota_*\mathbb{Q}_{\X_0}[-1]\to \psi_f[-1]:=\psi_f\mathbb{Q}_\X[-1]\xrightarrow{\mathrm{can}} \varphi_f[-1]:=\varphi_f\mathbb{Q}_\X[-1]\rightarrow \iota_*\mathbb{Q}_{\X_0}.
\end{equation}
By taking the Verdier dual of the above triangle, we obtain the distinguished triangle:
\begin{equation}\label{defvar}
\varphi_f[-1]\xrightarrow{\mathrm{var}} \psi_f[-1]\to \iota_*\iota^!\mathbb{Q}_{\X}[1]\rightarrow \varphi_f.
\end{equation}
We have that $\mathrm{can}\circ\mathrm{var}=1-\mathrm{T}$, where $\mathrm{T}$ is the monodromy operator acting on $\varphi_f$.

By combining the above 
sequences, we obtain the 
following diagram:
\begin{equation}\label{diagsec4}
    \begin{tikzcd}
        \iota_*\mathbb{Q}_{\X_0}\arrow[r, "\alpha"]& \iota_*\iota^!\mathbb{Q}_{\X}[2]\arrow[r]& \varphi_f^{\mathrm{inv}}\\
        \iota_*\mathbb{Q}_{\X_0}\arrow[u, "\text{id}"]\arrow[r]& \psi_f \arrow[u]\arrow[r, "\mathrm{can}"]& \varphi_f\arrow[u]\\
        0\arrow[u]\arrow[r]& \varphi_f\arrow[u, "\mathrm{var}"]\arrow[r, "\sim"]& \varphi_f\arrow[u, "1-\mathrm{T}"'].
    \end{tikzcd}
\end{equation}
In the above diagram, (up to shift) the second row is \eqref{defcan}, the second column is \eqref{defvar}, and the third column is \eqref{thirdcolumn}.
Therefore the first row is also a distinguished triangle. 
\end{proof}

\begin{remark}\label{rmk:cofib}
By Lemma~\ref{lem:dist:minv}, we have 
\begin{align*}
    \varphi_f^{\mathrm{inv}} \simeq \mathrm{Cofib}(\alpha)=\omega_{\X_0}^{\mathrm{sg}}
\end{align*}
where $\omega_{\X_0}^{\mathrm{sg}}$ is given in (\ref{omega:sg}). 
    
\end{remark}

\subsection{The Chern character map for matrix factorizations}\label{subsec:ChernMF}
Let $\X$ be a smooth quotient stack with a regular 
function $f \colon \X \to \mathbb{C}$. 
In this subsection, we construct a Chern character map
for the category of matrix factorizations, where 
the target is the cohomology of monodromy invariant vanishing cycles. 

We first consider the case that $\X$ is a smooth quasi-projective variety. 
The category $\mathrm{MF}(\X, f)$ is linear over $\mathrm{Perf}(\X)$, 
and we consider its relative topological K-theory 
\begin{align*}
	\mathcal{K}_{\X}^{\rm{top}}(\mathrm{MF}(\X, f)) \in \mathrm{Sh}_{\mathrm{Sp}}(\mathcal{X}^{\mathrm{an}}).
	\end{align*}
 Its rationalization gives an object in $D(\mathrm{Sh}_{\mathbb{Q}}(\X^{\rm{an}}))$. 
 We note the following lemma: 
\begin{lemma}\label{lem:sheaf:K}
	Suppose that $\X$ is a smooth quasi-projective variety. 
	There is a canonical equivalence
 in $D(\mathrm{Sh}_{\mathbb{Q}}(\X))$:
	\begin{align}\label{equiv:KMF}
		\mathcal{K}_{\X}^{\rm{top}}(\mathrm{MF}(\X, f))_{\mathbb{Q}} \simeq \varphi_f^{\rm{inv}}[\beta^{\pm 1}]. 
	\end{align}
	\end{lemma}
\begin{proof}
From the equivalence (\ref{dsgmf}), there is an exact sequence of dg-categories 
\begin{align}\label{exact:dgMF}
	\mathrm{Perf}(\X_0) \to D^b(\X_0) \to \mathrm{MF}(\X, f). 
	\end{align}
Therefore, we obtain the commutative diagram 
\begin{align*}
	\xymatrix{
	\mathcal{K}_{\X}^{\rm{top}}(\mathrm{Perf}(\X_0))_{\mathbb{Q}} \ar[r] \ar[d]_-{\sim} & \mathcal{K}_{\X}^{\rm{top}}(D^b(\X_0))_{\mathbb{Q}} \ar[r] \ar[d]_-{\sim} 
	& \mathcal{K}_{\X}^{\rm{top}}(\mathrm{MF}(\X, f))_{\mathbb{Q}} \ar@{.>}[d] \\
	\iota_{\ast}\mathbb{Q}_{\X_0}[\beta^{\pm 1}] \ar[r] & \iota_{\ast}\iota^! \mathbb{Q}_{\X}[\beta^{\pm 1}] \ar[r] & \varphi_{f}^{\rm{inv}}[\beta^{\pm 1}]. 
}
	\end{align*}
 Here, the left square commutes by the diagram (\ref{chernKG}) 
 (and Lemma~\ref{prop4zero} if $f=0$ case on some component) and 
 the vertical isomorphisms 
 follow from (\ref{isom:ktopalpha}) and Lemma~\ref{lem:KBM}.
By taking cofibers of the top arrows in the $\infty$-category 
$D(\mathrm{Sh}_{\mathbb{Q}}(\X))$, we obtain the equivalence (\ref{equiv:KMF}). 
	\end{proof}

When $\X$ is a smooth quasi-projective variety, 
we define the Chern character map of $\mathrm{MF}(\X, f)$
by taking the cohomologies of the global sections of (\ref{equiv:KMF}) and using Proposition~\ref{prop:gsection}:
\begin{align}\label{cmap:MF:var}
	\ch \colon K_i^{\rm{top}}(\mathrm{MF}(\X, f)) 
 \to K_i^{\rm{top}}(\mathrm{MF}(\X, f))_{\mathbb{Q}}\stackrel{\cong}{\to} H^{i+2\ast}(\X, \varphi_f^{\rm{inv}}). 
 	\end{align}
  
  In the case that $\X$ is a smooth quotient stack, the Chern character map 
is defined via approximation by varieties. 
Let $\X=X/G$ and $V_n$ be $G$-representations with open subset 
$U_n \subset V_n$ 
as in Subsection~\ref{subsub25}. 
Below we use the same symbol $f$ to denote the 
composition $(X\times U_n)/G \to \X \stackrel{f}{\to} \mathbb{C}$. 

\begin{lemma}\label{lem:Hl}
For a fixed $\ell$, there are isomorphisms for $n\gg 0$: 
\begin{align*}
    H^{\ell}(\X, \varphi_f^{\rm{inv}}) \stackrel{\cong}{\to} H^{\ell}((X\times U_n)/G, \varphi_f^{\rm{inv}}). 
\end{align*}
\end{lemma}
\begin{proof}
    The lemma follows from the distinguished triangle (\ref{defC})
    together with the isomorphisms (\ref{isom:BM:pullback}), (\ref{isom:coh:pullback}). 
\end{proof}
Let $\mathrm{res}_{V_n}$ be the restriction map 
\begin{align*}
    \mathrm{res}_{V_n} \colon K_i^{\rm{top}}(\mathrm{MF}(\X, f))
    \to K_i^{\rm{top}}(\mathrm{MF}((X\times U_n)/G, f))
\end{align*}
and define 
$\ch_{V_n}$ to be 
\begin{align*}
    \ch_{V_n} \colon 
    K_i^{\rm{top}}(\mathrm{MF}((X\times U_n)/G, f)) 
    \stackrel{\ch}{\to} H^{i+2\ast}((X\times U_n)/G, \varphi_f^{\rm{inv}}). 
\end{align*}
Here the last map is the Chern character map (\ref{cmap:MF:var}) for the 
variety $(X\times U_n)/G$. 
We then define the Chern character map for the stack $\X$
\begin{align}\label{cmap:quot}
	\ch \colon K_i^{\rm{top}}(\mathrm{MF}(\X, f)) \to
 	\widetilde{H}^i(\X, \varphi_f^{\rm{inv}})
	\end{align}
 to be the limit as in (\ref{equal:limit}):
 \[\ch(y):= \lim_{n\to \infty} \ch_{V_n}(\mathrm{res}_{V_n}(y)).\] 

Let $d:=\dim_\mathbb{C}\X$.
Similarly to (\ref{def:filtration}), we
define the filtration 
\begin{equation}\label{deffiltrationKsg}
E_\ell K_i^{\rm{top}}(\mathrm{MF}(\X, f)):=\mathrm{ch}^{-1}\left(H^{\geq 2d-i-2\ell}(\X, \varphi^{\mathrm{inv}}_f)\right).
\end{equation}
We obtain cycle maps on the associated graded pieces:
\begin{equation}\label{cycleKsg}
    \mathrm{c}\colon \mathrm{gr}_\ell K_i^{\rm{top}}(\mathrm{MF}(\X, f))\to H^{2d-i-2\ell}(\X, \varphi^{\mathrm{inv}}_f).
\end{equation}

%\begin{prop}\label{prop41}
%    The maps \eqref{cycleKsg} are isomorphisms over $\mathbb{Q}$ for all $i,\ell\in\mathbb{Z}$. 
%\end{prop}
%\begin{proof}
%	Since the Chern character map (\ref{cmap:quot})
%	is defined via approximations by varieties, the argument of Proposition~\ref{prop:cherniso}
%	applies. 
%   \end{proof}

We give some compatibility of Chern character maps. 
From the exact sequence (\ref{exact:dgMF}),
there is a long exact sequence of abelian groups:
\begin{equation}\label{LESKtheory}
\cdots\to K^{\mathrm{top}}_i(\X_0)\to G^{\mathrm{top}}_i(\X_0)\to K_i^{\rm{top}}(\mathrm{MF}(\X, f))\to K^{\mathrm{top}}_{i-1}(\X_0)\to G^{\mathrm{top}}_{i-1}(\X_0)\to\cdots.
\end{equation}
By the long exact sequence of cohomologies associated with the triangle (\ref{defC}), 
there is a long exact sequence of $\mathbb{Q}$-vector spaces:
\begin{align}\label{LESBM}
\cdots\to H^{2d-i-2}(\X_0)&\xrightarrow{\alpha} H^{\mathrm{BM}}_i(\X_0)\to H^{2d-i}(\X, \varphi_f^{\mathrm{inv}}[-2])= H^{2d-i-2}(\X, \varphi_f^{\mathrm{inv}})\\\notag&\to H^{2d-i-1}(\X_0)\to H^{\mathrm{BM}}_{i-1}(\X_0)\to\cdots.
\end{align}

The following compatibility of the Chern character maps
is direct from the construction of (\ref{cmap:MF:var})
and Lemma~\ref{prop4zero}. 

\begin{prop}\label{prop413}
    The following diagram commutes, where the top sequence is \eqref{LESKtheory} and the bottom sequence is \eqref{LESBM}:
    \begin{equation}\label{diagprop412}
        \begin{tikzcd}
           \cdots\arrow[r]& G^{\mathrm{top}}_i(\X_0)\arrow[d, "\mathrm{ch}"]\arrow[r]& K^{\mathrm{top}}_i(\mathrm{MF}(\X, f))\arrow[d, "\mathrm{ch}"]\arrow[r]& K^{\mathrm{top}}_{i-1}(\X_0)\arrow[d, "\mathrm{ch}"]\arrow[r]& \cdots\\
            \cdots\arrow[r]&\widetilde{H}^{\mathrm{BM}}_i(\X_0)\arrow[r]& \widetilde{H}^i(\X, \varphi_f^{\mathrm{inv}})\arrow[r]& \widetilde{H}^{i-1}(\X_0)\arrow[r]&\cdots.
        \end{tikzcd}
    \end{equation}
\end{prop}

% \begin{remark}
%     If $\X$ is a variety, the Chern character \eqref{Ch:Dsg} is the same as the Chern character (see \cite{Blanc})
%     \begin{equation}\label{categoricalCH}
%     K^{\mathrm{top}}\to \mathrm{HP}
%     \end{equation}  from the topological K-theory to the periodic cyclic homology of the dg-category $D_{\mathrm{sg}}(\X_0)$. Indeed, apply the Chern character \eqref{categoricalCH} to the sequence of dg-categories $\mathrm{Perf}(\X_0)\to D^b(\X_0)\to D_{\mathrm{sg}}(\X_0)$ to obtain the diagram:
%     \begin{equation*}
%         \begin{tikzcd}
%            \cdots\arrow[r]& K^{\mathrm{top}}_{i}(\mathrm{Perf}(\X_0))
%            \arrow[d, "\mathrm{ch}"]\arrow[r]& K^{\mathrm{top}}_i(D^b\mathrm{Coh}(\X_0))
%            \arrow[d, "\mathrm{ch}"]\arrow[r]& K^{\mathrm{sg}}_i(\X_0)\arrow[d, "\mathrm{ch}"]\arrow[r]& \cdots\\
%             \cdots\arrow[r]&\mathrm{HP}_{i}(\mathrm{Perf}(\X_0))\arrow[r]&\mathrm{HP}_i(D^b\mathrm{Coh}(\X_0))\arrow[r]& \mathrm{HP}_i(D_{\mathrm{sg}}(\X_0))\arrow[r]& \cdots.
%         \end{tikzcd}
%     \end{equation*}
%     The left square is the same as the diagram from Proposition \ref{prop4zero}, and then the conclusion follows.
% \end{remark}

% The map (\ref{cmap:MF:var}) is an isomorphism when $\X$ is a variety, 
% as (\ref{Ch:relative}) is an isomorphism in this case. 

\subsection{The Grothendieck-Riemann-Roch theorem for matrix factorizations}

We discuss a version of the Grothendieck-Riemann-Roch theorem for categories of matrix factorizations. 

We first state functorial properties of monodromy invariant vanishing cycles. 
Let $h\colon \X\to\mathscr{Y}$ be a morphism of smooth quotient stacks, 
$f \colon \mathscr{Y} \to \mathbb{C}$ a regular function, 
and set $g:=f\circ h$. 
There is a natural map $h^* \varphi_f \to \varphi_g$ which 
is compatible with monodromy operators. 
Hence it induces a map $h^* \varphi_f^{\mathrm{inv}} \to 
\varphi_g^{\mathrm{inv}}$, and after taking global sections:
\begin{align*}
    h^* \colon H^k(\mathscr{Y}, \varphi_f^{\mathrm{inv}}) \to H^k(\X, \varphi_g^{\mathrm{inv}}). 
\end{align*}

Suppose that $h$ is a proper, smoothable lci morphism of relative dimension $d$. Then, as $h_{*}=h_{!}$, there is a natural map $h_{*}\mathbb{Q}_{\X} \to \mathbb{Q}_{\mathscr{Y}}[-2d]$. 
Since the vanishing cycle functor commutes with $h_{*}$ for 
a proper map $h$, we obtain 
\begin{align*}
    h_{*}\varphi_g(\mathbb{Q}_{\X})=\varphi_g(h_{*}\mathbb{Q}_{\X}) \to\varphi_f(\mathbb{Q}_{\mathscr{Y}})[-2d].
\end{align*}
Since the above map is compatible with monodromy operators, 
we have the induced map $h_{*}\varphi_g^{\mathrm{inv}} \to \varphi_f^{\mathrm{inv}}[-2d]$, and 
\begin{align*}
    h_{*} \colon H^k(\X, \varphi_g^{\mathrm{inv}}) \to H^{k-2d}(\mathscr{Y}, \varphi_f^{\mathrm{inv}}). 
\end{align*}
The Grothendieck-Riemann-Roch theorem for matrix factorizations we will be using is the following: 
\begin{thm}\label{GRRMFtop}  
Let $\X, \mathscr{Y}$ be smooth quotient stacks and $h\colon \X \to \mathscr{Y}$ be a representable morphism. 

(i) The following diagram commutes:
\begin{equation*}
    \begin{tikzcd}
        K^{\mathrm{top}}_i(\mathrm{MF}(\mathscr{Y}, f))\arrow[d, "\mathrm{ch}"']\arrow[r, "h^*"]& K^{\mathrm{top}}_i(\mathrm{MF}(\mathscr{X}, g))\arrow[d, "\mathrm{ch}"]\\
        \widetilde{H}^i(\mathscr{Y}, \varphi^{\mathrm{inv}}_f)\arrow[r, "h^*"]& \widetilde{H}^i(\mathscr{X}, \varphi^{\mathrm{inv}}_g).
    \end{tikzcd}
    \end{equation*}

(ii) Assume that $h$ is proper and smoothable lci. Let 
$\mathrm{Td}(\mathbb{T}_h)\in \widetilde{H}^0(\X_0)$ be the Todd class 
    of the tangent complex $\mathbb{T}_h$ of $h$ and let $h'_*(-):=h_*(\mathrm{Td}(\mathbb{T}_h)\cdot(-))$. 
Then
the following diagram commutes:
\begin{equation*}
    \begin{tikzcd}
        K^{\mathrm{top}}_i(\mathrm{MF}(\X, g))\arrow[d, "\mathrm{ch}"']\arrow[r, "h_*"]& K^{\mathrm{top}}_i(\mathrm{MF}(\mathscr{Y}, f))\arrow[d, "\mathrm{ch}"]\\
        \widetilde{H}^i(\X, \varphi^{\mathrm{inv}}_g)\arrow[r, "h'_*"]& \widetilde{H}^i(\mathscr{Y}, \varphi^{\mathrm{inv}}_f).
    \end{tikzcd}
    \end{equation*}
\end{thm}

\begin{proof}
In the case that $\X$ and $\mathscr{Y}$ are smooth schemes, 
(i) and (ii) follow from the diagram of spectra (\ref{GRR3}), together with the identification of $\omega_{\X_0}^{\mathrm{sg}}$ with $\varphi_f^{\mathrm{inv}}$, see Remark~\ref{rmk:cofib}. 
Note that the corresponding maps $h^*, h_{*}$ for $\omega_{\X_0}^{\mathrm{sg}}$ and $\varphi_f^{\mathrm{inv}}$ are also 
naturally identified, since these maps are also compatible with the 
sequence (\ref{lem:dist:minv}) by the construction. 
A generalization to stacks follows from the approximation by varieties
  as in the proof of Theorem~\ref{GRRS}, using Lemma~\ref{lem:Hl}. 
\end{proof}

We note the following functoriality of 
the associated graded of 
topological K-theory of the category of matrix factorizations.

\begin{prop}\label{functoriality:Ktop}
Let $h\colon \X\to \mathscr{Y}$ be a representable morphism of smooth quotient stacks of relative dimension $d$. Let $f\colon \mathscr{Y}\to \mathbb{C}$ be a regular function, and set $g:=f\circ h$. 
Let $\X_0$ and $\mathscr{Y}_0$ be the (derived) zero loci of $g$ and $f$, respectively. 
Then for $i, \ell\in\mathbb{Z}$, 
the morphism $h$ induces a pullback map:
\begin{align*}
    h^*\colon \mathrm{gr}_\ell K^{\mathrm{top}}_i(\mathrm{MF}(\mathscr{Y}, f))\to \mathrm{gr}_{\ell+d} K^{\mathrm{top}}_i(\mathrm{MF}(\X, g)).
\end{align*}
If $h$ is proper and smoothable lci, then there is a pushforward map:
\begin{align*}
    h_*\colon \mathrm{gr}_\ell K^{\mathrm{top}}_i(\mathrm{MF}(\X, g))\to \mathrm{gr}_\ell K^{\mathrm{top}}_i(\mathrm{MF}(\mathscr{Y}, f)).
\end{align*}
\end{prop}

\begin{proof}
The claim follows from Theorem \ref{GRRMFtop}
and Proposition~\ref{propchprime}. 
% By Proposition \ref{prop413}, it suffices to check that $h_*$ and $h^*$ induce maps:
%     \begin{align*}
%     h_*&\colon \mathrm{gr}_\ell G^{\mathrm{top}}_i(\X_0)\to \mathrm{gr}_\ell G^{\mathrm{top}}_i(\mathscr{Y}_0),\
%     h_* \colon \mathrm{gr}_\ell K^{\mathrm{top}}_i(\X_0)\to \mathrm{gr}_\ell K^{\mathrm{top}}_i(\mathscr{Y}_0),\\
%     h^*&\colon \mathrm{gr}_\ell G^{\mathrm{top}}_i(\mathscr{Y}_0)\to \mathrm{gr}_{\ell+d} G^{\mathrm{top}}_i(\X_0),\
%     h^*\colon \mathrm{gr}_\ell K^{\mathrm{top}}_i(\mathscr{Y}_0)\to \mathrm{gr}_{\ell+d} K^{\mathrm{top}}_i(\X_0).
% \end{align*} 
% The first map is constructed by Proposition \ref{functoriality:Ktop2}. 
% The other three maps are constructed by Corollary \ref{functoriality:Ktop3}. 
\end{proof}

For future reference, we also state explicitly the compatibility of the Chern character maps with Kn\"orrer periodicity, which is a particular case of Theorem \ref{GRRMFtop}.

\begin{cor}\label{corollary410}
    Let $X$ be a smooth affine variety with an action of a reductive group $G$, let $\X:=X/G$ and consider a regular function $f\colon\X\to\mathbb{C}$. 
    Let $U$ be a finite dimensional representation of $G$ and consider the natural pairing $w\colon U\times U^\vee\to\mathbb{C}$. Let $\mathscr{Y}:=(X\times U\times U^\vee)/G$ and consider the regular function $f+w\colon\mathscr{Y}\to\mathbb{C}$, where $f$ and $w$ are pulled-back from $X$ and $U\times U^\vee$, respectively. Consider the natural maps:
    \[X \stackrel{v}{\twoheadleftarrow} X\times U \stackrel{s}{\hookrightarrow} X\times U\times U^\vee\]
    where $v$ is the projection and $s(x, u)=(x, u, 0)$. 
    Let $\mathrm{ch}':=\mathrm{ch}\cdot \mathrm{Td}(T_s)$, where $T_s$ is the relative tangent complex of $s$.
    Then the following diagram commutes:
    \begin{equation*}
        \begin{tikzcd}
            K^{\mathrm{top}}_i(\mathrm{MF}(\X, f))\arrow[d, "\mathrm{ch}'"]\arrow[r, "s_*v^*"]& K^{\mathrm{top}}_i(\mathrm{MF}(\mathscr{Y}, f+w))\arrow[d, "\mathrm{ch}"]\\
            \widetilde{H}^i(\X, \varphi_f^{\rm{inv}})\arrow[r, "s_*v^*"]& \widetilde{H}^i(\mathscr{Y}, \varphi_{f+w}^{\rm{inv}}).
        \end{tikzcd}
    \end{equation*}
\end{cor}

    Note that the horizontal maps are isomorphisms by the Thom-Sebastiani theorem, see the proofs of Propositions \ref{BPSprime} and \ref{quasiBPSprime}. The top horizontal map induces an isomorphism called Kn\"orrer periodicity \cite{OrLG, Hirano}

\subsection{Injectivity of the Chern character map}

The Chern character maps \eqref{chernstack}, \eqref{chernstack2}, or \eqref{cmap:quot} may not be injective when $\X$ is a stack. However, they are all isomorphisms over $\mathbb{Q}$ 
when $\X$ is a variety. 
In some cases of interest, we can show that \eqref{cmap:MF:var} is injective for $\X$ a stack using the following propositions.

\begin{prop}\label{cherninj}
    Let $\X$ be a smooth quotient stack and let $f\colon \X\to\mathbb{C}$ be a regular function. Let $\mathbb{S}$ be a subcategory of $\mathrm{MF}(\X, f)$. Assume that there exists a smooth variety $Y$ and a morphism $r\colon Y\to\X$ such that $r^*\colon \mathbb{S}\to \mathrm{MF}(Y, g)$ is fully faithful whose image is (left or right) admissible, where $g:=f\circ r$.
    Then for $i \in \mathbb{Z}$, the Chern character map 
    \[\mathrm{ch}\colon K^{\mathrm{top}}_i(\mathbb{S})\to K^{\mathrm{top}}_i(\mathrm{MF}(\X, f))
    \to \widetilde{H}^i(\X, \varphi_f^{\mathrm{inv}})\] is injective over $\mathbb{Q}$. 
    Moreover, we have $\dim_{\mathbb{Q}} K_i^{\rm{top}}(\mathbb{S})_{\mathbb{Q}}<\infty$. 
\end{prop}

\begin{proof}
    The pullback map $r^*\colon K^{\mathrm{top}}_i(\mathbb{S})\hookrightarrow K^{\mathrm{top}}_i(\mathrm{MF}(Y, g))$ is injective
    by the admissibility 
    assumption. The claim then follows from the commutative diagram:
    \begin{equation*}
        \begin{tikzcd}
          K^{\mathrm{top}}_i(\mathbb{S})_{\mathbb{Q}}\arrow[r]\arrow[rr, hook, bend left=20]\arrow[dr]&  K^{\mathrm{top}}_i(\mathrm{MF}(\X, f))_{\mathbb{Q}}\arrow[d, "\mathrm{ch}"']\arrow[r, "r^*"]& K^{\mathrm{top}}_i(\mathrm{MF}(Y, g))_{\mathbb{Q}}\arrow[d, "\mathrm{ch}", "\sim" {anchor=south, rotate=90}]\\
            &\widetilde{H}^i(\X, \varphi^{\mathrm{inv}}_f)\arrow[r, "r^*"]& \widetilde{H}^i(Y, \varphi^{\mathrm{inv}}_g).
        \end{tikzcd}
    \end{equation*}
    The last claim holds from the above diagram and 
    $\dim_{\mathbb{Q}}\widetilde{H}^i(Y, \varphi^{\mathrm{inv}}_g)<\infty$.
\end{proof}

\subsection{Monodromy invariant vanishing cycles under semiorthogonal decomposition}
We will use a variant of Lemma~\ref{lem:sheaf:K} for semiorthogonal 
summands. 
Let $\X$ be a smooth quasi-projective variety 
with a proper morphism $h \colon \X \to B$. 
Consider a regular function $g \colon B \to \mathbb{C}$ and 
its composition with $h$:
\begin{align*}
    f \colon \X \stackrel{h}{\to} B \stackrel{g}{\to} \mathbb{C}. 
\end{align*}
Set $\X_0:=f^{-1}(0)$.
Let $D^b(\X)=\langle \mathcal{C}_1, \mathcal{C}_2 \rangle$ be a 
$B$-linear semiorthogonal decomposition. 
It induces semiorthogonal decompositions 
\begin{align*}
D^b(\X_0)=\langle \mathcal{C}_1', \mathcal{C}_2'\rangle, \ 
\mathrm{Perf}(\X_0)=\langle \mathcal{C}_1'', \mathcal{C}_2'' \rangle, \ 
D^{\rm{sg}}(\X_0)=\langle \mathcal{C}_1^{\rm{sg}}, \mathcal{C}_2^{\rm{sg}} \rangle
\end{align*}
where $\mathcal{C}_i^{\rm{sg}}=\mathcal{C}_i'/\mathcal{C}_i''$, 
see~\cite[Proposition~2.5]{PTzero}.

\begin{lemma}\label{lem:sod:induce}
There is an isomorphism $\mathcal{K}_B^{\rm{top}}(\mathcal{C}_i^{\rm{sg}})_{\mathbb{Q}} \cong 
    \varphi_g^{\rm{inv}}\mathcal{K}_B^{\rm{top}}(\mathcal{C}_i)_{\mathbb{Q}}$. 
\end{lemma}
\begin{proof}
By Theorem~\ref{thm:phiproper}, Lemma~\ref{lem:sheaf:K}, and the fact that the vanishing cycle functor 
commutes with proper push-forward, there are isomorphisms: 
\begin{align*}
    \mathcal{K}_B^{\rm{top}}(D^{\rm{sg}}(\X_0))_{\mathbb{Q}} &\cong h_{\ast}\mathcal{K}_{\X}^{\rm{top}}(D^{\rm{sg}}(\X_0))_{\mathbb{Q}} \\
    &\cong h_{\ast}\varphi_f^{\rm{inv}}\mathcal{K}_{\X}^{\rm{top}}(D^b(\X))_{\mathbb{Q}} \\
    &\cong \varphi_g^{\rm{inv}}\mathcal{K}_B^{\rm{top}}(D^b(\X))_{\mathbb{Q}}. 
\end{align*}
It is straightforward to check that, under the above isomorphism,
    the decomposition 
    \begin{align*}
        \mathcal{K}_{\X}^{\rm{top}}(D^{\rm{sg}}(\X_0))_{\mathbb{Q}} \cong 
        \mathcal{K}_{\X}^{\rm{top}}(\mathcal{C}_1^{\rm{sg}})_{\mathbb{Q}} \oplus \mathcal{K}_{\X}^{\rm{top}}(\mathcal{C}_2^{\rm{sg}})_{\mathbb{Q}}
    \end{align*}
    is identified with the decomposition obtained by applying 
    $\varphi_g^{\rm{inv}}$ to the decomposition 
    \begin{align*}
        \mathcal{K}_B^{\rm{top}}(D^b(\X))_{\mathbb{Q}}=\mathcal{K}_B^{\rm{top}}(\mathcal{C}_1)_{\mathbb{Q}} \oplus \mathcal{K}_B^{\rm{top}}(\mathcal{C}_2)_{\mathbb{Q}}. 
    \end{align*}
    Therefore we obtain the desired isomorphism.     
\end{proof}

\section{Dimensional reduction}\label{s5}
In this section, we show the compatibility 
of Koszul equivalences with
the dimensional reduction in cohomology via the Chern character map.  
We will use these compatibilities in Subsection \ref{subsection:preproj} to compute the topological K-theory of preprojective quasi-BPS categories using the computation of the topological K-theory of quasi-BPS categories of tripled quivers with potential.

\subsection{Dimensional reduction in cohomology}\label{subsection:dimred}
% We continue with the notation from the previous Subsection. Consider the maps:
% \[\mathscr{P}\xleftarrow{p'}\mathscr{E}^\vee|_{\mathscr{P}}\xrightarrow{i'}\mathscr{E}^{\vee}\text{ and }\mathscr{P}\xrightarrow{i}\X\xleftarrow{p}\mathscr{E}^\vee.\]
%Assume $\mathscr{E}$ has rank $r$. 
%Let $\dim \mathscr{E}^\vee_0=\dim \mathscr{E}^\vee-\varepsilon$ for $\varepsilon\in\{0,1\}$.
In this subsection, we review the dimensional reduction theorem in cohomology \cite[Theorem A.1]{MR3667216}, \cite[Subsection 4.8]{MR2851153}. 
%(note that, to obtain the maps in loc. cit., one needs to precompose all the following maps by $l_*$, see the isomorphism \eqref{l}).
Recall the setting of Subsection \ref{subsectionKoszul}, 
i.e. $\X$ is a smooth quotient stack, 
$\mathscr{E} \to \X$ is a vector bundle of rank $r$
with a section $s$, and $\mathscr{P} \hookrightarrow \X$
is its derived zero locus. 
We will use the notations of the various maps from the diagram~\eqref{diagmapskoszul}:
\begin{equation}\label{diagmapskoszul2}
    \begin{tikzcd}
        \mathscr{E}^{\vee}|_{\mathscr{P}} \arrow[d, "\eta'"] \arrow[rr, bend left, hook, "j'"] \arrow[r, hook, "\tau"]  & \mathscr{E}_0^{\vee} \arrow[r, hook, "\iota"] & \mathscr{E}^{\vee} 
   \arrow[r, "f"] \arrow[d, "\eta"] & \mathbb{C} \\
    \mathscr{P} \arrow[rr, hook, "j"] & & \X.
    \end{tikzcd}
\end{equation}

\begin{thm}\emph{(\cite[Theorem~A.1]{MR3667216})}
There is a natural isomorphism for $\bullet\in D^b(\mathrm{Sh}_{\mathbb{Q}}(\X))$:
\begin{align}\label{dimred3}
\eta_!\varphi_f[-1]\eta^*\bullet\xrightarrow{\sim}\eta_!\eta^*j_*j^*\bullet.
\end{align}
\end{thm}

Below we explain how to construct an 
isomorphism (\ref{dimred3}). 
For $\bullet\in D^b(\mathrm{Sh}_{\mathbb{Q}}(\mathscr{E}^\vee_0))$, there is a natural isomorphism:
\begin{equation}\label{varphiiota}
\varphi_f[-1]\iota_*\bullet\xrightarrow{\cong} \iota_*\bullet.
\end{equation}
On the other hand, for $\bullet\in D^b(\mathrm{Sh}_{\mathbb{Q}}(\X))$,
there is a natural transformation 
\begin{equation}\label{pii}
    \eta^*\bullet\to \eta^*j_*j^*\bullet
\end{equation}
The natural transformations \eqref{pii} and \eqref{varphiiota} induce a natural transformation for $\bullet\in D^b(\mathrm{Sh}_{\mathbb{Q}}(\X))$: 
\[\varphi_f[-1]\eta^*\bullet\to\varphi_f[-1](\eta^*j_*j^*\bullet)=
\eta^*j_*j^*\bullet.\]
The isomorphism (\ref{dimred3}) is obtained by 
applying $\eta_{!}$ to the above natural transformation. 

The cohomological version of the dimensional 
reduction theorem is obtained from the 
isomorphism (\ref{dimred3}) as follows. 
By taking the Verdier dual of (\ref{dimred3}), we obtain 
\begin{equation}\label{dimred}
\eta_*j'_*j'^!\eta^!\bullet=
\eta_*\eta^!j_*j^!\bullet \xrightarrow{\cong} 
\eta_*\varphi_f[-1]\eta^!\bullet,
\end{equation} 
which alternatively can be described as applying the functor $\eta_*\varphi_f[-1]$ to the natural transformation \begin{align*}j'_*j'^!\eta^!\bullet\to \eta^!\bullet
\end{align*}
for $\bullet\in D^b(\mathrm{Sh}_{\mathbb{Q}}(\X))$.
% 
% )} \textcolor{blue}{(Yes, you are correct. I think we should give a reference for the functor $\varphi_f$ and then say we use the shorthand notation $\varphi_f:=\varphi_f\mathbb{Q}[-1]$ in order to have the same convention as \cite{MR3667216}.)}
% \begin{equation}\label{dimred}
% p_*i'_*i'^!p^!\bullet\xrightarrow{\sim}p_*\varphi_f p^*\bullet.
% \end{equation}
By applying $\bullet=\mathbb{Q}_{\X}$ and 
taking the cohomology of the two sides in \eqref{dimred}, one obtains the \textit{dimensional reduction} isomorphism: 
\begin{equation}\label{dimred2}
j'_*\eta'^*\colon H^{\mathrm{BM}}_i(\mathscr{P})\xrightarrow{\cong} 
H^{\mathrm{BM}}_{i+2r}(\mathscr{E}^\vee|_\mathscr{P})\xrightarrow{\cong}
H^{2\dim \mathscr{E}-2r-i}(\mathscr{E}^\vee, \varphi_f\mathbb{Q}[-1]).
\end{equation}
Note that the monodromy action 
on the left hand side is trivial. 

We give some lemmas on 
vanishing cycle cohomologies in the above context. 
Recall that 
the monodromy invariant vanishing cycle 
$\varphi_f^{\mathrm{inv}}$ is the cone of the map $\alpha\colon \iota_{\ast}\mathbb{Q}_{\mathscr{E}^\vee_0}\to \iota_{\ast}\iota^!\mathbb{Q}_{\mathscr{E}^\vee}[2]$, see \eqref{defC}.
\begin{lemma}\label{lem:moninj}
There is a natural injective map:
\begin{equation}\label{gamma}
\gamma\colon H^\bullet(\mathscr{E}^\vee, \varphi_f[-1])\hookrightarrow H^\bullet(\mathscr{E}^\vee, \varphi_f^{\mathrm{inv}}[-2]).\end{equation}
\end{lemma}
\begin{proof}
The isomorphism \eqref{dimred} factors through:
\[\eta_*j'_*j'^!\eta^!\bullet\to \eta_*\iota_*\iota^!\eta^!\bullet
=\eta_*\varphi_f[-1]\iota_*\iota^!\eta^!\bullet
\to
\eta_*\varphi_f[-1] \eta^!\bullet.\]
The second map applied to $\bullet=\mathbb{Q}_{\X}$ gives 
a map 
\begin{align*}
    \eta_*\iota_*\iota^!\mathbb{Q}_{\mathscr{E}^\vee}[2]\to \eta_*\varphi_f\mathbb{Q}_{\mathscr{E}^\vee}[1].
\end{align*}
From the diagram \eqref{diagsec4}, the above 
map factors through $\eta_*\varphi_f^{\mathrm{inv}}$, and thus
there are maps whose composition is an isomorphism:
\begin{equation}\label{betaprime}
\beta'\colon \eta_*j'_*\omega_{\mathscr{E}^\vee|_\mathscr{P}}\to \eta_*\iota_*\iota^!\omega_{\mathscr{E}^\vee}\to \eta_*\varphi_f^{\mathrm{inv}}[2\dim \mathscr{E}^\vee-2]\xrightarrow{\beta} \eta_*\varphi_f\omega_{\mathscr{E}^\vee}[-1].
\end{equation}
We let 
$\beta^\diamond$ be a composition 
in the above sequence 
\begin{align*}\beta^\diamond\colon \eta_*j'_*\omega_{\mathscr{E}^\vee|_\mathscr{P}}\to \eta_*\iota_*\iota^!\omega_{\mathscr{E}^\vee}\to \eta_*\varphi_f^{\mathrm{inv}}[2\dim \mathscr{E}^\vee-2].
\end{align*}
The map $\beta^\diamond$ provides a splitting of the map $\beta$, thus the triangle \eqref{thirdcolumn} 
canonically splits to give natural isomorphism:
\begin{equation}\label{notdagger}
\eta_*\varphi_f^{\mathrm{inv}}[-2]\cong \eta_*\varphi_f[-2]\oplus \eta_*\varphi_f[-1].
\end{equation}
By taking global sections of this isomorphism, 
we obtain the desired injection. 
\end{proof}
We also note the following: 
\begin{lemma}
The isomorphism (\ref{dimred2}) 
is given by the composition 
\begin{align}\label{splittingmap}
 \beta'\colon H^{\mathrm{BM}}_i(\mathscr{P})&\xrightarrow{\cong} H^{\mathrm{BM}}_{i+2r}(\mathscr{E}^\vee|_\mathscr{P})
 \to H^{\mathrm{BM}}_{i+2r}(\mathscr{E}^\vee_0)\to 
 H^{2\dim\mathscr{E}-2r-i}
 (\mathscr{E}^\vee, \varphi_f^{\mathrm{inv}}[-2])\\
 \notag&\to 
 H^{2\dim\mathscr{E}-2r-i}
 (\mathscr{E}^\vee, \varphi_f\mathbb{Q}[-1])=H^{2\dim\X-i}(\mathscr{E}^\vee, \varphi_f\mathbb{Q}[-1]).\end{align}
 Here the composition of the maps in the top row of \eqref{splittingmap} is given by $\beta^\diamond$.
 \end{lemma}
 \begin{proof}
The lemma is obtained by taking global sections of the complexes in \eqref{betaprime}. 
      \end{proof}

\subsection{The Chern character for graded matrix factorizations}

%In this subsection, we assume that the stack \eqref{def:K} is classical. 
% We compare dimensional reduction in cohomology and K-theory. The next computation follows from the explicit description description of $\kappa$, see \cite[Section 2.3.2]{T}.
% %We consider the maps:\[\mathscr{P}\xleftarrow{p}\mathscr{E}^\vee|_{\mathscr{P}}\xrightarrow{i}\mathscr{X}.\]
% %Assume $\mathscr{E}$ has rank $r$. 
% %When \eqref{def:K} is classical, the equivalence \eqref{Koszul} is given by 
% \begin{prop}
%     Recall the Koszul equivalence $\kappa$ from \eqref{Koszul}. Then the composition 
%     \[D^b(\mathscr{P}^{\mathrm{cl}})\to D^b(\mathscr{P})\xrightarrow{\kappa} \mathrm{MF}^{\mathrm{gr}}(\mathscr{E}^{\vee}, f)\] is given by the functor $i_*p^*$.
%     \end{prop}

%\[i_*p^*\colon D^b(\mathscr{P})\xrightarrow{\sim} \mathrm{MF}^{\mathrm{gr}}(\mathscr{E}^{\vee}, f).\]
The purpose of this subsection is to construct a Chern character map:
\begin{equation}\label{def:cherncharactergr}
\mathrm{ch}\colon K^{\mathrm{top}}_i(\mathrm{MF}^{\mathrm{gr}}(\mathscr{E}^\vee, f))\to \widetilde{H}^i(\mathscr{E}^\vee, \varphi_f[-1])
\end{equation}
compatible with the Chern character map \eqref{chquasismooth} for $\mathscr{P}$ and the Chern character map \eqref{cmap:quot} for $\mathrm{MF}(\mathscr{E}^\vee, f)$, see Proposition \ref{prop522}. 

We begin with a few preliminaries. 
Recall the Koszul equivalence from Theorem~\ref{thm:Koszul}:
\begin{align}\label{koszul:recall}
    \kappa \colon D^b(\mathscr{P}) \stackrel{\sim}{\to}
    \mathrm{MF}^{\rm{gr}}(\mathscr{E}^{\vee}, f). 
\end{align}
Let $l$ be the natural closed immersion $l\colon \mathscr{P}^{\mathrm{cl}}\hookrightarrow \mathscr{P}$.
Note that the pushforward map induces an equivalence $l_*\colon G^{\mathrm{top}}(\mathscr{P}^{\mathrm{cl}})\xrightarrow{\sim} G^{\mathrm{top}}(\mathscr{P})$.
Recall the forget-the-grading functor: 
\begin{equation}\label{forgetthepotential}
\Theta\colon \mathrm{MF}^{\mathrm{gr}}(\mathscr{E}^\vee, f)\to \mathrm{MF}(\mathscr{E}^\vee, f)
\end{equation}
and the equivalence from Theorem~\ref{thm:orlov}:
\[\mathrm{MF}(\mathscr{E}^\vee, f)\xrightarrow{\sim} D_{\mathrm{sg}}(\mathscr{E}^{\vee}_0).\]
The following diagram commutes:
\begin{equation*}
    \begin{tikzcd}
        D^b(\mathscr{P}^{\mathrm{cl}})\arrow[r, "j'_*\eta'^*l_*"]\arrow[d, "\tau_*\eta'^*l_*"]&
    \mathrm{MF}^{\mathrm{gr}}(\mathscr{E}^\vee, f)
    \arrow[r, "\Theta"]&
    \mathrm{MF}(\mathscr{E}^\vee, f)\arrow[d, "\sim"{anchor=south, rotate=90}]\\
D^b(\mathscr{E}^{\vee}_0)\arrow[rr]&&
    D_{\mathrm{sg}}(\mathscr{E}^{\vee}_0).
    \end{tikzcd}
\end{equation*}
By the equivalence $l_*\colon G^{\mathrm{top}}(\mathscr{P}^{\mathrm{cl}})\xrightarrow{\sim} G^{\mathrm{top}}(\mathscr{P})$, we obtain a commutative diagram
\begin{equation}\label{psi}
    \begin{tikzcd}
        G^{\mathrm{top}}_\bullet(\mathscr{P})\arrow[r, "j'_*\eta'^*"]\arrow[d, "\tau_*\eta'^*"']& K^{\mathrm{top}}_\bullet(\mathrm{MF}^{\mathrm{gr}}(\mathscr{E}^\vee, f))\arrow[r, "\Theta"]& K^{\mathrm{top}}_\bullet(\mathrm{MF}(\mathscr{E}^\vee, f))\arrow[d]\\
          G^{\mathrm{top}}_\bullet(\mathscr{E}^\vee_0)\arrow[rr]& &K^{\mathrm{top}}_\bullet(D_{\mathrm{sg}}(\mathscr{E}_0)).
    \end{tikzcd}
\end{equation}
By Proposition~\ref{prop511}, the top left arrow is 
the isomorphism induced by the Koszul equivalence (\ref{koszul:recall}). 

Recall the Chern character \eqref{cmap:quot} and the splitting \eqref{splittingmap}.
% There is a natural map \[H^{\mathrm{BM}}_j(\mathscr{P}^{\mathrm{cl}})\xrightarrow{\tau_*p'^*} H^{\mathrm{BM}}_{j+2r}(\mathscr{E}_0^{\vee})\to H^{2\dim\X-j}(\mathscr{E}^{\vee}, \varphi_f[-1])\to H^{2\dim\X-j-1}(\mathscr{E}^{\vee}, C).\]
%Consider the closed immersion $k\colon \mathscr{P}\to \mathscr{E}^\vee_0$, let $T_k$ be its virtual tangent bundle of $k$.
Let $\mathbb{N}:=\mathbb{T}_{\mathscr{P}/\mathscr{E}^{\vee}}[1]$ be the normal complex, where 
$\mathbb{T}_{\mathscr{P}/\mathscr{E}^{\vee}}$ is the relative tangent 
complex for $\mathscr{P}\hookrightarrow \X \hookrightarrow \mathscr{E}^\vee$ (which is a vector bundle on $\X$ restricted to $\mathscr{P}$), where the second arrow $\X\hookrightarrow \mathscr{E}^{\vee}$ is the zero section. 
We set 
\begin{align*}\mathrm{ch}':=\mathrm{Td}(\mathbb{N}) \cdot \ch \colon 
G_i^{\mathrm{top}}(\mathscr{P}) \to \widetilde{H}_i^{\mathrm{BM}}(\mathscr{P}).
\end{align*}
By noting that the normal complex of $\mathscr{E}_0^{\vee} \hookrightarrow \mathscr{E}^{\vee}$ is a trivial line bundle, 
the following diagram commutes by Proposition \ref{propo36}:
\begin{equation}\label{diagramsec5}
    \begin{tikzcd}
        G^{\mathrm{top}}_i(\mathscr{P})
        \arrow[rr, bend left, "\Psi"]
        \arrow[d, "\mathrm{ch}'"]\arrow[r, "\tau_{\ast}\eta'^{\ast}"]& G^{\mathrm{top}}_i(\mathscr{E}_0^{\vee})\arrow[d, "\mathrm{ch}"]\arrow[r]& K^{\mathrm{top}}_i(\mathrm{MF}(\mathscr{E}^\vee, f))\arrow[d, "\mathrm{ch}"]\\
        \widetilde{H}^{\mathrm{BM}}_i(\mathscr{P})\arrow[r, "\tau_*\eta'^*"]\arrow[rr, bend right, "\beta^\diamond"]& \widetilde{H}^{\mathrm{BM}}_i(\mathscr{E}^\vee_0)\arrow[r]& \widetilde{H}^i(\mathscr{E}^\vee, \varphi_f^{\mathrm{inv}}).
    \end{tikzcd}
\end{equation}

\begin{prop}\label{prop522}
There is a Chern character 
map \eqref{def:cherncharactergr} 
which fits into the commutative diagram
\begin{equation}\label{dd1}
    \begin{tikzcd}
         K^{\mathrm{top}}_i(\mathrm{MF}^{\mathrm{gr}}(\mathscr{E}^\vee, f))\arrow[d, "\mathrm{ch}"]\arrow[r, hook, "\Theta"]& K^{\mathrm{top}}_i(\mathrm{MF}(\mathscr{E}^\vee, f))\arrow[d, "\mathrm{ch}"]\\
         \widetilde{H}^i(\mathscr{E}^\vee, \varphi_f[-1])\arrow[r, hook, "\gamma"]& \widetilde{H}^i(\mathscr{E}^\vee, \varphi_f^{\mathrm{inv}}[-2]). 
    \end{tikzcd}
\end{equation}
Here each horizontal arrow is injective over $\mathbb{Q}$, and $\widetilde{H}^i(\mathscr{E}^\vee, \varphi_f^{\mathrm{inv}}[-2])=\widetilde{H}^i(\mathscr{E}^\vee, \varphi_f^{\mathrm{inv}})$
by the 2-periodicity. 
Moreover the following diagram commutes as well for
the modified Chern character maps 
for the immersions of $\mathscr{E}^\vee|_\mathscr{P}$ and $\mathscr{E}^\vee_0$ in $\mathscr{E}^\vee$:
\begin{equation}\label{dd2}
    \begin{tikzcd}
         G^{\mathrm{top}}_i(\mathscr{P})\arrow[r, "j'_*\eta'^*", "\sim"']\arrow[d, "\mathrm{ch}'"]& K^{\mathrm{top}}_i(\mathrm{MF}^{\mathrm{gr}}(\mathscr{E}^\vee, f))\arrow[d, "\mathrm{ch}"]\\
         \widetilde{H}^{\mathrm{BM}}_i(\mathscr{P})\arrow[r, "j'_*\eta'^*", "\sim"']&
         \widetilde{H}^i(\mathscr{E}^\vee, \varphi_f[-1]).
    \end{tikzcd}
\end{equation}
\end{prop}

\begin{proof}
We define the Chern character map (\ref{def:cherncharactergr}) 
to be the unique 
map
such that the diagram \eqref{dd2} commutes. 
Let $\Psi$ be the map in the diagram (\ref{diagramsec5}). 
We have that $\gamma\circ j'_*\eta'^*=\beta^\diamond$ and $\Theta\circ j'_*\eta'^*=\Psi$, so the diagram (\ref{dd1}) commutes as well. 
By Proposition~\ref{prop:forg} below, 
the map $\Theta$ is an embedding into the 
direct summand, hence it is injective. 
%is injective. 
%The map $\beta^\diamond$ is injective by %\eqref{splittingmap}. 
%Then $\Psi$ is also injective by the commutativity of the diagram \eqref{diagramsec5}. By the factorization $\Theta\circ j'_*\eta'^*=\Psi$, the map $\Theta$ is indeed injective. 
\end{proof}

The topological K-group $K^{\mathrm{top}}_{\ast}(\mathrm{MF}(\mathscr{E}^\vee, f))$ is a
module over $K^{\mathrm{top}}_{\ast}(\Spec \mathbb{C}, 0)=\Lambda$, where 
$\Lambda=\mathbb{Z}[\epsilon]$
with $\deg \epsilon=-1$, see Subsection~\ref{subsec:exterior}. 
We have used the following computation of a $\Lambda$-module structure on the topological K-theory of a category of matrix factorizations,
which will be proved in Subsection~\ref{subsec:exterior}. 

\begin{prop}\label{prop:forg}
    The forget-the-grading functor induces an isomorphism 
    \begin{equation}\label{isolambdatheta}
        K^{\mathrm{top}}_{\ast}(\mathrm{MF}^{\mathrm{gr}}(\mathscr{E}^\vee, f))\otimes_{\mathbb{Z}}\Lambda\xrightarrow{\cong} K^{\mathrm{top}}_{\ast}(\mathrm{MF}(\mathscr{E}^\vee, f))
    \end{equation} of $\Lambda$-modules. 
    Thus for an admissible subcategory 
    $\mathbb{M} \subset D^b_{\mathbb{C}^*}(\mathscr{E}^{\vee})$,
    there is an isomorphism of $\Lambda$-modules:
    \[K^{\mathrm{top}}_{\ast}(\mathrm{MF}^{\mathrm{gr}}(\mathbb{M}, f))\otimes_\mathbb{Z}\Lambda\xrightarrow{\cong} K^{\mathrm{top}}_{\ast}(\mathrm{MF}(\mathbb{M}, f)).\]
    \end{prop}

Similarly to (\ref{def:filtration}), we define an increasing filtration \begin{align*}E_\ell K^{\mathrm{top}}_i(\mathrm{MF}^{\mathrm{gr}}(\mathscr{E}^\vee, f))\subset K^{\mathrm{top}}_i(\mathrm{MF}^{\mathrm{gr}}(\mathscr{E}^\vee, f))
\end{align*}
by the following
\[E_\ell K^{\mathrm{top}}_i(\mathrm{MF}^{\mathrm{gr}}(\mathscr{E}^\vee, f)):=\mathrm{ch}^{-1}\left(H^{\geq 2\dim\mathscr{E}^\vee-i-2\ell}(\mathscr{E}^\vee, \varphi_f[-1])\right).\]
We obtain the cycle map 
\[\mathrm{c}\colon\mathrm{gr}_\ell K^{\mathrm{top}}_i(\mathrm{MF}^{\mathrm{gr}}(\mathscr{E}^\vee, f))\to H^{2\dim \mathscr{E}^\vee-i-2\ell}(\mathscr{E}^\vee, \varphi_f[-1])\]
which is an isomorphism by the isomorphism (\ref{cherngraded3}) together with 
the commutative diagram (\ref{dd2}).
The following is a corollary of Proposition \ref{cherninj} and Proposition \ref{prop522}:

\begin{cor}\label{prop53}
The following diagram commutes:
\begin{equation*}
    \begin{tikzcd}
        \mathrm{gr}_\ell G^{\mathrm{top}}_i(\mathscr{P})\arrow[d, "\mathrm{c}"]\arrow[r, "\sim"', "j'_*\eta'^*"]& \mathrm{gr}_{\ell+r} K^{\mathrm{top}}_i(\mathrm{MF}^{\mathrm{gr}}(\mathscr{E}^\vee, f))\arrow[d, "\mathrm{c}"]\arrow[r, hook]& \mathrm{gr}_{\ell+r+1} K^{\mathrm{top}}_i(\mathrm{MF}(\mathscr{E}^\vee, f))\arrow[d, "\mathrm{c}"]\\
        H^{\mathrm{BM}}_{i+2\ell}(\mathscr{P})\arrow[r, "\sim"', "j'_*\eta'^*"]& H^{2\dim \mathscr{X}-i-2\ell}(\mathscr{E}^\vee, \varphi_f[-1])\arrow[r, hook, "\gamma"]& H^{2\dim \mathscr{X}-i-2\ell}(\mathscr{E}^\vee, \varphi_f^{\mathrm{inv}}[-2]).
    \end{tikzcd}
\end{equation*}
\end{cor}

\begin{proof}
The Chern character maps in the diagram (\ref{diagramsec5}) induce the cycle maps $\mathrm{c}$ on the associated graded, see Proposition \ref{propchprime}. 
\end{proof}

\section{Topological K-theory of quasi-BPS categories for quivers with potential}\label{s6}

In this section, we 
give statements of the main results in this paper
on
topological K-theory of quasi-BPS categories
in terms of BPS cohomologies, see Theorem \ref{thm1}. The main step in the proof of Theorem \ref{thm1} is the construction of the cycle map from topological K-theory of quasi-BPS categories to BPS cohomology, see Theorem \ref{thm2}. 
In this section, we prove Theorem~\ref{thm1} 
assuming Theorem~\ref{thm2}. 
The proof of Theorem~\ref{thm2} will be given 
in the next section.

Before we state Theorem \ref{thm1}, we introduce notation related to quasi-BPS categories and BPS sheaves. 
Below $(Q, W)$ is a symmetric quiver with
potential $W$. 
We use the notation of Subsection~\ref{section:prelim}.

\subsection{Sets of partitions}\label{subsec612}
In this subsection, we introduce certain sets of partitions. In the next subsection, we use these sets to define constructible sheaves whose cohomology are the target of the cycle map from topological K-theory of quasi-BPS categories. We also present some exampels of explicit computations of such sets of partitions.

Let $d=(d^j)_{j\in I}$ be a dimension vector
and let $\delta \in M(d)_{\mathbb{R}}^{W_d}$. 
For a cocharacter $\lambda$ of $T(d)$, recall 
the definition of $n_{\lambda}$ in~\eqref{nlambdadef}. 
We introduce the following definition: 
\begin{defn}
For $\lambda$ an antidominant cocharacter of $T(d)$, we define 
\begin{equation}\label{varepsilondeltalambda}
\varepsilon_{\lambda, \delta}:=\begin{cases}
1,\text{ if }n_\lambda/2+\langle \lambda, \delta\rangle\in\mathbb{Z},\\
0, \text{ otherwise}.
\end{cases}
\end{equation}
For a partition $\mathbf{d}=(d_i)_{i=1}^k$ of $d$, we define $\varepsilon_{\mathbf{d}, \delta}=1$ if $\varepsilon_{\lambda, \delta}=1$ for any cocharacter $\lambda$ with associated partition $\mathbf{d}$, and define $\varepsilon_{\mathbf{d}, \delta}=0$ otherwise.
We define $S^d_\delta$ to be the set of partitions $\mathbf{d}=(d_i)_{i=1}^k$ of $d$
such that $\varepsilon_{\mathbf{d}, \delta}=1$. If $\delta=v\tau_d$, we use the notation $S^d_v$ instead of $S^d_{v\tau_d}$.
\end{defn}
We have the following another 
characterization of $\varepsilon_{\mathbf{d}, \delta}$. 
\begin{lemma}\label{lem:epsilon}
Let $\mathbf{d}=(d_i)_{i=1}^k$ be a partition of $d$, and $\lambda$ an antidominant cocharacter 
with associated partition $\mathbf{d}$. 
Let $\theta_i\in \frac{1}{2}M(d_i)$ and $\delta_i \in M(d_i)_{\mathbb{R}}^{W_{d_i}}$ be defined 
by (which are independent of $\lambda$ with associated partition $\mathbf{d}$)
\[\sum_{i=1}^k \theta_i=-\frac{1}{2}R(d)^{\lambda>0}+\frac{1}{2}\mathfrak{g}(d)^{\lambda>0}, \quad 
\sum_{i=1}^k \delta_{i}=\delta.\]
Then we have $\varepsilon_{\mathbf{d}, \delta}=1$ if 
and only if $\langle 1_{d_i}, \theta_i+\delta_i \rangle \in \mathbb{Z}$ for all $i$. 
\end{lemma}
\begin{proof}
    By the definition, we have 
    $\varepsilon_{\mathbf{d}, \delta}=1$ if and only 
    if 
    $\langle \lambda, \sum_{i=1}^k (\theta_i+\delta_i)\rangle \in \mathbb{Z}$
    for all antidominant cocharacter $\lambda$
    with associated partition $(d_i)_{i=1}^k$. 
    Therefore, the claim follows. 
\end{proof}

Below we explain the meaning of the above definition 
in terms of categorical Hall product of quasi-BPS categories. 
Recall the definition of 
$\mathbb{M}(d; \delta)$, $\mathbb{S}(d; \delta)$ from Subsection~\ref{subsec27}. 
\begin{prop}\label{prop:partition}
Let $(d_i)_{i=1}^k$ be a partition 
of $d$ and define $\theta_i$, $\delta_i$ as in 
Lemma~\ref{lem:epsilon}. 
Then the categorical Hall products (\ref{def:Hallprod}), (\ref{def:Hallprod2})
restrict to the functors
\begin{align}\label{hall:restrict}
&m_{\lambda} \colon 
\boxtimes_{i=1}^k\mathbb{M}(d_i; \theta_i+\delta_{i})\to \mathbb{M}(d; \delta), \\
&\notag m_{\lambda} \colon \boxtimes_{i=1}^k \mathbb{S}(d_i; \theta_i+\delta_{i}) \to 
\mathbb{S}(d; \delta). 
\end{align}
\end{prop}
%such that the image of the Hall product of categories $\mathbb{M}(d_i)_{w_i}$ is in $\mathbb{M}(d)_w$. Alternatively, for any $\chi_i\in M(d_i)_{w_i}$ with $\chi_i+\rho_i\in \frac{1}{2}\textbf{W}(d_i)$ for $1\leq i\leq k$, and for $\lambda$ the anti-dominant cocharacter associated to the partition $(d_i)_{i=1}^k$, we have \[\sum_{i=1}^k \chi_i\in F(\lambda).\] 
\begin{proof}
Let $\chi_i \in M(d_i)$ be dominant weights such that 
$\chi_i+\rho_i-(\theta_i+\delta_{i}) \in \mathbf{W}(d_i)$. 
Here $\rho_i$ is half the sum of positive weights of $\mathfrak{g}(d_i)$. 
Let $\chi=\sum_{i=1}^k \chi_i$. 
Then we have 
\begin{align*}
    \chi+\rho-\delta=-\frac{1}{2}R(d)^{\lambda>0}+\gamma
\end{align*}
where $\gamma \in \sum_{i=1}^k \mathbf{W}(d_i)$. 
The above element lies in the boundary of $\mathbf{W}(d)$. 
Then applying the argument of~\cite[Propositions 3.5 and 3.6]{P}, 
one concludes that the categorical Hall product
$m_{\lambda}\left(\boxtimes_{i=1}^k \Gamma_{G(d_i)}(\chi_i)\otimes \mathcal{O}_{\X(d_i)}\right)$
is generated by $\Gamma_{G(d)}(\chi') \otimes \mathcal{O}_{\X(d)}$ 
where $\chi'+\rho-\delta \in \mathbf{W}(d)$. 
(Here in loc. cit. and using the notations used there, \cite[Proposition 3.6]{P} is stated 
in the case that the $r$-invariant is bigger than $1/2$, 
but the case when the $r$-invariant is $1/2$ still works).
 \end{proof}

 \begin{remark}\label{rmk:partition}
 The left hand sides in (\ref{hall:restrict}) are non-zero 
 only if $\langle 1_{d_i}, \theta_i+\delta_i \rangle \in \mathbb{Z}$. 
 Therefore by Proposition~\ref{prop:partition}, the condition 
    $(d_i)_{i=1}^k \in S_{\delta}^d$ is a necessary 
     condition of the left hand sides of (\ref{hall:restrict}) to be non-zero. 
     Indeed, this is also a sufficient condition 
     if $X(d_i)^{\circ} \neq \emptyset$ for all $i$. 
 \end{remark}

We next compute the set of partitions $S_{\delta}^d$ in some particular examples. 

\begin{prop}\label{computationsetsdv}
    Let $Q=(I,E)$ be a quiver with an even number of edges between 
    any two different vertices and an odd number of loops at every vertex. 
    Let $(d,v)\in \mathbb{N}^I\times\mathbb{Z}$. 
    The set $S^d_v$ consists of partitions $\mathbf{d}=(d_i)_{i=1}^k$ such that
    \begin{align*}
        v \cdot \frac{\dd_i}{\dd} \in \mathbb{Z} \ \mbox{ for all } 1\leq i\leq k. 
    \end{align*}
    In particular, if $\gcd(v,\dd)=1$, then $S^d_v$ contains only the one term partition of $d$.
\end{prop}

\begin{proof}
Let $d=(d^a)_{a\in I}\in \mathbb{N}^I$.
    Note that $n_\lambda=\langle \lambda, \mathbb{L}^{\lambda>0}_{\X(d)}|_0\rangle$ is even by the assumption of $Q$. Then $\varepsilon_{\mathbf{d},v\tau_d}=1$ if and only if $\langle \lambda, v\tau_d\rangle\in\mathbb{Z}$ for all cocharacters $\lambda$ with associated partition $\mathbf{d}$.
        We write $\lambda=(\lambda^a)_{a\in I}$, where $\lambda^a\colon\mathbb{C}^*\to T(d^a)$ is a cocharacter
\begin{align}\notag
    \lambda^{a}(t)=(t^{m_1}, \ldots, t^{m_1}, t^{m_2}, \ldots, t^{m_2}, \ldots, t^{m_k}),
\end{align}
where $m_i \in \mathbb{Z}$ (independent of $a$) appears $d^a_i$-times, and 
$m_1>\cdots>m_k$. 
Then the condition $\langle \lambda, v \tau_{d} \rangle \in \mathbb{Z}$
is equivalent to 
\[\frac{v}{\dd} \cdot \sum_{i=1}^k m_i \dd_i \in \mathbb{Z}\] for all tuples of integers $(m_i)_{i=1}^k\in\mathbb{Z}^k$
with $m_1>\cdots>m_k$, 
which implies the desired conclusion.
\end{proof}

\begin{lemma}\label{lem:onevertex}
Let $Q$ be a quiver with one vertex and $2e$-loops for some $e \geq 1$. 
Then $S_v^d$ consists of partitions $\mathbf{d}=(d_i)_{i=1}^k$ such that, for all $1\leq i\leq k$, we have 
\begin{align*}
    \frac{1}{2}d_i\left( \sum_{j<i}d_j-\sum_{j>i}d_j \right)+\frac{vd_i}{d} \in \mathbb{Z}. 
\end{align*}
Moreover,
$S_{v}^d=\{d\}$ if and only if $\gcd(d, v)=1$ and $d \not\equiv 2 \pmod 4$, 
    or if $\gcd(d, v)=2$ and $d \equiv 2 \pmod 4$. 
\end{lemma}
\begin{proof}
In the notation of the proof of Proposition~\ref{computationsetsdv}, 
we have $n_{\lambda}/2+\langle \lambda, v\tau_d\rangle \in \mathbb{Z}$
if and only if
\begin{align*}
     \frac{1}{2}m_i (2e-1)d_i\left( \sum_{j<i}d_j-\sum_{j>i}d_j \right)+\frac{vm_i d_i}{d}
     \in \mathbb{Z}
\end{align*}
for all tuples of pairwise distinct integers $(m_i)_{i=1}^k \in \mathbb{Z}^k$, 
which implies the first result.    

By the above description of $S_{v}^d$,
if $d=d_1+\cdots+d_k$ is an element of $S_v^d$, then 
the 2-length partition $d=d_1+d_2'$
with $d_2'=d_2+\cdots+d_k$ is also an element of $S_v^d$. 
Therefore the second claim follows from Lemma~\ref{lem:decom} below. 
\end{proof}
\begin{lemma}\label{lem:decom}
Let $m \in \mathbb{Z}$ be an odd number. 
    For $(d, v) \in \mathbb{Z}^2$ with $d>0$, 
    there is no decomposition 
    $(d, v)=(d_1, v_1)+(d_2, v_2)$ 
    in $\mathbb{Z}^2$ with 
    $d_i>0$
    satisfying 
    \begin{align}\label{decom:d}
\frac{w_1}{d_1}=\frac{w_2}{d_2}, \ 
w_1 :=v_1-\frac{m}{2}d_1 d_2, \ 
w_2 :=v_2+\frac{m}{2}d_1 d_2
    \end{align}
    if and only if
    $\gcd(d, v)=1$ and $d \not\equiv 2 \pmod 4$, 
    or if $\gcd(d, v)=2$ and $d \equiv 2 \pmod 4$.
\end{lemma}
\begin{proof}
Note that there is a decomposition (\ref{decom:d})
if and only if we have the decomposition in $\mathbb{Z}^2$:
\begin{align}\label{decom:d2}
    (2d, 2v)=(2d_1, 2w_1)+(2d_2, 2w_2)
\end{align}
such that 
$w_1/d_1=v/d$ and 
$v_1=w_1+m d_1 d_2/2 \in \mathbb{Z}$. 
We write $(d, v)=k(d_0, v_0)$ for $k>0$ and $(d_0, v_0) \in \mathbb{Z}^2$
such that $(d_0, v_0)$
is coprime. If $k\geq 3$, there is always
a decomposition (\ref{decom:d2})
by $(2d_1, 2w_1)=4(d_0, v_0)$, 
$(2d_2, 2w_2)=(2k-4)(d_0, v_0)$. 
So we focus on the case that $k=1, 2$. 

Suppose that $k=1$. 
Then the right hand side of (\ref{decom:d2}) 
has divisibility $2$, so if there is a decomposition (\ref{decom:d2})
then 
$(2d_i, 2w_i)=(d, v)$. 
It follows that $d$ is even and $v$ is odd. 
Then we have  
\begin{align*}v_1=\frac{v}{2}+\frac{md^2}{8}
\end{align*}
and it is an integer if and only if $d \equiv 2 \pmod 4$. 

Suppose that $k=2$ and there is a decomposition (\ref{decom:d2}). 
Then there is $1\leq j\leq 3$ 
such that $(2d_1, 2w_1)=j(d_0, w_0)$
and $(2d_2, 2w_2)=(4-j)(d_0, w_0)$. We have 
\begin{align*}v_1=\frac{jv_0}{2}+\frac{j(4-j)md_0^2}{8}
\end{align*}
and it is an integer only if $d_0$ is even. Conversely, if $d_0$ is even, 
then $v_1 \in \mathbb{Z}$ for $j=2$. 
\end{proof}

\begin{prop}\label{prop:delta}
For any symmetric quiver $Q=(I, E)$, there is $\delta \in M(d)_{\mathbb{R}}^{W_d}$ such that 
$S_{\delta}^d=\{d\}$. 
\end{prop}
\begin{proof}
    The condition $\{d\} \in S_{\delta}^d$ is equivalent to that 
    $\langle 1_d, \delta \rangle \in \mathbb{Z}$. 
    In the case that $\lvert I \rvert=1$, then $\delta=v\tau_d$ for some $v \in \mathbb{Z}$. 
    By Proposition~\ref{computationsetsdv} and Lemma~\ref{lem:onevertex}, we have 
    $S_{v}^d=\{d\}$ for either $v=1$ or $v=2$. 
    
    In the case that $\lvert I \rvert \geq 2$, 
    we first take $\delta \in M(d)_{\mathbb{R}}^{W_d}$
    such that $v:=\langle 1_d, \delta \rangle \in \mathbb{Z}$, and set $\delta'=\delta-v\tau_d$. 
    Then $\langle 1_d, \delta' \rangle =0$. 
Note that the space of 
    $\delta' \in M(d)_{\mathbb{R}}^{W_d}$ satisfying $\langle 1_d, \delta' \rangle=0$
    is $\mathbb{R}^{\lvert I \rvert-1}(:=M(d)_{\mathbb{R}, 0}^{W_d})$, hence uncountable. 
Therefore for a very general choice of $\delta'$, 
we have $n_{\lambda}/2+\langle \lambda, \delta' \rangle \notin \mathbb{Z}$
for any $\lambda$ such that the hyperplane 
\begin{align*}H_{\lambda}:=\{x\in M(d)_{\mathbb{R}, 0}^{W_d}: \langle \lambda, x \rangle=0\} 
\subset M(d)_{\mathbb{R}, 0}^{W_d}=\mathbb{R}^{\lvert I \rvert-1}
\end{align*}
is of codimension one. We have 
$H_{\lambda}=\mathbb{R}^{\lvert I \rvert-1}$ if and only 
if the associated partition 
$d=d_1+\cdots+d_k$ is proportional to $d$, i.e. 
by writing $d=m d_0$ with $d_0$ a primitive dimension 
vector and $m \in \mathbb{Z}$, we have $d_i=m_i d_0$
where $m=m_1+\cdots+m_k$ is a partition of $m$. 
Then we have 
\begin{align}\label{compute:lambda}
    \frac{n_{\lambda}}{2}+\langle \lambda, \delta \rangle 
    =\frac{1}{2}\left\langle \nu, (R(d_0)-\mathfrak{g}(d_0))
    \otimes \mathrm{End}(W)^{\nu >0} \right\rangle +\langle \nu, v\tau_m
    \rangle. 
\end{align}
Here $W$ is a $m$-dimensional vector space 
and $\nu$ is an antidominant cocharacter of the maximal 
torus of $GL(W)$ with associated partition $(m_i)_{i=1}^k$. 
Therefore, we are reduced to the same computation
for the one-vertex case
with $e$-loops, where 
$e=\dim R(d_0)-\dim \mathfrak{g}(d_0)+1$,
and 
(\ref{compute:lambda}) is not an integer
for any such $\lambda$ for $v=1$ if $e$ is odd, for $v=1$ if $e$ is even and  
$m \not\equiv 2 \pmod 4$, for $v=2$ if 
$e$ is even and $m\equiv 2 \pmod 4$. 
\end{proof}

\subsection{BPS sheaves}\label{sub613}
Let us consider the stack of $Q$-representations of 
dimension $d$ and its good moduli space:
\[\pi_d\colon \X(d):=R(d)/G(d)\to X(d):=R(d)\ssslash G(d).\]
Recall that the potential $W$ gives a 
function $\Tr W \colon \X(d) \to \mathbb{C}$. 
It factors through the good moduli space map
\begin{align*}
\Tr W \colon \X(d) \stackrel{\pi_d}{\to} X(d) \stackrel{\Tr W}{\to} \mathbb{C}. 
\end{align*}
The affine variety $X(d)$ parametrizes 
semisimple $Q$-representations. 
Let $X(d)^{\circ} \subset X(d)$ 
be the open subset consisting of 
simple $Q$-representations. 
Following~\cite{DM}, the BPS sheaf
is defined as follows: 
\begin{defn}(\cite{DM})
The BPS sheaf $\mathcal{BPS}_d\in \mathrm{Perv}(X(d))$ is defined by 
\begin{align*}\mathcal{BPS}_d:=\begin{cases}
     \varphi_{\mathrm{Tr}\,W} \mathrm{IC}_{X(d)}[-1],\text{ if }X(d)^{\circ}\neq \emptyset,\\
     0,\text{ if }X(d)^{\circ}=\emptyset.
 \end{cases}
 \end{align*}
 \end{defn}
Recall the definition of monodromy invariant 
vanishing cycle in Definition~\ref{def:moninv}. 
We define the \textit{monodromy invariant BPS sheaf}
as follows: 
\begin{defn}
We define $\mathcal{BPS}_d^{\rm{inv}} \in D^b(\mathrm{Sh}_{\mathbb{Q}}(X(d)))$ to be 
\begin{align*}\mathcal{BPS}_d^{\rm{inv}}:=\begin{cases}
     \varphi^{\rm{inv}}_{\mathrm{Tr}\,W} \mathrm{IC}_{X(d)}[-1],\text{ if }X(d)^{\circ}\neq \emptyset,\\
     0,\text{ if }X(d)^{\circ}=\emptyset.
 \end{cases}
 \end{align*}

\end{defn}

% Denote by \[\mathrm{BPS}^{\ast}_d:=\mathrm{H}^{\ast}\left(X(d), \mathcal{BPS}_d\right)\] the BPS cohomology of $(Q,W)$. For $i\in\{0, 1\}$, denote by 
% \begin{equation}\label{def:BPStilde}
% \widetilde{\mathrm{BPS}}^i_d:=\bigoplus_{j\in\mathbb{Z}}\mathrm{BPS}^{i+2j}_d.
% \end{equation}

%If $(Q,W)$ is a tripled quiver with potential, then the mixed Hodge structure on $\mathrm{BPS}^{\ast}_d$ is pure, of Tate type, see \cite[Theorem A]{Dav}.
We define some direct sums of BPS sheaves 
associated with partitions of $d$. 
For a partition $A=(d_i)_{i=1}^k$ of $d$, its length is defined to be $\ell(A):=k$. We write
\begin{align*}\{d_1, \ldots, d_k\}=\{e_1,\ldots, e_s\}, \ 
e_i \neq e_j \mbox{ for } i\neq j
\end{align*}
and that, for each $1\leq i\leq s$, there are $m_i$ elements in $\{d_1, \ldots, d_k\}$ equal to $e_i$. 
We define the following maps, 
given by the direct sums of semisimple 
$Q$-representations  
\begin{align*}\oplus_i\colon X(e_i)^{\times m_i}\to X(m_ie_i), \
\oplus'\colon \times_{i=1}^s X(m_ie_i)\to X(d).
\end{align*}
The above maps are finite maps.
We define the following perverse sheaves: 
\begin{align}\label{BPSAsheaf}
    \notag\mathrm{Sym}^{m_i}\big(\mathcal{BPS}_{e_i}\big)&:=(\oplus_{i})_{*}\left(\mathcal{BPS}_{e_i}^{\boxtimes m_i}\right)^{\mathfrak{S}_{m_i}}\in \mathrm{Perv}(X(m_ie_i)),\\
    \mathcal{BPS}_{A}&:=\oplus'_{\ast}\left(\boxtimes_{i=1}^s \mathrm{Sym}^{m_i}(\mathcal{BPS}_{e_i})\right)\in \mathrm{Perv}(X(d)).
\end{align}
% Let $\mathrm{Sym}^{m_i}(\mathcal{BPS}_{e_i})$ be defined by
% \[\mathrm{Sym}^{m_i}(\mathcal{BPS}_{e_i}):=\oplus_{\ast}\left(\mathcal{BPS}_{e_i}^{\boxtimes m_i}\right)^{\mathfrak{S}_{m_i}},\]
% where $\oplus$ is the addition map 
% $X(e_i)^{\times m_i} \to X(m_i e_i)$.
% Consider the constructible sheaf $\mathcal{BPS}_{A}$ on $X(d)$:
% \begin{equation}\label{BPSAsheaf}
% \mathcal{BPS}_{A}:=\oplus'_{\ast}\left(\boxtimes_{i=1}^s \mathrm{Sym}^{m_i}(\mathcal{BPS}_{e_i})\right)\in \mathrm{Perv}(X(d)).
% \end{equation}
% for the addition map $\oplus' \colon \times_{i=1}^s X(m_i e_i)  \to X(d)$.

\begin{defn}\label{def:BPS:sum}
For $\delta \in M(d)_{\mathbb{R}}^{W_d}$, 
we define the following direct sum of symmetric 
products of BPS sheaves:
\begin{align}\label{defBPSddelta}
\mathcal{BPS}_{d,\delta}&:=\bigoplus_{A\in S^d_\delta}\mathcal{BPS}_A[-\ell(A)]\in D^b(\mathrm{Sh}_{\mathbb{Q}}(X(d))). 
\end{align}
For $\delta=v \tau_d$, we write 
$\mathcal{BPS}_{d, v}:=\mathcal{BPS}_{d,v\tau_d}$ for simplicity. 
\end{defn}
As we will see in \eqref{DMW}, the complexes $\mathcal{BPS}_A$ and $\mathcal{BPS}_{d,\delta}$
are direct summands of $\pi_{d\ast}\varphi_{\Tr W}\mathrm{IC}_{\X(d)}[-1]$ preserved by $1-\mathrm{T}$. 
We define the monodromy invariant versions as follows: 
\begin{defn}
We define 
$\mathcal{BPS}_A^{\mathrm{inv}}, \mathcal{BPS}_{d,\delta}^{\mathrm{inv}}\in D^b(\mathrm{Sh}_{\mathbb{Q}}(X(d)))$ by the exact triangles:
\begin{align*}
    &\mathcal{BPS}_A^{\mathrm{inv}}[-1]\to \mathcal{BPS}_A\xrightarrow{1-\mathrm{T}} \mathcal{BPS}_A\to \mathcal{BPS}_A^{\mathrm{inv}},\\
    &\mathcal{BPS}_{d,\delta}^{\mathrm{inv}}[-1]\to \mathcal{BPS}_{d,\delta}\xrightarrow{1-\mathrm{T}} \mathcal{BPS}_{d,\delta}\to \mathcal{BPS}_{d,\delta}^{\mathrm{inv}}.
\end{align*}
\end{defn}

\begin{remark}\label{rmk:alt}
Alternatively, by the Thom-Sebastiani theorem
we have another description. Let $\mathrm{IC}_A$ be the sheaf defined similarly to $\mathcal{BPS}_A$ 
without potential. Then 
we have, see Lemma~\ref{lem:Q=P}
\begin{align*}
\mathcal{BPS}_A=\varphi_{\mathrm{Tr}\,W}\mathrm{IC}_A[-1], \ 
\mathcal{BPS}_A^{\rm{inv}}=\varphi_{\Tr W}^{\rm{inv}} 
\mathrm{IC}_A[-1]. 
\end{align*}
If $A$ is a trivial partition $A=\{d\}$, then 
$\mathcal{BPS}_A^{\rm{inv}}=\mathcal{BPS}_d^{\rm{inv}}$
by definition. 
\end{remark}

 %We define $\widetilde{\mathrm{BPS}}_A$ as in \eqref{def:BPStilde}.
\subsection{Statement of the main theorems}\label{subsec61}
In this subsection, we state the main theorems of the paper and we note a few corollaries.
We use the notation from the previous subsection. 

 % satisfying Assumption \ref{assum11} and let $W$ be a potential of $Q$. Note that the dimension of $\X(d)$ is even. 
% The Chern character map \eqref{Ch:Dsg} 
Let us consider the Chern character map \eqref{cmap:quot}:
 \begin{align}\label{chsdv}
 \mathrm{ch}\colon K^{\mathrm{top}}_i(\mathbb{S}(d; \delta))
 \to K^{\mathrm{top}}_i(\mathrm{MF}(\X(d), \mathrm{Tr}\,W))\to \widetilde{H}^i(\X(d), \varphi_{\mathrm{Tr}\,W}^{\mathrm{inv}}).\end{align}
%The category $\mathbb{S}(d)_v$ is admissible in $\mathrm{MF}(\X(d), \mathrm{Tr}\,W)$ by applying matrix factorizations to the decomposition in Theorems \ref{sodfullstackB} or \ref{theorem266}, so the map $\alpha$ is injective, and thus $\mathrm{ch}$ is injective. 
Here the first map is injective over $\mathbb{Q}$
by Proposition~\ref{cherninj}. 
Recall \eqref{deffiltrationKsg} and define the filtration:
\[E_\ell K^{\mathrm{top}}_i(\mathbb{S}(d; \delta)):=K^{\mathrm{top}}_i(\mathbb{S}(d; \delta))\cap E_\ell K^{\mathrm{top}}_i(\mathrm{MF}(\X(d), \mathrm{Tr}\,W))\subset K^{\mathrm{top}}_i(\mathbb{S}(d; \delta)).\]
There is an injective cycle map on the associated graded pieces:
\begin{align}\label{cyclesdv}
    \mathrm{c}\colon \mathrm{gr}_\ell K^{\mathrm{top}}_i(\mathbb{S}(d; \delta))&\to H^{2 \dim\X(d)-2\ell-i}(\X(d), \varphi_{\mathrm{Tr}\,W}^{\mathrm{inv}})\\
    &\notag=
    H^{\dim \X(d)-2\ell-i}(\X(d), \varphi_{\mathrm{Tr}\,W}^{\mathrm{inv}}\,\mathrm{IC}_{\X(d)}),
\end{align}
where we used $\varphi_{\mathrm{Tr}\,W}\mathrm{IC}_{\X(d)}=\varphi_{\mathrm{Tr}\,W}[\dim \X(d)]$ in the computation 
of the cohomological degree.
The following is the main result of this section:

\begin{thm}\label{thm1}
The cycle map \eqref{cyclesdv} induces an isomorphism for $i, \ell\in\mathbb{Z}$:
\begin{align}\label{cyclesdw0}
 \mathrm{c}\colon \mathrm{gr}_\ell K_i^{\mathrm{top}}\left(\mathbb{S}(d; \delta)\right)_{\mathbb{Q}}\xrightarrow{\cong}H^{\dim\X(d)-2\ell-i}(X(d), \mathcal{BPS}_{d, \delta}^{\mathrm{inv}}[1]). 
% \bigoplus_{A\in S^d_v}
% % &H^{\dim\X(d)_0-2a-i}(X(d),\mathrm{BPS}_{A})_{\mathrm{inv}}\oplus\\
% % \notag &H^{\dim\X(d)_0-2a-i+1}(X(d),\mathrm{BPS}_{A})^{\mathrm{inv}}.
\end{align}
\end{thm}
Note that, by Proposition~\ref{prop:delta}, there exists $\delta$ satisfying 
$S_{\delta}^d=\{d\}$. 
In this case, 
 we regard $\mathbb{S}(d; \delta)$ as a categorification of the monodromy invariant BPS cohomology of $(Q,W)$. 
The main part of proving Theorem \ref{thm1} is the construction of a cycle map from the topological K-theory of quasi-BPS categories to BPS cohomologies, 
whose proof will be given in the next section.

\begin{thm}\label{thm2}
Let $(Q, W)$ be a symmetric quiver with 
potential. 
Let $d\in \mathbb{N}^I$, $\delta\in M(d)^{W_d}_\mathbb{R}$, and $i, \ell\in\mathbb{Z}$.
The cycle map \eqref{cyclesdv} induces a map:  
\begin{equation}\label{cyclesdw2}
\mathrm{c}\colon \mathrm{gr}_\ell K_i^{\mathrm{top}}\left(\mathbb{S}(d; \delta)\right)\to 
H^{\dim \X(d)-2\ell-i}\left(X(d), \mathcal{BPS}_{d,\delta}^{\mathrm{inv}}[1]\right).
\end{equation} 
\end{thm}

 % By Theorem \ref{sodfullstackB}, the category $\mathbb{S}(d)_v$ is admissible in $\mathrm{MF}\left(\X(d), \mathrm{Tr}\,W\right)$, and thus by Proposition \ref{prop400} we have that:
% \begin{equation}\label{injectivitysdw}
%     K_0^{\mathrm{top}}\left(\mathbb{S}^{\mathrm{gr}}(d; v\tau_d)\right)\hookrightarrow K_0^{\mathrm{top}}\left(\mathrm{MF}^{\mathrm{gr}}\left(\X^f(d)^{\mathrm{ss}}, \mathrm{Tr}\,W\right)\right).
% \end{equation}

%The cycle map \eqref{cyclesdv} is injective, thus the cycle map \eqref{cyclesdw2} is also injective.

We note that the cycle map (\ref{cyclesdw2}) is injective, 
as (\ref{cyclesdv}) is injective
and the map (\ref{cyclesdw2}) is 
a map to the direct summand in the right hand side 
of (\ref{cyclesdv}). 
We mention a numerical corollary of Theorems \ref{thm1} and \ref{thm2}. 
Before this, we have the following lemma: 
\begin{lemma}\label{thm2inj}
There is an equality
    \begin{align}\label{dim:equal}
\dim_{\mathbb{Q}}K_i^{\rm{top}}(\mathbb{S}(d; \delta))_{\mathbb{Q}}
=\dim_{\mathbb{Q}}\mathrm{gr}_{\ast}K_i^{\rm{top}}(\mathbb{S}(d; \delta))_{\mathbb{Q}}. 
\end{align}
\end{lemma}
\begin{proof}
The Chern 
character map (\ref{chsdv}) is injective 
over $\mathbb{Q}$ by 
Proposition \ref{cherninj} and Theorem \ref{theorem266}.
Thus $E_{\ell}K_i^{\rm{top}}(\mathbb{S}(d; \delta))_{\mathbb{Q}}=0$ for $\ell \ll 0$, 
so we obtain the equality (\ref{dim:equal}). 
\end{proof}

\begin{cor}\label{corollarytheorem61}
Let $(Q, W)$ be a symmetric quiver with
potential. Let $d \in \mathbb{N}^I$ and $\delta \in M(d)_{\mathbb{R}}^{W_d}$. 
Then there is an equality
\begin{equation}\label{ineqq}
\dim_{\mathbb{Q}} K^{\mathrm{top}}_i(\mathbb{S}(d; \delta))_{\mathbb{Q}}= \dim_{\mathbb{Q}} H^{\ast}(X(d), \mathcal{BPS}_{d, \delta})^{\mathrm{inv}}.
\end{equation}
\end{cor}
\begin{proof}[Proof of Corollary~\ref{corollarytheorem61} from Theorems~\ref{thm1},~\ref{thm2}]
Note that there is a (non-canonical) isomorphism
\begin{align}\label{isom:noncan}
    H^\bullet(X(d), \mathcal{BPS}_{d,\delta})^{\mathrm{inv}} \cong H^\bullet(X(d), \mathcal{BPS}_{d,\delta})_{\mathrm{inv}}.
    \end{align}
The claim then follows from Theorem~\ref{thm1} and Lemma~\ref{thm2inj}.  
\end{proof}

The result of 
 Corollary \ref{corintro} follows easily from Theorem \ref{thm1}, as follows: 

 \begin{proof}[Proof of Corollary \ref{corintro}]
 Note that Proposition \ref{computationsetsdv} implies that $\mathbf{d}=(d_i)_{i=1}^k\in S^d_v$ if and only if $\dd/\gcd(\dd,v)$ divides $\dd_i$ for $1\leq i\leq k$. Then
 $S^d_v=S^d_{v'}$ for $v,v'\in\mathbb{Z}$ such that $\gcd(\dd,v)=\gcd(\dd,v')$. 
     The statement then follows from Theorem \ref{thm1}.
 \end{proof}

%In Section \ref{subsection:preproj}, we compute the topological K-theory of quasi-BPS categories of preprojective algebras of quivers using Theorem \ref{thm1}, see Theorem \ref{thm1plus}. In \cite{PTK3}, we further use Theorem \ref{thm1plus} to compute the topological K-theory of quasi-BPS categories of K3 surfaces. In particular, we obtain categorifications of the BPS cohomology of a large class of preprojective algebras and of K3 surfaces. 

We finish this subsection by the discussion of the zero potential case of Theorem \ref{thm1}.
In this case $\mathcal{BPS}_d=\mathrm{IC}_{X(d)}$, 
and we denote by $I\!H^\bullet(X(d)):=H^\bullet\big(X(d), \mathrm{IC}_{X(d)}\big)$. By Theorem \ref{thm1} and Proposition \ref{Lambdastructure}, we obtain the following:
% Then $H^{\mathrm{odd}}(X(d))=0$ because $H^\bullet(\X(d))\twoheadrightarrow I\!H^{\bullet-\dim \X(d)+2}(X(d))$ by Kirwan surjectivity, alternatively by the decomposition \ref{DM}, and $H^{\mathrm{odd}}(\X(d))=0$. We denote by $I\!H^{\ast}(X(d)):=\mathrm{H}^{\ast}(X(d), \mathrm{IC})$.

\begin{thm}\label{thmWzero}
Let $Q=(I, E)$ be a symmetric quiver, 
let $d\in \mathbb{N}^I$
and $\delta \in M(d)_{\mathbb{R}}^{W_d}$
such that $S_{\delta}^d=\{d\}$. 
For $\ell\in\mathbb{Z}$, the cycle map induces an isomorphism:
   \[\mathrm{c}\colon \mathrm{gr}_\ell K^{\mathrm{top}}_0(\mathbb{M}(d; \delta))_{\mathbb{Q}}\xrightarrow{\cong} I\!H^{\dim \X(d)-2\ell-1}(X(d)).\]
   In particular, we have 
    \[\dim_\mathbb{Q} K^{\mathrm{top}}_0(\mathbb{M}(d; \delta))_{\mathbb{Q}}=\dim_\mathbb{Q} I\!H^{\ast}(X(d)).\]   
\end{thm}
\begin{remark}
  Note that $H^{\mathrm{odd}}(\X(d))=I\!H^{\mathrm{odd}-\dim \X(d)}(\X(d))=0$. We then have that $I\!H^{\mathrm{even}-\dim \X(d)}(X(d))=0$ because $\mathrm{IC}_{X(d)}[-1]$ is a direct summand of $R\pi_{d\ast}\mathrm{IC}_{\X(d)}$, see \eqref{DM}; alternatively, the vanishing $I\!H^{\mathrm{even}-\dim \X(d)}(X(d))=0$ follows from Kirwan surjectivity.   
\end{remark}

%We note a consequence of Corollary \ref{corollarytheorem61}, alternatively a numerical corollary of Theorem \ref{thmWzero}.

%\begin{cor}\label{corollarythmWzero}
%    Let $Q=(I, E)$ be a quiver satisfying Assumption~\ref{assum11}. 
%    For $d\in \mathbb{N}^I$
%   and $v\in \mathbb{Z}$ such that $\gcd\left(\underline{d}, v\right)=1$ and $i\in \mathbb{Z}$, we have 
%    \[\dim_\mathbb{Q} K^{\mathrm{top}}_0(\mathbb{M}%(d)_v)=\dim_\mathbb{Q} I\!H^{\ast}(X(d)).\]
%\end{cor}

For $\delta \in M(d)_{\mathbb{R}}^{W_d}$
with $S_{\delta}^d=\{d\}$, we 
regard $\mathbb{M}(d; \delta)$ as a categorification of the intersection cohomology of $X(d)$. Note that, in general, $X(d)$ is a singular variety, and it is not straightforward to construct (natural) categorifications of intersection cohomology for a singular variety.

\begin{example}\label{exam:loop}
We discuss the isomorphism in Theorem~\ref{thm1} in a simple example. 
Let $Q$ be a quiver with one vertex and $(2e+1)$-loops, and 
consider the case of $d=2$ and $\delta=\tau_2$. Let $V$ be a two dimensional $\mathbb{C}$-vector space. 
The moduli stack is \[\X(2)=\mathfrak{gl}(2)^{\oplus (2e+1)}/GL(2)=\mathrm{End}(V)^{\oplus (2e+1)}/\mathrm{Aut}(V),\] 
and we have 
\begin{align*}
    K_0^{\rm{top}}(\X(2))=\mathbb{Z}[q_1^{\pm 1}, q_2^{\pm 1}]^{\mathfrak{S}_2}. 
\end{align*}
The quasi-BPS category $\mathbb{M}(2)_1$ is generated by 
$\mathrm{Sym}^{2a+1}(V) \otimes (\det V)^{-a}$ for $0\leq a\leq e-1$, 
so we have 
\begin{align*}
    K^{\rm{top}}_0(\mathbb{M}(2)_1)=\bigoplus_{a=0}^{e-1}
    \mathbb{Z}\cdot(q_1^{a+1} q_2^{-a}+q_1^{a}q_2^{1-a}+\cdots+q_1^{-a}q_2^{a+1}). 
\end{align*}
The associated graded has the description: 
\begin{align}\label{summandsgr}
     \mathrm{gr}_{\ast}K^{\rm{top}}_0(\mathbb{M}(2)_1)=
     \bigoplus_{a=0}^{e-1}\mathbb{Z}\cdot q_1^{-a}q_2^{-a}(q_1+q_2)(q_1-q_2)^{2a}. 
\end{align}
On the other hand, we will recall
in Theorem~\ref{Dav-Mein:deco}
the decomposition theorem of intersection cohomologies of 
moduli stacks of representations of quivers by Davison-Meinhardt~\cite{DM}. 
In the case of this example, the above decomposition is 
\begin{align*}
I\!H^{\ast}(\X(2))=&I\!H^{\ast}(X(2)) \otimes (\mathbb{Q}[-1]\oplus \mathbb{Q}[-3]\oplus \cdots ) \\
&\oplus \mathrm{Sym}^2(I\!H^{\ast}(X(1)) \otimes (\mathbb{Q}[-1] \oplus \mathbb{Q}[-3] \oplus \cdots)).
\end{align*}
We have $I\!H^{\ast}(\X(2))=\mathbb{Q}[\hbar_1, \hbar_2]^{\mathfrak{S}_2}$ and 
the above decomposition is 
(ignoring grading)
\begin{align*}
    \mathbb{Q}[\hbar_1, \hbar_2]^{\mathfrak{S}_2}=&\oplus_{a=0}^{e-1}\mathbb{Q}\cdot (\hbar_1-\hbar_2)^{2a}
    \otimes (\mathbb{Q}\oplus \mathbb{Q}(\hbar_1+\hbar_2) \oplus \cdots) \\
    &\oplus_{m_1 \geq m_2 \geq 0} \mathbb{Q}(\hbar_1-\hbar_2)^{2e}(\hbar_1^{m_1}\hbar_2^{m_2}+\hbar_1^{m_2}\hbar_2^{m_1}),  
\end{align*}
where $I\!H^{\ast}(X(2))$ is 
\begin{align}\label{summandsIH}
    I\!H^{\ast}(X(2))=\bigoplus_{a=0}^{e-1}\mathbb{Q} \cdot (\hbar_1-\hbar_2)^{2a}.
\end{align}
The cycle map sends summands of \eqref{summandsgr} to summands of \eqref{summandsIH} as follows:
\begin{align*}
     c\colon q_1^{-a}q_2^{-a}(q_1+q_2)(q_1-q_2)^{2a}
     \mapsto 2(\hbar_1-\hbar_2)^{2a}. 
\end{align*}
Note that we have 
\begin{align*}
\mathrm{Sym}^2(I\!H^{\ast}(X(1))[-1])[-2m_1-2m_2]=\mathbb{Q}(\hbar_1-\hbar_2)^{2e}(\hbar_1^{m_1}\hbar_2^{m_2}+\hbar_1^{m_2}\hbar_2^{m_1})
\end{align*}
via cohomological Hall product (\ref{coha}), see the explicit formula in~\cite[Theorem~2]{MR2851153} for the cohomological Hall product. 
\end{example}

\subsection{The decomposition theorem}\label{subsection:decompositiontheorem}

Let $\alpha\in\mathbb{N}$ and recall the construction of framed quivers $Q^{\alpha f}$ from Subsection \ref{subsec:sod}.

We review the explicit computation of summands in the BBDG decomposition theorem \cite{BBD} for the pushforward of the constant sheaf along the maps: 
\begin{align*}
    \pi_{\alpha f,d} \colon \X^{\alpha f}(d)^{\text{ss}}\to X(d),\
    \pi_d \colon \X(d)\to X(d)
\end{align*} due to Meinhardt--Reineke \cite{MeRe} and Davison--Meinhardt \cite{DM}.
The maps $\pi_{\alpha f,d}$ ``approximate" the map $\pi_d$, see \cite[Subsection 4.1]{DM}.
% in particular $\pi_{d*}$ commutes with $\varphi_{\mathrm{Tr}\,W}$. 
The computation of $\pi_{d*}\mathrm{IC}_{\X(d)}$ is deduced from the computation of $\pi_{\alpha f, d *}\mathbb{Q}_{\X^{\alpha f}(d)^{\mathrm{ss}}}[\dim \X(d)]$.
% We begin by discussing the computation for $\pi_{d*}\mathrm{IC}_{\X(d)}$.
% The perverse t-structure on $X(d)$ induced a perverse filtration on 
% \[\varphi_{\mathrm{Tr}\,W}\pi_{d*}\mathrm{IC}_{\X(d)}=\pi_{d*}\varphi_{\mathrm{Tr}\,W}\mathrm{IC}_{\X(d)}.\]  

We introduce some objects of $D^b(\mathrm{Sh}_{\mathbb{Q}}(X(d)))$, 
which generalize the construction in Subsection~\ref{sub613}. 
Let $A$ be a tuplet $(e_i, m_{i,a})$ for $1\leq i\leq s$ and for $a\geq 0$, with $(e_i)_{i=1}^s\in \mathbb{Z}_{\geq 1}^s$ pairwise distinct and $m_{i,a}\geq 0$ such that 
\begin{align*}\sum_{i=1}^s \sum_{a\geq 0}e_im_{i,a}=d.
\end{align*}
Let $\mathscr{P}$ be the set of all such tuplets $A$ and let $\mathscr{P}_\alpha\subset \mathscr{P}$ be the subset of such tuplets with $m_{i,a}=0$ for $a\geq \alpha e_i$. Note that each $A$ has a corresponding partition with terms $e_i$ with multiplicity $\sum_{a\geq 0}m_{i,a}$ for $1\leq i\leq s$.
Below we often regard a partition 
$\{d_1, \ldots, d_k\}=\{e_1, \ldots, e_s\}$ of $d$ 
where $e_i$ appears $m_i$-times
as an element $(e_i, m_{i, a})$ of
$\mathscr{P}$ such that $m_{i, 0}=m_i$ and $m_{i, a}=0$ 
for $a\geq 1$. 
We denote by $\mathscr{P}^{\circ} \subset \mathscr{P}$ 
the subset of such tuplets. 
%Let $A$ be a partition of $d$; as in Subsection \ref{sub613}, we write $A$ as a tuplet $(e_i, m_{i,a})$ for $1\leq i\leq s$ and for $a\geq %0$, with $(e_i)_{i=1}^s\in \mathbb{Z}_{\geq 1}^s$ pairwise distinct and $m_{i,a}\geq 0$ such that $\sum_{i=1}^s \sum_{a\geq 0}e_im_{i,a}=d$. 
%Let $\mathscr{P}$ be the set of all such tuplets $A$ and let $\mathscr{P}_\alpha\subset \mathscr{P}$ be the subset of such tuplets with $m_{i,a}=0$ for $a\geq \alpha e_i$.

We consider the addition maps:
\begin{align}
    &\oplus_{i,a}\colon X(e_i)^{\times m_{i,a}}\to X(m_{i,a}e_i),\\
    &\notag \oplus'\colon \times_{i=1}^s\times_{a\geq 0} X(m_{i,a}e_i)\to X(d).
\end{align}
Using the above addition maps, we 
define the following complexes of 
constructible sheaves:
\begin{align*}
    \mathrm{Sym}^{m_{i,a}}\left(\mathrm{IC}_{X(e_i)}[-2a-1]\right)&:=\oplus_{i,a *} \left(\left(\mathrm{IC}_{X(e_i)}[-2a-1]\right)^{\boxtimes m_{i,a}}\right)^{\mathfrak{S}_{m_{i,a}}},\\
\mathrm{P}_A&:=\oplus'_*\left(\boxtimes_{1\leq i\leq s, a\geq 0} \mathrm{Sym}^{m_{i,a}} \left(\mathrm{IC}_{X(e_i)}[-2a-1]\right)\right).
\end{align*}
% Let $\Delta_{i,a}\subset X(d_i)^{\times m_{i,a}}$ be the set of ordered $m_{i,a}$ points in $X(d_i)$, not all different. Consider the direct sum maps
% \begin{align*}
%     \pi^\circ_{i,a}&\colon \times_{i=1}^k \times_{a\geq 0} \left(X(d_i)^{\times m_{i,a}}\setminus \Delta_{i,a}\right)\to X(d),\\
%     \pi'_{i,a}&\colon \times_{i=1}^k \times_{a\geq 0} X(d_i)^{\times m_{i,a}}\to X(d).
% \end{align*}
% There is a natural map on the fibers of $\pi^\circ_{i,a}$ and $\pi'_{i,a}$.
% We denote by 
% \[\mathrm{P}_A:=\bigotimes_{i=1}^k\bigotimes_{a\geq 0} \mathrm{Sym}^{m_{i,a}} (\mathrm{IC}_{d_i}[-2a-1])\] the corresponding shift of the IC extension of the line bundle $\left(\pi^\circ_{i,a*}\mathbb{Q}\right)^{\mathfrak{S}_A}$ on the image on $\pi'_{i,a}$. 
Then $\mathrm{P}_A$ is supported on the image of $\oplus'$ and is a shifted perverse sheaf of degree
\[p_A:=\sum_{i=1}^s\sum_{a\geq 0}m_{i,a}(2a+1),\] 
namely we have that $\mathrm{P}_A[p_A]\in \mathrm{Perv}(X(d))$.
% Then $\mathrm{P}_A$ is supported on the image of $\oplus'$. 

Let $\alpha$ be an even positive natural number.
The following explicit form of the BBDG decomposition theorem for $\pi_{\alpha f,d}$ was determined by Meinhardt--Reineke:
\begin{thm}\emph{(\cite[Proposition 4.3]{MeRe})}\label{thm:MeiRei}
There is a decomposition:
\begin{equation}\label{MR}
    \pi_{\alpha f,d*}\left(\mathbb{Q}_{\X^{\alpha f}(d)^{\text{ss}}}[\dim \X(d)]\right)=\bigoplus_{A\in \mathscr{P}_\alpha}P_A.
\end{equation}
\end{thm}
\begin{remark}
The result in loc. cit. is stated as an equality in the Grothendieck group of constructible sheaves, but the above stronger statement holds by the argument in \cite[Proof of Theorem 4.10]{DM}.
\end{remark}
Using Theorem~\ref{thm:MeiRei}, 
the following decomposition is obtained by Davison--Meinhardt:
\begin{thm}\emph{(\cite[Theorem~C]{DM})}\label{Dav-Mein:deco}
There is a decomposition: 
\begin{equation}\label{DM}
\pi_{d*}\mathrm{IC}_{\X(d)}=\bigoplus_{A\in\mathscr{P}}\mathrm{P}_A.
\end{equation}
\end{thm}

In the case that there is a potential $W$, we define  
\begin{equation}\label{defqa}
\mathrm{Q}_A:=\oplus'_*\left(\boxtimes_{1\leq i\leq s, a\geq 0} \mathrm{Sym}^{m_{i,a}} \left(\varphi_{\mathrm{Tr}\,W}\mathrm{IC}_{X(e_i)}[-2a-2]\right)\right).
\end{equation}
\begin{lemma}\label{lem:Q=P}
We have 
\[\mathrm{Q}_A=\varphi_{\mathrm{Tr}\,W}\mathrm{P}_A[-1].\]
\end{lemma}
\begin{proof}
The lemma follows from the Thom-Sebastiani theorem. 
\end{proof}
The proper pushforward commutes with the vanishing cycle functor. The map $\pi_d$ can be approximated by the proper maps $\pi_{\alpha f, d}$, thus $\pi_{d*}$ also commutes with the vanishing cycle functor.
From the above observations and 
using 
the results of Theorem~\ref{thm:MeiRei}, Theorem~\ref{Dav-Mein:deco}, 
the following vanishing cycle versions of these 
are also obtained in~\cite{DM}:
\begin{thm}\emph{(\cite{DM})}
There are decompositions: 
\begin{align}\label{MRW}
    &\pi_{\alpha f,d*}\varphi_{\mathrm{Tr}\,W}\left(\mathbb{Q}_{\X^{\alpha f}(d)^{\text{ss}}}[\dim \X(d)-1]\right)=\bigoplus_{A\in \mathscr{P}_\alpha}Q_A, \\
    \label{DMW}
    &\pi_{d*}\varphi_{\mathrm{Tr}\,W}\mathrm{IC}_{\X(d)}[-1]=\bigoplus_{A\in\mathscr{P}}Q_{A}.
\end{align}
The summands in all the above decompositions are induced via cohomological Hall product. 
\end{thm}

\subsection{Relative topological K-theory of quasi-BPS categories}\label{subsec:reltop}
In this subsection, we prove Theorem~\ref{thm1}
assuming Theorem~\ref{thm2}. 
A key point is to 
study the relative topological K-theory 
of quasi-BPS categories over $X(d)$, and show that 
it contains the BPS sheaf as a direct summand. 
Then the dimension of the K-theory of the quasi-BPS category is bigger 
than or equal to the dimension of the BPS 
cohomology, thus Theorem~\ref{thm2} implies 
Theorem~\ref{thm1}. 

\begin{proof}[Proof of Theorem \ref{thm1} assuming Theorem \ref{thm2}]
By a standard argument (whose proof is postponed to Proposition~\ref{reduce} and Lemma~\ref{clambdazero}), 
	we may assume that $Q$ has at least two loops at each vertex
 so that $X(d) \neq \emptyset$. 
	Let $\alpha$ be an even positive integer and consider 
	the proper morphism 
	\begin{align*}
		\pi_{\alpha, d} \colon \X^{\alpha f}(d)^{\rm{ss}} \to \X(d) \stackrel{\pi_d}{\to} X(d). 
		\end{align*}
	For simplicity, we write $\X^{\alpha f}(d)^{\rm{ss}}=\X^{\dag}$ and 
	$\pi^{\dag}=\pi_{\alpha, d}$. 
		The local Chern character map is an isomorphism, see Section~\ref{sec:topK:dg}:
	\[\ch \colon \mathcal{K}_{\X^{\dag}}(D^b(\X^{\dag}))_{\mathbb{Q}}\stackrel{\cong}{\to}
	\mathbb{Q}_{\X^{\dag}}[\beta^{\pm 1}].\] 
	By considering its push-forward to $X(d)$ and
 using Theorem~\ref{thm:phiproper} and the isomorphism (\ref{isom:KU}), 
 we obtain 
	the sheaf-theoretic Chern character map 
	\begin{align*}
		\ch \colon 
		\mathcal{K}_{X(d)}^{\rm{top}}(D^b(\X^{\dag}))_{\mathbb{Q}} \cong
  \pi^{\dag}_{\ast}\mathcal{K}_{\X^{\dag}}(D^b(\X^{\dag}))_{\mathbb{Q}} \stackrel{\cong}{\to}
  \pi^{\dag}_{\ast}\mathbb{Q}_{\X^{\dag}}[\beta^{\pm 1}]. 
		\end{align*}
Together with Theorem~\ref{thm:MeiRei}, there is an isomorphism 
	\begin{align}
		\ch \colon \label{KX(d)}
			\mathcal{K}_{X(d)}^{\rm{top}}(D^b(\X^{\dag}))_{\mathbb{Q}}\stackrel{\cong}{\to} \bigoplus_{A \in \mathscr{P}_{\alpha}}P_A[\beta^{\pm 1}][-\dim \X(d)]. 
		\end{align}
  
	By Theorem~\ref{theorem266}, 
	the  subcategory $\mathbb{M}(d; \delta) \subset D^b(\X^{\dag})$ 
	is right admissible and $X(d)$-linear. 
	Therefore, there is a decomposition by Lemma~\ref{lem:sod:Ktheory}:
	\begin{align}\label{decomp1}
		\mathcal{K}_{X(d)}^{\rm{top}}(D^b(\X^{\dag}))_{\mathbb{Q}}=
		\mathcal{K}_{X(d)}^{\rm{top}}(\mathbb{M}(d; \delta))_{\mathbb{Q}} \oplus K'. 		
		\end{align}
In particular, 
$\mathcal{K}_{X(d)}^{\rm{top}}(\mathbb{M}(d; \delta))_{\mathbb{Q}}$
is a direct summand of 
the right hand side of (\ref{KX(d)}), 
therefore it 
 splits into the direct sum of its perverse cohomologies. 
As it is also $2$-periodic, there exists a subset $\mathscr{P}_{\alpha}' \subset \mathscr{P}_{\alpha}$
such that the isomorphism (\ref{KX(d)}) restricts 
to the isomorphism
\begin{align}\label{Ktop:Mdelta}
	\ch \colon 
	\mathcal{K}_{X(d)}^{\rm{top}}(\mathbb{M}(d; \delta))_{\mathbb{Q}} \stackrel{\cong}{\to}
	 \bigoplus_{A \in \mathscr{P}_{\alpha}'}
	P_A[\beta^{\pm 1}][-\dim \X(d)]. 
	\end{align}

For a partition $A \in \mathcal{S}_{\delta}^{d}$, we regard
it as an element of $\mathscr{P}^{\circ} \subset \mathscr{P}$, 
see Subsection~\ref{subsection:decompositiontheorem}. 
Below we show that $\mathcal{S}_{\delta}^{d} \subset \mathscr{P}_{\alpha}'$. 
We first show the case that $A$
is the 1-length partition $A=\{d\}$. 
Recall that the open subset 
$X(d)^{\circ} \subset X(d)$ of simple 
representations is non-empty. 
The maps 
\begin{align*}
    &\X(d)^{\circ}:=\pi_d^{-1}(X(d)^{\circ}) \to X(d)^{\circ}, \\ 
    &\X^{\dag\circ}:=(\pi^{\dag})^{-1}(X(d)^{\circ}) \to X(d)^{\circ}
\end{align*}
are a $\mathbb{C}^{\ast}$-gerbe and a 
$\mathbb{P}^{N-1}$-bundle for $N=\alpha \dd$, respectively. 
Note that, as we assumed $\{d\} \in S_{\delta}^{d}$, 
we have $v=\langle 1_d, \delta \rangle \in \mathbb{Z}$.
There are semiorthogonal decomposition from the projective bundle formula 
\begin{align}\label{sod:pbundle}
	D^b(\X^{\dag\circ})=\big\langle D^b(\X(d)^{\circ})_{w} \,\big|\, v\leq w <N+v \big\rangle. 
	\end{align}
The weight $w$-part is equivalent to 
$D^b(X(d)^{\circ}, \alpha^w)$ for a Brauer class $\alpha$. 
Since the associated element $
w\hat{\alpha} \in H^3(X(d)^{\circ}, \mathbb{Z})$ is torsion, 
it does not affect the K-theory after rationalization. 
It follows that 
\begin{align*}
	\mathcal{K}_{X(d)}^{\rm{top}}(D^b(\X^{\dag}))_{\mathbb{Q}}|_{X(d)^{\circ}}\cong
	\mathcal{K}_{X(d)^{\circ}}^{\mathrm{top}}(\X^{\dag\circ})_{\mathbb{Q}}\cong
	\bigoplus_{w=v}^{N+v-1}\mathbb{Q}_{X(d)^{\circ}}[\beta^{\pm 1}]. 
	\end{align*}
On the other hand, the base change of the semiorthogonal 
decomposition in Theorem~\ref{theorem266} via $X(d)^{\circ}\subset X(d)$ 
is the semiorthogonal decomposition (\ref{sod:pbundle}), and 
$\mathbb{M}(d; \delta)$ corresponds to $D^b(\X(d)^{\circ})_{v}$ under the 
base change. Here, we refer to Subsection~\ref{basechange}
for the base change of semiorthogonal decomposition. 
Therefore, we have 
\begin{align*}
	\mathcal{K}_{X(d)}^{\mathrm{top}}(\mathbb{M}(d; \delta))_{\mathbb{Q}}|_{X(d)^{\circ}} \cong 
	\mathbb{Q}_{X(d)^{\circ}}[\beta^{\pm 1}]. 
	\end{align*}
By comparing with (\ref{Ktop:Mdelta}),
there exists $A \in \mathscr{P}_{\alpha}'$ such that 
$P_A=\mathrm{IC}_{X(d)}[-1]$, i.e. $\{d\} \in \mathscr{P}_{\alpha}'$.  
In particular, 
we have an isomorphism of the form 
\begin{align}\label{decomp2}
	\ch \colon 
	\mathcal{K}_{X(d)}^{\rm{top}}(\mathbb{M}(d; \delta))_{\mathbb{Q}} \stackrel{\cong}{\to} \mathrm{IC}_{X(d)}[\beta^{\pm 1}][1-\dim \X(d)] \oplus K''. 
	\end{align}

We next show that any $A \in S_{\delta}^d$ is contained in 
$\mathscr{P}_{\alpha}'$. 
Let $A$ correspond to a partition 
$\mathbf{d}=\{d_i\}_{i=1}^k$, and $\lambda$ an antidominant 
cocharacter with associated partition $\mathbf{d}$. 
By Proposition~\ref{prop:partition}, the categorical 
Hall product restricts to the functor
\begin{align*}
	\boxtimes_{i=1}^k \mathbb{M}(d_i; \theta_i+\delta_i) \to \mathbb{M}(d; \delta). 
	\end{align*}
The above functor is linear over $X(d)$, 
so it induces the morphism in $D(\mathrm{Sh}_{\mathbb{Q}}(X(d)))$:
\begin{align*}
	i_{\lambda\ast}\boxtimes_{i=1}^k \mathcal{K}_{X(d_i)}^{\rm{top}}(\mathbb{M}(d_i; \theta_i+\delta_i))_{\mathbb{Q}} \to \mathcal{K}_{X(d)}^{\rm{top}}(\mathbb{M}(d; \delta))_{\mathbb{Q}}. 
	\end{align*}
Here we have used the notation of the diagram (\ref{diagram:natural}). 
From Remark~\ref{rmk:partition}, we have $\langle 1_{d_i}, \theta_i+\delta_i \rangle \in \mathbb{Z}$. 
Therefore we have the decomposition (\ref{decomp2})
for $(d, \delta)$ replaced by $(d_i, \theta_i+\delta_i)$:
\begin{align}\notag
	\ch \colon 
	\mathcal{K}_{X(d_i)}^{\rm{top}}(\mathbb{M}(d_i; \theta_i+\delta_i))_{\mathbb{Q}} \stackrel{\cong}{\to} \mathrm{IC}_{X(d_i)}[\beta^{\pm 1}][1-\dim \X(d_i)] \oplus K_i. 
	\end{align}
We have the following diagram 
\begin{align*}
	\xymatrix{
i_{\lambda\ast}\boxtimes_{i=1}^k \mathcal{K}_{X(d_i)}^{\rm{top}}(\mathbb{M}(d_i; \theta_i+\delta_i)) \ar[r]^-{p_{\lambda\ast}q_{\lambda}^{\ast}} \ar[d]_-{\ch} & \mathcal{K}_{X(d)}^{\rm{top}}(\mathbb{M}(d; \delta)) \ar[d]_-{\ch} \\
i_{\lambda\ast}(\boxtimes_{i=1}^k \mathrm{IC}_{X(d_i)}[-1])[-\dim \X(d)^{\lambda}][\beta^{\pm 1}] \oplus K'''	& \bigoplus_{A' \in \mathscr{P}_{\alpha}'}P_{A'}[-\dim \X(d)][\beta^{\pm 1}]. 
}
	\end{align*}
 Here $K'''$ is of the form $K'''=i_{\lambda\ast}\boxtimes_{i=1}^k K_i'$
 such that $K_i'=K_i$ for some $i$. 
By taking the total cohomologies, we obtain 
the commutative diagram 

\begin{align*}
	\xymatrix{
		\mathcal{H}^{\ast}(i_{\lambda\ast}\boxtimes_{i=1}^k \mathcal{K}_{X(d_i)}^{\rm{top}}(\mathbb{M}(d_i; \theta_i+\delta_i))_{\mathbb{Q}}) \ar[r]^-{p_{\ast}q^{\ast}} \ar[d]_-{c} & 
		\mathcal{H}^{\ast}(\mathcal{K}_{X(d)}^{\rm{top}}(\mathbb{M}(d; \delta)) \ar[d]_-{c}) \\
		\mathcal{H}^{\ast-\dim \X(d)^{\lambda}}(i_{\lambda\ast}(\boxtimes_{i=1}^k \mathrm{IC}_{X(d_i)}[-1]))[\beta^{\pm 1}] \oplus \mathcal{H}^{\ast}(K''')\ar[r] \ar[d]	& \bigoplus_{A' \in \mathscr{P}_{\alpha}'}\mathcal{H}^{\ast-\dim \X(d)}(P_{A'})[\beta^{\pm 1}] \ar[d] \\
 \mathcal{H}^{\ast-\dim \X(d)^{\lambda}}(i_{\lambda\ast} \pi_{\ast}\mathrm{IC}_{\X(d)^{\lambda}} )
 \ar[r]^-{\mathcal{H}^{\ast}(m_{\lambda}^{\rm{co}})}[\beta^{\pm 1}] &
 \mathcal{H}^{\ast-\dim \X(d)}(\pi_{\ast}\mathrm{IC}_{\X(d)})[\beta^{\pm 1}]. 
	}
\end{align*}
Here 
the bottom vertical arrows are embeddings into 
direct summands, and 
$m_{\lambda}^{\rm{co}}$ is the 
 sheaf-theoretic cohomological Hall algebra (\ref{coha}). 
By~\cite[Theorem~C]{DM}, the image of 
$i_{\lambda\ast}\boxtimes_{i=1}^k \mathrm{IC}_{X(d_i)}[-1]$
in ${}^p \mathcal{H}^{\ast}(\pi_{d\ast}\mathrm{IC}_{X(d)})$
under 
${}^p \mathcal{H}^{\ast}(m_{\lambda}^{\rm{co}})$ 
is $P_A$. 
Therefore, noting that 
$\mathcal{H}^{\ast}({}^p \mathcal{H}^{\ast}(\pi_{d\ast}\mathrm{IC}_{\X(d)}))=\mathcal{H}^{\ast}(\pi_{d\ast}\mathrm{IC}_{\X(d)})$, 
we see that the image of the bottom right arrow in the above commutative diagram
contains elements with non-zero elements from 
the summand $\mathcal{H}^{\ast}(P_A)$. 
Therefore, we must have $A \in \mathscr{P}_{\alpha}'$. 

It follows that there is an isomorphism of the form 
\begin{align}\label{decomp2.5}
	\mathcal{K}_{X(d)}^{\rm{top}}(\mathbb{M}(d; \delta))_{\mathbb{Q}} \stackrel{\cong}{\to}
	\bigoplus_{A \in S_{\delta}^d}P_A[\beta^{\pm 1}][1-\dim \X(d)]\oplus K'''. 
	\end{align}
By applying the functor $\varphi_{\Tr W}^{\rm{inv}}$ to (\ref{decomp2.5}), we have 
\begin{align*}
	\varphi_{\Tr W}^{\rm{inv}}	\mathcal{K}_{X(d)}^{\rm{top}}(\mathbb{M}(d; \delta))_{\mathbb{Q}}
	\stackrel{\cong}{\to}\mathcal{BPS}^{\rm{inv}}_{d, \delta}[\beta^{\pm 1}][-\dim \X(d)] \oplus 
	\varphi_{\Tr W}^{\rm{inv}}(K'''). 
	\end{align*}
 On the other hand, by Lemma~\ref{lem:sod:induce}, there is
 an isomorphism 
\begin{align*}
	\varphi_{\Tr W}^{\rm{inv}}\mathcal{K}_{X(d)}^{\rm{top}}(\mathbb{M}(d; \delta))_{\mathbb{Q}}
	\cong \mathcal{K}_{X(d)}^{\rm{top}}(\mathbb{S}(d; \delta))_{\mathbb{Q}}. 	
\end{align*}
 It follows that 
$\mathcal{BPS}^{\rm{inv}}_{d, \delta}[\beta^{\pm 1}][-\dim \X(d)]$ is a direct 
summand of $\mathcal{K}_{X(d)}^{\rm{top}}(\mathbb{S}(d; \delta))_{\mathbb{Q}}$, therefore 
\begin{align}\label{ineq:BPS}
	\dim H^{\ast}(X(d), \mathcal{BPS}_{d, \delta})^{\rm{inv}} \leq 
	\dim_{\mathbb{Q}}K_i^{\rm{top}}(\mathbb{S}(d; \delta))_{\mathbb{Q}}. 
	\end{align}
 However, note that Theorem~\ref{thm2} and Lemma~\ref{thm2inj} imply that 
 \[\dim H^{\ast}(X(d), \mathcal{BPS}_{d, \delta})^{\rm{inv}} \geq 
	\dim_{\mathbb{Q}}K_i^{\rm{top}}(\mathbb{S}(d; \delta))_{\mathbb{Q}}.\]
Thus equality holds in \eqref{ineq:BPS} and the conclusion follows from Theorem~\ref{thm2}.
	\end{proof}

The above proof of Theorem~\ref{thm1}
also implies the following sheaf-theoretic version: 
\begin{thm}\label{thm:Ksheaf2}
There is an isomorphism in $D(\mathrm{Sh}_{\mathbb{Q}}(X(d)))$:
\begin{align*}
    \mathcal{K}_{X(d)}^{\rm{top}}(\mathbb{S}(d; \delta))_{\mathbb{Q}}
    \cong \mathcal{BPS}_{d, \delta}^{\rm{inv}}[\beta^{\pm 1}][-\dim \X(d)]. 
\end{align*}
    \end{thm}
\begin{proof}
From the isomorphism (\ref{decomp2.5})
in the proof of Theorem~\ref{thm1}, it is enough to show 
that $\varphi_{\Tr W}^{\rm{inv}}(K''')=0$. 
The inequality (\ref{ineq:BPS}) is indeed an equality, 
thus we have $H^{\ast}(\varphi_{\Tr W}^{\rm{inv}}(K'''))=0$. 
Let $g \colon U \to X(d)$ be an étale map. 
One can apply the same argument in Theorem~\ref{thm1} to show 
that the cycle map 
\begin{align*}
    c \colon \mathrm{gr}_{\ell}K_i^{\rm{top}}(\mathbb{S}_U(d; \delta))_{\mathbb{Q}}
    \to H^{\dim \X(d)-2\ell-i}(U, g^{\ast}\mathcal{BPS}^{\mathrm{inv}}_{d, \delta}[1])
\end{align*}
is an isomorphism. 
Here $\X_U(d)=\X(d) \times_{X(d)} U$, 
and 
\begin{align*}\mathbb{S}_U(d; \delta)
:=\mathrm{MF}(\mathbb{M}_U(d; \delta), g^{\ast}\Tr W)
\end{align*}
where $\mathbb{M}_U(d; \delta) \subset D^b(\X_U(d))$
is the base-change of $\mathbb{M}(d; \delta)$
under $U \to X(d)$, see Subsection~\ref{basechange}.
Then we have $H^{\ast}(g^{\ast}\varphi_{\Tr W}^{\rm{inv}}(K'''))=0$. 
As it holds for any étale map $g \colon U \to X(d)$, 
we have that $\varphi_{\Tr W}^{\rm{inv}}(K''')=0$.
    \end{proof}

\section{The cycle maps on quasi-BPS categories}\label{sec:proof}
In this section, we construct the cycle map from topological K-theory of quasi-BPS categories to BPS cohomology, see Theorem \ref{thm2}. 
The construction is based on the fact that the weight conditions for complexes in $\mathbb{S}(d;\delta)$ restrict the possible perverse degree of their image under the cycle map, see Proposition \ref{prop06} and Corollary~\ref{onesupport2}.
In Subsection \ref{reduce}, we reduce the proof for a general symmetric quiver to that of a quiver with at least two loops at every vertex.
%note that quivers satisfying Assumption~\ref{assum11} already have this property. 
In Subsection \ref{subsec64}, we prove a restriction statement of the image under the cycle map of an object in a quasi-BPS category. In Subsection \ref{subsec65}, we combine the above restriction with the decomposition theorems \eqref{DM} to prove Theorem \ref{thm2}.
% In Subsection \ref{subsec64}, we show that the images under the cycle map of complexes in quasi-BPS categories are restricted with respect to . In Subsection \ref{subsec65}, we show that this restriction implies that the images under the cycle map of complexes in quasi-BPS categories
% We now collect preliminary results to be used in the proof of Proposition \ref{thm2}.

\subsection{Reduction to quivers with enough loops}\label{subsec:reduce}
In this subsection, 
we reduce Theorem~\ref{thm2} to the case of quivers 
with potentials where there exist at least
two loops at each vertex. 
Let $Q=(I, E)$ be a symmetric quiver 
with potential $W$. 
We produce another quiver with potential 
$(Q^{\sharp}, W^{\sharp})$ as follows. 
For each vertex $i\in I$, let $\omega_i, \omega'_i$ be two loops at $i$. 
We then set 
\begin{align*}
E^\sharp:=E\sqcup\{\omega_i, \omega'_i\mid i\in I\}, \ W^q:=\sum_{i\in I}\omega_i\omega'_i.
\end{align*}
The quiver with potential $(Q^{\sharp}, W^{\sharp})$
is defined by 
    \[Q^\sharp:=(I,E^\sharp),\, W^\sharp:=W+W^q.\]

\begin{prop}\label{reduce}
Assume Theorem \ref{thm2} holds for $(Q^\sharp, W^\sharp)$. Then Theorem \ref{thm2} holds for $(Q,W)$.
\end{prop}

    Recall that the stack of $Q$-representations 
    is given by $\X(d)=R(d)/G(d)$.
    For the quiver $Q^\sharp$ and for $d\in \mathbb{N}^I$, 
    the moduli stack of $Q^{\sharp}$-representations 
    and its good moduli space is given by
    \begin{align*}\pi^\sharp_d\colon\X^\sharp(d):=\left(R(d)\oplus \mathfrak{g}(d)^{\oplus 2}\right)/G(d)\to X^\sharp(d).\end{align*}
We have the 
 BPS sheaves $\mathcal{BPS}^\sharp_{d,v}$ on $X^{\sharp}(d)$ as defined in \eqref{defBPSddelta}, 
    %the sheaves $\mathrm{Q}'_A$ as defined in \eqref{defqa},
    the polytope $\mathbb{W}^\sharp(d)$ as in \eqref{defpolytope}, the integers $n^\sharp_\lambda$ as in \eqref{nlambdadef}, the quasi-BPS categories $\mathbb{M}^\sharp(d; \delta)$ from \eqref{defmdw} and  $\mathbb{S}^\sharp(d; \delta)$ from \eqref{def:quasiBPS}. 
Let $\mathscr{S}(d)$ be defined by 
\[\mathscr{S}(d):=(R(d)\oplus \mathfrak{g}(d))/G(d)\] and consider the following 
diagram 
  \begin{equation}\label{diagprop68}
       \begin{tikzcd}
           \X(d)\arrow[d, "\pi_d"] & \mathscr{S}(d)\arrow[r, "s"]\arrow[l, "v"'] & \X^\sharp(d)\arrow[d, "\pi^\sharp_d"]\\
             X(d) \arrow[rr, "u"]& &X^\sharp(d).
         \end{tikzcd}
     \end{equation}
Here $v$ is the natural projection and $s$ is the natural inclusion, 
the vertical maps are good moduli space morphisms and 
the bottom horizontal arrow is the induced map on good moduli spaces. 
We discuss some preliminary lemma and propositions. 

%\begin{equation*}
%    \begin{tikzcd}
%        \X(d) & \mathscr{S}(d)\arrow[r, hook, "s"]\arrow[l, "v"'] & \X^\sharp(d).\arrow[ll, bend right, "t"]
%    \end{tikzcd}
%\end{equation*}

\begin{lemma}\label{lem:BPSprime}
For $i\in\mathbb{Z}$, 
there is an isomorphism
\begin{align}\label{isom:sharp}
    s_*v^*\colon H^i(\X(d), \varphi_{\mathrm{Tr}\,W}\mathrm{IC}_{\X(d)}[-1])&\xrightarrow{\cong} H^i(\X^\sharp(d), \varphi_{\mathrm{Tr}\,W^\sharp}\mathrm{IC}_{\X^\sharp(d)}[-1]).
    \end{align}
   \end{lemma}

\begin{proof}
    % Consider the diagram:
    % \begin{equation*}
    %     \begin{tikzcd}
    %         \X(d)\arrow[d, "\pi_d"] & \mathscr{S}(d)\arrow[r, "p"]\arrow[l, "q"'] & \X'(d)\arrow[d, "\pi'_d"]\\
    %         X(d) \arrow[rr, "u"]& &X'(d).
    %     \end{tikzcd}
    % \end{equation*}
   It is enough to show the isomorphism in $D^b(\mathrm{Sh}_{\mathbb{Q}}(\X^\sharp(d)))$:
\begin{equation}\label{isopidstar}
    s_*v^*\colon u_*\pi_{d*}\varphi_{\mathrm{Tr}\,W}\mathrm{IC}_{\X(d)}[-1]\xrightarrow{\cong} \pi^\sharp_{d*}\varphi_{\mathrm{Tr}\,W^\sharp}\mathrm{IC}_{\X^\sharp(d)}[-1].
\end{equation}
By the dimensional reduction isomorphism (\ref{dimred}), 
we have the isomorphism 
\begin{align*}
    i_{\ast}\mathrm{IC}_{\X(d)}
    \stackrel{\cong}{\to} \varphi_{\Tr W^q}\mathrm{IC}_{\X^{\sharp}(d)}[-1]. 
\end{align*}
Here $i \colon \X(d) \hookrightarrow \X^{\sharp}(d)$
is the natural inclusion. 
By applying $\pi_{d\ast}$ to both sides, 
we obtain the isomorphism 
in $D^b(\mathrm{Sh}_{\mathbb{Q}}(\X^\sharp(d)))$:
\begin{equation}\label{isopprop699}
    s_*v^*\colon u_*\pi_{d*}\mathrm{IC}_{\X(d)}\xrightarrow{\cong} \pi^\sharp_{d*}\varphi_{\mathrm{Tr}\,W^q}\mathrm{IC}_{\X^\sharp(d)}[-1].
\end{equation}
%We write $G:=G(d)$ and $\mathfrak{g}:=\mathfrak{g}(d)$
%for simplicity.
%By replacing $\X(d)$ with $BG$, the above map is 
%\begin{equation}\label{isopprop68}
%s_{0*}v_0^*\colon u_{0\ast}\pi_{d,0*}\mathrm{IC}_{BG}\xrightarrow{\cong}\pi_{d,0*}\varphi_{\mathrm{Tr}\,W^q}\mathrm{IC}_{\mathfrak{g}^{\oplus 2}/G}[-1],
%\end{equation} where $s_0, v_0, u_0$ are the maps as in \eqref{diagprop68} for $\X(d)$ replaced
%by $BG$, and $\pi_{d,0}\colon \mathfrak{g}^{\oplus 2}/G\to \mathfrak{g}^{\oplus 2}\ssslash G$
%is the good moduli space map. 
% Alternatively, we need to show that there is an isomorphism:
% \begin{equation}\label{isopprop688}
% s_{0*}v_0^*\colon \mathbb{Q}_\mathrm{pt}\xrightarrow{\sim} \varphi_{\mathrm{Tr}\,W^q}\mathrm{IC}_{\mathfrak{g}^{\oplus 2}}[-1].
% \end{equation}
%By a direct computation, we have that \[\varphi_{\mathrm{Tr}\,W^q}\mathrm{IC}_{\mathfrak{g}^{\oplus 2}/G}[-1]=\mathrm{IC}_{BG}\] because $W^q$ is a Morse function with critical locus $BG$, that 
%is the origin in $\mathfrak{g}^{\oplus 2}/G$. Further, \eqref{isopprop68} is an isomorphism for global sections by dimensional reduction, see Subsection \ref{subsection:dimred}, so \eqref{isopprop68} is an isomorphism. Then \eqref{isopprop699} is also an isomorphism
%by the base-change with respect to $X^{\sharp}(d) \to \mathfrak{g}^{\oplus 2}\ssslash G$. 

By abuse of notation, we write 
$\Tr W \colon \X^{\sharp}(d) \to \mathbb{C}$ the 
pull-back of $\Tr W \colon \X(d) \to \mathbb{C}$
via the projection $\X^{\sharp}(d) \to \X(d)$. 
%\begin{align*}\mathrm{Tr}\,W\colon \X^\sharp(d)\xrightarrow{\mathrm{proj}}\X(d)\xrightarrow{\mathrm{Tr}\,W}\mathbb{C}.
%\end{align*}
Note that $\pi_{d*}$ commutes with $\varphi_{\mathrm{Tr}\,W}$ because $\pi_d$ can be approximated with the proper maps $\pi_{\alpha f,d}$, see Subsection \ref{subsection:decompositiontheorem}. Further, $\varphi_{\mathrm{Tr}\,W}$ commutes with proper pushforward and smooth pullback. 
By applying $\varphi_{\mathrm{Tr}\,W}$ to both sides of \eqref{isopprop699} and using the Thom-Sebastiani theorem for vanishing cycles, we obtain the isomorphisms
\begin{align*}
    s_*v^*\colon u_*\pi_{d*}\varphi_{\mathrm{Tr}\,W}\mathrm{IC}_{\X(d)}[-1]&\xrightarrow{\cong} \pi^\sharp_{d*}\varphi_{\mathrm{Tr}\,W}\left(\varphi_{\mathrm{Tr}\,W^q}\mathrm{IC}_{\X^\sharp(d)}[-1]\right)[-1]\\
    &\cong \pi^\sharp_{d*}\varphi_{\mathrm{Tr}\,W^\sharp}\mathrm{IC}_{\X^\sharp(d)}[-1].
\end{align*} 
Therefore we obtain the isomorphism (\ref{isopidstar}). 
% Indeed, we have $\varphi_{\mathrm{Tr}\,W^q}\mathrm{IC}_{\X^\star(d)}[-1]=\mathrm{IC}_{\X(d)}$ by a direct computation using that $W^q$ is a Morse function with critical locus $\X(d)\hookrightarrow \X^\star(d)$. 
% By the Thom-Sebastiani theorem, there is an isomorphism in $D^b_{\mathrm{con}}(\X^\star(d))$, where the tensor product is taken over $BG$:
% \[\varphi_{\mathrm{Tr}\,W}\mathrm{IC}_{\X(d)}[-1]\otimes \varphi_{\mathrm{Tr}\,W^q}\mathrm{IC}_{\mathfrak{g}^{\oplus}/G}[-1]\xrightarrow{\sim} \varphi_{\mathrm{Tr}\,W^\star}\mathrm{IC}_{\X^\star(d)}[-1].\]
% Consider the maps $BG \xleftarrow{v_0}\mathfrak{g}/G\xrightarrow{s_0}\mathfrak{g}^{\oplus 2}/G$.
% There is an isomorphism \[s_*v^*\colon \mathbb{Q}_{BG}[-\dim\,\mathfrak{g}]\xrightarrow{\sim} \varphi_{\mathrm{Tr}\,W^q}\mathbb{Q}_{\mathfrak{g}^{\oplus 2}/G}[-1],\] and thus \eqref{isopprop69} follows. 
% Consider the diagram:
%      \begin{equation*}
%        \begin{tikzcd}
%            \X(d)\arrow[d, "\pi_d"] & \mathscr{S}(d)\arrow[r, "p"]\arrow[l, "q"'] & \X'(d)\arrow[d, "\pi'_d"]\\
%              X(d) \arrow[rr, "u"]& &X'(d).
%          \end{tikzcd}
%      \end{equation*}
% There is then an isomorphism in $D^b_{\mathrm{con}}(X'(d))$ induced by $p_*q^*$: 
%     \begin{align}\label{isopidstar}
%         u_*\pi_{d*}\varphi_{\mathrm{Tr}\,W}\mathrm{IC}_{\X(d)}[-1]\xrightarrow{\sim} \pi'_{d*}\varphi_{\mathrm{Tr}\,W'}\mathrm{IC}_{\X'(d)}[-1]
%     \end{align}
  \end{proof}

\begin{prop}\label{BPSprime}
The isomorphism (\ref{isom:sharp}) in Lemma~\ref{lem:BPSprime}
restricts to the isomorphism 
    \begin{align*}
    H^i(X(d), \mathcal{BPS}_{d,\delta})&\stackrel{\cong}{\to}  H^i(X^\sharp(d), \mathcal{BPS}^\sharp_{d,\delta}).
\end{align*}
\end{prop}
\begin{proof}
      It is enough to show that the isomorphism \eqref{isopidstar} induces an isomorphism in $D^b(\mathrm{Sh}_{\mathbb{Q}}(\X^\sharp(d)))$
    \[u_*\mathcal{BPS}_{d,\delta}\xrightarrow{\cong} \mathcal{BPS}^\sharp_{d,\delta}.\]
   We first show that $S^d_\delta(Q)=S^d_\delta(Q^\sharp)$.
   For a cocharacter $\lambda$ of $T(d)$, let $n_\lambda$ and $n^\sharp_\lambda$ be the integers \eqref{nlambdadef} for $Q$ and $Q^\sharp$, respectively. Let $\varepsilon_{\lambda, \delta}$ and $\varepsilon_{\lambda, \delta}^\sharp$ be the integers \eqref{varepsilondeltalambda} for $Q$ and $Q^\sharp$, respectively.
Then
\[n^\sharp_\lambda-n_\lambda=2\langle \lambda, \mathfrak{g}(d)^{\lambda>0}\rangle \in 2\mathbb{Z}.\]
Therefore $\varepsilon_{\lambda, \delta}=\varepsilon_{\lambda, \delta}^\sharp$, so indeed $S^d_\delta(Q)=S^d_\delta(Q^\sharp)=:S^d_\delta$. 
    
    It suffices to check that \eqref{isopidstar} induces isomorphisms:
    \begin{equation}\label{BPSssheaves}
    u_*\mathcal{BPS}_A\xrightarrow{\cong}\mathcal{BPS}^\sharp_A
     \end{equation} for any $A\in S^d_\delta$.     
     The isomorphism \eqref{isopidstar} is obtained by applying the functor $\varphi_{\mathrm{Tr}\,W}$ to the isomorphism \eqref{isopprop699}. 
     Therefore it suffices to check \eqref{BPSssheaves} when $W=0$, so we assume that $W=0$ in the rest of the proof.
    %  Assume $\mathrm{Q}_A$ is a summand of ${}^p\mathcal{H}^k(\pi_{d*}\varphi_{\mathrm{Tr}\,W}\mathrm{IC}_{\X(d)}[-1])$. There is an isomorphism
    % The map \eqref{isopidstar} induces maps
    % \[u_*\mathrm{Q}_A\xrightarrow{\sim} \mathrm{Q}'_A\] for any $A\in \mathscr{P}$. Because \eqref{isopidstar}
    % By the Thom-Se
    % Indeed, it suffices to check that, under \eqref{isopidstar}, there is an isomorphism: 
    % \begin{equation}\label{BPSssheaves}
    % u_*\mathcal{BPS}_A\xrightarrow{\sim}\mathcal{BPS}'_A
    %  \end{equation}
    % for any $A\in S^d_v$. 
    Suppose that $A$ has a corresponding partition $(d_i)_{i=1}^k$ of $d$. Let $X_A$ be the image of the addition map $\oplus\colon \times_{i=1}^k X(d_i)\to X(d)$. 
    By (\ref{isopprop699}), there is an isomorphism:
    \[u_*{}^p\mathcal{H}^{k}(\pi_{d*}\mathrm{IC}_{\X(d)})\xrightarrow{\sim} {}^p\mathcal{H}^{k}(\pi^\sharp_{d*}\varphi_{\mathrm{Tr}\,W^q}\mathrm{IC}_{\X^\sharp(d)}[-1]).\]
    There are either no summands of support $X_A$ on both sides, case in which both $u_*\mathcal{BPS}_A$ and $\mathcal{BPS}^\sharp_A$ are zero, or there are unique summands of support $X_A$ on both sides, namely $u_*\mathcal{BPS}_A$ and $\mathcal{BPS}^\sharp_A$, and thus \eqref{BPSssheaves} follows.
\end{proof}

We note the following corollary of Lemma~\ref{lem:BPSprime} and Proposition \ref{BPSprime}
on the isomorphism of the monodromy invariants parts: 

\begin{cor}\label{corBPSprime}
Let $\delta\in M(d)^{W_d}_\mathbb{R}$ and $i\in\mathbb{Z}$. 
There is an isomorphism
\begin{align*}
    s_*v^*\colon H^i(\X(d), \varphi^{\mathrm{inv}}_{\mathrm{Tr}\,W}\mathrm{IC}_{\X(d)}[-1])\xrightarrow{\cong} H^i(\X^\sharp(d), \varphi^{\mathrm{inv}}_{\mathrm{Tr}\,W^\sharp}\mathrm{IC}_{\X^\sharp(d)}[-1]),
    \end{align*}
    which restricts to the isomorphism 
    \begin{align*}
    H^i(X(d), \mathcal{BPS}^{\mathrm{inv}}_{d,\delta})\xrightarrow{\cong} H^i(X^\sharp(d), \mathcal{BPS}^{\sharp, \mathrm{inv}}_{d,\delta}).
\end{align*}
\end{cor}

We recall the relation 
between quasi-BPS categories under Kn\"orrer periodicity: 
\begin{prop}\label{quasiBPSprime}\emph{(\cite[Proposition~2.14]{PTquiver})}
    There is an equivalence:
    \begin{align*}
        s_*v^*\colon \mathrm{MF}(\X(d), \mathrm{Tr}\,W)&\xrightarrow{\sim} \mathrm{MF}(\X^\sharp(d), \mathrm{Tr}\,W^\sharp),
    \end{align*}
    which restricts to the equivalence 
   $\mathbb{S}(d; \delta)\xrightarrow{\sim} \mathbb{S}^\sharp(d; \delta)$.
\end{prop}

We now give a proof of Proposition~\ref{reduce}: 
\begin{proof}[Proof of Proposition \ref{reduce}]
   Let $i,\ell\in\mathbb{Z}$.
   By Corollary \ref{corollary410}, there is a commutative diagram, where $b=\dim \X(d)-i-2\ell=\dim \X^\sharp(d)-i-2(\ell+\dim \mathfrak{g}(d))$:
   \begin{equation*}
       \begin{tikzcd}
           \mathrm{gr}_\ell K^{\mathrm{top}}_i(\mathrm{MF}(\X(d), \mathrm{Tr}\,W))
           \arrow[r, "\cong"', "s_*v^*"]
           \arrow[d, "\mathrm{c}"]& \mathrm{gr}_{\ell+\dim \mathfrak{g}} K^{\mathrm{top}}_i(\mathrm{MF}(\X^\sharp(d), \mathrm{Tr}\,W^\sharp))
           \arrow[d, "\mathrm{c}"]\\
           H^b(\X(d), \varphi_{\mathrm{Tr}\,W}^{\mathrm{inv}}\mathrm{IC}_{\X(d)})\arrow[r, "\cong"', "s_*v^*"]& H^{b}(\X^\sharp(d), \varphi_{\mathrm{Tr}\,W^\sharp}^{\mathrm{inv}}\mathrm{IC}_{\X^{\sharp}(d)}).
       \end{tikzcd}
       \end{equation*}
       The conclusion follows from Corollary \ref{corBPSprime} and Proposition \ref{quasiBPSprime}.
\end{proof}

\subsection{Coproduct-like maps in K-theory}\label{subsec64}
By Proposition~\ref{reduce}, it is enough to consider quivers 
satisfying the following assumption: 

\begin{assum}\label{assum3}
    Assume that the quiver $Q$ is symmetric and has at least two loops at any vertex.  
\end{assum}

In this subsection, under the above assumption, 
we construct coproduct-like maps in topological K-theory 
of matrix factorizations. 

We introduce some notation. For a cocharacter $\lambda$ of $T(d)$ 
with associated partition $\mathbf{d}$, we define 
\begin{align}\label{def:cd}
c_{\mathbf{d}}:=c_\lambda:=\dim \X(d)-\dim \X(d)^{\lambda\geq 0}.
\end{align}
The above assumption of the quiver is used in the following lemma: 

\begin{lemma}\label{clambdazero}
    Let $Q=(I,E)$ be a quiver which satisfies Assumption~\ref{assum3} and let $d\in\mathbb{N}^I$ be a non-zero dimension vector.

    (i) For any cocharacter $\lambda$ of $T(d)$, we have $c_\lambda\geq 0$, and the inequality is strict if $\lambda$ has an associated partition with at least two terms. Moreover, we have \[\dim \X(d)^{\lambda\geq 0}-\dim \X(d)^\lambda=c_\lambda.\]

    (ii) There exists a simple $Q$-representation of dimension $d$.
    In particular, the map $\pi_d\colon \X(d)\to X(d)$ is generically a $\mathbb{C}^*$-gerbe.  
\end{lemma}

\begin{proof}
    Let $S(d)$ be the affine space of dimension $d$ representations of 
the quiver obtained from $Q$ by deleting one loop at every vertex in $I$.
Then
\[\X(d)=\left(S(d)\oplus\mathfrak{g}(d)\right)/G(d).\] 
We have that 
$\dim \X(d)=\dim S(d)$ and $\dim \X(d)^{\lambda \geq 0}=\dim S(d)^{\geq 0}$. 
Note that $S(d)$ contains $\mathfrak{g}(d)$ as a 
sub $G(d)$-representation. 
Therefore we have 
\begin{align*}c_\lambda=\dim S(d)^{\lambda<0} \geq 
\dim \mathfrak{g}(d)^{\lambda<0}\geq 0,
\end{align*}
and 
the last inequality is strict if $\lambda$ corresponds
to a partition of length at least two. Therefore (i) 
holds. 
Then (ii) holds since the union of the images of
$a_{\lambda} \colon \X(d)^{\lambda} \to \X(d)$
for every cocharacters $\lambda$ corresponding to 
partitions of length at least two is strictly 
smaller than $\X(d)$ because of $c_{\lambda}>0$.

\end{proof}

Below we assume that $Q$ satisfies Assumption \ref{assum3}.
Let $\lambda$ be an antidominant cocharacter of $T(d)$. 
We have the natural 
morphism $a_{\lambda} \colon \X(d)^{\lambda} \to \X(d)$ 
induced by the inclusion $R(d)^{\lambda} \hookrightarrow R(d)$, 
see the diagram (\ref{diagram:natural}). 
Its pull-back
and cycle maps give the following commutative 
diagram 
  \begin{equation}\label{comm:gra}
        \begin{tikzcd}
            \mathrm{gr}_\ell K^{\mathrm{top}}_i(\mathrm{MF}(\X(d), \Tr W)) \arrow[d, hook, "\mathrm{c}"]\arrow[r, "a^*_\lambda"]& 
            \mathrm{gr}_{\ell-2c_\lambda} K^{\mathrm{top}}_i\left(\mathrm{MF}\left(\X(d)^{\lambda}, \mathrm{Tr}\,W\right)\right)\arrow[d, hook, "\mathrm{c}"]\\
            H^{2\dim \X(d)-2\ell-i}\left(\X(d), \varphi_{\mathrm{Tr}\,W}^{\mathrm{inv}}\right)\arrow[r, "a^*_\lambda"]&  H^{2\dim \X(d)^{\lambda}-2\ell-i}\left(\X(d)^{\lambda}, \varphi_{\mathrm{Tr}\,W}^{\mathrm{inv}}\right). 
        \end{tikzcd}
    \end{equation} 
Here we have used that $2c_\lambda:=\dim \X(d)-\dim \X(d)^{\lambda}$ by Lemma~\ref{clambdazero}. 
Since $\lambda$ acts on $\X(d)^{\lambda}$ trivially, 
the stack $\X(d)^{\lambda}$ is a $\mathbb{C}^{\ast}$-gerbe 
over another stack $\X(d)'^\lambda$
\begin{align}\label{gerb:lambda}
    \X(d)^{\lambda} \to \X(d)'^\lambda:=R(d)^\lambda/G(d)'
\end{align}
where $G(d)':=G(d)^\lambda/\text{image}(\lambda)$. 

Note that the relative tangent complex of $\X(d)^{\lambda \geq 0} \to \X(d)$
is isomorphic to $\mathbb{L}_{\X(d)}^{\lambda>0}$. Its determinant 
line bundle descends to a line bundle $\det \left(\mathbb{L}_{\X(d)}^{\lambda>0}\right)$ 
on $\X(d)^{\lambda}$ 
with $\lambda$-weight $n_{\lambda}$. Note that $n_{\lambda}>0$ by 
Assumption~\ref{assum3} if the associated partition $(d_i)_{i=1}^k$ 
has length $k\geq 2$. 
We set
\begin{align*}
    h:=\begin{cases}
        \frac{1}{n_{\lambda}} c_1(\det \mathbb{L}_{\X(d)}^{\lambda>0}) \in 
    H^2(\X(d)^{\lambda}, \mathbb{Q}),\text{ if }k\geq 2,\\
    c_1(\tau_d),\text{ if }k=1.
    \end{cases}
\end{align*}
We have the isomorphism 
\begin{align}\label{isom:h}
    H^{\ast}(\X(d)'^{\lambda}, \varphi_{\Tr W})[h] \stackrel{\cong}{\to}
    H^{\ast}(\X(d)^{\lambda}, \varphi_{\Tr W}). 
\end{align}

We need a little more care of having such an isomorphism for 
K-theory, as the $n_{\lambda}$-th root of $\det \left(\mathbb{L}_{\X(d)}^{\lambda>0}\right)$ does not 
necessary exist. 
Instead we consider the following. 
Let $(d_i)_{i=1}^k$ be a partition associated with $\lambda$, 
and consider the 
subgroup $SG(d)^{\lambda} \subset G(d)^{\lambda}$
defined as the kernel of the morphism
\begin{align*}
    G(d)^{\lambda} \to \mathbb{C}^{\ast}, \ (g_i^a)_{a \in I, 1\leq i\leq k} \mapsto 
    \prod_{a, i} \det g_{a}^i. 
\end{align*}
There is a surjective group homomorphism 
\begin{align}\label{surj:Gd}
    \mathbb{C}^{\ast} \times SG(d)^{\lambda} \to G(d)^{\lambda}, \ 
    (t, g) \mapsto \lambda(t) g.
\end{align}
The kernel of the above homomorphism is a finite
abelian group. We set 
\begin{align}\label{Xtilde}
    \widetilde{\X}(d)^{\lambda}:=R(d)^{\lambda}/(\mathbb{C}^{\ast}\times SG(d)^{\lambda}) \to 
    \widetilde{\X}(d)'^{\lambda}:=R(d)^{\lambda}/SG(d)^{\lambda}. 
\end{align}
Here $\mathbb{C}^{\ast}\times SG(d)^{\lambda}$ acts on $R(d)^{\lambda}$ through 
the surjection (\ref{surj:Gd}). 
The morphism (\ref{Xtilde}) is a trivial $\mathbb{C}^{\ast}$-gerbe, so we have 
the corresponding isomorphism 
\begin{align}\label{isom:MFq}
K^{\mathrm{top}}_i\left(\mathrm{MF}\left(\widetilde{\X}(d)'^{\lambda}, \mathrm{Tr}\,W\right)\right)[q^{\pm 1}]
\stackrel{\cong}{\to} 
K^{\mathrm{top}}_i\left(\mathrm{MF}\left(\widetilde{\X}(d)^{\lambda}, \mathrm{Tr}\,W\right)\right).
\end{align}
Here $q$ is the class of weight one $\mathbb{C}^{\ast}$-character 
pulled back via the projection $\widetilde{\X}(d)^{\lambda} \to B\mathbb{C}^{\ast}$. 
Note that $q^{n_{\lambda}}$ and the pull-back of 
$\det \mathbb{L}_{\X(d)^{\lambda>0}}$ via $\widetilde{\X}(d)^{\lambda} \to \X(d)^{\lambda}$
only differs by a character of $SG(d)^{\lambda}$. 
We have the commutative diagram 
\begin{align}\label{comm:lambda}
    \xymatrix{
\widetilde{\X}(d)^{\lambda} \ar[r] \ar[d] & \widetilde{\X}(d)'^{\lambda} \ar[d] \\
\X(d)^{\lambda} \ar[r] & \X(d)'^{\lambda}.     
    }
\end{align}
Here each vertical arrow is a gerbe of a finite abelian group
and the horizontal arrows are $\mathbb{C}^{\ast}$-gerbes. 
We have the following lemma. 
\begin{lemma}\label{lem:comm:gerb}
The diagram (\ref{comm:lambda}) induces the commutative diagram 
\begin{equation*}
       \begin{tikzcd}
            \mathrm{gr}_\ell K^{\mathrm{top}}_i(\mathrm{MF}(\X(d)^{\lambda}, \Tr W)) \arrow[d, hook, "\mathrm{c}"]\arrow[r]& 
            \mathrm{gr}_{\ell} K^{\mathrm{top}}_i\left(\mathrm{MF}\left(\widetilde{\X}(d)^{\lambda}, \mathrm{Tr}\,W\right)\right)\arrow[d, hook, "\mathrm{c}"]\\
            H^{\ast}\left(\X(d)'^{\lambda}, \varphi_{\mathrm{Tr}\,W}^{\mathrm{inv}}\right)[h]\arrow[r, "\cong"]&  H^{\ast}\left(\widetilde{\X}(d)'^{\lambda}, \varphi_{\mathrm{Tr}\,W}^{\mathrm{inv}}\right)[h]. 
        \end{tikzcd}
\end{equation*}
Here the horizontal arrows are given by pull-backs, 
and the bottom arrow is an isomorphism. 
\end{lemma}
\begin{proof}
The diagram commutes by the construction. 
The bottom arrow is an isomorphism since 
$\widetilde{\X}(d)'^{\lambda} \to \X(d)'^{\lambda}$ is a gerbe over 
a finite abelian group, and because $H^{\ast}(B\mu_d, \mathbb{Q})=\mathbb{Q}$ for any $d\geq 1$. 
\end{proof}

Recall the definition of $c_{\mathbf{d}}$ and $c_{\lambda}$ in (\ref{def:cd}). 
We define 
\begin{equation}\label{def:width}
c_{\lambda, \delta}:=c_\lambda+\varepsilon_{\lambda,\delta},\, c_{\mathbf{d}, \delta}:=c_\mathbf{d}+\varepsilon_{\mathbf{d},\delta}.
\end{equation}
The following is the main technical result of this subsection. 
\begin{prop}\label{prop06}
For each partition $\mathbf{d}=(d_i)_{i=1}^k$ of $d$, 
there exists an antidominant cocharacter $\lambda$ of $T(d)$ with 
associated partition $\mathbf{d}$ such that, 
in the commutative diagram (\ref{comm:gra}),
       the image of $\mathrm{c}\,a^*_\lambda\mathrm{gr}_{\ell}K^{\mathrm{top}}_i\left(\mathbb{S}(d; \delta)\right)$ lies in the subspace 
    \[\bigoplus_{j=0}^{c_{\lambda, \delta}-1}H^{2\dim \X(d)^{\lambda}-2\ell-i-2j}\left(\X(d)'^{\lambda}, \varphi_{\Tr W}^{\mathrm{inv}}\right) h^j\subset 
    H^\ast\left(\X(d)'^{\lambda}, \varphi_{\Tr W}^{\mathrm{inv}}\right)[h].\]
    Moreover, the cocharacter $\lambda$ satisfies $\varepsilon_{\lambda, \delta}=\varepsilon_{\mathbf{d}, \delta}$, 
    hence $c_{\lambda, \delta}=c_{\mathbf{d}, \delta}$. 
    %Note that $a_\lambda$ only depends on the partition $\mathbf{d}$, so we obtain that the image of $\mathrm{c}\,a^*_\lambda\mathrm{gr}_\ell K^{\mathrm{top}}_i\left(\mathbb{S}(d; \delta)\right)$ lies in the subspace 
    %\[\bigoplus_{j=0}^{c_{\mathbf{d}, \delta}-1}H^{2\dim \X(d)-2\ell-i-2j}\left(\X(d)'^{\lambda}, \varphi^{\mathrm{inv}}[-2]\right) h^j\subset 
    %H^{\ast}\left(\X(d)'^{\lambda}, \varphi^{\mathrm{inv}}[-2]\right)[h].\]
\end{prop}

\begin{proof}
Let $\lambda$ be an antidominant cocharacter with corresponding 
partition $\mathbf{d}$. We may take such $\lambda$ such
that $\lambda \colon \mathbb{C}^{\ast} \to T(d)$ is injective. 
For an object $A \in \mathbb{S}(d; \delta)$, 
by the definition of quasi-BPS category, the pull-back $a_\lambda^*(A)$
has $\lambda$-weights contained in the interval
%lies in the subcategory of $\mathrm{MF}(\X(d)^{\lambda}, \mathrm{Tr}\,W)$ generated by 
%$\mathrm{MF}(\X(d)'^\lambda, \mathrm{Tr}\,W)_v$ for 
\[S_{\lambda, \delta}:=
\left[-\frac{n_{\lambda}}{2}+\langle \lambda,\delta\rangle, \frac{n_{\lambda}}{2}+\langle \lambda,\delta\rangle\right]\cap \mathbb{Z}.\]
%Note that $|S_{\lambda, \delta}|=w_{\lambda, \delta}$.
%Thus the image of $\ell_\lambda^*\left(K^{\mathrm{top}}_0\left(\mathbb{S}^{\mathrm{gr}}(d; \delta)_w\right)\right)$ is in the subspace \[ K^{\mathrm{top}}_0\left(\mathrm{MF}\left(\X(d)'^{\lambda}, \mathrm{Tr}\,W\right)\right)\otimes \mathscr{A},\] 
Thus, by setting $b_{\lambda}$
to be the composition of $a_{\lambda}$ with 
$\widetilde{\X}(d)^{\lambda} \to \X(d)^{\lambda}$, we obtain 
\begin{equation}\label{subsetPhi}
    b^*_\lambda K^{\mathrm{top}}_i\left(\mathbb{S}(d; \delta)\right)\subset K^{\mathrm{top}}_i\left(\mathrm{MF}\left(\widetilde{\X}(d)'^{\lambda}, \mathrm{Tr}\,W\right)\right)\otimes
\mathscr{A},
\end{equation}
where 
$\mathscr{A}$ is defined by 
\begin{align*}
\mathscr{A}:=\bigoplus_{j\in S_{\lambda, \delta}}\mathbb{Q} \cdot q^j
\subset K_0^{\rm{top}}(B\mathbb{C}^{\ast})=\mathbb{Q}[q^{\pm 1}]. 
\end{align*}
There 
exists a filtration in the right hand side 
of (\ref{subsetPhi})
induced by the Chern character maps for both $K_i^{\mathrm{top}}\left(\mathrm{MF}(\widetilde{\X}(d)'^\lambda, \mathrm{Tr}\,W)\right)$ and $K_0^{\mathrm{top}}\left(B\mathbb{C}^*\right)$, and there is an isomorphism obtained by the Kunneth formula:
\begin{align}\label{kunnethgr}
&\mathrm{gr}_\ell K_i^{\mathrm{top}}\left(\mathrm{MF}(\widetilde{\X}(d)^\lambda, \mathrm{Tr}\,W)\right)\\
&\notag\cong \bigoplus_{a+b=\ell}\mathrm{gr}_a K_i^{\mathrm{top}}\left(\mathrm{MF}(\widetilde{\X}(d)'^\lambda, \mathrm{Tr}\,W)\right)\otimes \mathrm{gr}_b K_0^{\mathrm{top}}\left(B\mathbb{C}^*\right).
\end{align}

The filtration on $\mathscr{A}$ is induced by the filtration on
$\mathbb{Q}[q^{\pm 1}]$ given by 
$(q-1)^b \mathbb{Q}[q^{\pm 1}] \subset \mathbb{Q}[q^{\pm 1}]$
for $b\geq 0$, see Example~\ref{exam:BGL}.  
Note that we have 
\begin{align}\label{inclu:A}
    \mathscr{A} \cap (q-1)^b \mathbb{Q}[q^{\pm 1}] =0, \ b\geq n_{\lambda}+\varepsilon_{\lambda, \delta}. 
\end{align}
Suppose that $k=2$, i.e. $\mathbf{d}$ is a 
two length partition $d=d_1+d_2$. 
In this case, we choose $\lambda$ 
which acts on $\beta^i_a$ for $i\in I$ with weight one 
for $1\leq a \leq d^i_1$ and zero for $a>d^i_1$. 
Then the $\lambda$-weights on $R(d)$ and $\mathfrak{g}(d)$ 
are either zero or $\pm 1$, so we have $n_{\lambda}=c_{\lambda}$. 
Therefore the conclusion follows from (\ref{inclu:A})
and Lemma~\ref{lem:comm:gerb}. 

Suppose that $k\geq 3$. The stack $\X(d)^{\lambda}$
only depends on a partition $\bf{d}$, so we may write it as 
$\X(\mathbf{d})=R(\mathbf{d})/G(\mathbf{d})$. It is a $(\mathbb{C}^{\ast})^k$-gerbe
\begin{align*}
    \X(\mathbf{d}) \to \X(\mathbf{d})'=R(\mathbf{d})/G(\mathbf{d})'
\end{align*}
where $G(\mathbf{d})'=G(\mathbf{d})/Z(\mathbf{d})$
and $Z(\mathbf{d})=(\mathbb{C}^{\ast})^k \subset T(d_1)\times \cdots \times T(d_k)$ is the diagonal 
torus. 
As it is not necessarily a trivial gerbe, we apply a similar construction 
as in (\ref{Xtilde}).
By setting $S(\mathbf{d})=\times_{i=1}^k SG(d_i)$, 
we have the surjection $Z(\mathbf{d}) \times S(\mathbf{d})\to G(\mathbf{d})$, 
and trivial $Z(\mathbf{d})$-gerbe
\begin{align*}
    \widetilde{\X}(\mathbf{d}):=R(\mathbf{d})/(Z(\mathbf{d})\times S(\mathbf{d})) \to 
    \widetilde{\X}(\mathbf{d})':=R(\mathbf{d})/S(\mathbf{d}). 
\end{align*}
Similarly to (\ref{isom:MFq}), we have the isomorphism 
\begin{align*}
K_i^{\rm{top}}(\mathrm{MF}(\widetilde{\X}(\mathbf{d})', \Tr W))_{\mathbb{Q}} \otimes \mathbb{Q}[q_1^{\pm 1}, \ldots, q_{k}^{\pm 1}] \stackrel{\cong}{\to} 
    K_i^{\rm{top}}(\mathrm{MF}(\widetilde{\X}(\mathbf{d}), \Tr W))_{\mathbb{Q}}. 
\end{align*}
Giving an antidominant cocharacter of $T(d)$ with corresponding partition $\mathbf{d}$
is equivalent to giving an element 
\begin{align}\label{lambda:Zd}
    \lambda=(\lambda_1, \ldots, \lambda_k) \in \Hom(\mathbb{C}^{\ast}, Z(\mathbf{d}))=\mathbb{Z}^k, \ 
    \lambda_1>\cdots>\lambda_k. 
\end{align}
Let $\gamma$ be the composition 
$\gamma \colon \widetilde{\X}(\mathbf{d}) \to \X(\mathbf{d}) \to \X(d)$. 
We need to find such $\lambda$ such that 
the vanishing orders of the elements 
\begin{align}\label{pull:gamma}
    \gamma^{\ast}K_i^{\rm{top}}(\mathbb{S}(d; \delta)) \subset 
    K_i^{\rm{top}}(\mathrm{MF}(\widetilde{\X}(\mathbf{d})', \Tr W)) \otimes \mathbb{Q}[q_1^{\pm 1}, \ldots, q_{k}^{\pm 1}]
\end{align}
at the hypersurface 
\begin{align*}
    F_{\lambda}:=(f_{\lambda}=0) \subset Z(\mathbf{d}), \ 
f_{\lambda} :=q_1^{\lambda_1}\cdots q_k^{\lambda_k}-1
\end{align*}
are
at most $c_{\lambda, \delta}-1$. 
By the argument of the length two partition case, 
by setting $\lambda_0=(1, 0, \cdots, 0)$
the vanishing orders of the elements (\ref{pull:gamma}) at $F_{\lambda_0}=\{q_1=1\}$ are at most $c_{\lambda_0, \delta}-1$. 
We then perturb $\lambda_0$ to a rational 
cocharacter of $Z(\mathbf{d})$, 
\begin{align*}
 \lambda'=(1, \lambda_2', \cdots, \lambda_k'), \ 1\gg \lambda_2'>\cdots>\lambda_k'
 \end{align*}
 where $\lambda_i' \in \mathbb{Q}$. By a generic perturbation, the elements 
 from the left hand side of (\ref{subsetPhi}) are divisible by 
 $f_{\lambda'}$ at most $(c_{\lambda_0, \delta}-1)$-times. 
 
Indeed, the locus in the left hand side in (\ref{subsetPhi}) 
divisible $m$-times by $(q_1-1)$ is given by 
the intersection $H_1 \cap \cdots \cap H_m$, 
where $H_i$
is the linear subspace determined by the condition 
$(\partial/\partial q_1)^{i-1}(-)|_{q_1=1}=0$. 
Then we have $H_1 \cap \cdots \cap H_{c_{\lambda_0, \delta}}=0$. 
By a small perturbation of $\lambda_0$, 
the locus divisible $m$-times by $f_{\lambda'}$ is the intersection $H_1' \cap \cdots \cap H_m'$,
where $H_i'$ is a perturbation of $H_i$. 
Therefore, we have $H_1' \cap \cdots \cap H_{c_{\lambda_0, \delta}}'=0$
for an enough small perturbation of $\lambda_0$. 

We then take $N>0$ such that $\lambda=N\lambda'$ is integral. 
Since the vanishing orders at $F_{\lambda'}$ and $F_{N\lambda'}$ are 
the same and we have the inequality $c_{\lambda_0}<c_{\lambda}$,
we obtain a desired cocharacter $\lambda$. 
Moreover, the above choice of $\lambda$ can be taken 
arbitrary of the form (\ref{lambda:Zd}) 
such that $\lambda_1\gg \lambda_2>\cdots>\lambda_k$, 
which are general enough to assure that
$\varepsilon_{\lambda, \delta}=\varepsilon_{\mathbf{d}, \delta}$. 
\end{proof}

\subsection{Coproduct-like maps in cohomology}\label{subsec65}

In this subsection, we 
construct coproduct-like maps in cohomologies and 
use them together with Proposition~\ref{prop06} to give a proof of 
Theorem~\ref{thm2}. 
Below we 
assume that the quiver $Q=(I,E)$ satisfies Assumption \ref{assum3}. 
By Proposition \ref{reduce}, we obtain Theorem \ref{thm2} for general symmetric quivers $Q$.

For $\lambda$ an antidominant cocharacter with associated partition $\mathbf{d}=(d_i)_{i=1}^k$,
let $\pi_{\lambda}$ be the good moduli space map 
\begin{align*}\pi_\lambda:=\times_{i=1}^k \pi_{d_i}\colon \X(d)^\lambda\to X(d)^\lambda:=\times_{i=1}^k X(d_i).\end{align*}
Let $t_{\lambda}$ be the projection $t_\lambda\colon \X(d)^\lambda\to \X(d)'^\lambda$, 
which is a $\mathbb{C}^{\ast}$-gerbe. The maps introduced fit in the following diagram:

\begin{align*}
\xymatrix{
\X(d)^{\lambda} \ar[r]_-{t_{\lambda}}  \ar@/^20pt/[rr]_-{\pi_{\lambda}} \ar[d]_{a_{\lambda}} & \X(d)^{'\lambda} \ar[r] & X(d)^{\lambda} \ar[d]_-{i_{\lambda}} \\
\X(d) \ar[rr]^-{\pi_d} & & X(d). 
}    
\end{align*}
% \begin{align*}
%\xymatrix{
%\X(d)^{\lambda} \ar[r]^-{a_{\lambda}} \ar[d]^-{t_{\lambda}} \ar@/_20pt/[dd]_-{\pi_{\lambda}}& \X(d) \ar[dd]^-{\pi_d} \\
%\X(d)'^{\lambda} \ar[d] &  \\
%X(d)^{\lambda} \ar[r]^-{i_{\lambda}} & X(d). 
%}  
% \end{align*}
 We consider the following perverse truncation
\begin{align}
\label{Sdef} \mathrm{S}\colon t_{\lambda*}\mathrm{IC}_{\X(d)^{\lambda}}[c_{\lambda}] &\cong \mathrm{IC}_{\X(d)'^\lambda}[c_\lambda-1]\otimes \mathbb{Q}[h]  \\ \notag
&\to {}^p\tau^{\geq c_\lambda+1}\left(\mathrm{IC}_{\X(d)'^\lambda}[c_\lambda-1]\otimes \mathbb{Q}[h]\right)\\ &\cong 
\notag \bigoplus_{j\geq 0}\mathrm{IC}_{\X(d)'^\lambda}[-c_\lambda-1-2j]  \cong
t_{\lambda*}\mathrm{IC}_{\X(d)^\lambda}[-c_\lambda].
\end{align}
%Further, the map $\widetilde{q}_\lambda$ is smooth of relative dimension $c_\lambda$. By taking the lowest non-zero perverse filtration on $\widetilde{q}_{\lambda *}\mathrm{IC}_{\widetilde{X}(d)^{\lambda\geq 0}}$, one obtains a map\[\mathrm{S}\colon \widetilde{q}_{\lambda *}\mathrm{IC}_{\widetilde{X}(d)^{\lambda\geq 0}}\to \mathrm{IC}_{\X(d)^\lambda}[-c_\lambda].\]
%Define the map $\Delta_\lambda$ as the composition:
%\[\Delta_\lambda\colon \pi_*\mathrm{IC}_{\X(d)}\to 
%\pi_{\lambda*}p_{\lambda*}\mathrm{IC}_{\X(d)^{\lambda\geq 0}}[c_\lambda]=
%i_{\lambda*}\pi_{\lambda*}q_{\lambda*}\mathrm{IC}_{\X(d)^{\lambda\geq 0}}%[c_\lambda]\xrightarrow{\mathrm{S}}i_{\lambda*}\pi_{\lambda*}\mathrm{IC}_{\X(d)^\lambda}.\]
%Alternatively, if $\ell_\lambda\colon \X(d)^\lambda\to \X(d)$, then 
The map $\mathrm{S}$ sends $h^i$ for $0\leq i\leq c_{\lambda}-1$ to zero, 
and $h^i$ for $i\geq c_{\lambda}$ to $h^{i-c_{\lambda}}$. 
We then define the map $\Delta_{\lambda}$ to be the composition
\begin{equation}\label{defdeltamap}
\Delta_\lambda\colon \pi_{d \ast}\mathrm{IC}_{\X(d)} \to 
\pi_{d \ast}a_{\lambda\ast}\mathrm{IC}_{\X(d)^{\lambda}}[2c_{\lambda}] 
=i_{\lambda\ast}\pi_{\lambda\ast}
\mathrm{IC}_{\X(d)^{\lambda}}[2c_{\lambda}] 
\stackrel{\mathrm{S}}{\to}
i_{\lambda\ast}\pi_{\lambda\ast}
\mathrm{IC}_{\X(d)^{\lambda}}.
\end{equation}
%Consider a triplet $A$ of tuplets $(d_i)_{i=1}^k$, $(m_i)_{i=1}^k$, $(j_i)_{i=1}^k$, such that $d_i$ are pairwise different positive integers, $m_i$ are positive integers such that $\sum_{i=1}^k m_id_i=d$, and $j_i$ are nonnegative integers. Consider the simple perverse sheaf in perverse degree $d_A:=\sum_{i=1}^k (2m_i+1)$:\[\mathrm{P}_A:=\bigotimes_{i=1}^k\mathrm{Sym}^{m_i}\mathrm{IC}_{d_i}[-2j_i-1].\]
%Meinhardt--Reineke proved that\[\pi_*\mathrm{IC}_{\X(d)}=\bigoplus_A \mathrm{P}_A,\] where the sum is after all data $A$ as above. For such $A$, let $\lambda$ be a dominant cocharacter such that $\X(d)^\lambda\cong \times_{i=1}^k\X(d_i)$.
% \ell_{\lambda\ast}\mathbb{Q}_{\X(d)^{\lambda}}$, i.e. 
% \begin{align*}
% \pi_{\ast}\mathrm{IC}_{\X(d)} \to 
% \pi_{\ast}\ell_{\lambda\ast}\mathrm{IC}_{\X(d)^{\lambda}}[2c_{\lambda}] 
% =i_{\lambda\ast}\pi_{\lambda\ast}
% \mathrm{IC}_{\X(d)^{\lambda}}[2c_{\lambda}] 
% \stackrel{\mathrm{S}}{\to}
% i_{\lambda\ast}\pi_{\lambda\ast}
% \mathrm{IC}_{\X(d)^{\lambda}}.
% \end{align*}
% )}

Recall the notations from Subsection \ref{subsection:decompositiontheorem} and the decomposition theorem \eqref{DM}.
We consider the total perverse cohomology 
\begin{align*}{}^p\mathcal{H}^{\ast}\left(\pi_{d*}\mathrm{IC}_{\X(d)}\right):=\bigoplus_{i\in\mathbb{Z}}{}^p\mathcal{H}^i\left(\pi_{d*}\mathrm{IC}_{\X(d)}\right)[-i].
\end{align*}
For $A\in\mathscr{P}$ as in Subsection \ref{subsection:decompositiontheorem}, 
we have natural maps into and onto the direct summand $\mathrm{P}_A$
\begin{align*}
\mathrm{P}_A\hookrightarrow {}^p\mathcal{H}^{\ast}\left(\pi_{d*}\mathrm{IC}_{\X(d)}\right) \twoheadrightarrow \mathrm{P}_A.
\end{align*}
Suppose that the corresponding partition of $A$ is $\mathbf{d}=(d_i)_{i=1}^k$
as above, and let
$B \in \mathscr{P}$ be another element. 
We have the following map 
\begin{align*}
    \Delta_{\lambda} \colon \mathrm{P}_{\mathrm{B}} \hookrightarrow 
    {}^p\mathcal{H}^{\ast}\left(\pi_{d*}\mathrm{IC}_{\X(d)}\right)
    \stackrel{\Delta_{\lambda}}{\to} {}^p\mathcal{H}^{\ast}\left(i_{\lambda*}\pi_{\lambda*}\mathrm{IC}_{\X(d)^{\lambda}}\right)
    \twoheadrightarrow \mathrm{P}_{A}. 
\end{align*}

\begin{prop}\label{prop310}
    Let $A, B\in \mathscr{P}$ with corresponding sheaves $\mathrm{P}_A$ and $\mathrm{P}_B$ of different support. Assume that $p_B\leq p_A$. 

    (i) %It suffices to show the claim for $W=0$.
        The map \eqref{defdeltamap} induces an isomorphism 
    \begin{equation}\label{mapdeltaaa}
        \Delta_\lambda\colon \mathrm{P}_A\xrightarrow{\cong} \mathrm{P}_A%\subset \bigotimes_{i=1}^k\bigotimes_{a\geq 0}  (\mathrm{IC}_{d_i}[-2a-1])^{\times m_{i,a}}
    .\end{equation}

    (ii) The map $\Delta_\lambda \colon \mathrm{P}_B\to \mathrm{P}_A$ is zero. 
\end{prop}

\begin{proof}
    (i) 
   Recall the subset $\mathscr{P}^{\circ} \subset \mathscr{P}$ in Subsection~\ref{subsec:reltop}. 
We first prove the case that $A \in \mathscr{P}^{\circ}$, i.e. 
$A$ corresponds to a tuple $(e_i, m_{i, a})_{i=1}^s$ such that 
$m_{i, a}=0$ for $a\geq 1$. 
Let $(d_i)_{i=1}^k$ be a partition of $d$
where $e_j$ appear $m_{j}:=m_{j, 0}$-times, and $\lambda$ 
an antidominant cocharacter with associated partition $(d_i)_{i=1}^k$. 
    Recall the sheaf-theoretic cohomological Hall algebra (\ref{coha})
    \begin{align*}
        m_{\lambda}^{\rm{co}} \colon i_{\lambda\ast}\pi_{\lambda*}\mathrm{IC}_{\X(d)^{\lambda}} \to 
        \pi_{d\ast}\mathrm{IC}_{\X(d)}. 
    \end{align*}
   The lowest non-zero piece of the perverse filtration on $i_{\lambda \ast}\pi_{\lambda*}\mathrm{IC}_{\X(d)^{\lambda}}$ is given by \[
    {}^p\tau^{\leq k}i_{\lambda \ast}\pi_{\lambda*}\mathrm{IC}_{\X(d)^\lambda}=i_{\lambda \ast}\boxtimes_{i=1}^s\left(\mathrm{IC}_{X(e_i)}[-1]\right)^{\boxtimes m_i}.\]
    The (shifted) perverse sheaf $i_{\lambda\ast}\boxtimes_{i=1}^s\left(\mathrm{IC}_{X(e_i)}[-1]\right)^{\boxtimes m_i}$ splits as a direct sum of simple perverse sheaves, and one of 
    them is $\mathrm{P}_A$. 
    There is thus a natural inclusion $\mathrm{P}_A\subset i_{\lambda \ast}\pi_{\lambda*}\mathrm{IC}_{\X(d)^\lambda}$.
    The map \begin{equation}\label{pa0}
    \Delta_\lambda m_\lambda^{\rm{co}}\colon \mathrm{P}_A\to i_{\lambda \ast}\pi_{\lambda*}\mathrm{IC}_{\X(d)^\lambda}
   \end{equation}
    has image in the lowest non-zero perverse truncation of $i_{\lambda \ast}\pi_{\lambda*}\mathrm{IC}_{\X(d)^\lambda}$,
    and thus \eqref{pa0} induces a map:
    \begin{equation}\label{pa}
    \Delta_\lambda m_\lambda^{\rm{co}}\colon \mathrm{P}_A\to {}^p\tau^{\leq s}i_{\lambda \ast}\pi_{\lambda*}\mathrm{IC}_{\X(d)^\lambda}=\boxtimes_{i=1}^k\left(\mathrm{IC}_{X(e_i)}[-1]\right)^{\boxtimes m_i}.
   \end{equation}
    The (shifted) perverse sheaf $i_{\lambda \ast}\boxtimes_{i=1}^k\left(\mathrm{IC}_{X(e_i)}[-1]\right)^{\boxtimes m_i}$ has only one summand isomorphic to $\mathrm{P}_A$, which is a simple (shifted) perverse sheaf. Thus the map \eqref{pa} restricts to a map 
    \begin{equation}\label{pa2}
        \mathrm{P}_A\to \mathrm{P}_A
    \end{equation} 
    All such maps are given by multiplication by scalars. It is thus an isomorphism if it is not the zero map. It suffices to show that the maps \eqref{pa0} or \eqref{pa} are not zero. We will show this after passing to a convenient open set of $X(d)^\lambda$. %We next explain the construction of $U$.

%Recall that for any $e\in\mathbb{N}^I$, we denote by $\pi_e\colon \X(e)\to X(e)$. 
For any non-zero $e\in \mathbb{N}^I$, by the same argument used to prove Lemma~\ref{clambdazero},  there exists a point in $R(e)$ corresponding to a simple representation. 
For $1\leq i\leq k$, let $R_i$ be a simple representation of $Q$ of dimension $d_i$ such that $R_i$ and $R_j$ are not isomorphic for $1\leq i<j\leq k$. Let $R:=\bigoplus_{i=1}^k R_i$. Note that the stabilizer of $R$ is $T=(\mathbb{C}^*)^k$. By the etale slice theorem, there is an analytic smooth open substack $R\in \mathscr{U}/T\subset \X(d)$ such that \[\mathscr{U}\ssslash T\to X(d)\text{ and }\mathscr{U}^\lambda \ssslash T \to X(d)^\lambda\] are analytic neighborhoods of $\pi_d(R)$ and $\times_{i=1}^k \pi_{d_i}(R_i)=\pi_\lambda(R)$, respectively. After possibly shrinking $\mathscr{U}$, we may assume that $\mathscr{U}$ and $\mathscr{U}^\lambda$ are contractible. The maps 
\begin{equation*}
    \begin{tikzcd}
        \X(d)^\lambda \arrow[rr, bend right, "a_\lambda"] & \X(d)^{\lambda\geq 0}\arrow[l, "q_\lambda"']\arrow[r, "p_\lambda"] & \X(d)
    \end{tikzcd}
\end{equation*}
are, analytically locally over $\pi_d(R)\in X(d)$, isomorphic to the following:
\begin{equation}\label{dia:locU}
    \begin{tikzcd}
        \mathscr{U}^\lambda/T \arrow[rr, hook, bend right, "a_\lambda"] & \mathscr{U}^{\lambda\geq 0}/T\arrow[l,  "q_\lambda"']\arrow[r, hook, "p_\lambda"] & \mathscr{U}/T.
    \end{tikzcd}
\end{equation}
Note that the maps $p_\lambda$ and $a_\lambda$ in (\ref{dia:locU}) are closed immersions.
    % Note that $\dim X(e)\geq 1$ for any non-zero dimension $e$. Indeed, let $R^\circ(e)$ be the space of dimension $e$ representations of $Q^\circ$ and write \[\X(e)=\left(R^\circ(e)\oplus \overline{R^\circ(e)}\oplus\mathfrak{g}(e)\right)/G(e).\]
    % The diagonal cocharacter acts trivially on $\left(R^\circ(d)\oplus \overline{R^\circ(d)}\oplus\mathfrak{g}(d)\right)$, and thus \[\dim X(e)\geq 1+\dim \X(e)\geq 1+\dim \mathfrak{g}(e)-\dim G(e)=1.\]
    % For $1\leq i\leq s$, 
    % choose $U_i\subset X(e_i)$ such that $\pi'^{-1}_{e_i}(U_i)\cong U_i/\mathbb{C}^*$, and further, such that $U_i$ and $U_j$ are disjoint if $1\leq i<j\leq s$ and $e_i=e_j$.
    % Choose a point $(p_i)_{i=1}^s\in X(d)^\lambda=\times_{i=1}^s X(e_i)$ with $p_i$ pairwise distinct
    % for $1\leq i<j\leq s$. Choose small open analytic sets $p_i\in U_i$ pairwise disjoint and let $U:=\times_{i=1}^k U_i$. 
    % Let $T\cong (\mathbb{C}^*)^s$ be the torus acting trivially on $\X(d)^\lambda$.
    % Let $\mathscr{U}\subset R(d)$ be an open subset, invariant under the action of $T$, such that $i_\lambda^{-1}(\mathscr{U}\ssslash T)=\mathscr{U}^\lambda=\mathscr{U}^T=U$. 
    %
    To show that the map \eqref{pa2} is non-zero, it suffices to check that the map 
    \begin{equation}\label{pa00}
    \Delta_\lambda m_\lambda|_{\mathscr{U}^\lambda \ssslash T}\colon \mathrm{P}_A|_{\mathscr{U}^\lambda \ssslash T}\to \mathrm{P}_A|_{\mathscr{U}^\lambda \ssslash T}
    \end{equation} is non-zero. It suffices to check that the map is non-zero after passing to global sections. We drop the restriction to $\mathscr{U}^\lambda$ from the notation from now on.
    %The element $1\in H^0(\mathscr{U}^\lambda/T)$ is in $P^{\leq s}H^{\ast}(\mathscr{U}^\lambda/T)$. 
    We check by a direct computation that 
    \begin{equation}\label{computationpushpull}
    \Delta_\lambda m_\lambda(1)=1\in H^0(\mathscr{U}^\lambda/T).
    \end{equation}
    Note that the computation \eqref{computationpushpull} shows that the map \eqref{pa00} is non-zero, and thus the conclusion follows. It suffices to check the computation in \eqref{computationpushpull} for $\mathscr{U}^\lambda/\mathbb{C}^*$, where by $\mathbb{C}^*$ we denote the image of $\lambda$, because $H^0(\mathscr{U}^\lambda/\mathbb{C}^*)\cong H^0(\mathscr{U}^{\lambda}/T)\cong \mathbb{Q}$. Note that we have $H^{\ast}(\mathscr{U}^\lambda/\mathbb{C}^*)\cong \mathbb{Q}[h]$ and that \[m_\lambda(1)=p_{\lambda*}q_\lambda^*(1)=h^{c_\lambda}\] because $p_\lambda$ has relative dimension $-c_\lambda$. Note that $\Delta_\lambda(h^{c_\lambda})=1$ from the construction of $\Delta_\lambda$, and thus the conclusion follows. 
    
    %when $Q$ has an even number of edges between any two vertices $i$ and $j$ of the quiver. We may assume that the quiver $Q$ has such a property by using a construction as in Subsection \ref{symmquivers}. Note that the computation \eqref{computationpushpull} shows that the map \eqref{pa00} is non-zero, and thus the conclusion follows. 

We next reduce the case of $A \in \mathscr{P} \setminus \mathscr{P}^{\circ}$
to the case as above. 
Let $A \in \mathscr{P}$ corresponds to 
a tuple $(e_i, m_{i, a})$ and let $A^{\circ} \in \mathscr{P}^{\circ}$ be 
a tuple $(e_i, m_{i, a}')$ with $m_{i, 0}'=\sum_{a}m_{i, a}$ and $m_{i, a}'=0$ for $a\geq 1$. 
For $d\in \mathbb{N}^I$, we set $\hbar_d:=c_1(\mathcal{O}(\sigma_d))\in H^2(\X(d))$.
By \cite[Theorem C]{DM},
the summand $\mathrm{P}_{A}$ of ${}^p\mathcal{H}^{\ast}\left(\pi_{d*}\mathrm{IC}_{\X(d)}\right)$ is obtained from $\mathrm{P}_{A^{\circ}}$
by multiplication with 
$u :=\mathrm{Sym}(\hbar_{d_1}^{b_1} \cdots \hbar_{d_k}^{b_k})$. 
Here $b_i$ is the sum of $a \geq 0$ for which 
$m_{j, a} \neq 0$ and $e_j=d_i$, and $\mathrm{Sym}(-)$ is the symmetrizer with respect to the 
$\mathfrak{S}_k$-action on $\{1, \cdots, k\}$. 
Then we have the following commutative diagram 
\begin{align*}
    \xymatrix{
P_{A^{\circ}} \ar[r]^-{\Delta_{\lambda}} \ar[d]_-{u}  & P_{A^{\circ}} \ar[d]_-{u}\\
P_A \ar[r]^-{\Delta_{\lambda}} & P_A.   
    }
\end{align*}
In order to show that the above diagram commutes, by simpleness of $P_A$ it is 
enough to check at $\mathcal{U}\ssslash T$ as in the above arguments. 
Then the above diagram commutes by
$\Delta_{\lambda}(h^{c_{\lambda}}u)=u$
from the definition of $\Delta_{\lambda}$. 

    (ii) If $p_B<p_A$, the map $\Delta_\lambda\colon\mathrm{P}_B\to\mathrm{P}_A$ is zero by considering the perverse degree. If $p_B=p_A$, then the map is zero because, after a shift, it is a map of simple perverse sheaves with different support. 
\end{proof}

\begin{example}\label{example:nabla}
In Example~\ref{exam:loop},
let $\lambda \colon \mathbb{C}^{\ast} \to (\mathbb{C}^{\ast})^2=T(2)$
be an antidominant cocharacter $t\mapsto (t^{\lambda_1}, t^{\lambda_2})$
for $\lambda_1 \geq \lambda_2$. 
When $\lambda_1>\lambda_2$, we have $h=\hbar_1-\hbar_2$, $c_{\lambda}=2e$ and  
\begin{align*}
    &\Delta_{\lambda}((\hbar_1-\hbar_2)^{2e})(\hbar_1^{m_1} \hbar_2^{m_2}+\hbar_1^{m_2} \hbar_2^{m_1})
    =\hbar_1^{m_1} \hbar_2^{m_2}+\hbar_1^{m_2} \hbar_2^{m_1}, 
    \end{align*}
    and $\Delta_{\lambda}((\hbar_1-\hbar_2)^{2a})=0$ for $\ 0\leq a \leq e-1$. 
    
When $\lambda_1=\lambda_2$, we have $h=(\hbar_1+\hbar_2)/2$, $c_{\lambda}=0$ and 
\begin{align*}
\Delta_{\lambda}((\hbar_1-\hbar_2)^{2a}(\hbar_1+\hbar_2)^k)=(\hbar_1-\hbar_2)^{2a}(\hbar_1+\hbar_2)^k
\end{align*}
for $0\leq a\leq e-1$ and $k\geq 0$.     
\end{example}

We next prove an analogue of Proposition \ref{prop310} for a non-zero potential.
Let $W$ be a potential of $Q$. Recall the shifted perverse sheaves $\mathrm{Q}_A$ from Subsection \ref{subsection:decompositiontheorem}. Let ${}^p\mathcal{H}^{\ast}(\pi_*\varphi_{\mathrm{Tr}\,W}\mathrm{IC}_{\X(d)}[-1])$ be the total perverse cohomology of $\pi_*\varphi_{\mathrm{Tr}\,W}\mathrm{IC}_{\X(d)}[-1]$.
We have natural maps into and onto the 
direct summands
\[\mathrm{Q}_A \hookrightarrow {}^p\mathcal{H}^{\ast}(\pi_*\varphi_{\mathrm{Tr}\,W}\mathrm{IC}_{\X(d)}[-1])\twoheadrightarrow \mathrm{Q}_A.\]
By applying the vanishing cycle functor to the maps \eqref{defdeltamap}, we obtain:
\begin{align}\label{defdeltamapbis}
\Delta_\lambda\colon \pi_{d*}\varphi_{\mathrm{Tr}\,W}\mathrm{IC}_{\X(d)}[-1]&\to i_{\lambda*}\pi_{\lambda *}\varphi_{\mathrm{Tr}\,W}\mathrm{IC}_{\X(d)^\lambda}[-1] \\
\notag &= i_{\lambda*} \boxtimes_{i=1}^k\left(\pi_{d_i*}\varphi_{\mathrm{Tr}\,W}\mathrm{IC}_{\X(d_i)}[-1]\right).
\end{align}
Let $\mathrm{Q}_A^{\mathrm{inv}}$ be defined by the exact triangle 
\[\mathrm{Q}_A^{\mathrm{inv}}[-1]\to \mathrm{Q}_A\xrightarrow{1-\mathrm{T}}\mathrm{Q}_A\to \mathrm{Q}_A^{\mathrm{inv}}.\]

\begin{prop}\label{prop310bis}
    Let $A, B\in \mathscr{P}$ with corresponding perverse sheaves $\mathrm{Q}_A$ and $\mathrm{Q}_B$ such that $P_A$ and $P_B$ have different supports. Assume that $p_B\leq p_A$. 

    (i) %It suffices to show the claim for $W=0$.
        The map \eqref{defdeltamapbis} induces isomorphisms 
    \begin{equation}\label{mapdeltaaabis}
        \Delta_\lambda\colon \mathrm{Q}_A\xrightarrow{\sim} \mathrm{Q}_A,\, \Delta_\lambda\colon \mathrm{Q}^{\mathrm{inv}}_A\xrightarrow{\sim} \mathrm{Q}^{\mathrm{inv}}_A %\subset \bigotimes_{i=1}^k\bigotimes_{a\geq 0}  (\mathrm{IC}_{d_i}[-2a-1])^{\times m_{i,a}}
    .\end{equation}

    (ii) The maps $\Delta_\lambda \colon \mathrm{Q}_B\to \mathrm{Q}_A$ and $\Delta_\lambda \colon \mathrm{Q}^{\mathrm{inv}}_B\to \mathrm{Q}^{\mathrm{inv}}_A$ are zero. 
\end{prop}

\begin{proof}
    The maps above are induced from the map \eqref{defdeltamap}, thus the conclusion follows from Proposition \ref{prop310}.
\end{proof}

We now record a corollary to be used in the proof of Theorem \ref{thm2}.
Fix a splitting 
\begin{equation}\label{xa}
H^\bullet\left(\X(d), \varphi^{\mathrm{inv}}_{\mathrm{Tr}\,W}\mathrm{IC}_{\X(d)}[-1]\right)=\bigoplus_{A\in\mathscr{P}}H^\bullet(X(d), \mathrm{Q}^{\mathrm{inv}}_A).
\end{equation}
For an element $x\in H^\bullet\left(\X(d), \varphi^{\mathrm{inv}}_{\mathrm{Tr}\,W}\mathrm{IC}_{\X(d)}[-1]\right)$, using the decomposition above we write 
\begin{align}\label{xa2}
x=\sum_{A\in\mathscr{P}}x_A, \ x_A\in H^\bullet(X(d), \mathrm{Q}^{\mathrm{inv}}_A). 
\end{align}

\begin{cor}\label{onesupport2}
Let $\delta\in M(d)_\mathbb{R}^{W_d}$
and $\lambda$ be an antidominant cocharacter of $T(d)$ with associated partition $\mathbf{d}$ such that $\varepsilon_{\lambda,\delta}=\varepsilon_{\mathbf{d},\delta}$.
For $x\in H^i(\X(d), \varphi^{\mathrm{inv}}_{\mathrm{Tr}\,W}\mathrm{IC}_{\X(d)}[-1])$, suppose that
\begin{align}\label{a:pullback}a^*_\lambda(x)\in \bigoplus_{j=0}^{c_{\mathbf{d}, \delta}-1}H^{i-2j}(\X(d)'^\lambda, \varphi^{\mathrm{inv}}_{\mathrm{Tr}\,W}\mathrm{IC}_{\X(d)'^\lambda}[-2])h^j.
\end{align}
%Recall the decomposition \eqref{xa2}. 
% Let $x\in H^\bullet(\X(d), \varphi])$ and assume that 
% \[\ell^*_\lambda(x)\in \bigoplus_{i=0}^{w_{\lambda, \delta}-1}H^{\ast}(\X(d)'^\lambda, \varphi)h^i.\]
Then we have the following: 

(i) If $\varepsilon_{\mathbf{d}, \delta}=0$, then $x_A=0$ for all tuples $A\in\mathscr{P}$ with corresponding partition 
$\mathbf{d}$. 

(ii) If $\varepsilon_{\mathbf{d}, \delta}=1$, then $x_A=0$ for all tuples $A\in\mathscr{P}$ with corresponding partition 
$\mathbf{d}$ such that $A \neq A^\circ$.
\end{cor}

\begin{proof}
If $\varepsilon_{\mathbf{d}, \delta}=0$, 
then $\mathrm{S}a_{\lambda}^{\ast}(x)=0$
by the definition of the map $\mathrm{S}$ in (\ref{Sdef})
and the condition (\ref{a:pullback}). 
Then we have $\Delta_{\lambda}(x)=0$. 
If $\varepsilon_{\mathbf{d}, \delta}=1$, 
then from the definition of $\mathrm{S}$ in (\ref{Sdef}), 
we have 
\begin{align*}
\mathrm{S}a_\lambda^*(x)\in H^{\ast}\left(X(d)^\lambda, \varphi_{\mathrm{Tr}\,W}^{\mathrm{inv}}\mathrm{IC}_{\X(d)'^\lambda}[-c_\lambda-2]\right). 
\end{align*}
Therefore 
$\Delta_\lambda(x)$ lies in the image of 
\begin{align*}
H^i(\X(d)'^\lambda, \varphi^{\mathrm{inv}}_{\mathrm{Tr}\,W}\mathrm{IC}_{\X(d)'^\lambda}[-2])\hookrightarrow  H^i(\X(d)^\lambda, \varphi^{\mathrm{inv}}_{\mathrm{Tr}\,W}\mathrm{IC}_{\X(d)^\lambda}[-1]).\end{align*}
Now both claims follow from Proposition \ref{prop310bis}
using the induction on $p_A$. 
\end{proof}

% We next state a version of Corollary \ref{onesupport2} for $H^\bullet(\X(d), C[-2])=H^\bullet(\X(d), \varphi[-1])^\dagger$.
% Fix a splitting \[H^\bullet(\X(d), \varphi\,\mathrm{IC}[-1])^\dagger=\bigoplus_{A\in \mathscr{P}} H^\bullet(X(d), \mathrm{Q}_A)^\dagger.\]
% Let $y\in H^\bullet\left(\X(d), \varphi\,\mathrm{IC}[-1]\right)^\dagger$. Use the decomposition above to write 
% \[y=\sum_{A\in\mathscr{P}}y_A\] with $y_A\in H^\bullet(X(d), \mathrm{Q}_A)^\dagger$.

% \begin{cor}\label{onesupport3}
% Let $\lambda$ be an antidominant cocharacter and let $\delta\in M(d)_\mathbb{R}^{W_d}$.
% Let $y\in H^\bullet(\X(d), C)$ and assume that 
% \[\ell^*_\lambda(y)\in \bigoplus_{i=0}^{w_{\lambda, \delta}-1}H^{\ast}(\X(d)'^\lambda, C)h^i.\]

% (a) If $\varepsilon_{\lambda, \delta}=1$, then $y_A=0$ for all tuples $A\in\mathscr{P}$ with corresponding cocharacter $\lambda$. 

% (b) If $\varepsilon_{\lambda, \delta}=0$, then $y_A=0$ for all tuples $A\in\mathscr{P}$ with corresponding cocharacter $\lambda$ and different from $A^\circ$.
% \end{cor}

% \begin{proof}
% The analogues of Corollary \eqref{onesupport2} hold for $H^\bullet(\X(d), \varphi)_{\mathrm{inv}}$ and $H^\bullet(\X(d), \varphi)^{\mathrm{inv}}$. Thus the claims of Corollary \ref{onesupport3} follow from the decomposition \eqref{notdagger}.
% \end{proof}

Finally in this section, we give a proof of Theorem~\ref{thm2}. 
\begin{proof}[Proof of Theorem \ref{thm2}]
Recall the cycle map in (\ref{cyclesdv})
\[\mathrm{c}\colon \mathrm{gr}_a K_i^{\mathrm{top}}(\mathbb{S}(d; \delta))\to H^{\dim \X(d)-2a-i}(\X(d), \varphi_{\Tr W}^{\rm{inv}}
\mathrm{IC}_{\X(d)}).\]
By Proposition \ref{reduce}, we may assume that $Q$ has at least two loops at every vertex.
    Let $y$ be in the image of the above map. 
    By Proposition \ref{prop06} and Corollary \ref{onesupport2},
    we have that $y_A=0$ unless $A=A^\circ$ for some partition $\mathbf{d}=(d_i, m_i)_{i=1}^k$ of $d$ with $m_i\geq 1$ and $d_i$ pairwise distinct with $\varepsilon_{\mathbf{d}, \delta}=1$. The statement thus follows.
\end{proof}

% \begin{remark}
%     Note that Theorem \ref{thm2} holds for an arbitrary quiver with potential $(Q,W)$. 
%     Let $Q=(I,E)$. Construct a quiver $Q'=(I,E')$, where $E'=E\sqcup\{\omega_i, \omega'_i\mid i\in I\}$, where $\omega_i, \omega'_i$ are loops at $i\in I$. 
%     Let $W'=W+\sum_{i\in I}\omega_i\omega'_i$. Then Theorem \ref{thm}
% \end{remark}

\section{Topological K-theory of quasi-BPS categories for preprojective algebras}\label{subsection:preproj}

%\subsection{BPS sheaves for preprojective algebras}

In this section, we use the results of Sections \ref{s5} and \ref{s6} to compute the topological K-theory of quasi-BPS categories for 
preprojective algebras
in terms of BPS cohomology, see Theorem \ref{thm1plus}. The results in this section can be applied (in conjunction with those from Subsection \ref{subsec67}) to compute the topological K-theory of quasi-BPS categories for local Calabi-Yau surfaces, see \cite{PTK3}.

\subsection{The preprojective BPS sheaf}\label{subsec71}
Let $Q^\circ=(I, E^\circ)$ be a quiver
and let $(Q, W)$ be its tripled quiver. 
Recall the moduli stack of $Q$-representations of 
dimension $d$ and its good moduli space
\[\pi_{X,d}:=\pi_d\colon \X(d)\to X(d).\]
Recall also the moduli stack of representations 
of the preprojective algebra of $Q^{\circ}$:
\[\pi_{P,d}\colon \mathscr{P}(d)^{\rm{cl}}\to P(d).\]
We consider the moduli stack of dimension $d$ representations of the double quiver of $Q^\circ$ and its good moduli space:
\[\pi_{Y, d} \colon \mathscr{Y}(d):=(R^\circ(d)\oplus R^\circ(d)^\vee)/G(d) \to Y(d).\]
We have the diagram:
\begin{equation*}
    \begin{tikzcd}
        \mathscr{P}(d)^\mathrm{cl}
        \arrow[d, "\pi_{P,d}"]\arrow[r, hook, "j"]& \mathscr{Y}(d)
        \arrow[d, "\pi_{Y,d}"]&\mathscr{X}(d)\arrow[l, "\eta"']\arrow[d, "\pi_{X,d}"]\\
        P(d) \arrow[r, hook, "j"]& Y(d)& X(d).\arrow[l, "\eta"']
    \end{tikzcd}
\end{equation*} 
Here $\eta \colon \X(d) \to \mathscr{Y}(d)$ is the projection which forgets the $\mathfrak{g}(d)$-component and the bottom horizontal 
arrows are induced maps on good moduli spaces. 
Let $\mathbb{C} \hookrightarrow \mathfrak{g}(d)$
be the diagonal embedding, which induces the closed immersion 
\begin{align*}
\gamma \colon  X'(d) :=(R^{\circ}(d) \oplus R^{\circ}(d)^{\vee} \oplus \mathbb{C})\ssslash G(d) \hookrightarrow 
X(d). 
\end{align*}
Let $\overline{\eta}:=\eta|_{X'(d)}$. 
By \cite[Theorem/Definition 4.1]{D}, there exists a \textit{preprojective BPS sheaf} 
\[\mathcal{BPS}^p_d\in \mathrm{Perv}(P(d))\]  such that the BPS sheaf of the tripled quiver with potential $(Q,W)$ associated to $Q^\circ$ is 
\begin{equation}\label{etai}
\mathcal{BPS}_d=\gamma_{\ast}\overline{\eta}^{\ast}j_*(\mathcal{BPS}^p_d)[1]\in\mathrm{Perv}(X(d)).
\end{equation}
For a partition $A=(d_i)_{i=1}^k$ of $d$, define $\mathcal{BPS}^p_A\in \mathrm{Perv}(P(d))$ as in \eqref{BPSAsheaf}
using addition maps on $P(d)$. 
For $\delta\in M(d)^{W_d}_\mathbb{R}$,
define the following perverse sheaves on $P(d)$:
\begin{equation}\label{defpreprojbps}
\mathcal{BPS}^p_{d, \delta}:=\bigoplus_{A\in S^d_\delta}\mathcal{BPS}^p_A,\, \mathcal{BPS}^p_{d,v}:=\mathcal{BPS}^p_{d, v\tau_d},
\end{equation}
where the set of partitions $S^d_{\delta}$ is defined from the tripled quiver $Q$, see Subsection~\ref{subsec612}.
Then $\mathcal{BPS}^p_{d,v}$ is a direct summand of $\pi_{P,d*}\omega_{\mathscr{P}(d)^{\mathrm{cl}}}$, see \cite[Theorem A]{D}, and so $H^{-a}(P(d), \mathcal{BPS}^p_{d,v})$ is a direct summand of \[H^{\mathrm{BM}}_a(\mathscr{P}(d)^{\mathrm{cl}})=H^{-a}\big(P(d), \pi_{P,d*}\omega_{\mathscr{P}(d)^{\mathrm{cl}}}\big).\] 
% Actually, $\mathcal{BPS}_{d,v}$ is a summand of ${}^p\tau^{\leq 0}\pi_{P*}\omega_{\mathscr{P}(d)^{\mathrm{cl}}}={}^p\mathcal{H}^0\left(\pi_{P*}\omega_{\mathscr{P}(d)^{\mathrm{cl}}}\right)$, but we do not use this in the paper.

Recall the maps 
\[\mathscr{P}(d)\xleftarrow{\eta'}\eta^{-1}(\mathscr{P}(d))\xrightarrow{j'}\X(d).\]
The dimension of $\mathscr{P}(d)$ as a quasi-smooth stack is $\dim \mathscr{P}(d):=\dim \mathscr{Y}(d)-\dim \mathfrak{g}(d)$.
Recall the dimensional reduction isomorphism from Subsection \ref{subsection:dimred}:
\begin{align*}
j'_*\eta'^*\colon H^{\mathrm{BM}}_a(\mathscr{P}(d)^{\mathrm{cl}})\cong H^{\mathrm{BM}}_a(\mathscr{P}(d))&\xrightarrow{\sim} H^{2\dim\mathscr{Y}(d)-a}(\X(d), \varphi_{\mathrm{Tr}\,W}\mathbb{Q}_{\X(d)}[-1])\\ &=H^{\dim\mathscr{P}(d)-a}(\X(d), \varphi_{\mathrm{Tr}\,W}\mathrm{IC}_{\X(d)}[-1]).
\end{align*}
By the construction of the PBW isomorphism for preprojective Hall algebras \cite[Equation (31)]{D}, 
the above isomorphism preserves the BPS cohomologies:
\begin{equation}\label{BPSmatched}
j'_*\eta'^*\colon H^{-a}(P(d), \mathcal{BPS}^p_{d,v})\xrightarrow{\sim} H^{\dim\mathscr{P}(d)-a}(X(d), \mathcal{BPS}_{d,v}).
\end{equation}
% By \cite[Theorem A]{Dav}, 
% we have that $H^{\mathrm{BM}}_{\mathrm{odd}}(\mathscr{P}(d)^{\mathrm{cl}})=0$. Further, the monodromy is trivial on $H^\bullet(\X(d), \varphi_{\mathrm{Tr}\,W}\mathrm{IC}_{\X(d)}[-1])$ because $W$ is linear, see \cite[Theorem A.1]{MR3667216}. 

\subsection{Computations}

Recall the categories 
\[\mathbb{T}(d; \delta)\subset D^b(\mathscr{P}(d))\text{ and }\mathbb{T}(d; \delta)^{\mathrm{red}}\subset D^b(\mathscr{P}(d)^{\mathrm{red}})\] from Subsection \ref{subsec:qbps:pre}. 
Recall also that 
we have natural closed immersions 
\begin{align*}
    l \colon \mathscr{P}(d)^{\rm{cl}} \hookrightarrow 
    \mathscr{P}(d)^{\rm{red}} \stackrel{l'}{\to} 
    \mathscr{P}(d). 
\end{align*}

\begin{prop}\label{prop71}
Let $Q^\circ=(I, E^\circ)$ be a quiver, let $d\in \mathbb{N}^I$, and let $\delta\in M(d)^{W_d}_\mathbb{R}$.
    The closed immersion $l'\colon \mathscr{P}(d)^{\mathrm{red}}\hookrightarrow \mathscr{P}(d)$ induces an equivalence of spectra:
    \[l'_*\colon K^{\rm{top}}(\mathbb{T}(d; \delta)^{\mathrm{red}})\xrightarrow{\sim} K^{\rm{top}}(\mathbb{T}(d; \delta)).\]
\end{prop}
\begin{proof}
    There is an equivalence of spectra $l'_*\colon G^{\mathrm{top}}(\mathscr{P}(d)^{\mathrm{red}})\xrightarrow{\sim} G^{\mathrm{top}}(\mathscr{P}(d))$. The claim follows from Proposition \ref{redunred} and Theorem \ref{theorem266}.
\end{proof}

%\begin{prop}\label{prop71}
%Let $Q$ be a symmetric quiver.
%    Then there is an equivalence of spectra $l'_*\colon K^{\mathrm{top}}(\mathbb{T}(d)^{\mathrm{red}}_v)\to K^{\mathrm{top}}(\mathbb{T}(d)_v)$.
%\end{prop}
%
%\begin{proof}
%    There is an equivalence of spectra $l'_*\colon G^{\mathrm{top}}(\mathscr{P}(d)^{\mathrm{red}})\xrightarrow{\sim} G^{\mathrm{top}}(\mathscr{P}(d))$. The claim then follows from Theorem \ref{theorem266}. 
%\end{proof}

% Consider the category of graded matrix factorizations $\mathrm{MF}^{\mathrm{gr}}(\X(d), \mathrm{Tr}\,W)$ as in Subsection \ref{subsec26}. We let \[\mathbb{T}(d)_v\subset D^b(\mathscr{P}(d))\] be the category such that the Koszul equivalence induces:
% \[\kappa\colon \mathbb{T}(d)_v\xrightarrow{\sim}\mathbb{S}^{\mathrm{gr}}(d)_v.\]
% By \cite[Theorem A]{Dav}, 
% we have that $H^{\mathrm{BM}}_{\mathrm{odd}}(\mathscr{P}(d)^{\mathrm{cl}})=0$. Further, the monodromy is trivial on $H^\bullet(\X(d), \varphi_{\mathrm{Tr}\,W}\mathrm{IC}_{\X(d)}[-1])$ because $W$ is linear, see \cite[Theorem A.1]{MR3667216}. We thus obtain that:

For $i\in\mathbb{Z}$, consider the Chern character map \eqref{chquasismooth} for the quasi-smooth stack $\mathscr{P}(d)$:
\begin{equation}\label{chpd}
    \mathrm{ch}\colon G^{\mathrm{top}}_i\left(\mathscr{P}(d)\right)\to \widetilde{H}^{\mathrm{BM}}_i(\mathscr{P}(d)).
\end{equation}
It induces a Chern character map: 
\begin{equation}\label{chtdv}
    \mathrm{ch}\colon K^{\mathrm{top}}_i(\mathbb{T}(d; \delta))\hookrightarrow G^{\mathrm{top}}_i\left(\mathscr{P}(d)\right)\to \widetilde{H}^{\mathrm{BM}}_i(\mathscr{P}(d)).
\end{equation}

\begin{cor}\label{cor615bis}
    The map \eqref{chtdv} is injective over $\mathbb{Q}$. 
\end{cor}

\begin{proof}
 The injectivity of (\ref{chtdv}) follows from Proposition \ref{cherninj}, Theorem \ref{theorem266}, and the Koszul equivalence \eqref{Koszul}.
\end{proof}

\begin{cor}\label{cor615}
    We have $K^{\mathrm{top}}_1(\mathbb{T}(d; \delta))_{\mathbb{Q}}=0$.
\end{cor}
\begin{proof}
 We have that $H^{\mathrm{BM}}_{\mathrm{odd}}(\mathscr{P}(d)^{\mathrm{cl}})=0$ by \cite[Theorem A]{Dav}. The conclusion follows by Corollary~\ref{cor615bis}.
\end{proof}

Recall the filtration $E_\ell G^{\mathrm{top}}_0(\mathscr{P}(d))$ of $G^{\mathrm{top}}_0(\mathscr{P}(d))$ from Subsection \ref{subsec331}. Define the filtration:
\[E_\ell K^{\mathrm{top}}_0(\mathbb{T}(d; \delta)):=E_\ell G^{\mathrm{top}}_0(\mathscr{P}(d))\cap K^{\mathrm{top}}_0(\mathbb{T}(d; \delta))\subset K^{\mathrm{top}}_0(\mathbb{T}(d; \delta)).\]
We denote by $\mathrm{gr}_\ell K^{\mathrm{top}}_0(\mathbb{T}(d; \delta))$ the associated graded piece, and note that it is a direct summand of $\mathrm{gr}_\ell G^{\mathrm{top}}_0(\mathscr{P}(d))$ by Theorem \ref{theorem266}. 
Similarly, we define filtrations 
\begin{align*}
&E_\ell G^{\mathrm{top}}_0(\mathscr{P}(d)^{\mathrm{red}})\subset G^{\mathrm{top}}_0(\mathscr{P}(d)^{\mathrm{red}}), \\ 
&E_\ell 
K^{\mathrm{top}}_0(\mathbb{T}(d; \delta)^{\mathrm{red}}
)\subset K^{\mathrm{top}}_0(\mathbb{T}(d; \delta)^{\mathrm{red}}).
\end{align*}
Recall the Koszul equivalence (\ref{kappamagic}) and the
forget-the-grading functor 
\begin{align*}
  D^b(\mathscr{P}(d)) \stackrel{\kappa}{\to} 
  \mathrm{MF}^{\rm{gr}}(\X(d), \Tr W) \stackrel{\Theta}{\to}
  \mathrm{MF}(\X(d), \Tr W). 
\end{align*}
The above functors restrict to functors 
\begin{align*}
    \mathbb{T}(d; \delta) \stackrel{\sim}{\to}
    \mathbb{S}^{\rm{gr}}(d; \delta) \stackrel{\Theta}{\to}
    \mathbb{S}(d; \delta). 
\end{align*}

\begin{cor}\label{cor616}
The functor $\Theta$ induces 
an isomorphism:  
    \begin{equation}\label{isograd}
        \mathrm{gr}_\ell K^{\mathrm{top}}_0\left(\mathrm{MF}^{\mathrm{gr}}(\X(d), \mathrm{Tr}\,W)\right)_{\mathbb{Q}}\xrightarrow{\sim} \mathrm{gr}_\ell K^{\mathrm{top}}_0\left(\mathrm{MF}(\X(d), \mathrm{Tr}\,W)\right)_{\mathbb{Q}}.
    \end{equation}
    There are thus also isomorphisms:
    \begin{align*}\mathrm{gr}_{\ell} K^{\mathrm{top}}_0\left(\mathbb{T}(d; \delta)\right)_{\mathbb{Q}}&\xrightarrow{\sim} \mathrm{gr}_{\ell+\dim \mathfrak{g}(d)} K^{\mathrm{top}}_0\left(\mathbb{S}^{\mathrm{gr}}(d; \delta)\right)_{\mathbb{Q}} \\
    &\xrightarrow{\sim} \mathrm{gr}_{\ell+\dim \mathfrak{g}(d)+1} K^{\mathrm{top}}_0\left(\mathbb{S}(d; \delta)\right)_{\mathbb{Q}}.\end{align*}
\end{cor}

\begin{proof}
    The isomorphism \eqref{isograd} follows from Proposition~\ref{prop:forg}, 
    and $H_{\mathrm{odd}}^{\mathrm{BM}}(\mathscr{P}(d)^{\mathrm{cl}})=0$
    by~\cite[Theorem~A]{Dav}.  The other isomorphisms follow from the Koszul equivalence, see Corollary~\ref{prop53} for an explanation of the degree of the graded pieces.
\end{proof}

    \begin{cor}\label{cor617}
    There is a commutative diagram,
    where the vertical maps are cycle maps and the left horizontal maps are the dimensional reduction maps $j'_*\eta'^*$: 
\begin{equation*}
    \begin{tikzcd}
        \mathrm{gr}_{\ast} G^{\mathrm{top}}_0(\mathscr{P}(d))_{\mathbb{Q}}\arrow[d, "\mathrm{c}"]\arrow[r, "\sim"]& \mathrm{gr}_{\ast} K^{\mathrm{top}}_0(\mathrm{MF}^{\mathrm{gr}}(\X(d), \mathrm{Tr}\,W))_{\mathbb{Q}}
        \arrow[d, "\mathrm{c}"]\arrow[r, "\sim"]&\mathrm{gr}_{\ast} K^{\mathrm{top}}_0(\mathrm{MF}(\X(d), \mathrm{Tr}\,W))_{\mathbb{Q}}\arrow[d, "\mathrm{c}"]
        \\
        \widetilde{H}^{\mathrm{BM}}_{0}(\mathscr{P}(d))\arrow[r, "\sim"]& \widetilde{H}^{0}(\X(d), \varphi_{\mathrm{Tr}\,W}[-1])\arrow[r, "\sim"]& \widetilde{H}^{0}(\X(d), \varphi_{\mathrm{Tr}\,W}^{\mathrm{inv}}).
    \end{tikzcd}
\end{equation*}
Here we have suppressed the cohomological degrees to make the diagram simpler. 
\end{cor}

\begin{proof}
    The claim follows from Corollary \ref{prop53} and Corollary \ref{cor616}.
\end{proof}

\begin{thm}\label{thm1plus}
For a quiver $Q^{\circ}$, $d\in \mathbb{N}^I$, $\delta \in M(d)_{\mathbb{R}}^{W_d}$ and $\ell \in \mathbb{Z}$,  
there is an isomorphism 
\begin{align}\label{isom:Tred}
    \mathrm{gr}_\ell K^{\mathrm{top}}_0(\mathbb{T}(d; \delta))_{\mathbb{Q}}\cong \mathrm{gr}_\ell K^{\mathrm{top}}_0(\mathbb{T}(d; \delta)^{\mathrm{red}})_{\mathbb{Q}}
\end{align}
and 
the cycle map for $\mathscr{P}(d)$ induced by (\ref{chpd})
gives 
an isomorphism 
\begin{equation}\label{cyclemappreproj}
    \mathrm{c}\colon \mathrm{gr}_\ell K^{\mathrm{top}}_0(\mathbb{T}(d; \delta))_{\mathbb{Q}}\stackrel{\cong}{\to}
    H^{-2\ell}(P(d), \mathcal{BPS}^p_{d, \delta}).
     \end{equation}
\end{thm}

\begin{proof}%[Proof of Theorem \ref{thm1plus}]
The isomorphism (\ref{isom:Tred}) follows from Proposition \ref{prop71}.
Consider the diagram, whose lower square commutes from Corollary \ref{cor617} and the top horizontal map is an isomorphism by Corollary \ref{cor616}:
\begin{equation*}
    \begin{tikzcd}
    \mathrm{gr}_\ell K^{\mathrm{top}}_0(\mathbb{T}(d; \delta))_{\mathbb{Q}}\arrow[d, hook]\arrow[r, "\sim"]\arrow[dd, bend right=85, "\alpha"']& \mathrm{gr}_{\ell+\dim \mathfrak{g}(d)+1} K^{\mathrm{top}}_0(\mathbb{S}(d; \delta))_{\mathbb{Q}}\arrow[d, hook]\arrow[dd, bend left=85, "\beta"]\\
        \mathrm{gr}_\ell G^{\mathrm{top}}_0(\mathscr{P}(d))_{\mathbb{Q}}\arrow[d, "\mathrm{c}"]\arrow[r, "\sim"]& \mathrm{gr}_{\ell+\dim \mathfrak{g}(d)+1} K^{\mathrm{top}}_0(\mathrm{MF}(\X(d), \mathrm{Tr}\,W))_{\mathbb{Q}}\arrow[d, "\mathrm{c}"]
        \\
        H^{\mathrm{BM}}_{2\ell}(\mathscr{P}(d))\arrow[r, "\sim"]& H^{2\dim \mathscr{Y}(d)-2\ell}(\X(d), \varphi_{\Tr W}[-1]).
    \end{tikzcd}
\end{equation*}
By Theorem \ref{thm2},
the map $\beta$ has image in \[H^{\dim \mathscr{P}(d)-2\ell}(\X(d), \mathcal{BPS}_{d,\delta})\subset H^{2\dim\mathscr{Y}(d)-2\ell}(\mathscr{X}(d), \varphi_{\mathrm{Tr}\,W}[-1]).\] 
Therefore by \eqref{BPSmatched}, the map $\alpha$ has image in 
\begin{align}\label{image:TP}
H^{-2\ell}(P(d), \mathcal{BPS}^p_{d, \delta})
\subset H_{2\ell}^{\mathrm{BM}}(\mathscr{P}(d)).
\end{align}
It is an isomorphism over $\mathbb{Q}$ onto $H^{-2\ell}(P(d), \mathcal{BPS}^p_{d, \delta})$ by Theorem~\ref{thm1}. 
\end{proof}

\begin{remark}
    There are two perverse filtrations on $H^{\mathrm{BM}}_{\ast}(\mathscr{P}(d))$ for any quiver $Q^\circ$. One of them is induced from the tripled quiver with potential $(Q,W)$ and studied in \cite{DM}; the first non-zero piece in the perverse filtration is ${}^p\tau^{\leq 1}\pi_{d*}\varphi_{\mathrm{Tr}\,W}\mathrm{IC}_{\X(d)}=\mathcal{BPS}_d$. 
    Another filtration is induced from the map $\pi_{P,d}$ and studied in \cite{D}, where it is called ``the less perverse filtration"; the first non-zero piece in the perverse filtration is ${}^p\tau^{\leq 0}\pi_{P,d*}\omega_{\mathscr{P}(d)^{\mathrm{cl}}}$.
    Note that, for any $v\in\mathbb{Z}$, 
    ${}^p\tau^{\leq 1}\pi_{d*}\varphi_{\mathrm{Tr}\,W}\mathrm{IC}_{\X(d)}$ is a direct summand of 
    $\mathcal{BPS}_{d,v}$, which itself is a direct summand of  ${}^p\tau^{\leq 0}\pi_{P,d*}\omega_{\mathscr{P}(d)^{\mathrm{cl}}}$. Thus the topological K-theory of quasi-BPS categories lies between the first non-zero pieces of these two perverse filtrations. 
\end{remark}

\begin{remark}
     Davison--Hennecart--Schlegel Mejia computed in \cite{DHSM, DHSM2} the preprojective BPS sheaves in terms of the intersection complexes of the varieties $P(d)$. 
\end{remark}

We note the following numerical corollary of Theorem \ref{thm1plus}.

\begin{cor}\label{corollarythm1plus}
    Let $Q^\circ=(I, E^{\circ})$ be a quiver, let $d \in \mathbb{N}^{I}$, and let $\delta\in M(d)_\mathbb{R}^{W_d}$. 
    There is an equality 
    \[\dim_\mathbb{Q} K^{\mathrm{top}}_0(\mathbb{T}(d; \delta))_{\mathbb{Q}}=\dim_\mathbb{Q} H^{\ast}(P(d), \mathcal{BPS}^p_{d, \delta}).\]
\end{cor}

\begin{proof}
The map \eqref{chtdv} is injective by Corollary~\ref{cor615bis}.
%The map \eqref{chtdv} is injective by Proposition \ref{cherninj}, Theorem \ref{theorem25}, and the Koszul equivalence \eqref{Koszul}. 
The conclusion then follows from Theorem \ref{thm1plus}. 
\end{proof}

\section{Examples}\label{s8}

In this section, we discuss some explicit examples of computations of the topological K-theory of quasi-BPS categories. All vector spaces considered in this section are $\mathbb{Q}$-vector spaces.
We first note a preliminary proposition.

\begin{prop}\label{prop81}
    Let $Q=(I,E)$ be a symmetric quiver, let $d\in \mathbb{N}^I$, and let $v\in\mathbb{Z}$. Then
    \[\dim K^{\mathrm{top}}_0(\mathbb{M}(d)_v)_{\mathbb{Q}}=\# \left(M(d)^+\cap (\mathbf{W}(d)+v\tau_d-\rho)\right).\]
\end{prop}

\begin{proof}
Since $\X(d) \to BGL(d)$ is an affine fibration, 
there is a natural isomorphism \[K_0(\X(d))\xrightarrow{\sim} K_0^{\mathrm{top}}(\X(d))\cong K_0(BG(d)).\] The category $\mathbb{M}(d)_v$ is admissible in $D^b(\X(d))$, so the above isomorphism restricts 
    to the isomorphism \[K_0(\mathbb{M}(d)_v)\xrightarrow{\sim} K_0^{\mathrm{top}}(\mathbb{M}(d)_v).\] The basis of $K_0(\X(d))$ is given by
    the classes of the vector bundles $\mathcal{O}_{\X(d)}\otimes \Gamma_{G(d)}(\chi)$, where $\chi$ is a dominant weight of $T(d)$ and $\Gamma_{G(d)}(\chi)$ is the irreducible representation of $G(d)$ of highest weight $\chi$. The computation \[\dim K_0(\mathbb{M}(d)_v)_{\mathbb{Q}}=\# \left(M(d)^+\cap (\textbf{W}(d)+v\tau_d-\rho)\right)\] follows then from the definition of $\mathbb{M}(d)_v$.
\end{proof}

\begin{remark}\label{remarksub8}
  In view of Proposition \ref{prop81} and  Theorem \ref{thmWzero}, the total intersection cohomology of the spaces $X(d)$ can be determined by counting lattice points inside the polytope $M(d)^*\cap(\textbf{W}(d)+v\tau_d-\rho)$. 
\end{remark}

\subsection{Toric examples}

Let $g\in \mathbb{N}$.
Consider the quiver $Q=(I,E)$, where $I=\{1,2\}$ and $E$ has one loop at $1$, one loop at $2$, $2g+1$ edges \{$e_1,\ldots, e_{2g+1}\}$ from $1$ to $2$ and $2g+1$ edges $\{\overline{e}_1,\ldots, \overline{e}_{2g+1}\}$ from $2$ to $1$. The following is a figure for $g=1$.
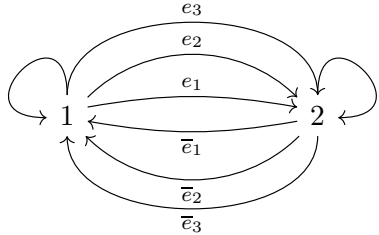
\begin{equation*}
    \begin{tikzcd}
 1 \arrow[out=90, in=180, loop, swap] \arrow[rrr, bend left=45, "e_2"]
 \arrow[rrr, bend left=90, "e_3"]
\arrow[rrr, shift left=1, bend left=10, "e_1"] &&&2 \arrow[out=90, in=0, loop, swap] \arrow[lll, bend left=10, "\overline{e}_1"] \arrow[lll, bend left=45, "\overline{e}_2"] \arrow[lll, bend left=90, "\overline{e}_3"]
 \end{tikzcd}
\end{equation*}
Fix the dimension vector $d=(1,1)\in \mathbb{N}^I$. 
Then \[\X(d)=\left(\mathbb{C}^2\oplus \mathbb{C}^{2(2g+1)}\right)\big/(\mathbb{C}^*)^2.\]
The diagonal $\mathbb{C}^*\hookrightarrow (\mathbb{C}^*)^2$ acts trivially on $\mathbb{C}^2\oplus \mathbb{C}^{2(2g+1)}$. The factor $\mathbb{C}^*$ corresponding to the vertex $1$ acts with weight $0$ on $\mathbb{C}^2$, weight $1$ on $\mathbb{C}^{2g+1}$, and weight $-1$ on $\mathbb{C}^{2g+1}$. We consider the stack, which is the $\mathbb{C}^*$-rigidification of $\X(d)$:  
\[\X'(d)=\left(\mathbb{C}^2_0\oplus \mathbb{C}^{2g+1}_1\oplus \mathbb{C}^{2g+1}_{-1}\right)\big/\mathbb{C}^*.\]
The GIT quotient for any non-trivial stability condition provides a small resolution of singularities:
\[\tau\colon Y:=\left(\mathbb{C}^2_0\oplus \mathbb{C}^{2g+1}_1\oplus \mathbb{C}^{2g+1}_{-1}\right)^{\mathrm{ss}}\big/\mathbb{C}^*=\mathbb{C}^2\times\mathrm{Tot}_{\mathbb{P}^{2g}}\left(\mathcal{O}(-1)^{2g+1}\right)\to X(d).\]
Here, \textit{small} means that $\dim Y\times_{X(d)} Y=\dim X(d)$ and $Y\times_{X(d)} Y$ has a unique irreducible component of maximal dimension.
Then, by the BBDG decomposition theorem, we have that $\tau_*\mathrm{IC}_{Y}=\mathrm{IC}_{X(d)}$.
We decorate the BPS sheaves with a superscript zero to indicate that the potential is zero.
We obtain that:
\begin{equation}\label{computations}
    \mathcal{BPS}^0_d=\tau_*\mathrm{IC}_{Y}=\mathrm{IC}_{X(d)}\text{ and }\mathcal{BPS}^0_{(1,0)}=\mathcal{BPS}^0_{(0,1)}=\mathrm{IC}_{\mathbb{C}}.
\end{equation}

\begin{prop}\label{prop82}
(i) If $v$ is odd, then 
we have an equivalence \begin{align*}\mathbb{M}(d)_v \stackrel{\sim}{\to} D^b(Y)
\end{align*}
and $\mathcal{BPS}^0_{d,v}=\mathcal{BPS}^0_d$.

(ii) If $v$ is even, then we have the semiorthogonal decomposition 
\begin{align*}\mathbb{M}(d)_v=\langle D^b(Y), D^b(\mathbb{C}^2) \rangle, 
\end{align*}
and the direct sum decomposition 
\begin{align*}
\ \mathcal{BPS}^0_{d,v}=\mathcal{BPS}^0_d\oplus \mathcal{BPS}^0_{(1,0)}\boxtimes \mathcal{BPS}^0_{(0,1)}.
\end{align*}
\end{prop}

\begin{proof}
    The category $\mathbb{M}(d)_v \subset D^b(\X(d))$ is generated by the line bundles $\mathcal{O}_{\X(d)}(w\beta_2+(v-w)\beta_1)$ for $w\in \mathbb{Z}$ such that 
    \begin{equation}\label{boundsprop83}
    \frac{v}{2}\leq w\leq 2g+1+\frac{v}{2}.
    \end{equation}
    One can show that $\mathbb{M}(d)_v$ is equivalent to the ``window subcategory" (in the sense of \cite{halp}) of $D^b(\X'(d))$ containing objects $F$ such that 
    \begin{align*}
        \mathrm{wt}(F|_{0}) \subset \left[\frac{v}{2}, \frac{v}{2}+2g+1  \right] \cap \mathbb{Z}. 
    \end{align*}
    
    If $v$ is odd, then $\mathbb{M}(d)_v\stackrel{\sim}{\to} D^b(Y)$ by \cite[Theorem 2.10]{halp}. The boundary points 
    $v/2$ and $v/2+2g+1$ are not integers, so $\mathcal{BPS}^0_{d,v}=\mathcal{BPS}^0_d$.

    If $v$ is even, then 
    \begin{align*}\mathcal{BPS}^0_{d,v}=\mathcal{BPS}^0_d\oplus \mathcal{BPS}^0_{(1,0)}\boxtimes \mathcal{BPS}^0_{(0,1)}.
    \end{align*}
    The fixed locus of the unique Kempf-Ness locus in the construction of $Y$ is $(\mathbb{C}^2_0\oplus \mathbb{C}^{2g+1}_1\oplus \mathbb{C}^{2g+1}_{-1})^{\mathbb{C}^*}=\mathbb{C}^2$.
    As a corollary of \cite[Theorem 2.10]{halp}, see the remark in \cite[Equation (3)]{MR3552550}, the category $\mathbb{M}(d)_v$ has a semiorthogonal decomposition with summands $D^b(Y)$ and $D^b(\mathbb{C}^2)$.
\end{proof}

As a corollary of the above proposition and of the computations \eqref{computations}, we obtain the following equality: 
\[\dim K^{\mathrm{top}}_0(\mathbb{M}(d)_v)\stackrel{(*)}{=}\dim H^{\ast}(X(d), \mathcal{BPS}^0_{d,v})=\begin{cases} 2g+1,\text{ if }v\text{ is odd},\\
2g+2,\text{ if }v\text{ is even}.
\end{cases}\]
The equality $(*)$ is also the consequence \eqref{corollarytheorem61} of Theorem \ref{thm1}.
Note that the dimensions of $K^{\mathrm{top}}(\mathbb{M}(d)_v)$ can be computed immediately using Proposition \ref{prop81} and \eqref{boundsprop83}, and then $(*)$ can be checked without using window categories.
However, by Proposition \ref{prop82}, the equality $(*)$ is obtained as a corollary of the Atiyah-Hirzebruch theorem for the smooth varieties $Y$ and $\mathbb{C}^2$. 

Further, Proposition \ref{prop82} is useful when considering a non-zero potential for $Q$. 
For example, consider the potential \[W:=\sum_{i=1}^{2g+1}e_i\overline{e}_i.\]
Note that $\Tr W\colon Y\to \mathbb{C}$ is smooth. The computation \eqref{computations} implies that:
\begin{equation}\label{computations2}
    \mathcal{BPS}_d=\tau_*\varphi_{\Tr W}\mathrm{IC}_{Y}[-1]=0, \  \mathcal{BPS}^0_{(1,0)}=\mathcal{BPS}^0_{(0,1)}=\mathrm{IC}_{\mathbb{C}}.
\end{equation}
In this case, the monodromy actions of
BPS sheaves are trivial.
Further, Proposition \ref{prop82} implies that:
\begin{align*}
    \mathbb{S}(d)_v&\simeq \mathrm{MF}(Y, \Tr W)=0 \text{ if }v\text{ is odd},\\ \mathbb{S}(d)_v&=\langle \mathrm{MF}(Y, \Tr W), \mathrm{MF}(\mathbb{C}^2, 0)\rangle\simeq \mathrm{MF}(\mathbb{C}^2, 0)\text{ if }v\text{ is even}.
\end{align*}
Let $i\in \mathbb{Z}$.
The following equality (which also follows by \eqref{corollarytheorem61}) holds by a direct computation:
\[\dim K^{\mathrm{top}}_i(\mathbb{S}(d)_v)_{\mathbb{Q}}=\dim H^{\ast}(X(d), \mathcal{BPS}_{d,v})=\begin{cases} 0,\text{ if }v\text{ is odd},\\
1,\text{ if }v\text{ is even}.
\end{cases}\]

\begin{remark}
  A similar analysis can be done for any symmetric quiver $Q=(I,E)$ and a dimension vector $d=(d^i)_{i\in I}\in \mathbb{N}^I$ such that $d^i\in\{0,1\}$ for every $i\in I$. We do not give the details for the proofs. Let $v\in\mathbb{Z}$ such that $\gcd(\dd, v)=1$. Assume $W=0$.
  
  One can show that, for a generic GIT stability $\ell\in M(d)^{W_d}_\mathbb{R}\cong M(d)_\mathbb{R}$, the GIT quotient $Y:=R(d)^{\ell\text{-ss}}/ G(d)$ 
  gives a small resolution
  $\tau \colon Y \to X(d)$. 
  Then $\mathcal{BPS}^0_d=\tau_*\mathrm{IC}_{Y}.$
  By \cite{hls}, there is an equivalence:
  $\mathbb{M}(d)_v\stackrel{\sim}{\to} D^b(Y)$. 
  The following equality (which is a corollary of Theorem \ref{thm1}) follows then by the Atiyah-Hirzebruch theorem for the smooth variety $Y$:
  \[\dim K^{\mathrm{top}}_0(\mathbb{M}(d)_v)_{\mathbb{Q}}= \dim K^{\mathrm{top}}_0(Y)_{\mathbb{Q}}=\dim H^{\ast}(Y)=\dim H^{\ast}(X(d), \mathcal{BPS}^0_d).\]
  Similar computations can be done also for a general $v\in \mathbb{Z}$.
\end{remark}

\subsection{Quivers with one vertex and an odd number of loops}
For $g \in \mathbb{N}$, let 
$Q$ be the quiver with one vertex and $2g+1$ loops. The following is a picture for $g=1$.
\begin{equation}\label{quiverthreeloops}
    \begin{tikzcd}
\bullet \arrow[out=0,in=90,loop,swap]
  \arrow[out=120,in=210,loop,swap]
  \arrow[out=240,in=330,loop,swap]
\end{tikzcd}
\end{equation}
For $d\in \mathbb{N}$, recall the good moduli space map:
\[\X(d):=\mathfrak{gl}(d)^{\oplus (2g+1)}/GL(d)\to X(d):=\mathfrak{gl}(d)^{\oplus (2g+1)}\ssslash GL(d).\]
%Note that $X(d)$ is a variety of matrix invariants. 
For $g>0$, the variety $X(d)$ is singular.
For every stability condition $\ell\in M(d)^{W_d}_\mathbb{R}$, we have that $\X(d)^{\ell\text{-ss}}=\X(d)$, so we do not obtain resolutions of singularities of $X(d)$ as in the previous example. 
There are no known crepant geometric resolutions (in particular, small resolutions) of $X(d)$. For $\gcd(d,v)=1$, \v{S}penko--Van den Bergh \cite{SVdB} proved that $\mathbb{M}(d)_v$ is a twisted noncommutative crepant resolution of $X(d)$. In view of Theorem \ref{thmWzero}, we regard $\mathbb{M}(d)_v$ as a categorical analogue of a small resolution of $X(d)$.

Reineke \cite{MR2889742} and Meinhardt--Reineke \cite{MeRe} provided multiple combinatorial formulas for the dimensions of the intersection cohomology groups $I\!H^i(X(d))$. 
As noted in Remark \ref{remarksub8}, 
the result of Theorem \ref{thmWzero} also provides combinatorial formulas for the total intersection cohomology of $X(d)$. We explain that our formula recovers a formula already appearing in the work of Reineke \cite[Theorem 7.1]{MR2889742}. 

Fix $v\in \mathbb{Z}$.
By Proposition \ref{prop81}, we need to determine the number of (integral, dominant) weights $\chi=\sum_{i=1}^d c_i\beta_i\in M(d)^+$ with $\sum_{i=1}^d c_i=v$ and $c_i\geq c_{i-1}$ for every $2\leq i\leq d$, such that
\begin{equation}\label{polytopeg}
\chi+\rho-v\tau_d\in \frac{2g+1}{2}\mathrm{sum}[0, \beta_i-\beta_j],
\end{equation}
where the Minkowski sum is taken over all $1\leq i,j\leq d$. 
We define $\widetilde{\chi}\in M(d)$ and $\widetilde{c}_i\in \mathbb{Z}$ for $1\leq i\leq d$ as follows:
\[\widetilde{\chi}:=\chi-g\cdot(2\rho)=\sum_{i=1}^d \widetilde{c}_i\beta_i.\]
Note that, for every $2\leq i\leq d$, the inequality $c_i\geq c_{i-1}$ becomes:
\begin{equation}\label{dominanttilde}
    \widetilde{c}_i-\widetilde{c}_{i-1}+2g\geq 0.
\end{equation}

A dominant weight $\chi$ satisfies \eqref{polytopeg} if and only if, for all dominant cocharacters $\lambda$ of $T(d)\subset GL(d)$, we have:
\begin{equation}\label{bounds}
\langle \lambda, \chi+\rho-v\tau_d\rangle\leq \frac{2g+1}{2}\langle \lambda, \mathfrak{g}^{\lambda>0}\rangle=(2g+1)\langle \lambda, \rho\rangle.
\end{equation}

\begin{prop}\label{propboundslambdal}
    The inequalities \eqref{bounds} hold for all dominant cocharacters $\lambda$ if and only if they hold for the cocharacters \begin{align*}\lambda_k(z)=(\overbrace{1\ldots, 1}^{d-k},\overbrace{z,\ldots, z}^k)\in T(d)
    \end{align*}
    for $1\leq k\leq d$.
\end{prop}

\begin{proof}
    In the cocharacter lattice, any dominant cocharacter $\lambda$ is a linear combination with nonnegative coefficients of $\lambda_k$ for $1\leq k\leq d$. Then, if \eqref{bounds} holds for all $\lambda_k$, it also holds for all dominant $\lambda$.
\end{proof}

% We do not give the details, but 
% one can show that \eqref{polytopeg} holds if and only if, for every $1\leq k\leq d$, we have that
% \[\langle \lambda_k, \chi+\rho-v\tau_d\rangle \leq \langle \lambda_k, (2g+1)\rho\rangle,\] where $\lambda_k(z)=(\overbrace{z,\ldots, z}^k,\overbrace{1\ldots, 1}^{d-k})\in T(d)$. 
The conditions \eqref{bounds} for $\lambda_k$ 
is equivalent to 
that $\langle \lambda_k, \widetilde{\chi}\rangle\leq \langle \lambda_k, v\tau_d\rangle$, or 
equivalently 
\begin{equation}\label{average}
    \sum_{i=d-k+1}^d\widetilde{c}_i\leq \frac{vk}{d}.
\end{equation}

\begin{defn}\label{def86}
Let $\mathscr{H}^{2g+1}_{d,v}$ be the set of tuplets of integers $(\widetilde{c}_i)_{i=1}^d\in \mathbb{Z}^d$ satisfying the inequality \eqref{dominanttilde} and \eqref{average} for every $2\leq k\leq d$ and such that $\sum_{i=1}^d \widetilde{c}_i=v$. 
Let $H^{2g+1}_{d,v}:=\# \mathscr{H}^{2g+1}_{d,v}$.
\end{defn}

\begin{remark}
The numbers $H^{2g+1}_{d,0}$ appear in combinatorics as ``score sequences of complete tournaments", and in the study of certain $\mathbb{C}^*$-fixed points in the moduli of $SL(n)$-Higgs bundles, see \cite[Section 7]{MR2889742}. By Proposition \ref{prop81}, we have that:
\[\dim K^{\mathrm{top}}_0(\mathbb{M}(d)_v)_{\mathbb{Q}}=H^{2g+1}_{d,v}.\]
By Theorem \ref{thmWzero}, for any $v\in\mathbb{Z}$ such that $\gcd(d,v)=1$, we obtain that:
\begin{equation}\label{hgdv}
\dim I\!H^{\ast}(X(d))=H^{2g+1}_{d,v}.
\end{equation}
The above statement was already proved (by different methods) by Reineke and Meinhardt--Reineke by combining \cite[Theorem 7.1]{MR2889742} and \cite[Theorem 4.6]{MeRe}, see also \cite[Section 4.3]{MeRe}. Note that we assume that the number of loops is odd in order to apply Theorem \ref{thm1}. In loc. cit., Reineke also provided combinatorial formulas for $m$-loop quivers for $m$ even.
\end{remark}

\begin{remark}
    Note that, as a corollary of \eqref{hgdv}, we obtain that $H^{2g+1}_{d,v}=H^{2g+1}_{d,v'}$ if $\gcd(d,v)=\gcd(d,v')=1$. There are natural bijections $\mathscr{H}^{2g+1}_{d,v}\xrightarrow{\sim} \mathscr{H}^{2g+1}_{d,v'}$ for $d|v-v'$ or for $v'=-v$, but we do not know such natural bijections for general $v,v'$ coprime with $d$.
\end{remark}

For $\gcd(d, v)=1$ and $n\geq 1$, the topological 
K-theory 
$K_i^{\rm{top}}(\mathbb{M}(nd)_{nv})$
is computed from the intersection cohomology of $X(e)$ for $e\in\mathbb{N}$, and the set $S_{nv}^{nd}$.
The following is a corollary of Proposition \ref{computationsetsdv}: 
\begin{cor}\label{setSnd}
For $\gcd(d, v)=1$ and $n\geq 1$,
the set $S_{nv}^{nd}$ consists of partitions $(d_i)_{i=1}^k$ of
$nd$ such that $d_i=n_i d$
for $(n_i)_{i=1}^k\in\mathbb{N}^k$ a partition of $n$.
\end{cor}
% \begin{proof}
% Consider a partition $(d_i)_{i=1}^k$ of $nd$. It
% is an element of $S_{nv}^{nd}$ if and only if, 
% for any cocharacter $\lambda$ with associated partition 
% $(d_i)_{i=1}^k$, 
% we have 
% $n_{\lambda}/2 +\langle \lambda, nv\tau_{nd} \rangle \in \mathbb{Z}$. 
% Since $n_{\lambda}/2=g \cdot \langle \lambda, \mathfrak{g}(d)^{\lambda>0} \rangle \in \mathbb{Z}$, 
% this is equivalent to that $\langle \lambda, nv \tau_{nd} \rangle \in \mathbb{Z}$
% for any $\lambda$ as above. 
% We write $\lambda$ as 
% \begin{align*}
%     \lambda(t)=(t^{m_1}, \ldots, t^{m_1}, t^{m_2}, \ldots, t^{m_2}, \ldots, t^{m_k})
% \end{align*}
% where $m_i$ appears $d_i$-times, and $m_i \neq m_j$ for $1\leq i\neq j\leq k$. 
% Then the condition $\langle \lambda, nv \tau_{nd} \rangle \in \mathbb{Z}$
% is equivalent to that 
% \[v/d \cdot \sum_{i=1}^k m_i d_i \in \mathbb{Z}\] for all tuples of pairwise distinct integers $(m_i)_{i=1}^k\in\mathbb{Z}^k$, 
% which is equivalent to $d|d_i$ for all $1\leq i\leq k$. 
% \end{proof}

\begin{example}
Suppose that $g=0$. In this case, the 
variety $X(d)$ is smooth:
\[X(d)=\mathfrak{gl}(d)\ssslash GL(d) \stackrel{\sim}{\to} \mathrm{Sym}^d(\mathbb{C}) \cong\mathbb{C}^d.\] 
The above isomorphism is given by sending an element of $\mathfrak{gl}(d)$ to the
multiset of generalized eigenvalues. 
However $X(d)^{\mathrm{st}}=\emptyset$ if $d>1$, 
thus $\mathcal{BPS}_d=\mathrm{IC}_{\mathbb{C}}$ if 
$d=1$, and $\mathcal{BPS}_d=0$ for $d>1$.  
Then by Corollary~\ref{setSnd}, we have 
$\mathcal{BPS}_{d,v}=0$ unless $d|v$, in which case $\mathcal{BPS}_{d,v}=\mathrm{Sym}^d(\mathcal{BPS}_1)=\mathrm{IC}_{X(d)}$.
Thus for $g=0$, we have 
\begin{align}\label{com:g=0}
\dim H^{\ast}(X(d), \mathcal{BPS}_{d,v})=\begin{cases}
    1,\text{ if }d|v,\\ 0, \text{ otherwise.}
\end{cases}
\end{align}

On the other hand, by \cite[Lemma 3.2]{Toquot2} we have that $\mathbb{M}(d)_v=0$ unless $d|v$, in which case it is the subcategory of $D^b(\X(d))$ generated by $\mathcal{O}_{\X(d)}(v\tau_d)$, and thus equivalent to $D^b(X(d))$, see \cite[Lemma 3.3]{Toquot2}. Then:
\begin{align}\label{com2:g=0}\dim K^{\mathrm{top}}_0(\mathbb{M}(d)_v)_{\mathbb{Q}}=\begin{cases}
    1,\text{ if }d|v,\\ 0, \text{ otherwise.}
\end{cases}
\end{align}
For $g=0$, we can thus verify \eqref{corollarytheorem61} by the direct computations (\ref{com:g=0}), (\ref{com2:g=0}). 
\end{example}
% \begin{proof} 
% The dimensional reduction isomorphism \eqref{dimred} maps the cohomology of $\mathcal{BPS}^p_d$ isomorphically onto the cohomology of $\mathcal{BPS}_d$.
%     The claim follows from Theorem \ref{thm1} and the compatibility of the Koszul equivalence and the dimensional reduction in cohomology from Proposition \eqref{prop53}. 
% \end{proof}

\subsection{The three loop quiver}

In this subsection, we make explicit the corollary of Theorem \ref{thm1} for the three loop quiver $Q$ as in \eqref{quiverthreeloops} with loops $\{x,y,z\}$ and with potential \begin{align*}W=x[y,z]=xyz-xzy.\end{align*}
Note that $(Q, W)$ is the tripled quiver 
with potential of the quiver $Q^{\circ}$ with one 
vertex and one loop. 
The quasi-BPS categories $\mathbb{S}(d)_v$ 
in this case are studied in~\cite{PTzero}
as a categorification of BPS invariants 
for $\mathbb{C}^3$, 
and they are admissible 
subcategories 
of the DT category $\mathcal{DT}(d)$
for $\mathbb{C}^3$ studied in \cite{PTzero}. The quasi-BPS categories $\mathbb{T}(d)_v$
for the double of $Q^{\circ}$ are admissible
subcategories 
in $D^b(\mathscr{C}(d))$, where $\mathscr{C}(d)$ is the 
derived moduli stack of zero dimensional 
sheaves on $\mathbb{C}^2$ with length $d$. 
In~\cite[Theorem~4.13]{PTzero}, we 
computed torus-equivariant algebraic K-theory of 
$\mathbb{T}(d)_v$. 
The following is an analogue of~\cite[Theorem~4.13]{PTzero}
for topological K-theory: 

\begin{prop}\label{prop89}
Let $(d,v)\in \mathbb{N}\times\mathbb{Z}$ be coprime, let $n\in\mathbb{N}$ and $i\in \mathbb{Z}$. 
Then we have the following identities: 
\[\dim K^{\mathrm{top}}_i(\mathbb{S}(nd)_{nv})_{\mathbb{Q}}=\dim K^{\mathrm{top}}_0(\mathbb{T}(nd)_{nv})_{\mathbb{Q}}=p_2(n),\]
where $p_2(n)$ is the number of partitions of $n$.
\end{prop}
\begin{proof}%[Proof of Proposition \ref{prop89}]
    % Let $S_n$ be the set of partitions of $n$. By Proposition \ref{prop89}, there is a bijection $S_n\xrightarrow{\sim}S^{nv}_{nd}$. 
    It is proved in~\cite[Theorem 5.1]{Dav}
    that  
    $\mathcal{BPS}_e$
    is isomorphic to $\Delta_{\ast}\mathrm{IC}_{\mathbb{C}^3}$ for every $e\in\mathbb{N}$, where $\Delta \colon \mathbb{C}^3\hookrightarrow X(e)$ is the subvariety parameterizing three diagonal matrices. Then $\dim H^{\ast}(X(e), \mathcal{BPS}_e)=1$, and so 
    $\mathrm{Sym}^k \left(H^{\ast}(X(e), \mathcal{BPS}_e)\right)$
    is one dimensional 
    for every positive integers $e, k$. 
    Then $H^{\ast}(X(nd), \mathcal{BPS}_A)$ is also one dimensional for every $A\in S^{nd}_{nv}$. Note that $\# S^{nd}_{nv}=p_2(n)$ by Corollary \ref{setSnd}.
    Then 
    \[H^{\ast}(X(nd), \mathcal{BPS}_{nv, nd})=\bigoplus_{A\in S^{nv}_{nd}}H^{\ast}(X(nd), \mathcal{BPS}_A)=\mathbb{Q}^{\oplus p_2(n)}.\]
    The monodromy is trivial on $H^{\ast}(X(nd), \mathcal{BPS}_{nd, nv})$. By Theorems \ref{thm1} and \ref{thm1plus}, we obtain
    the desired computations.
\end{proof}

\begin{remark}
    By Theorem \ref{thm1}, the topological K-theory of quasi-BPS categories may be determined whenever one can compute the BPS cohomology and the set $S^d_v$. Proposition \ref{prop89} is an example of such a computation. 
    We mention two other computations for the three loop quiver with potentials $W':=x[y,z]+z^a$ (for $a\geq 2$) and $W'':=x[y,z]+yz^2$. 
    
    Let $\mathcal{BPS}'_d$ and $\mathbb{S}'(d)_v$ be the BPS sheaves and the quasi-BPS categories of $(Q,W')$. Denote similarly the BPS sheaves and the quasi-BPS categories of $(Q,W'')$.
    By \cite[Theorem 1.5]{DP}, we have that $H^{\ast}(X(d), \mathcal{BPS}'_d)^{\mathrm{inv}}=0$ because $H^{\ast}(\mathbb{C}, \varphi_{t^a})^{\mathrm{inv}}=0$. Then Theorem \ref{thm1} implies that, for every $i, v\in \mathbb{Z}$ with $\gcd(d,v)=1$
    we have $K^{\mathrm{top}}_i(\mathbb{S}'(d)_v)_{\mathbb{Q}}=0$. 
  
    By \cite[Corollary 7.2]{DP}, we have that $H^{\ast}(X(d), \mathcal{BPS}''_d)^{\mathrm{inv}}=H^{\ast}(X(d), \mathcal{BPS}''_d)$ is one dimensional. As in Proposition \ref{prop89}, we have that, for every $i,v\in\mathbb{Z}$, 
    the dimension of $K^{\mathrm{top}}_i(\mathbb{S}''(d)_v)_{\mathbb{Q}}$
    equals $p_2(\gcd(d,v))$.  
\end{remark}

\section{\'Etale covers of preprojective algebras}\label{subsec67}

In this section, we prove an extension of Theorem \ref{thm1plus} to étale covers of preprojective stacks which we use to compute the topological K-theory of quasi-BPS categories of K3 surfaces in \cite{PTK3}.

We will use the notations and constructions from Subsection \ref{subsec71}.
Throughout this section, we fix a quiver $Q^\circ=(I, E^{\circ})$ and a dimension vector $d\in\mathbb{N}^I$.
We begin by discussing the setting and by stating Theorem \ref{prop614}, the main result of this section.

\subsection{Preliminaries}\label{subsec91}
Let $E$ be an affine variety with an action of $G:=G(d)$ and with a $G$-equivariant étale map \[e\colon E\to R^\circ(d)\oplus R^\circ(d)^\vee.\]
We assume that the induced map 
\begin{align*}
    e \colon U:=E\ssslash G \to Y(d)
\end{align*}
is \'etale. In this subsection, we use the notation 
$(-)_U$ by the pull-back for the above map. So from the diagram 
\begin{align*}
    \xymatrix{
\mathscr{P}(d)^{\rm{cl}} \inclusion \ar[rrd] & \mathscr{P}(d)^{\rm{red}} \inclusion & 
\mathscr{P}(d) \ar[d] \inclusion & 
\mathscr{Y}(d) \ar[d] & \mathscr{X}(d) \ar[l]_-{\eta} \ar[d] \ar[rd]^-{f} & \\
& & P(d) \inclusion & Y(d) & \ar[l]_-{\eta} X(d) \ar[r]   & \mathbb{C}
    }
\end{align*}
the pull-back with respect to $e \colon U\to Y(d)$ is denoted by 
\begin{align*}
    \xymatrix{
\mathscr{P}(d)^{\rm{cl}}_U \inclusion \ar[rrd] & \mathscr{P}(d)^{\rm{red}}_U \inclusion  & \mathscr{P}(d)_U \ar[d] \inclusion & 
\mathscr{Y}(d)_U \ar[d] & \mathscr{X}(d)_U \ar[l]_-{\eta_U} \ar[d] \ar[rd]^-{f_U} \\
& &P(d)_U \inclusion & Y(d)_U & \ar[l]_-{\eta_U} X(d)_U  \ar[r] & \mathbb{C}. 
    }
\end{align*}

\begin{defn}
Let $\mathcal{BPS}^p_d\in \mathrm{Perv}(P(d))$ be the preprojective BPS sheaf. 
We define 
\begin{equation}\label{def:BPSLsheaf}
\mathcal{BPS}^p_{d, U}:=e^*(\mathcal{BPS}^p_d)\in \mathrm{Perv}(P(d)_U).
\end{equation}
Here $e \colon P(d)_U\to P(d)$ is the pull-back from $e \colon U \to Y(d)$. 
We also define $\mathcal{BPS}^p_{d, \delta, U}$ for $\delta\in M(d)^{W_d}_\mathbb{R}$ as in \eqref{defpreprojbps}.
\end{defn}

We take $\delta \in M(d)_{\mathbb{R}}^{W_d}$ such that $\langle 1_d, \delta\rangle=v \in \mathbb{Z}$. 
By Theorem \ref{theorem266} and Remark~\ref{rmk:version}, there is a semiorthogonal decomposition:
\begin{equation*}
    D^b\left(\mathscr{P}(d)\right)_v=\big\langle \mathbb{A}(d; \delta), \mathbb{T}(d; \delta)\big\rangle. 
\end{equation*}
The purpose of this section is to prove the following:

\begin{thm}\label{prop614}
Let $\langle 1_d, \delta \rangle =v\in\mathbb{Z}$.
    There exist subcategories 
    \begin{align*}\mathbb{T}_U=\mathbb{T}(d; \delta)_U, 
    \ \mathbb{A}_U=\mathbb{A}(d; \delta)_U \subset D^b(\mathscr{P}(d)_U)_v
    \end{align*}such that:
\begin{enumerate}
    \item there is a semiorthogonal decomposition $D^b(\mathscr{P}(d)_U)_v=\langle \mathbb{A}_U, \mathbb{T}_U\rangle$,
    \item if $U=P(d)$ and $e$ is the identity, then $\mathbb{T}_U=\mathbb{T}(d; \delta)$ and $\mathbb{A}_U=\mathbb{A}(d; \delta)$,
    \item if $h\colon E'\to E$ is a $G$-equivariant étale map inducing
    an étale map
    \begin{align*}e':=e\circ h\colon 
    U':=E'\ssslash G \to Y(d)
    \end{align*}
    and if we consider the categories $\mathbb{A}_{U'}$, 
    $\mathbb{T}_{U'}$ for $E'$, 
    then $h$ induces functors 
    \begin{align*}h^*\colon \mathbb{T}_U\to \mathbb{T}_{U'}, \ 
    h^*\colon \mathbb{A}_{U}\to \mathbb{A}_{U'}, 
    \end{align*}
    \item for any $i,\ell\in\mathbb{Z}$, the cycle map \eqref{cherngraded3} induces isomorphisms \[\mathrm{c}\colon \mathrm{gr}_\ell K^{\mathrm{top}}_i(\mathbb{T}_U)\xrightarrow{\sim} 
         H^{-2\ell-i}(P(d)_U, \mathcal{BPS}^p_{d, \delta, U}).\]
\end{enumerate}
Further, one can also define categories 
\begin{align*}\mathbb{T}^{\mathrm{red}}_U, \mathbb{A}^{\mathrm{red}}_U\subset D^b(\mathscr{P}(d)_U^{\mathrm{red}})_v
\end{align*}
which satisfy the analogous conditions to (1)-(4) above. In particular, the map $l'\colon \mathscr{P}(d)_U^{\mathrm{red}}\to\mathscr{P}(d)_U$ induces an isomorphism  
\begin{align}\label{lprimec}
\mathrm{c}\circ l'_*\colon \mathrm{gr}_\ell K^{\mathrm{top}}_i(\mathbb{T}_U^{\mathrm{red}})_{\mathbb{Q}}&\xrightarrow{\sim}
\mathrm{gr}_\ell K^{\mathrm{top}}_i(\mathbb{T}_U)_{\mathbb{Q}} \\
&\notag \xrightarrow{\sim} 
         H^{-2\ell-i}(P(d)_U, \mathcal{BPS}^p_{d,\delta,U}).
         \end{align}
    % \begin{itemize}
    %     \item there is a semiorthogonal decomposition $D^b(\mathscr{L})_v=\langle \mathbb{A}, \mathbb{T}\rangle$,
    %     \item $e^*\colon \mathbb{T}(d)^{\mathrm{red}}_v\to \mathbb{T}$ and $e^*\colon \mathbb{A}(d)^{\mathrm{red}}_v\to \mathbb{A}$, and
    %     \item for any $i,a\in\mathbb{Z}$, the cycle map \eqref{cherngraded} for $\mathscr{L}$ induces isomorphisms \[\mathrm{gr}_a K^{\mathrm{top}}_i(\mathbb{T})\xrightarrow{\sim} 
    %      H^{-2a-i}(L, \mathcal{BPS}^L_{d,v}).\] 
    % \end{itemize}
    
    % Further, let $h\colon E'\to E$ be an étale map which induces the étale map $e':=eh\colon E'\to R^\circ(d)\oplus R^\circ(d)^\vee$, and consider $\pi_{L'}\colon \mathscr{L'}\to L'$ and the categories $\mathbb{A}(L'), \mathbb{T}(L')\subset D^b(\mathscr{L}')$.
    % Then $h$ induces functors $h^*\colon \mathbb{T}(L)_v\to \mathbb{T}(L')_v$ and $h^*\colon \mathbb{A}(L)_v\to \mathbb{A}(L')_v$. 
\end{thm}

We will only explain the constructions for $\mathscr{P}(d)$, the case of $\mathscr{P}(d)_U^{\mathrm{red}}$ is similar.
%In Subsection \ref{subsec670}, we explain how to ``lift" semiorthogonal decompositions along an \'etale cover, see~\cite{MR2801403} for the case of derived categories of varieties.
In Subsection \ref{subsec671}, we define the categories $\mathbb{T}$ and $\mathbb{A}$ using graded matrix factorizations and the Koszul equivalence. In Subsection \ref{subsec672}, we prove the fourth claim in Theorem \ref{prop614}.

\subsection{Quasi-BPS categories for étale covers}\label{subsec671}

We will use the setting from Subsection \ref{subsec91}. 
There is a Cartesian diagram, where the horizontal maps $e$ are étale maps:
\begin{equation}\label{dia:Fe}
    \begin{tikzcd}
        \mathscr{X}(d)_U\arrow[dr, phantom, "\square"]\arrow[d, "\pi_F"']\arrow[r, "e"]& \X(d)\arrow[d,"\pi_{X,d}"]\\
        X(d)_U\arrow[r, "e"]& X(d).
    \end{tikzcd}
\end{equation}
By Theorem \ref{theorem266}, there is a semiorthogonal decomposition 
\begin{equation}\label{SODBM}
D^b\big(\X(d)\big)_v=\big\langle \mathbb{B}(d; \delta), \mathbb{M}(d; \delta)\big\rangle.
\end{equation}
We define subcategories 
\begin{align*}\mathbb{B}_U=\mathbb{B}(d; \delta)_U, \ 
\mathbb{M}_U=\mathbb{M}(d; \delta)_U \subset D^b(\mathscr{X}(d)_U)_v
\end{align*}to be 
classically generated (see Subsection \ref{basechange}) by $e^{\ast}\mathbb{B}(d; \delta)$, $e^{\ast}\mathbb{M}(d; \delta)$ respectively. 
Then Lemma \ref{propSODetale} implies that:

% Let $\mathbb{M}:=\mathbb{T}(L)_v$ be the smallest idempotent subcategory of $D^b(\mathscr{F})$ containing $e^*(\mathbb{M}(d)_v)$. Define similarly $\mathbb{B}:=\mathbb{B}(L)_v\subset D^b(\mathscr{F})$. 

% The map $e$ induces an étale map
% \[e\colon E\oplus\mathfrak{g}\ssslash G\to X(d).\]
% The category $\mathbb{M}(d; v\tau_d)\subset D^b(\X(d))$ is right admissible, and consider the semiorthogonal decomposition 
% \begin{equation}\label{SODBM}
% D^b(\X(d))=\langle \mathbb{B}(d; v\tau_d), \mathbb{M}(d; v\tau_d)\rangle.
% \end{equation}

\begin{cor}\label{prop615}
There is a semiorthogonal decomposition
\begin{align}\label{sod:U}
D^b(\mathscr{X}(d)_U)_v=\langle \mathbb{B}_U, \mathbb{M}_U\rangle.
\end{align}
\end{cor}

% \begin{remark}\label{remark85}
%     There are also analogous semiorthogonal decompositions in the non-reduced case \[D^b(\mathscr{F})=\langle \mathbb{B}^\dagger, \mathbb{M}^\dagger\rangle\] with $\mathbb{M}^\dagger(P(d))=\mathbb{M}(d)_v$. 
%     Further, for $\delta\in M(d)^{W_d}_\mathbb{R}$, we may define categories $\mathbb{M}(\delta)=\mathbb{M}(L; \delta)$ admissible in $D^b(\mathscr{F}^{\mathrm{red}})_v$ and $\mathbb{M}(\delta)^\dagger$ admissible in $D^b(\mathscr{F})$ such that $\mathbb{M}(P(d); \delta)=\mathbb{M}(d; \delta)^{\mathrm{red}}$ and $\mathbb{M}(P(d); \delta)^\dagger=\mathbb{M}(d; \delta)$. 
% \end{remark}

Consider the category of graded matrix factorizations $\mathrm{MF}^{\mathrm{gr}}(\mathscr{X}(d)_U, f_U)$, where the grading is given by weight two $\mathbb{C}^{\ast}$-action on
fibers of $\X(d)_U \to \mathscr{Y}(d)_U$. 
By the Koszul equivalence, we have that:
\begin{equation}\label{koszuletale}
\kappa_L\colon D^b(\mathscr{P}(d)_U)\xrightarrow{\sim}\mathrm{MF}^{\mathrm{gr}}(\mathscr{X}(d)_U, f_U).
\end{equation}
Define the subcategories of $D^b(\mathscr{P}(d)_U)$:
\[\mathbb{T}_U:=\kappa_L^{-1}\left(\mathrm{MF}^{\mathrm{gr}}(\mathbb{M}_U, f_U)\right), \ \mathbb{A}_U:=\kappa_L^{-1}\left(\mathrm{MF}^{\mathrm{gr}}(\mathbb{B}_U, f_U)\right).\]
By \cite[Proposition 2.5]{PTzero}, we obtain:

\begin{cor}\label{prop6155}
The properties (1), (2), and (3) in the statement of Theorem \ref{prop614} hold for the categories $\mathbb{A}_U$ and $\mathbb{T}_U$ of $D^b(\mathscr{P}(d)_U)$.
\end{cor}

\subsection{Comparison with BPS cohomology}\label{subsec672}

Recall the notation from Subsection \ref{subsec71}.
%Consider the commutative diagram:
%\begin{equation*}
%\begin{tikzcd}[row sep=scriptsize, column sep=scriptsize]
%& & & \mathscr{P}(d)_U^{\mathrm{cl}} \arrow[dl, hook'] \arrow[rr] \arrow[dd] & & \mathscr{P}(d)^{\mathrm{cl}} \arrow[dl, hook'] \arrow[dd] \\ 
%\mathscr{F}\arrow[rr, "\eta"]\arrow[dd, "\pi_F"'] & & E/G \arrow[rr, crossing over] \arrow[dd, "\pi_E"] & & \mathscr{Y}(d) \\
%& & & L \arrow[dl, hook'] \arrow[rr] & & P(d) \arrow[dl, hook'] \\
%F\arrow[rr, "\eta"]& & E\ssslash G \arrow[rr] & & Y(d) \arrow[from=uu, crossing over]\\
%\end{tikzcd}
%\end{equation*}
Recall the BPS sheaf $\mathcal{BPS}_d\in \mathrm{Perv}(X(d))$ for the tripled quiver with potential $(Q,W)$ associated to $Q^\circ$.
We define the following perverse sheaf, 
see the diagram (\ref{dia:Fe}):
\[\mathcal{BPS}_{d, U}=e^*(\mathcal{BPS}_d)\in \mathrm{Perv}(X(d)_U).\] 
For $\delta\in M(d)^{W_d}_\mathbb{R}$, we define 
\[\mathcal{BPS}_{d, \delta, U}\in D^b(\mathrm{Sh}_{\mathbb{Q}}(X(d)_U))\]
to be the direct sum 
as in \eqref{defBPSddelta}.
On the other hand, the monodromy 
on $H^\bullet(\mathscr{X}(d)_U, \varphi_{f_U})$ is trivial, so 
we have the cycle map
\begin{align}\label{cycleF}
\mathrm{c}\colon \mathrm{gr}_a K^{\mathrm{top}}_i\left(\mathrm{MF}(\mathscr{X}(d)_U, f_U)\right)_{\mathbb{Q}} \to &
H^{2\dim\X(d)-i-2a+1}(\mathscr{X}(d)_U, \varphi_{f_U}\mathbb{Q}[-1]) \\
&\notag \oplus H^{2\dim\X(d)-i-2a+2}(\mathscr{X}(d)_U, \varphi_{f_U}\mathbb{Q}[-1]).
\end{align}
For $\delta \in M(d)_{\mathbb{R}}^{W_d}$, we set
\begin{align*}
    \mathbb{S}_U:=\mathrm{MF}(\mathbb{M}(d; \delta)_U, f_U), \ 
    \mathbb{S}_U^{\rm{gr}} :=\mathrm{MF}^{\rm{gr}}(\mathbb{M}(d; \delta)_U, f_U). 
\end{align*}
We now define a cycle map from topological K-theory of quasi-BPS categories to BPS cohomology, which is the analogue of Theorem \ref{thm1}.

\begin{prop}\label{prop617}
For $\delta\in M(d)^{W_d}_\mathbb{R}$, the 
cycle map \eqref{cycleF} induces the morphism 
\begin{align}\label{cyclemapetale0}
\mathrm{c}\colon \mathrm{gr}_a K^{\mathrm{top}}_i\left(\mathbb{S}_U\right)\to &H^{\dim \mathscr{X}(d)-i-2a+1}(X(d)_U, \mathcal{BPS}_{d, \delta, U}) \\
\notag&\oplus 
H^{\dim \mathscr{X}(d)-i-2a+2}(X(d)_U, \mathcal{BPS}_{d, \delta, U}).
\end{align}
\end{prop}

\begin{proof}
    The same argument used in the proof of Theorem \ref{thm2} applies here. The $\lambda$-widths (see \eqref{def:width}) of the category $\mathbb{M}(d; \delta)_U$ are equal to the $\lambda$-widths of the category $\mathbb{M}(d; \delta)$ for all cocharacters $\lambda$. 
    The analogue of Proposition \ref{prop06} then holds for
    $\mathbb{S}_U$. 
    \end{proof}

We next prove the analogue of Theorem \ref{thm1}.

\begin{prop}\label{propthm1plus}
    The cycle map \eqref{cyclemapetale0} is an isomorphism over $\mathbb{Q}$.
\end{prop}
\begin{proof}
    Similarly to the proof of Theorem~\ref{thm1}, it is enough to show that both 
    sides in (\ref{cyclemapetale0}) have the same dimension. 
    By pulling back the isomorphism in Theorem~\ref{thm:Ksheaf2}, and noting 
    that the monodromy for $f_U$ is trivial, we obtain 
    \begin{align*}
        \mathcal{K}_{X(d)_U}(\mathbb{S}_U)_{\mathbb{Q}} \cong 
        \mathcal{BPS}_{d, \delta, U}[\beta^{\pm 1}] \oplus 
        \mathcal{BPS}_{d, \delta, U}[\beta^{\pm 1}][1]. 
    \end{align*}
    Therefore, the dimensions of both sides in (\ref{cyclemapetale0}) are equal. 
\end{proof}

We prove the analogue of Theorem \ref{thm1plus}.

\begin{prop}\label{prop98}
The cycle map induces an isomorphism 
\begin{equation}\label{cyclemapetalee}
\mathrm{c}\colon \mathrm{gr}_a K^{\mathrm{top}}_i\left(\mathbb{S}_U^{\rm{gr}}\right)_{\mathbb{Q}}
\xrightarrow{\sim} H^{\dim \mathscr{X}(d)-i-2a}(X(d)_U, \mathcal{BPS}_{d,\delta,U}).
\end{equation}
Further, there is an isomorphism:
\[\mathrm{gr}_a K^{\mathrm{top}}_i(\mathbb{T}_U)_{\mathbb{Q}}\xrightarrow{\sim} 
         H^{-2a-i}(P(d)_U, \mathcal{BPS}^p_{d,\delta,U})\]
\end{prop}

\begin{proof}
    The isomorphism \eqref{cyclemapetalee} follows from Corollary \ref{prop53}
    and Proposition~\ref{propthm1plus}.  The last isomorphism follows from \eqref{cyclemapetalee} and the compatibility between dimensional reduction and the Koszul equivalence from Corollary \ref{prop53}. 
\end{proof}

\begin{proof}[Proof of Theorem \ref{prop614}]
    The first three properties hold by Corollary \ref{prop6155}. The fourth property follows from Proposition \ref{prop98}.
    The statement for reduced stacks follows similarly. The isomorphism \eqref{lprimec} also follows directly from Proposition \ref{propo36}.
    % the
    % compatibility between the Koszul equivalence and dimensional reduction, see the proof of Theorem \ref{thm1plus}.
\end{proof}

We also note the following analogue of Corollary \ref{cor615bis}.

\begin{prop}
    The following Chern character map 
    is injective over $\mathbb{Q}$ \[\mathrm{ch}\colon K^{\mathrm{top}}_i(\mathbb{T}_U)\hookrightarrow G^{\mathrm{top}}_i(\mathscr{P}(d)_U)\to 
    \widetilde{H}_i^{\rm{BM}}(\mathscr{P}(d)_U).\] 
\end{prop}

\begin{proof}
The proof is analogous to that of Corollary \ref{cor615bis}.
    The claim follows from Proposition \ref{cherninj}, a framed version of 
    the semiorthogonal decomposition (\ref{sod:U})
    \begin{align*}
        D^b(\X^{\alpha f}(d)^{\rm{ss}}_U)=\langle \mathbb{B}_U', \mathbb{M}_U'\rangle
    \end{align*}
    and the Koszul equivalence. 
\end{proof}

\section{Some auxiliary results}\label{sec:aux}
In this section, we prove some postponed results. 

\subsection{Proof of Proposition~\ref{prop226}}\label{subsec:sod2}
\begin{proof}
    Let 
    \begin{align*}
        \X^{\alpha f}(d)=\mathcal{S}_1 \sqcup \cdots \sqcup \mathcal{S}_N \sqcup 
        \X^{\alpha f}(d)^{\rm{ss}}
    \end{align*}
    be the Kempf-Ness stratification with center $\mathcal{Z}_i \subset \mathcal{S}_i$
    and associated cocharacters $\lambda_i \colon \mathbb{C}^{\ast} \to T(d)$, 
    see~\cite[Section~2]{halp} for the above notions. 
    Let $\mathbb{G}^{\alpha}(d; \delta)$
    be the subcategory of $D^b(\X^{\alpha f}(d))$ consisting 
    of objects $\mathcal{E}$ such that $\mathcal{E}|_{\mathcal{Z}_i}$ has 
    $\lambda_i$-weights contained in the interval 
    \begin{align*}
        \left[-\frac{1}{2}n_{\lambda_i}^{\alpha}, \frac{1}{2}n_{\lambda_i}^{\alpha} \right]
        +\langle \lambda_i, \delta \rangle,
    \end{align*}
    where $n_{\lambda}^{\alpha}$ is defined by 
    \begin{align*}
        n_{\lambda}^{\alpha}:=\langle \lambda, (R^{\alpha f}(d)^{\vee})^{\lambda>0}
        -(\mathfrak{g}(d)^{\vee})^{\lambda>0} \rangle, 
    \end{align*} where $R^{\alpha f}(d)=R(d) \oplus V^{\alpha}(d)$ as in (\ref{Rf:framed}). 
    By the window theorem~\cite{halp, MR3895631}, the following 
    composition is an equivalence 
    \begin{align*}
        \mathbb{G}^{\alpha}(d; \delta) \subset D^b(\X^{\alpha f}(d)) \to D^b(\X^{\alpha f}(d)^{\rm{ss}}). 
    \end{align*}
    By~\cite[Lemma~5.1.9]{T}, the cocharacter $\lambda_i$ acts on 
    $V^{\alpha}(d)$ with strictly negative weights. 
    Therefore $n_{\lambda_i}^{\alpha}>n_{\lambda_i}$, so the pull-back 
    along $\pi_{\alpha f} \colon \X^{\alpha f}(d) \to \X(d)$ 
    restricts to the functor 
    \begin{align*}
    \pi_{\alpha f}^{\ast} \colon \mathbb{M}(d; \delta) \to \mathbb{G}^{\alpha}(d; \delta). 
    \end{align*}
    The above functor admits a right adjoint given by 
    \begin{align*}
    \Phi \colon 
        \mathbb{G}^{\alpha}(d; \delta)
        \xrightarrow{p_v \circ \pi_{\alpha f\ast}}
        D^b(\X(d))_v \stackrel{i^R}{\to} \mathbb{M}(d; \delta).
    \end{align*}
    The first functor is the composition of 
    $\pi_{\alpha f\ast}$ with the projection 
    onto the weight $v$-part, and $i^R$ is the 
    projection with respect to the semiorthogonal decomposition (\ref{SODtheorem266Pd}). 
    Then we have $\Phi \circ \pi_{\alpha f}^{\ast}=\id$, thus the functor 
    \begin{align*}
        \pi_{\alpha f}^{\ast} \colon \mathbb{M}(d; \delta) \to D^b(\X^{\alpha f}(d)^{\rm{ss}})
    \end{align*}
    is fully-faithful with right adjoint. 
\end{proof}

%\begin{prop}\label{cherninj2}
 %   Let $\X$ be a smooth quotient stack and let $f\colon \X\to\mathbb{C}$ be a regular function. Assume that there is a semiorthogonal decomposition $\mathrm{MF}(\X, f)=\langle \mathbb{B}_i\mid i\in I\rangle$ and a collection of finite subsets $I_n\subset I$ for $n\in \mathbb{N}$ with the following two properties:
 %   \begin{itemize}
 %       \item for any finite subset $S\subset I$, there exists $n\in\mathbb{N}$ such that $S\subset I_n$,
 %       \item for all $n\in \mathbb{N}$, there exists a smooth variety $Y_n$ and a morphism $r_n\colon Y_n\to \X$ such that 
 %       the subcategory $\mathbb{B}^n:=\langle \mathbb{B}_i\mid i\in I_n\rangle$ 
 %       in $\mathrm{MF}(\X, f)$
 %       is (left or right) admissible in $\mathrm{MF}(Y_n, g_n)$ via $r_n^*$, where $g_n :=f \circ r_n$.  
 %   \end{itemize}
 %   Then
 %   for $i\in \mathbb{Z}$ the Chern character map 
 %   \[\mathrm{ch}\colon K^{\mathrm{top}}_i(\mathrm{MF}(\X, f))\to \widetilde{H}^i(\X, \varphi_f^{\mathrm{inv}}) \]
 %   is injective over $\mathbb{Q}$. 
%\end{prop}
%
%\begin{proof}
%    Let $x\in K^{\mathrm{top}}_i(\mathrm{MF}(\X,f))=\bigoplus_{j\in I}K^{\mathrm{top}}_i(\mathbb{B}_j)$. Let $S\subset I$ be a finite set such that $x\in \bigoplus_{j\in S}K^{\mathrm{top}}_i(\mathbb{B}_j)$. Then there exists $n\in \mathbb{N}$ such that $x\in K_i^{\mathrm{top}}(\mathbb{B}^n)$. The Chern character map 
%    \[\mathrm{ch}\colon K^{\mathrm{top}}_i(\mathbb{B}^n)\to \widetilde{H}^i(\X, \varphi_f^{\mathrm{inv}})\] is injective over $\mathbb{Q}$ by Proposition \ref{cherninj}, and the claim follows.
%\end{proof}

\subsection{Proof of Lemma~\ref{lem:rho}}\label{subsec:lem:rho}
\begin{proof} 
   Let $\mathrm{Emb}$ be the category whose objects consist of quasi-projective varieties $X$ with 
    morphisms $X \leadsto Z$ given by diagrams
    \begin{align}\label{dia:XZ}
X \stackrel{\pi}{\leftarrow} V \stackrel{i}{\hookrightarrow} Z
\end{align}
where $V\to X$ can be factored as a composition of torsors for locally free sheaves, and $V\hookrightarrow Z$ is a closed immersion. The composition is given by pull-back of closed subschemes. 
Then it is proved in~\cite{Thoma1, ThomasonAS} that given $X\leadsto Z_1$ and $X\leadsto Z_2$, there is 
$Z_1 \leadsto \mathbb{A}^n$ and $Z_2 \leadsto \mathbb{A}^n$ such that the two 
compositions $X\leadsto \mathbb{A}^n$ agree. 
We set 
\begin{align*}
    E^1(X)=\mathcal{K}_X^{\rm{top}}(D^b(X)) \in \mathrm{Sh}_{\mathrm{Sp}}(X^{\rm{an}}), \ 
    E^2(X)=\underline{KU}_{X, c}^{\vee}. 
\end{align*}
For a diagram (\ref{dia:XZ}) and for $k=1, 2$ we set 
\begin{align*}
    E^k_{V}(Z):=\mathrm{fib}(E^k(Z) \to E^k(Z \setminus V)) \in \mathrm{Sh}_{\mathrm{Sp}}(V^{\rm{an}})
\end{align*}
and define 
\begin{align*}
    E^k(X\leadsto Z):=\pi_{*}i^{-1}E_{V}^k(Z) \in \mathrm{Sh}_{\mathrm{Sp}}(X^{\rm{an}}).
\end{align*}
Below we omit $i^{-1}$ since $E_V(Z)$ is supported on $V$. 
As in the diagram (\ref{dia:Ktop:KU}), we have an equivalence 
\begin{align*}
    \eta_{X\leadsto Z} \colon E^1(X\leadsto Z) \stackrel{\sim}{\to} E^2(X\leadsto Z).
\end{align*}

By the devissage and Lemma~\ref{lem:Kdecompose}, we have an equivalence 
\begin{align*}\rho_{X\leadsto Z}\colon E^k(X) \stackrel{\sim}{\to} E^k(X\leadsto Z).
\end{align*}
For $Z\leadsto Z'$, we have the commutative diagram 
\begin{align*}
    \xymatrix{
    & & \ar[rrd]^-{\pi''} \ar[rd] V'' \ar[d]\ar[ld] \ar[lld]_-{i''} & & \\
X & \ar[l]^-{\pi} V \inclusion_-{i} & Z & V' \ar[l]^-{\pi'} \inclusion_-{i'} & Z'.
    }
\end{align*}
Here $V''=V\times_Z V'$, and the diagram
$X\leftarrow V'' \hookrightarrow Z'$. 
corresponds to 
the composition of $X\leadsto Z$ and $Z\leadsto Z'$. 
Then applying devissage and Lemma~\ref{lem:Kdecompose}, the above diagram 
induces an equivalence
\begin{align*}
    \rho_{Z\leadsto Z'} \colon \pi_{*}E_{V}^k(Z) \stackrel{\sim}{\to} \pi_{*}''E_{V''}^k(V')
    \stackrel{\sim}{\to} \pi''_{*}E_{V''}^k(Z').
\end{align*}
It is straightforward to check that 
\begin{align*}
    \rho_{X\leadsto Z'}\simeq \rho_{Z\leadsto Z'} \circ \phi_{X\leadsto Z} \colon 
    E^k(X) \stackrel{\sim}{\to} E^k(X\leadsto Z'). 
\end{align*}

Now for an embedding $M\subset A$ for a smooth $A$, we have 
$\rho_{M\subset A}=\rho_{M\leadsto A}$ where $M\leadsto A$ 
is a morphism in $\mathrm{Emb}$ given by $M=M \hookrightarrow A$. 
For an another embedding $M\subset A'$ for smooth $A'$, there are 
$A\leadsto \mathbb{A}^n$ and $A'\leadsto \mathbb{A}^n$ such that 
the compositions $M\leadsto \mathbb{A}^n$ agree. 
Then both of $\rho_{M\subset A}$ and $\rho_{M\subset A'}$ fit into 
a diagram 
\begin{align*}
    \xymatrix{
E^1(X) \ar[r]^-{\sim} \ar[d]_-{\rho_{M\subset A}}^-{\rho_{M\subset A'}} & 
E^1(X\leadsto \mathbb{A}^n) \ar[d]^-{\eta_{X\leadsto \mathbb{A}^n}}_-{\sim} \\
E^2(X) \ar[r]^-{\sim} & E^2(X\leadsto \mathbb{A}^n)
    }
\end{align*}
which is commutative in $\mathrm{Ho}(\mathrm{Sh}_{\mathrm{Sp}}(M^{\rm{an}}))$. 
Therefore $\rho_{M\subset A}$ is unique in $\mathrm{Ho}(\mathrm{Sh}_{\mathrm{Sp}}(M^{\rm{an}}))$.  
\end{proof}

We have used the following lemma: 
\begin{lemma}\label{lem:Kdecompose}
For a torsor $\pi \colon V \to X$ of a locally free sheaf $\mathcal{V}$ on $X$, 
we have a natural equivalence $E^k(X) \stackrel{\sim}{\to} \pi_{*}E^k(V)$.
\end{lemma}
\begin{proof}
    The proof is essentially same as in~\cite[Lemma~2.2, Lemma~2.6]{HLP}. 
    A torsor $V$ corresponds to an exact sequence 
    \begin{align*}
        0\to \mathcal{V} \to \mathcal{V}' \to \mathcal{O}_X \to 0
    \end{align*}
    such that $V=\mathbb{P}(\mathcal{V}') \setminus \mathbb{P}(\mathcal{V})$.
Let $i \colon \mathbb{P}(\mathcal{V}) \hookrightarrow \mathbb{P}(\mathcal{V}')$ be the 
closed immersion and $j \colon V
\subset \mathbb{P}(\mathcal{V}')$ be the open immersion. 
    Let $r=\mathrm{rank}(\mathcal{V})$ and $h\colon \mathbb{P}(\mathcal{V}') \to X$ the projection. 
    We have the $X$-linear semiorthogonal decomposition 
    \begin{align*}
        D^b(\mathbb{P}(\mathcal{V}'))=\langle h^* D^b(X), h^* D^b(X)\otimes \mathcal{O}(1), \ldots, h^* D^b(X) \otimes \mathcal{O}(r) \rangle
    \end{align*}
which gives an equivalence by Lemma~\ref{lem:sod:Ktheory}
\begin{align}\label{decom:V1}
\bigoplus_{i=0}^r \mathcal{K}_X^{\rm{top}}(D^b(X)) \stackrel{\sim}{\to}
\pi_{*}\mathcal{K}^{\rm{top}}_{\mathbb{P}(\mathcal{V}')}(D^b(\mathbb{P}(\mathcal{V}'))).
\end{align}
   Similarly we have an equivalence 
   \begin{align}\label{decom:V2}
\bigoplus_{i=0}^{r-1} \mathcal{K}_X^{\rm{top}}(D^b(X)) \stackrel{\sim}{\to}
h_{*}\mathcal{K}^{\rm{top}}_{\mathbb{P}(\mathcal{V})}(D^b(\mathbb{P}(\mathcal{V}))).
   \end{align}
By the devissage and Lemma~\ref{lem:openimm}, there is a fiber sequence
\begin{align*}
    i_{*}\mathcal{K}^{\rm{top}}_{\mathbb{P}(\mathcal{V})}(D^b(\mathbb{P}(\mathcal{V})))
    \to \mathcal{K}^{\rm{top}}_{\mathbb{P}(\mathcal{V}')}(D^b(\mathbb{P}(\mathcal{V}')))
    \to j_{*}\mathcal{K}^{\rm{top}}_{V}(D^b(V)). 
\end{align*}
By push-forward to $X$, we obtain
\begin{align*}
    h_{*}\mathcal{K}^{\rm{top}}_{\mathbb{P}(\mathcal{V})}(D^b(\mathbb{P}(\mathcal{V})))
    \to h_{*}\mathcal{K}^{\rm{top}}_{\mathbb{P}(\mathcal{V}')}(D^b(\mathbb{P}(\mathcal{V}')))
\to \pi_{*}\mathcal{K}^{\rm{top}}_V(D^b(V)). 
\end{align*}
Under the decompositions (\ref{decom:V1}), (\ref{decom:V2}), the first arrow is a split 
injection so that its cofiber is equivalent to $\mathcal{K}_X^{\rm{top}}(D^b(X))$.
Therefore the lemma for $E^1$ follows. The proof for $E^2$ is similarly proved as in 
the proof of~\cite[Lemma~2.6]{HLP}.
\end{proof}

We have also used the following lemma: 

\begin{lemma}\label{lem:sod:Ktheory}
Let $\mathscr{D}=\langle \mathcal{C}_1, \mathcal{C}_2\rangle$
be a $M$-linear semiorthogonal decomposition. 
Then we have 
\begin{align}\label{splitting}
    \mathcal{K}_M^{\rm{top}}(\mathscr{D}) =\mathcal{K}_M^{\rm{top}}(\mathcal{C}_1) \oplus \mathcal{K}_M^{\rm{top}}(\mathcal{C}_2).
\end{align}  
\end{lemma}
\begin{proof}
    For $i\in\{1,2\}$, let $F_i$ be the composition 
    \begin{align}\notag
        F_i \colon \mathscr{D} \twoheadrightarrow \mathcal{C}_i \hookrightarrow \mathscr{D},
        \end{align}
        where the first functor is the projection onto the semiorthogonal 
        summand and the second functor is the inclusion. 
        Since $F_i$ is $M$-linear, the functors $F_i$ induce 
        the morphisms $F_{i\ast}$ on $\mathcal{K}_M^{\rm{top}}(\mathscr{D})$ giving 
        idempotents for a splitting (\ref{splitting}). 
\end{proof}

\subsection{Action of exterior algebra on the K-theory of matrix factorizations}\label{subsec:exterior}
In this subsection, we prove Proposition \ref{prop:forg}.
We denote by $\mathrm{pt}:=\mathrm{Spec}\,\mathbb{C}$.
The following computation follows as in Lemma~\ref{prop4zero}. 

\begin{lemma}\label{lemmalambda}
As a $\mathbb{Z}/2$-algebra, we have
\begin{align*}
    K_{\ast}^{\mathrm{top}}(\mathrm{MF}(\mathrm{pt}, 0))
    =\Lambda :=\mathbb{Z}[\epsilon]
\end{align*}
where $\deg \epsilon=-1$. 
\end{lemma}
\begin{proof}
    Let $T:=\Spec \left(\Lambda_{\mathbb{C}}\right)$. 
    Since $\mathrm{MF}(\mathrm{pt}, 0)$ is equivalent to $D^b(T)/\mathrm{Perf}(T)$
    by Theorem~\ref{thm:orlov}, 
    there are exact sequences
    \begin{align*}
      \cdots \to
      K_i^{\mathrm{top}}(T) \to G_i^{\mathrm{top}}(T) \to K_i^{\mathrm{top}}(\mathrm{MF}(\mathrm{pt}, 0)) \to K_{i+1}^{\mathrm{top}}(T) \to \cdots.  
    \end{align*}
    As $K_i^{\mathrm{top}}(T)=G_i^{\mathrm{top}}(T)=\mathbb{Z}$, it is enough to show that 
    $K_i^{\mathrm{top}}(T) \to G_i^{\mathrm{top}}(T)$ is zero, which follows as in Lemma~\ref{prop4zero}. 
\end{proof}

Note that, for any regular function on a smooth stack $h\colon \mathscr{Y}\to\mathbb{C}$, the category $\mathrm{MF}(\mathscr{Y}, h)$ is a module over $\mathrm{MF}(\mathrm{pt}, 0)$, 
  so $K_{\ast}^{\mathrm{top}}(\mathrm{MF}(\mathscr{Y}, h))$
  is a $\mathbb{Z}/2$-graded 
  $\Lambda$-module by Lemma \ref{lemmalambda}. 

\begin{prop}\label{Lambdastructure}
    Let $\X$ be a smooth stack. Then \[K^{\mathrm{top}}_{\ast}(\mathrm{MF}(\X, 0))\cong K^{\mathrm{top}}_{\ast}(\X)\otimes_\mathbb{Z}\Lambda\] as $\Lambda$-modules.
    Moreover, if $\mathbb{M}\subset D^b(\X)$ is an admissible subcategory of $D^b(\X)$, there is an isomorphism of $\Lambda$-modules: \[K^{\mathrm{top}}_{\ast}(\mathrm{MF}(\mathbb{M}, 0))\cong K^{\mathrm{top}}_{\ast}(\mathbb{M})\otimes_\mathbb{Z}\Lambda.\]
\end{prop}

\begin{proof}
It suffices to prove the first isomorphism.
Let $\X_0$ be the derived zero locus of $0\colon \X\to\mathbb{C}$. 
%Let $r=\mathrm{Spec}\,\mathbb{C}[\epsilon]$.
By the long exact sequence \eqref{LESKtheory}, it suffices to show that the map $\alpha'\colon \mathrm{Perf}(\X_0)\to D^b(\X_0)$ induces the zero map:
\[\alpha'\colon K^{\mathrm{top}}_\bullet(\X_0)\to G^{\mathrm{top}}_\bullet(\X_0),\]
which we showed in Lemma~\ref{prop4zero}.
% Consider the natural projection $\pi\colon \X_0=\X\times r\to \X$ and let $l\colon \X_0^{\mathrm{cl}}\cong \X\to \X_0$. Then $\pi^*\colon K^{\mathrm{top}}_\bullet(\X)\xrightarrow{\sim}K^{\mathrm{top}}_\bullet(\X_0)$ and $l_*\colon K^{\mathrm{top}}_\bullet(\X_0)\xrightarrow{\sim}K^{\mathrm{top}}_\bullet(\X)$.
% For any topological vector bundle $E$ on $\X$,
% there is an exact triangle:
% \[\pi^*(E)\to l_*(E)\to l_*(E)[1]\to \pi^*(E)[1],\] thus $\left[\alpha(\pi^*(E))\right]=\left[l_*(E)\right]+\left[l_*(E)[1]\right]=0\in G^{\mathrm{top}}_\bullet(\X_0)$, so the conclusion follows.
\end{proof}

\begin{proof}[Proof of Proposition~\ref{prop:forg}]
    It is enough to prove \eqref{isolambdatheta}. 
    Recall the notation $\mathrm{pt}:=\Spec \mathbb{C}$ and $T:=\mathrm{Spec}\,\Lambda_{\mathbb{C}}$ and the Koszul equivalence \eqref{Koszul}. Using \cite[Proposition 3.24]{HLP}
    (also see~\cite[Proposition~3.9]{T4} and note that 
    $\mathrm{MF}(\mathrm{pt}, 0)\simeq D^b(T)/\mathrm{Perf}(T)$), the equivalence \eqref{Koszul} induces an equivalence:
    \[\kappa'\colon D^b(\mathscr{P})\otimes_{D^b(\mathrm{pt})}\mathrm{MF}(\mathrm{pt}, 0)\xrightarrow{\sim} \mathrm{MF}(\mathscr{E}^\vee, f).\]
       Let $\mathscr{P}_0:=\mathscr{P}\times T$ and let $\pi\colon \mathscr{P}_0\to\mathscr{P}$ and $t\colon T\to \mathrm{pt}$ be the natural projections. We have that $\mathrm{MF}(\mathrm{pt}, 0)\cong D^b(T)/t^*(D^b(\mathrm{pt}))$. Then 
       there is an equivalence
    \[D^b(\mathscr{P})\otimes_{D^b(\mathrm{pt})}\mathrm{MF}(\mathrm{pt}, 0)\simeq D^b(\mathscr{P}_0)/\pi^*(D^b(\mathscr{P})).\] It suffices to show that the map 
    \[\pi^*\colon G^{\mathrm{top}}_i(\mathscr{P})\to G^{\mathrm{top}}_i(\mathscr{P}_0)\stackrel{\cong}{\to} G^{\mathrm{top}}_i(\mathscr{P})\] is zero, which follows as in the proof of Lemma~\ref{prop4zero}. 
\end{proof}

\subsection{Proof of Lemma~\ref{lem:openimm}}\label{subsec:motivic}
\begin{proof}
	We use the construction of $\mathcal{K}_M^{\rm{top}}$ via motivic stable homotopy theory, 
	see~\cite[Proposition~7.7]{Moulinos}: 
	\begin{align*}
		\mathcal{K}_M^{\rm{top}}(\mathscr{D})=\mathrm{Betti}_M(KGL_M(\mathscr{D})). 
	\end{align*}
	Here $\mathrm{Betti}_M$ is the Betti realization of the stable 
	motivic category 
	\begin{align*}
		\mathrm{Betti}_M \colon \mathrm{SH}(M) \to \mathrm{Sh}_{\mathrm{Sp}}(M^{\rm{an}})
	\end{align*}
	and $KGL_M(\mathscr{D}) \in \mathrm{SH}(M)$ is the constant spectrum 
	associated to the homotopy K-theory 
	$\underline{KH}_{M}(\mathscr{D}):=L_{\mathbb{A}^1}(\underline{K}_M(\mathscr{D}))$, 
	see~\cite[Section~7.3]{Moulinos} for details of the notation. 
	
	We have the following diagram 
	\begin{align*}
		\xymatrix{
			\mathrm{SH}(U) \ar[r]^-{\mathrm{Betti}_U} \ar[d]_-{j_{\ast}} & \mathrm{Sh}_{\mathrm{Sp}}(U^{\rm{an}}) \ar[d]_-{j_{\ast}} \\
			\mathrm{SH}(M) \ar[r]_-{\mathrm{Betti}_M} & \mathrm{Sh}_{\mathrm{Sp}}(M^{\rm{an}}).         
		}
	\end{align*}
	The above diagram may not be commutative as $j$ is not proper. 
	However, there is a natural morphism 
	\begin{align*}
		\mathrm{Betti}_M \circ j_{\ast} \to j_{\ast} \circ \mathrm{Betti}_U
	\end{align*}
	which is an equivalence when applied to $\Lambda$-constructible objects
	for a fixed class of objects $\Lambda \subset \mathrm{SH}(\mathrm{pt})$, where $\mathrm{pt}=\Spec \mathbb{C}$,
	see~\cite[Corollary~3.17]{Ayoub}.
	In particular, we apply it for the object
	\begin{align*}KGL_U(\mathrm{Perf}(U))=\phi^{\ast}KGL_{\mathrm{pt}}(\mathrm{Perf}(\mathrm{pt}))
	\end{align*}
	for the structure morphism $\phi \colon U \to \mathrm{pt}$ to obtain the equivalence 
	\begin{align*}
		\mathrm{Betti}_M \circ j_{\ast}KGL_{U}(\mathrm{Perf}(U)) \stackrel{\sim}{\to} 
		j_{\ast}K_U^{\rm{top}}(\mathrm{Perf}(U)). 
	\end{align*}
	On the other hand, we have 
	\begin{align*}
		j_{\ast}KGL_U(\mathrm{Perf}(U)) \simeq KGL_M(\mathrm{Perf}(U)). 
	\end{align*}
	The above equivalence follows from the construction of $KGL_M$ and noting 
	that, for a smooth morphism $Y \to M$, we have 
	\begin{align*}
		\underline{K}_M(\mathrm{Perf}(U))(Y) &=K(\mathrm{Perf}(U) \otimes_{\mathcal{O}_M}
		\mathrm{Perf}(Y)) \\
		&=K(\mathrm{Perf}(Y_U)) \\
		&=\underline{K}_U(\mathrm{Perf}(U))(Y_U)
	\end{align*}
	where $Y_U=Y\times_M U$. Therefore we obtain the desired equivalence (\ref{equiv:open}). 
\end{proof}

\subsection{Base-change and semiorthogonal decompositions}\label{basechange}

In this paper, we need to construct semiorthogonal decompositions for \'etale maps (or open immersions) of moduli of representations of a quiver. We use the following base-change result for semiorthogonal decompositions, see~\cite{MR2801403} for the case of derived categories of varieties.
 
For a pretriangulated dg-category $\mathscr{D}$ and a subcategory $\mathscr{C} \subset \mathscr{D}$, we say 
that $\mathscr{D}$ \textit{is classically generated by} $\mathscr{C}$ if the smallest pretriangulated subcategory of $\mathscr{D}$ 
which contains $\mathscr{C}$ and is closed under direct summands is $\mathscr{D}$. 

\begin{prop}\label{propSODetale}
Let $\X$ be a QCA (quasi-compact with affine stabilizers) derived stack with a morphism 
$\pi \colon \X \to S$ to a scheme $S$. Let \[D^b(\X)=\langle \mathscr{C}_i \mid i \in I\rangle\]
be a $S$-linear semiorthogonal decomposition. 
Then, for any \'etale map $f \colon T \to S$ and $f_T \colon \X_T=\X \times_S T \to \X$, 
there is a semiorthogonal decomposition 
\[D^b(\X_T)=\langle \mathscr{C}_{i, T} \mid i \in I\rangle,\] 
where $\mathscr{C}_{i, T} \subset D^b(\X_T)$ is the subcategory classically 
generated by $f_T^{\ast}\mathscr{C}_i$. 
\end{prop}
\begin{proof}
    The image of $f_T^{\ast} \colon \Ind D^b(\X) \to \Ind D^b(\X_T)$ 
    classically generates $\Ind D^b(\X_T)$, as 
    any $A \in \Ind D^b(\X_T)$ is a direct summand of $f_T^{\ast}f_{T\ast}A$. Indeed, consider the diagram:
    \begin{equation*}
        \begin{tikzcd}
           \X':=\X_T\times_{\X}\X_T\arrow[d, "g_T"]\arrow[r, "g_T"]& \X_T\arrow[d, "f_T"]\\
            \X_T\arrow[r, "f_T"]&\X.
        \end{tikzcd}
    \end{equation*}
    Then $f_T^*f_{T*}A=g_{T*}g_T^*A=A\otimes g_{T*}\mathcal{O}_{\X'}$. The map $g_T$ has a section given by the diagonal map $\Delta\colon \X_T\to \X'$, thus $g_{T*}\mathcal{O}_{\X'}$ has $\mathcal{O}_{\X_T}$ as a direct summand, and so $A$ is indeed a direct summand of $f_T^{\ast}f_{T\ast}A$.
    
    By the QCA assumption, objects in $D^b(\X_T) \subset \Ind D^b(\X_T)$ are compact, 
    see~\cite{MR3037900}. 
    Therefore $D^b(\X_T)$ is classically generated by 
    $f_T^{\ast} D^b(\X)$, thus by $\mathscr{C}_{i, T}$ for $i\in I$. 
    
    To show semiorthogonality, 
    consider $i,j\in I$ such that $\Hom(A_i, A_j)=0$ for all $A_i\in \mathscr{C}_i$ and $A_j\in \mathscr{C}_j$. 
We have 
\begin{align}\notag
    \Hom_{D^b(\X_T)}(f_T^{\ast}A_i, f_T^{\ast}A_j)&=\Hom_{\Ind D^b(\X)}(A_i, f_{T\ast}f_T^{\ast}A_j) \\
 \label{hom:tensor}   &=\Hom_{\Ind D^b(\X)}(A_i, A_j \otimes_{\mathcal{O}_S}f_{\ast}\mathcal{O}_T). 
\end{align}
    Here $f_{\ast}\mathcal{O}_T \in D_{\rm{qc}}(S)=\Ind \mathrm{Perf}(S)$, 
    and the $S$-linearity of $\mathscr{C}_j$ implies 
    $A_j \otimes_{\mathcal{O}_S} f_{\ast}\mathcal{O}_T \in \Ind \mathscr{C}_j$. 
    Then the vanishing of (\ref{hom:tensor}) follows from 
    the compactness of $A_i$ (see the end of the proof of \cite[Lemma 5.5]{PTK3} for how compactness is used). 
\end{proof}

\appendix

\section{Topological Grothendieck Riemann-Roch theorem via motivic stable homotopy theory}\label{subsec:append}
In the appendix, we derive the topological Riemann-Roch formulas which we need in this paper from the existing references of motivic stable homotopy theory~\cite{DJK, Deg2, Ayoub}. 
\subsection{Bivariant theories}
For a quasi-compact and quasi-separated scheme $S$ over $\mathbb{C}$, 
we denote by $\mathrm{SH}(S)$ the stable $\infty$-category of motivic 
spectra~\cite{KhanPhD}. Its homotopy category is equivalent to the stable 
$\mathbb{A}^1$-homotopy category by Voevodsky~\cite{Voev}. 
We refer to~\cite[Section~2]{DJK} for notations. 

Let $\mathbb{E}\in \mathrm{SH}(S)$ be a motivic ring spectrum. 
For any $p \colon X \to S$ which is an $s$-morphism (separated morphism of finite type), and $v\in K(X)$, the \textit{$v$-twisted bivariant spectrum} of $X$ over $S$ is 
defined to be the mapping spectrum 
\begin{align*}
    \mathbb{E}(X/S, v):=\mathrm{Map}(\mathbb{S}_S, p_{*}p^!(\mathbb{E}_S\otimes \mathrm{Th}_X(-v))).
\end{align*}
Here $\mathrm{Th}_X(-v)$ is the Thom spectrum associated with 
the virtual vector bundle, see~\cite[Section~2.1.5]{DJK}. 
For a morphism $f\colon X \to Y$ over $S$, there is a 
natural product, see~\cite[Section~2]{DJK}:
\begin{align*}
    \mathbb{E}(X/Y, v) \times \mathbb{E}(Y/S, w) \to 
    \mathbb{E}(X/S, v+f^* w).
\end{align*}

We write $r:=[\mathcal{O}_X^{\oplus r}] \in K(X).$
The particular classes of interest are 
Borel-Moore/cohomology theories
\begin{align*}
\mathbb{E}^{\mathrm{BM}}(X, r):=\mathbb{E}(X/S, r), \
\mathbb{E}(X, r):=\mathbb{E}(X/X, r).
\end{align*}
The spectrum $\mathbb{E}(X):=\mathbb{E}(X, 0)$ is a ring spectrum and $\mathbb{E}^{\mathrm{BM}}(X, r)$ is a module over it. 

\subsection{Induced maps}
For a $S$-morphism $f \colon X \to Y$, there is a 
natural pull-back 
\begin{align*}
    f^* \colon \mathbb{E}(Y, r) \to \mathbb{E}(X, r). 
\end{align*}
If $f$ is proper, there is a natural push-forward
\begin{align*}
    f_{*} \colon \mathbb{E}^{\mathrm{BM}}(X, r) \to \mathbb{E}^{\mathrm{BM}}(Y, r)
\end{align*}
given by the adjunction $f_{!} f^! \to \id$ and $f_!=f_*$ as $f$ is proper. 

Let $c$ be an orientation of $\mathbb{E}$ as in~\cite[Definition~4.4.1]{DJK}, which is the data of equivalences for each $S$-scheme $p\colon X \to S$ and $v\in K(X)$
\begin{align}\label{orient}
\tau_v^c \colon 
    \mathbb{E}(X/S, v) \simeq \mathbb{E}(X/S, r)
\end{align}
where $r$ is the rank $v$, 
which are functorial and respect the $E_{\infty}$-group structure on $K(X)$ up to a homotopy coherent 
system of compatibilities. 
For example, an $\mathrm{MGL}$-module structure on $\mathbb{E}$ gives rise to an orientation of $\mathbb{E}$, see~\cite[Example 4.4.2]{DJK}, \cite[Section~2.1.12]{Deg2}.
Then by~\cite[Theorem~3.3.2, Definition~4.1.3, Definition~4.1.4, Section~4.4.3]{DJK},
there exists a system of fundamental classes 
\begin{align}\label{fund:eta}
    \eta_f^{(\mathbb{E}, c)} \in \mathbb{E}(X/Y, d_f)
\end{align}
for smoothable lci $S$-morphisms $f\colon X \to Y$, 
satisfying expected properties, see~\cite[Theorem~3.3.2]{DJK}. 
Here $d_f$ is the virtual dimension of $f$.

The fundamental classes (\ref{fund:eta}) define, for a smoothable lci morphism $f$, a pull-back 
of Borel-Moore theories
\begin{align*}
    f_c^{*} \colon \mathbb{E}^{\mathrm{BM}}(Y, r) \to 
    \mathbb{E}^{\mathrm{BM}}(X, r+d_f), \ (-)\mapsto (-) \cdot \eta_f^{(\mathbb{E}, c)}.
\end{align*}
If $f$ is furthermore proper, we have the push-forward of cohomology theories
\begin{align*}
    f_{c*} \colon \mathbb{E}(X, r) \to \mathbb{E}(Y, r+d_f)
\end{align*}
given by 
\begin{align*}
    \mathbb{E}(X, r) \stackrel{\cdot \eta_f^{(\mathbb{E}, c)}}{\to} \mathbb{E}(X/Y, r+d_f) \stackrel{f_{*}}{\to} \mathbb{E}(Y/Y, r+d_f)=\mathbb{E}(Y, r+d_f).
\end{align*}

\subsection{Riemann-Roch theorem}
Let $\psi \colon \mathbb{E} \to \mathbb{F}$ be a map of motivic ring spectra. 
It naturally induces the maps 
\begin{align*}
    \psi \colon \mathbb{E}^{\mathrm{BM}}(X, r) \to \mathbb{F}^{\mathrm{BM}}(X, r), \  \psi \colon \mathbb{E}(X, r) \to \mathbb{F}(X, r)
\end{align*}
From the construction, for $f \colon X \to Y$,
there is a commutative diagram 
\begin{align}\label{com:E1}
    \xymatrix{
\mathbb{E}(Y, r) \ar[r]^-{f^*}\ar[d]_-{\psi}  & \mathbb{E}(X, r) \ar[d]^-{\psi} \\
\mathbb{F}(Y, r) \ar[r]^-{f^*} & \mathbb{F}(X, r). 
    }
\end{align}

If $f$ is furthermore proper, there is a commutative diagram 
\begin{align}\label{com:E2}
    \xymatrix{
\mathbb{E}^{\mathrm{BM}}(X, r) \ar[r]^-{f_*}\ar[d]_-{\psi}  & \mathbb{E}^{\mathrm{BM}}(Y, r) \ar[d]^-{\psi} \\
\mathbb{F}^{\mathrm{BM}}(X, r) \ar[r]^-{f_*} & \mathbb{F}^{\mathrm{BM}}(Y, r). 
    }
\end{align}

Let $c$ be an orientation of $\mathbb{E}$ and $d$ an orientation of $\mathbb{F}$. 
Then for a smoothable lci morphism $f\colon X \to Y$, by~\cite[Theorem~3.2.6]{Deg2} 
we have 
\begin{align}\label{Td}
\psi(\eta_{f}^{(\mathbb{E}, c)}) \simeq \mathrm{Td}_{\psi}(\mathbb{T}_f) \cdot \eta_f^{(\mathbb{F}, d)}.
\end{align}
Here $\mathrm{Td}_{\psi}(\mathbb{T}_f)\in \mathbb{F}(Y)$ is the Todd class~\cite[Definition~3.2.4]{Deg2}
of the tangent complex $\mathbb{T}_f$, which keeps track of the 
difference of the orientations $\psi(c)$ and $d$ of $\mathbb{F}$. 

We have the commutative diagram~\cite[Proposition~3.3.11]{Deg2}
\begin{align}\label{com:E3}
    \xymatrix{
\mathbb{E}^{\mathrm{BM}}(Y, r) \ar[rr]^-{f_c^*}\ar[d]_-{\psi} & & \mathbb{E}^{\mathrm{BM}}(X, r+d_f) \ar[d]^-{\psi} \\
\mathbb{F}^{\mathrm{BM}}(Y, r) \ar[rr]^-{\mathrm{Td}_{\psi}(\mathbb{T}_f)f_d^*} & & \mathbb{F}^{\mathrm{BM}}(X, r+d_f). 
    }
\end{align}
If $f$ is furthermore proper, there is a commutative diagram (the Grothendieck-Riemann-Roch theorem)~\cite[Proposition~3.3.11]{Deg2}
\begin{align}\label{com:E4}
    \xymatrix{
\mathbb{E}(X, r) \ar[rr]^-{f_{c*}}\ar[d]_-{\psi}  && \mathbb{E}(Y, r+d_f) \ar[d]^-{\psi} \\
\mathbb{F}(X, r) \ar[rr]^-{f_{d*}(\mathrm{Td}_{\psi}(\mathbb{T}_f \cdot ))} && \mathbb{F}(Y, r+d_f). 
    }
\end{align}
Note that in~\cite[Proposition~3.3.11]{Deg2} the commutative diagrams (\ref{com:E3}), (\ref{com:E4}) are stated 
for homotopy groups of spectra, but from the constructions they are lifted to 
commutative diagrams of spectra~\cite[Section~4.4.3]{DJK}, where the commutative structure is induced by a choice of a homotopy (\ref{Td}) and orientations (\ref{orient}). 

\begin{remark}\label{rmk:ch}
If $p_X \colon X \to S$ is lci, 
we can define $\psi^{\mathrm{BM}}$ to be 
\begin{align}\label{psiBM}
    \psi^{\mathrm{BM}} :=\mathrm{Td}_{\psi}(\mathbb{T}_{p_X})^{-1} \cdot \psi \colon \mathbb{E}^{\mathrm{BM}}(X) \to \mathbb{F}^{\mathrm{BM}}(X).
\end{align}
Then the commutative diagrams (\ref{com:E2}), (\ref{com:E3}) are written as 
the commutative diagrams:
\begin{align}\label{com:E5}
    &\xymatrix{
\mathbb{E}^{\mathrm{BM}}(X, r) \ar[rr]^-{f_*}\ar[d]_-{\psi^{\mathrm{BM}}}  && \mathbb{E}^{\mathrm{BM}}(Y, r) \ar[d]^-{\psi^{\mathrm{BM}}} \\
\mathbb{F}^{\mathrm{BM}}(X, r) \ar[rr]^-{f_{*}(\mathrm{Td}_{\psi}(\mathbb{T}_f \cdot ))} && \mathbb{F}^{\mathrm{BM}}(Y, r), 
    }\\
   \label{com:E55} &\xymatrix{
\mathbb{E}^{\mathrm{BM}}(Y, r) \ar[rr]^-{f_c^*}\ar[d]_-{\psi} & & \mathbb{E}^{\mathrm{BM}}(X, r+d_f) \ar[d]^-{\psi} \\
\mathbb{F}^{\mathrm{BM}}(Y, r) \ar[rr]^-{f_d^*} & & \mathbb{F}^{\mathrm{BM}}(X, r+d_f). 
    }
\end{align}    
\end{remark}
\subsection{Riemann-Roch theorem for the singularity spectrum}
For $p_X \colon X \to S$, by taking the product with 
$\eta_{p_X}^{(\mathbb{E}, c)}$, we obtain the map of spectra
\begin{align*}
    i_X^{(\mathbb{E}, c)}(-):=\eta_{p_X}^{(\mathbb{E}, c)} \cdot (-) \colon \mathbb{E}(X) \to \mathbb{E}^{\mathrm{BM}}(X, d_{p_X}).
\end{align*}
The \textit{singularity spectrum} is defined to be the cofiber 
\begin{align*}
    \mathbb{E}^{\mathrm{sg}}(X):=\mathrm{Cofib}(i_X^{(\mathbb{E}, c)}).
\end{align*}
Let $\psi \colon \mathbb{E} \to \mathbb{F}$ be a map of motivic ring spectra, 
with orientations $c$, $d$, respectively. 

Let $p_X \colon X \to S$ is lci, and consider the map (\ref{psiBM}). 
By the associativity of the fundamental classes in~\cite[Theorem~2.5.3]{Deg2},
we have the following commutative diagram 
\begin{align}\label{psi:E}
    \xymatrix{
    \mathbb{E}(X) \ar[r]^-{i_X^{(\mathbb{E}, c)}} \ar[d]_-{\psi} & \mathbb{E}^{\mathrm{BM}}(X, d_{p_X}) \ar[d]^-{\psi^{\mathrm{BM}}} \\
    \mathbb{F}(X) \ar[r]^-{i_X^{(\mathbb{F}, d)}} & \mathbb{F}^{\mathrm{BM}}(X, d_{p_X}).
    }
\end{align}
By taking the induced map on cofibers, we obtain 
\begin{align*}
    \psi^{\mathrm{sg}} \colon \mathbb{E}^{\mathrm{sg}}(X) \to \mathbb{F}^{\mathrm{sg}}(X).
\end{align*}

Let $f \colon X \to Y$ be a smoothable lci morphism. Then we have the commutative 
diagram 
\[
\begin{array}{c@{\qquad}c@{\qquad}c}
\vcenter{\hbox{
\xymatrix{
  \mathbb{E}(Y) \ar[r]^-{i_Y^{(\mathbb{E}, c)}} \ar[d]_-{\psi}
    & \mathbb{E}^{\mathrm{BM}}(Y, d_{p_Y}) \ar[d]^-{\psi^{\mathrm{BM}}} \\
  \mathbb{F}(Y) \ar[r]^-{i_Y^{(\mathbb{F}, d)}}
    & \mathbb{F}^{\mathrm{BM}}(Y, d_{p_Y})
}}}
&
\xrightarrow[(f^*,f_c^*)]{(f^*,f_c^*)}
&
\vcenter{\hbox{
\xymatrix{
  \mathbb{E}(X) \ar[r]^-{i_X^{(\mathbb{E}, c)}} \ar[d]_-{\psi}
    & \mathbb{E}^{\mathrm{BM}}(X, d_{p_X}) \ar[d]^-{\psi^{\mathrm{BM}}} \\
  \mathbb{F}(X) \ar[r]^-{i_X^{(\mathbb{F}, d)}}
    & \mathbb{F}^{\mathrm{BM}}(X, d_{p_X})
}}}
\end{array}
\]
By taking the induced maps on cofibers, we obtain the commutative diagram 
\begin{align}\label{com:Esg1}
    \xymatrix{
\mathbb{E}^{\mathrm{sg}}(Y) \ar[r]^-{f^{\mathrm{sg}*}}\ar[d]_-{\psi^{\mathrm{sg}}} & \mathbb{E}^{\mathrm{sg}}(X) \ar[d]^-{\psi^{\mathrm{sg}}} \\
\mathbb{F}^{\mathrm{sg}}(Y) \ar[r]^-{f^{\mathrm{sg}*}} & \mathbb{F}^{\mathrm{sg}}(X). 
    }
\end{align}
Here $f^{\mathrm{sg}*}$ is induced by $(f^*, f_c^*)$. 

If $f$ is furthermore proper, we have the commutative diagram 
\[
\begin{array}{c@{\qquad}c@{\qquad}c}
\vcenter{\hbox{
\xymatrix{
  \mathbb{E}(X) \ar[r]^-{i_X^{(\mathbb{E}, c)}} \ar[d]_-{\psi}
    & \mathbb{E}^{\mathrm{BM}}(X, d_{p_X}) \ar[d]^-{\psi^{\mathrm{BM}}} \\
  \mathbb{F}(X) \ar[r]^-{i_X^{(\mathbb{F}, d)}}
    & \mathbb{F}^{\mathrm{BM}}(X, d_{p_X})
}}}
&
\xrightarrow[\mathrm{Td}_{\psi}(\mathbb{T}_f)(f_{c*},f_{*})]{(f_{c*},f_{*})}
&
\vcenter{\hbox{
\xymatrix{
  \mathbb{E}(Y) \ar[r]^-{i_Y^{(\mathbb{E}, c)}} \ar[d]_-{\psi}
    & \mathbb{E}^{\mathrm{BM}}(Y, d_{p_Y}) \ar[d]^-{\psi^{\mathrm{BM}}} \\
  \mathbb{F}(Y) \ar[r]^-{i_Y^{(\mathbb{F}, d)}}
    & \mathbb{F}^{\mathrm{BM}}(Y, d_{p_Y})
}}}
\end{array}
\]
By taking the induced maps on cofibers, we obtain the commutative diagram 
\begin{align}\label{com:Esg2}
    \xymatrix{
\mathbb{E}^{\mathrm{sg}}(X) \ar[rr]^-{f^{\mathrm{sg}}_{*}}\ar[d]_-{\psi^{\mathrm{sg}}} & &\mathbb{E}^{\mathrm{sg}}(Y) \ar[d]^-{\psi^{\mathrm{sg}}} \\
\mathbb{F}^{\mathrm{sg}}(X) \ar[rr]^-{\mathrm{Td}_{\psi}(\mathbb{T}_f)f^{\mathrm{sg}}_{*}} & &\mathbb{F}^{\mathrm{sg}}(Y). 
    }
\end{align}
Here $f^{\mathrm{sg}}_{*}$ is induced by $(f_{c*}, f_{*})$. 

\subsection{Betti realization}
We now apply the Betti realization functor~\cite{Ayoub}
\begin{align*}
    \mathrm{Betti} \colon \mathrm{SH}(S) \to D(\mathrm{Sh}_{\mathbb{Q}}(S^{\mathrm{an}})).
\end{align*}
We set
\begin{align*}
    \varphi:=\mathrm{Betti}(\psi) \colon \mathcal{E}:=\mathrm{Betti}(\mathbb{E}) \to 
    \mathcal{F}:=\mathrm{Betti}(\mathbb{F}).
\end{align*}
We also define $\mathcal{E}^{\mathrm{BM}}$, $\mathcal{E}^{\mathrm{sg}}$ in a way similar to 
$\mathbb{E}^{\mathrm{BM}}$, $\mathbb{E}^{\mathrm{sg}}$. 
All the diagrams in the previous subsections also hold in the Betti setting. 
By replacing $(\mathbb{E}, \mathbb{E}^{\mathrm{BM}}, \mathbb{E}^{\mathrm{sg}})$, 
$(\mathbb{F}, \mathbb{F}^{\mathrm{BM}}, \mathbb{F}^{\mathrm{sg}})$
with $(\mathcal{E}, \mathcal{E}^{\mathrm{BM}}, \mathcal{E}^{\mathrm{sg}})$, 
$(\mathcal{F}, \mathcal{F}^{\mathrm{BM}}, \mathcal{F}^{\mathrm{sg}})$,  
we have the following topological Riemann-Roch formulas for $\star \in \{\emptyset, \mathrm{BM}, \mathrm{sg}\}$
\begin{align}\label{GRR:top}
   & \xymatrix{
\mathcal{E}^{\star}(Y) \ar[r]^-{f^{\star*}}\ar[d]_-{\varphi^{\star}} & \mathcal{E}^{\star}(X) \ar[d]^-{\varphi^{\star}} \\
\mathcal{F}^{\star}(Y) \ar[r]^-{f^{\star*}} & \mathcal{F}^{\star}(X),
    }
    &    \xymatrix{
\mathcal{E}^{\star}(X) \ar[rr]^-{f^{\star}_{*}}\ar[d]_-{\varphi^{\star}} & &\mathcal{E}^{\star}(Y) \ar[d]^-{\varphi^{\star}} \\
\mathcal{F}^{\star}(X) \ar[rr]^-{\mathrm{Td}_{\varphi}(\mathbb{T}_f)f^{\star}_{*}} & &\mathcal{F}^{\star}(Y). 
    }
\end{align}
In the first diagram, $f$ is smoothable lci for $\star\in \{\mathrm{BM}, \mathrm{sg}\}$, in the second diagram $f$ is smoothable lci and proper, and 
$X, Y$ are lci for $\star \in \{\mathrm{BM}, \mathrm{sg}\}$. 

In particular, we set $S=\mathrm{pt}$, and consider the motivic Chern character
\begin{align*}
    \psi=\ch \colon \mathbb{E}=\mathrm{KGL} \to \mathbb{F}=\bigoplus_{n\in \mathbb{Z}}H\mathbb{Q}(n)[2n].
\end{align*}
Note that both sides admit canonical orientations such that the 
associated Todd class equals the one which appears in the classical 
GRR formula, see~\cite[2.5.3.3]{Deg1}. 
Its topological realization is 
\begin{align*}
   \varphi= \ch \colon \mathcal{E}=KU \to \mathcal{F}=\mathbb{Q}[\beta^{\pm 1}],
\end{align*}
where $\deg \beta=2$. We note that 
$\mathcal{E}^{\mathrm{BM}}(X, r)$ is independent of $r$ by Bott periodicity, 
and write it as $\mathcal{E}^{\mathrm{BM}}(X)$. 
We have the induced Chern character map 
\begin{align}\label{ch:Dsg}
    \varphi^{\star}=\ch \colon \mathcal{E}^{\star}(X) \to \mathcal{F}^{\star}(X).
\end{align}
Below we investigate the above map. 

\begin{lemma}\label{lem:EBM}
For a smoothable lci scheme $X$, there are natural equivalences 
\begin{align*}
    \mathcal{E}(X) \simeq K^{\mathrm{top}}(X), \ \mathcal{E}^{\mathrm{BM}}(X) \simeq G^{\mathrm{top}}(X).
\end{align*}
    
\end{lemma}
\begin{proof}
    The first equivalence follows from 
    $\mathcal{E}(X)=p_{*}\underline{KU}_X=K^{\mathrm{top}}(X)$ 
    by the definition of $\mathcal{E}(X)$. 

    As for the second equivalence, let $X\subset A$ be a regular embedding where $A$ is smooth
    with $d=\dim A$, 
    and $U=A\setminus X$. 
    We have the commutative diagram 
    \begin{align*}
        \xymatrix{
\mathbb{E}^{\mathrm{BM}}(X, d) \ar[r] & \mathbb{E}^{\mathrm{BM}}(A, d) \ar[r] & \mathbb{E}^{\mathrm{BM}}(U, d) \\
&  \mathbb{E}(A) \ar[r] \ar[u]_-{\sim}^-{\eta_A^{(\mathbb{E}, c)}} & \mathbb{E}(U). \ar[u]_-{\sim}^-{\eta_U^{(\mathbb{E}, c)}} 
        }
    \end{align*}
    Here the top horizontal sequence 
    is a fiber sequence by the Morel-Voevodsky localization property~\cite[Proposition~2.2.10]{DJK}, and 
    vertical arrows are equivalences since $A, U$ are smooth, see~\cite[Example~2.3.13]{DJK}.
    By applying the Betti realization, using the first equivalence and Bott periodicity, we have 
    the commutative diagram, where each horizontal sequence is a fiber sequence
     \begin{align*}
        \xymatrix{
\mathcal{E}^{\mathrm{BM}}(X) \ar[r] & \mathcal{E}^{\mathrm{BM}}(A) \ar[r] & \mathcal{E}^{\mathrm{BM}}(U) \\
G^{\mathrm{top}}(X) \ar[r] &  K^{\mathrm{top}}(A) \ar[r] \ar[u]_-{\sim} & K^{\mathrm{top}}(U). \ar[u]_-{\sim} 
        }
    \end{align*}
    Therefore we obtain the equivalence $G^{\mathrm{top}}(X) \simeq \mathcal{E}^{\mathrm{BM}}(X)$
    by taking the induced map on fibers. An independence of a choice 
    of $X\subset A$ can be proved along with the same argument of Lemma~\ref{lem:rho}. 
    \end{proof}

    We obtain a commutative diagram of fiber sequences of spectra 
    \begin{align}\label{com:fiber}
    \xymatrix{
K^{\mathrm{top}}(\mathrm{Perf}(X)) \ar[r] \ar[d]_-{\sim} & K^{\mathrm{top}}(\mathrm{Coh}(X)) \ar[r] \ar[d]_-{\sim}& K^{\mathrm{top}}(D_{\mathrm{sg}}(X))  \\
K^{\mathrm{top}}(X) \ar[r] \ar[d]_-{\sim} & G^{\mathrm{top}}(X) \ar[r] \ar[d]_-{\sim} & G^{\mathrm{top}}(X)/K^{\mathrm{top}}(X) \\
\mathcal{E}(X) \ar[r] & \mathcal{E}^{\mathrm{BM}}(X) \ar[r] & \mathcal{E}^{\mathrm{sg}}(X).
    }
    \end{align}
    By taking the induced maps on cofibers, we obtain equivalences 
    \begin{align*}
        K^{\mathrm{top}}(D_{\mathrm{sg}}(X)) \simeq G^{\mathrm{top}}(X)/K^{\mathrm{top}}(X) \simeq 
        \mathcal{E}^{\mathrm{sg}}(X).
    \end{align*}
    Similarly for the spectrum $\mathcal{F}$, we have 
    equivalences 
    \begin{align*}
        \mathcal{F}(X) \simeq \Gamma(\mathbb{Q}_X)[\beta^{\pm 1}], \ 
        \mathcal{F}^{\mathrm{BM}}(X) \simeq \Gamma(\omega_X)[\beta^{\pm 1}].
    \end{align*}
    We define 
    \begin{align}\label{omega:sg}
        \omega_X^{\mathrm{sg}}:=\mathrm{Cofib}(\mathbb{Q}_X \stackrel{\alpha}{\to} \omega_X[-2\dim X])
    \end{align}
    where $\alpha$ is given by capping with the fundamental class of smoothable lci scheme $X$. Then we have 
    an equivalence 
    \begin{align*}
        \mathcal{F}^{\mathrm{sg}}(X) \simeq \Gamma(\omega_X^{\mathrm{sg}})[\beta^{\pm 1}].
    \end{align*}
By the Betti realization of (\ref{psi:E}), we obtain the Chern character map of 
fiber sequences of spectra
\begin{align}\label{GRR:KG}
    \xymatrix{
K^{\mathrm{top}}(X) \ar[r] \ar[d]^-{\ch} & G^{\mathrm{top}}(X) \ar[r] \ar[d]^-{\ch} & K^{\mathrm{top}}(D_{\mathrm{sg}}(X)) \ar[d]^-{\ch} \\
\Gamma(\mathbb{Q}_X)[\beta^{\pm 1}] \ar[r] & \Gamma(\omega_X)[\beta^{\pm 1}] \ar[r] & 
\Gamma(\omega_X^{\mathrm{sg}})[\beta^{\pm 1}].
    }
\end{align}
Moreover each vertical arrow satisfies the GRR transformation formulas (\ref{GRR:top}); 
\begin{align}\label{GRR1}
  &  \xymatrix{
K^{\mathrm{top}}(X) \ar[r]^-{f_{*}} \ar[d]_-{\ch} & K^{\mathrm{top}}(Y) \ar[d]^-{\ch} \\
\Gamma(\mathbb{Q}_{X})[\beta^{\pm 1}] \ar[r]^{\mathrm{Td}_{f}f_{*}} & 
\Gamma(\mathbb{Q}_{Y})[\beta^{\pm 1}],
    } \
   \xymatrix{
K^{\mathrm{top}}(Y) \ar[r]^-{f^{*}} \ar[d]_-{\ch} & K^{\mathrm{top}}(X) \ar[d]^-{\ch} \\
\Gamma(\mathbb{Q}_{Y})[\beta^{\pm 1}] \ar[r]^{f^{*}} & 
\Gamma(\mathbb{Q}_{X})[\beta^{\pm 1}]
    }\\
\label{GRR2} & \xymatrix{
G^{\mathrm{top}}(X) \ar[r]^-{f_{*}} \ar[d]_-{\ch} & G^{\mathrm{top}}(Y) \ar[d]^-{\ch} \\
\Gamma(\omega_{X})[\beta^{\pm 1}] \ar[r]^{\mathrm{Td}_{f}f_{*}} & 
\Gamma(\omega_{Y})[\beta^{\pm 1}],
    } \
   \xymatrix{
G^{\mathrm{top}}(Y) \ar[r]^-{f^{*}} \ar[d]_-{\ch} & G^{\mathrm{top}}(X) \ar[d]^-{\ch} \\
\Gamma(\omega_{Y})[\beta^{\pm 1}] \ar[r]^{f^{*}} & 
\Gamma(\omega_{X})[\beta^{\pm 1}]
    }\\
  \label{GRR3} &  \xymatrix{
K^{\mathrm{top}}(D_{\mathrm{sg}}(X)) \ar[r]^-{f_{*}} \ar[d]_-{\ch} & K^{\mathrm{top}}(D_{\mathrm{sg}}(Y)) \ar[d]^-{\ch} \\
\Gamma(\omega_{X}^{\mathrm{sg}})[\beta^{\pm 1}] \ar[r]^{\mathrm{Td}_{f}f_{*}} & 
\Gamma(\omega_{Y}^{\mathrm{sg}})[\beta^{\pm 1}],
    } \
   \xymatrix{
K^{\mathrm{top}}(D_{\mathrm{sg}}(Y)) \ar[r]^-{f^{*}} \ar[d]_-{\ch} & K^{\mathrm{top}}(D_{\mathrm{sg}}(X)) \ar[d]^-{\ch} \\
\Gamma(\omega_{Y}^{\mathrm{sg}})[\beta^{\pm 1}] \ar[r]^{f^{*}} & 
\Gamma(\omega_{X}^{\mathrm{sg}})[\beta^{\pm 1}].
    }
\end{align}
Here $f$ is proper and smoothable lci in the left diagrams, 
$X, Y$ are lci in the left (\ref{GRR2}), (\ref{GRR3}), $f$ is smoothable lci in the right (\ref{GRR2}), (\ref{GRR3}). 

\begin{remark}\label{rmk:GRR}
For $\mathcal{E}=KU$ the push-forward $f_{*}$ in the diagrams 
(\ref{GRR1}) is induced by the fundamental class of $f$ with respect 
to the canonical orientation of $KU$. 
On the other hand, for a proper and smoothable lci morphism $f$, we have 
\[f_{*} \colon \mathrm{Perf}(X) \to \mathrm{Perf}(Y),\] 
which induces a map $K^{\mathrm{top}}(X) \to K^{\mathrm{top}}(Y)$. 
Indeed, this map agrees with the pushforward induced by the fundamental class, see~\cite[Remark~2.5.3.7]{Deg1}. 
\end{remark}
\subsection{Quasi-smooth case}    
We also use a version of the GRR formula in the quasi-smooth case. 
Suppose that $X$ is a derived scheme such that $p_X \colon X \to S$
is quasi-smooth. We set 
\begin{align*}
\mathbb{E}^{\mathrm{BM}}(X, r):=\mathbb{E}^{\mathrm{BM}}(X^{\mathrm{cl}}, r). 
\end{align*}
For $\psi\colon \mathbb{E} \to \mathbb{F}$, we define 
\begin{align*}
    \psi^{\mathrm{BM}}:=\psi \cdot \mathrm{Td}_{\psi}^{-1}(\mathbb{T}_{p_X}) \colon \mathbb{E}^{\mathrm{BM}}(X, r)
    \to \mathbb{F}^{\mathrm{BM}}(X, r).
\end{align*}
    Let $p_Y \colon Y \to S$ be also quasi-smooth and 
    $f\colon X \to Y$ is a proper map. 
    Then we have the commutative diagram (\ref{com:E5})
    by applying it for $f^{\mathrm{cl}} \colon X^{\mathrm{cl}} \to Y^{\mathrm{cl}}$. Similarly if $f$ is smooth, then 
    $f^{\mathrm{cl}} \colon X^{\mathrm{cl}} \to Y^{\mathrm{cl}}$ is 
    also smooth, and we have the commutative diagram (\ref{com:E55}). 

    Let $G^{\mathrm{top}}(X):=G^{\mathrm{top}}(X^{\mathrm{cl}})$
    and $\omega_X=\omega_{X^{\mathrm{cl}}}$. 
    By applying the Betti realization, we obtain the commutative diagrams 
    \begin{align}\label{com:Gtop1}
    \xymatrix{
G^{\mathrm{top}}(X) \ar[r]^-{f_{*}} \ar[d]_-{\ch} & G^{\mathrm{top}}(Y) \ar[d]^-{\ch} \\
\Gamma(\omega_{X})[\beta^{\pm 1}] \ar[r]^{\mathrm{Td}_{f}f_{*}} & 
\Gamma(\omega_{Y})[\beta^{\pm 1}],
    } \
   \xymatrix{
G^{\mathrm{top}}(Y) \ar[r]^-{f^{*}} \ar[d]_-{\ch} & G^{\mathrm{top}}(X) \ar[d]^-{\ch} \\
\Gamma(\omega_{Y})[\beta^{\pm 1}] \ar[r]^{f^{*}} & 
\Gamma(\omega_{X})[\beta^{\pm 1}]. 
    }   
    \end{align}
    Here $f$ is proper and quasi-smooth in the first diagram and $f$ is smooth in the 
    second diagram. In the first case, $\mathrm{Td}_f$ is the Todd class of the tangent complex $\mathbb{T}_f$, and $f_{*}$ is induced by 
    $f_{*}\omega_X=f_{!} f^! \omega_Y \to \omega_Y$, and in the second case 
    $f^*$ is induced by $f^* \omega_Y=\omega_X[-2 \dim f]$.

 \bibliographystyle{amsalpha}
\bibliography{math}
\medskip

\textsc{\small Tudor P\u adurariu: Sorbonne Université and Université Paris Cité, CNRS, IMJ-PRG, F-75005 Paris, France.}\\
\textit{\small E-mail address:} \texttt{\small padurariu@imj-prg.fr}\\

\textsc{\small Yukinobu Toda: Kavli Institute for the Physics and Mathematics of the Universe (WPI), University of Tokyo, 5-1-5 Kashiwanoha, Kashiwa, 277-8583, Japan.}\\
\textit{\small E-mail address:} \texttt{\small yukinobu.toda@ipmu.jp}\\

 \end{document}